\documentclass[11pt, reqno ]{amsart} 

\usepackage[utf8]{inputenc}
\usepackage[margin=1in]{geometry}
\usepackage[T1]{fontenc}
\usepackage[colorlinks=true, linkcolor=blue, citecolor=red, urlcolor=blue]{hyperref}
\usepackage{amsmath, amssymb, amsthm}
\usepackage{graphicx}
\usepackage{enumitem}
\usepackage{tikz}
\usepackage{mathtools}
\usepackage{xcolor}
\usepackage[normalem]{ulem}

\newcommand{\act}[1]{\,\sp{#1}\!}
\newcommand{\del}{\partial}
\newcommand{\red}[1]{\textcolor{red}{#1}}
\newcommand{\obj}{{\mathop{\mathrm{Obj}}}}
\newcommand{\mor}{{\mathop{\mathrm{Mor}}}}
\newcommand{\Hom}{{\mathop{\mathrm{Hom}}}}
\newcommand{\id}{{\mathop{\mathrm{id}}}}

\newcommand{\TXMod}{{\bf X_3Mod}}
\newcommand{\Gray}{{\bf Gray4Group}}

\newcommand{\bG}{{\mathbb{G}}}

\newcommand{\cC}{{\mathcal{C}}}
\newcommand{\cD}{{\mathcal{D}}}

\newcommand{\blue}[1]{\textcolor{blue}{#1}}
\newcommand{\green}[1]{\textcolor{green!50!black}{#1}}

\newcommand{\strrect}[3][]{
  \begin{scope}[every node/.append style={overlay},#1]
    \foreach \x [count=\i] in {#3}{
      \ifx\x\empty
      \else
        \pgfmathsetmacro{\modResult}{mod(\i-1, #2)}
        \pgfmathsetmacro{\quotResult}{(\i-\modResult-1)/#2}
        \node at (\modResult,-\quotResult) {\x};
      \fi
    }
    \foreach \x in {0,1,...,#2}{
      \draw ({-1/2+\x},{1/2})--({-1/2+\x},{-#2+1/2});
      \draw ({-1/2},{1/2-\x})--({#2-1/2},{1/2-\x});
    }
  \end{scope}
}

\newcommand{\strsquare}[4][]{
  \begin{scope}[every node/.append style={overlay},#1]
    \foreach \x [count=\i] in {#4}{
      \ifx\x\empty
      \else
        \pgfmathsetmacro{\modResult}{mod(\i-1, #2)}
        \pgfmathsetmacro{\quotResult}{(\i-\modResult-1)/#2}
        \node at (\modResult,-\quotResult) {\x};
      \fi
    }
    \foreach \x in {0,1,...,#2}{
      \draw ({-1/2+\x},{1/2})--({-1/2+\x},{-#3+1/2});
    }
    \foreach \y in {0,1,...,#3}{
      \draw ({-1/2},{1/2-\y})--({#2-1/2},{1/2-\y});
    }
  \end{scope}
}

\setlist[enumerate]{label=(\roman*)}

\newtheorem{theorem}{Theorem}[section]
\newtheorem{lemma}[theorem]{Lemma}

\theoremstyle{definition}
\newtheorem{definition}[theorem]{Definition}

\theoremstyle{remark}
\newtheorem{remark}[theorem]{Remark}

\numberwithin{equation}{section}

\title{An Equivalence of Categories between 3-Crossed Modules and Gray 4-Groups}
\author{Masaki Fukuda and Tommy Shu}
\address{Department of Physics, Tohoku University\\
 \\
 Department of Mathematics, Tohoku
University\\67 \\}
\email{masaki.fukuda.r8@dc.tohoku.ac.jp,seiyo.shu.r7@dc.tohoku.ac.jp}
\date{\today}
\subjclass[2020]{Primary 18N20; Secondary 18G45,18B40}
\keywords{Gray category, crossed module, higher category theory, braided monoidal category, Zamolodchikov equation}
\thanks{We would like to express our sincere gratitude to Professor Yuji Terashima for his invaluable guidance, insightful comments, and constant support throughout this research.}

\begin{document}

\begin{flushright}
TU-1322 \\
\vspace{2em}
\end{flushright}

\begin{abstract}
  In this paper, we investigate the relation between the category of 3-crossed modules and the category of Gray-type 4-groups. The notion of a 3-crossed module was first introduced by Arvasi \textit{et al.}, motivated by the question of what kind of algebraic structure completely encodes a homotopy 4-type. On the other hand, from the point of view that higher groups are equivalent to algebraic realizations of higher categories -- as exemplified by the relationship between 2-crossed modules and Gray 3-groups established by Sarikaya--Ulualan -- it had not been clear how the 3-crossed modules of Arvasi \textit{et al.} relate to any higher category. In our previous paper, we proposed a new definition of a 3-crossed module and observed that it admits a natural interpretation in terms of higher categories. In this paper, we make this interpretation precise: we introduce a 4-category, which reduces to a semistrict braided monoidal 2-category when restricted to a single object and a single 1-morphism, and prove that the category of our 3-crossed modules is equivalent to the category of Gray 4-groups, defined as single-object versions of this 4-category in which all morphisms are invertible. Furthermore, while Gray 4-groups provide a conceptual framework for 4-dimensional higher structures, they are often computationally intractable. Our equivalence establishes 3-crossed modules as a concrete, group-theoretic calculus for Gray 4-groups, providing a powerful computational tool for studying surface knots, state-sum invariants, and higher gauge theories.
\end{abstract}

\maketitle

\tableofcontents
\clearpage

\section{Introduction} \label{sec:Introduction}
The study of crossed modules has developed along two main lines of motivation:
\begin{itemize}
  \item [A.] To realize a homotopy $n$-type algebraically, in terms of an $(n-1)$-crossed module.
  \item [B.] To realize an $(n-1)$-crossed module as an algebraic model for an $n$-category, or more specifically, for a single-object $n$-groupoid.
\end{itemize}
In the case $n=2$, Whitehead introduced the notion of a crossed module, motivated by problem A \cite{Whitehead:1949}. Subsequently, Brown--Spencer and Noohi showed that crossed modules are equivalent to strict 2-groups, thereby introducing problem B \cite{Brown:1976}. Similarly, in the case $n=3$, Conduch\'e introduced the notion of a 2-crossed module as an algebraic model for a homotopy 3-type, motivated by problem A \cite{Conduche:1984}, and Sarikaya--Ulualan proved an equivalence between the category of Gray 3-groups and the category of 2-crossed modules, motivated by problem B \cite{Sarikaya:2024}. Thus, for $n=2$ and $n=3$, answers to both problems A and B have been obtained, establishing a complementary pair of interpretations.

Beyond classification problems, crossed modules and 2-crossed modules have also found applications in the construction of topological invariants \cite{MartinsPorter2007,MartinsPicken2013arXiv,Radenkovic:2022qkd}. Extending such applications to higher-dimensional invariants, such as generalizations of the Dijkgraaf--Witten invariant and invariants of surface knots, naturally requires an extension to higher $n$. The study of 3-crossed modules by Arvasi \textit{et al.} was an attempt to answer problem A in the case $n=4$ \cite{Arvasi:2009}. Their work is motivated by the goal of realizing an algebraic model of the homotopy 4-type (i.e., problem A), via the Moore complex of a simplicial group, an algebraic device that completely encodes the homotopy type of a space. 

On the other hand, the problem from the perspective of B had not, until recently, been seriously investigated. The motivation of this study lies in problem B: we generalize the work of Sarikaya--Ulualan to the case $n=4$. In our previous work \cite{Fukuda:2025}, we pointed out that the connection between the definition of a 3-crossed module given in \cite{Arvasi:2009} and problem B is unclear. We proposed a new definition of 3-crossed modules and discussed its relation to quasi-categories and to the Moore complex. To ensure that this new definition constitutes a valid and universal higher-dimensional generalization of the crossed module, we must, as in the cases $n=2,3$, establish an equivalence of categories between the category of 3-crossed modules and an appropriate higher categorical counterpart, which we call the category of Gray 4-groups. In this paper, we prove this equivalence of categories, which was left as future work in \cite{Fukuda:2025}.

However, if we try to follow approach B, we naturally run into an independent problem: what kind of 4-category is the appropriate generalization of a Gray 3-category?\footnote{In this paper, we refer to the usual notion of a Gray category as a Gray 3-category in order to distinguish it from the Gray 4-categories introduced later.} \footnote{For details on this problem, we refer the reader to \cite{Crans:1999,Bourke:2023,Miranda:2024}.} As mentioned above, the category of Gray 3-groups is categorically equivalent to the category of 2-crossed modules. This fact suggests that the category equivalent to the category of 3-crossed modules should be the category of some 4-categories generalizing Gray 3-categories. Thus, in order to answer problem B, we need a definition of such 4-categories. A solution to this problem can be obtained by observing the relation between Gray 3-categories and strict braided monoidal categories: according to the Baez--Dolan stabilization hypothesis \cite{Baez:1995xq}, a strict braided monoidal category can be considered as a Gray category with a single object and a single 1-morphism (see, e.g., \cite{Crans:1998}; for the analogous statement for tricategories, which requires more care, see \cite{ChengGurski2011}). 

From this observation, we deduce that the desired 4-category, appropriately generalizing a Gray 3-category, should correspond in the same way to a semistrict braided monoidal 2-category, a notion originating in the work of Kapranov--Voevodsky \cite{Kapranov:1994}.\footnote{The definition of semistrict braided monoidal 2-categories, however, has a convoluted history, so we should clarify which definition we use. We adopt the final version of the definition arrived at by Crans \cite{Crans:1998}, and generalize it to our 4-categories. See also \cite{Baez:1996}.} Kapranov--Voevodsky and subsequent authors \cite{Baez:1996,Crans:1998} do not give an explicit definition of the corresponding 4-category, but they appear to have been aware that such a 4-category should exist. Guided by this observation, we construct an algebraic definition of a 4-category directly, in such a way that our 4-categories reduce to semistrict braided monoidal 2-categories when restricted to a single object and a single 1-morphism.

The main result of this paper is a proof of the categorical equivalence between the category of Gray 4-groups introduced above and the category of 3-crossed modules defined in \cite{Fukuda:2025}.

The rest of this paper is organized as follows. In Section \ref{sec:Pre}, we introduce the definitions of 2- and 3-crossed modules together with their diagrammatic representations. These diagrammatic representations are helpful not only for an intuitive understanding of the definitions, but also for the comparison with 4-categories carried out later. We then recall the definition of Gray 3-categories and their diagrammatic representation, in preparation for introducing Gray 4-categories, which we define at the end of the section.

In Section \ref{sec:main}, we present our main result. The proof proceeds as follows. In Section \ref{sec:Gray323CM}, we construct a functor from the category of Gray 4-groups to the category of 3-crossed modules. In Section \ref{3CM_to_Gray4}, we construct a functor in the opposite direction. In Section \ref{sec:thm}, we show that the composite of these two functors is naturally isomorphic to the identity functor. At the end of Section \ref{sec:thm}, we also comment on the relation between our Gray 4-categories and semistrict braided monoidal 2-categories.

Section \ref{sec:con} is devoted to discussion. There, we comment on the relation to the Zamolodchikov tetrahedron equation \cite{Zamolodchikov1980JETP} and 2-knot (surface-knot) invariants. 

\vspace{1em}
\noindent
\underline{Notation}: \\
In this paper, we use the following notation:
\begin{itemize}
  \item Since we do not address the weakening of monoidal structures in this paper, we drop the adjective ``semistrict'' and simply refer to semistrict $n$-categories as $n$-categories.
  \item We denote sets by uppercase letters and their elements by the corresponding lowercase letters. For example, when we consider a set $X$, its elements are denoted by $x,x',x_i$ and so on.
  \item For any sets $A_1,\ldots,A_n$ and maps $f_1,g_1,\ldots,f_{n-1},g_{n-1}$ between them, an iterated fiber product, i.e., the direct product eliminated by some conditions, is defined by
  \begin{align}
    \nonumber
    &A_1\times_{f_1,g_1}\ldots \times_{f_{n-1},g_{n-1}}A_n \\
    \nonumber
    &= \{(a_1,\ldots,a_n)\in A_1\times \ldots A_n|f_1(a_1)=g_1(a_2),\ldots,f_{n-1}(a_{n-1})=g_{n-1}(a_n)\}.
  \end{align}
  For our purposes, however, it suffices to specify the degree of the morphisms involved; for example,
  \[
    X\times_iY := X\times_{s_i,t_i}Y,
  \]
  where the maps $s_i$ and $t_i$ are source and target maps into $i$-morphisms, respectively, and $X$ and $Y$ are higher morphisms whose degrees are greater than $i$.
\end{itemize}

\section{Preliminaries}\label{sec:Pre}

\subsection{Crossed modules} \label{sec:CM}
In this section, we give the definition of 3-crossed modules. As preparation, we begin with the 2-crossed module and then generalize it to the 3-crossed module. Our definition of a 2-crossed module is similar to the refined definition given in \cite{Sarikaya:2024}.

\begin{definition}\label{def:2CM}
A \textbf{2-crossed module} is given by a complex of groups:
\[
  L\overset{\partial_2}{\longrightarrow} H\overset{\partial_1}{\longrightarrow} G
\]
together with the following three structures:
\begin{itemize}
  \item a left action of $G$ by automorphisms on $H$~:~$\triangleright^1_2$,
  \item a left action of $G$ by automorphisms on $L$~:~$\triangleright^1_3$,
  \item a left action of $H$ by automorphisms on $L$~:~$\triangleright^2_3$,
  \item a function called \textbf{Peiffer lifting} : $\{-,-\}~:~H\times H\to L$,
\end{itemize}
which are required to satisfy the following seven axioms. In what follows, we use some shorthand notation. Subscripts attached to the maps are omitted unless confusion would arise; the context clarifies which maps are used. We denote the action of $x$ on $y$ by $\act{x}y$. Moreover, we may omit the parentheses around the map $\partial$. The axioms are
\begin{enumerate}
  \item \label{ax:2CM-exact}For any element of $L$, the boundary of the boundary must be the identity element of the lowest group $G$. That is
  \[
    \partial_2\circ\partial_1(l) = e_G,\qquad \forall l\in L.
  \]
  \item \label{ax:2CM-BAction}Actions and boundary maps commute with each other. For each $g\in G,h\in H$ and $l\in L$, we have
  \begin{align}
    \nonumber
    \partial(\act{g}h) &= \act{g}(\partial h), \\
    \nonumber
    \partial(\act{g}l) &= \act{l}(\partial h), \\
    \nonumber
    \partial(\act{h}l) &= \act{h}(\partial l),
  \end{align}
  where the action of a group on itself is the adjoint action.
  \item \label{ax:2CM-GonPeiffer}The Peiffer lifting is invariant under the $G$ action. For each $g\in G$ and $h,h'\in H$, we have
  \[
    \act{g}\{h,h'\} = \{\act{g}h,\act{g}h'\}
  \]
  \item \label{ax:2CM-BoundaryPeiffer}The boundary of the Peiffer lifting is
  \begin{equation}
    \label{eq:2CM-bPeiffer}
    \partial\{h,h'\} = hh'h^{-1}\act{\partial h}h'^{-1}
  \end{equation}
  \item For any elements $l,l'$ of the third group $L$, the Peiffer identity holds:
  \begin{equation}
    \label{eq:2CM-Peifferid}
    \act{\partial l}l' = ll'l^{-1}
  \end{equation}
  \item For any elements $h_z,h_y$ and $h_x$ of the second group $H$, the Peiffer lifting satisfies
  \begin{align}
    \label{eq:HH_H-2CM}
    \{h_zh_y,h_x\} &= \act{h_z}\{h_y,h_x\}\{h_z,\act{\partial h_y}h_x\} \\
    \label{eq:H_HH-2CM}
    \{h_z,h_yh_x\} &= \{h_z,h_y\}\act{\act{\partial h_z}h_y}\{h_z,h_x\}
  \end{align}
  \item \label{ax:2CM-Hactions}For any element $l$ of the third group $L$ and any element $h$ of the second group $H$, the Peiffer lifting satisfies
  \begin{align}
    \label{eq:LH-2CM}
    \act{h}l &= l\{\partial l^{-1},h\} \\
    \label{eq:HL-2CM}
    \act{\partial h}l &= \act{h}l\{h,\partial l^{-1}\}
  \end{align}
\end{enumerate}
\end{definition}

We add a few comments on this definition. As mentioned at the beginning of this section, our definition is similar to the one introduced in \cite{Sarikaya:2024}, except that we do not include the condition $\{\partial l,\partial l'\} = ll'l^{-1}l'^{-1}$. Instead, we regard the condition \eqref{eq:LH-2CM} as an axiom in order to emphasize the role of the action of $H$ on $L$. This is more consistent with the categorical perspective and also with our generalization to 3-crossed modules. One can derive the condition $\{\partial l,\partial l'\} = ll'l^{-1}l'^{-1}$ from the condition \eqref{eq:LH-2CM}. 

There is another remark. For instance, the condition \eqref{eq:H_HH-2CM} appears to differ from the one given in many other references, such as \cite{Martins:2009evc}, but they are in fact the same. This is a consequence of the modification above. In our definition, since we already have the action of the second group $H$ on the third group $L$, we can express the condition \eqref{eq:H_HH-2CM} using this action. In addition, this formulation is more compatible with the diagrammatic representation introduced in Section \ref{sec:3CM-Diagram}.

It is convenient to rewrite Equations \eqref{eq:LH-2CM} and \eqref{eq:HL-2CM} in the following forms:
\begin{align}
  \label{eq:2CM-2-funct1}
  l\{h_2,h_1\} &= \{\partial_2(l)h_2,h_1\}\act{\act{\partial_1(h_2)}h_1}l \\
  \label{eq:2CM-2-funct2}
  \act{h_2}l\{h_2,h_1\} &= \{h_2,\partial_2(l)h_1\}\act{\partial_1(h_2)}l
\end{align}
These correspond to the 2-functoriality in the diagrammatic representation of Gray 3-categories, which will be introduced in Section \ref{sec:diag-Gray3}.

\begin{definition}\label{def:3CM}
A \textbf{3-crossed module} is given by a complex of groups:
\[
  M\overset{\partial_3}{\longrightarrow} L\overset{\partial_2}{\longrightarrow} H\overset{\partial_1}{\longrightarrow} G
\]
together with the following three structures:\vspace{1em} \\
\begin{minipage}{.49\columnwidth}
\begin{itemize}
  \item a left action of $G$ on $H$~:~$\triangleright^1_2$,
  \item a left action of $G$ on $L$~:~$\triangleright^1_3$,
  \item a left action of $G$ on $M$~:~$\triangleright^1_4$,
\end{itemize}
\end{minipage}
\begin{minipage}{.49\columnwidth}
  \begin{itemize}
    \item a left action of $H$ on $L$~:~$\triangleright^2_3$,
    \item a left action of $H$ on $M$~:~$\triangleright^2_4$,
    \item a left action of $L$ on $M$~:~$\triangleright^3_4$,
  \end{itemize}
\end{minipage}\vspace{1em} \\
All actions are given by automorphisms. There are six types of lifting.
\begin{itemize}
  \item a function called \textbf{22-Peiffer lifting}; $\{-,-\}_{22}~:~H\times H\to L$,
  \item a function called \textbf{32-Peiffer lifting}; $\{-,-\}_{32}~:~L\times H\to M$,
  \item a function called \textbf{23-Peiffer lifting}; $\{-,-\}_{23}~:~H\times L\to M$,
  \item a function called \textbf{33-Peiffer lifting}; $\{-,-\}_{33}~:~L\times L\to M$,
  \item a function called \textbf{left-Homanian}; $\{-,-|-\}~:~H\times H\times H\to M$,
  \item a function called \textbf{right-Homanian}; $\{-|-,-\}~:~H\times H\times H\to M$,
\end{itemize}
which are required to satisfy the following eight axioms. In what follows, we use the same shorthand notation as for the 2-crossed modules.
\begin{enumerate}
  \item The following subsequent together with suitable maps is a 2-crossed module;
  \begin{equation}
    \label{eq:sub-2CM}
    M\overset{\partial_3}{\longrightarrow} L\overset{\partial_2}{\longrightarrow} H,\quad \triangleright^2_3,\triangleright^2_4,\triangleright^3_4,\quad \{-,-\}_{33}.
  \end{equation}
  In addition to relations of this 2-crossed module, we also have the following.
  \item \label{ax:3CM-exact}For any element of $L$, the boundary of the boundary must be the identity element of the lowest group $G$.\footnote{The exactness for $\partial_2$ and $\partial_3$ holds from the axiom (i).} That is
  \[
    \partial_1\circ\partial_2(l) = e_G,\qquad \forall l\in L
  \]
  \item \label{ax:3CM-BAction}Actions and boundary maps commute with each other. For any $g\in G,h\in H,l\in L$ and $m\in M$, we have
  \begin{align}
    \nonumber
    \partial(\act{g}h) &= \act{g}(\partial h), &
    \partial(\act{g}l) &= \act{g}(\partial l), &
    \partial(\act{g}m) &= \act{g}(\partial m). 
  \end{align}
  where the action of a group on itself is given by the adjoint action.
  \item \label{ax:3CM-ActionLifting}All liftings are invariant under the $G$ action. For each $g\in G,h,h',h_x,h_y,h_z\in H$ and $l,l'\in L$, we have
  \begin{align}
    \nonumber
    \act{g}\{h,h'\} &= \{\act{g}h,\act{g}h'\}, &
    \act{g}\{l,h\} &= \{\act{g}l,\act{g}h\}, &
    \act{g}\{h_z,h_y|h_x\} &= \{\act{g}h_z,\act{g}h_y|\act{g}h_x\}, \\
    \nonumber
    \act{g}\{l,l'\} &= \{\act{g}l,\act{g}l'\}, &
    \act{g}\{h,l\} &= \{\act{g}h,\act{g}l\}, &
    \act{g}\{h_z|h_y,h_x\} &= \{\act{g}h_z|\act{g}h_y,\act{g}h_x\}.
  \end{align}
  \item \label{ax:3CM-BoundaryLiftings}The boundaries of the liftings are
  \begin{align}
    \nonumber
    \partial\{h,h'\} &= hh'h^{-1}\act{\partial h}h'^{-1}, \\
    \label{eq:3CM-bLH}
    \partial\{l,h\} &= l\{\partial l^{-1},h\}\act{h}l^{-1}, \\
    \label{eq:3CM-bHL}
    \partial\{h,l\} &= \act{h}l\{h,\partial l^{-1}\}\act{\partial h}l^{-1}, \\
    \nonumber
    \partial\{h_z,h_y|h_x\} &= \{h_zh_y,h_x\}\{h_z,\act{\partial h_y}h_x\}^{-1}\act{h_z}\{h_y,h_x\}^{-1}, \\
    \nonumber
    \partial\{h_z|h_y,h_x\} &= \{h_z,h_yh_x\}\act{\act{\partial h_z}h_y}\{h_z,h_x\}^{-1}\{h_z,h_y\}^{-1}.
  \end{align}
    \item \label{ax:3CM-Tetrahedrons}Tetrahedrons:
      \begin{align}
        \label{eq:3-CM-Fund-left-Hom}
        \{h_wh_z,h_y|h_x\}\act{\act{h_wh_z}\{h_y,h_x\}}\{h_w,h_z|\act{\partial h_y}h_x\} &= \{h_w,h_zh_y|h_x\}\act{h_w}\{h_z,h_y|h_x\} \\
        \label{eq:3-CM-Fund-right-Hom}
        \{h_w|h_z,h_yh_x\}\act{\{h_w,h_z\}}(\act{\act{\partial h_w}h_z}\{h_w|h_y,h_x\}) &= \{h_w|h_zh_y,h_x\}\{h_w|h_z,h_y\}
      \end{align}
    \item \label{ax:3CM-KV}Kapranov-Voevodsky's polytope (KV-polytope):
      \begin{align}
        \nonumber
        &\{h_w,h_z|h_yh_x\}\act{\act{h_w}\{h_z,h_yh_x\}}\{h_w|\act{\partial h_z}h_y,\act{\partial h_z}h_x\}\act{h_w}\{h_z|h_y,h_x\} \\
        \nonumber
        &= \{h_wh_z|h_y,h_x\}\act{\{h_wh_z,h_y\}}(\act{\act{\partial (h_wh_z)}h_y}\{h_w,h_z|h_x\})\{h_w,h_z|h_y\} \\
        \label{eq:3CM-KV}
        &\quad 
        \act{\act{h_w}\{h_z,h_y\}}\{\{h_w,\act{\partial h_z}h_y\},\act{\act{\partial(h_wh_z)}h_yh_w}\{h_z,h_x\}\}
      \end{align}
  \item Prisms: \label{ax:3CM-Prisms}
      \begin{align}
        \label{eq:3CM-L-HH}
        \{l,hh'\} &= \act{l}\{\partial l^{-1}|h,h'\}\{l,h\}\act{h}\{l,h'\} \\
        \label{eq:3CM-HH-L}
        \{hh',l\} &= \act{\act{hh'}l}\{h,h',\partial l^{-1}\}\act{h}\{h',l\}\{h,\act{\partial h'}l\},
      \end{align}
      \begin{align}
         \label{eq:3cm_prism23-2}
        &\act{h_2}\left<l,h_1\right>\act{\act{h_2}\{\partial l,h_1\}}\{\{h_2,h_1\},\act{\act{\partial h_2}h_1h_2}l\}^{-1}\{h_2,\partial l|h_1\}^{-1} \\
        \nonumber
        &= \{\act{h_2}l,\{h_2,h_1\}\}\act{\act{\partial\act{h_2}l}\{h_2,h_1\}}\left<\act{h_2}l,\act{\partial h_2}h_1\right>\{\partial\act{h_2}l,h_2|h_1\}^{-1},
      \end{align}
      where we have defined $\left<l,h\right> := \act{l}\{l^{-1},h\}^{-1}$.
      \begin{align}
                \label{eq:3cm_prism2-23}
        &\{\{h_2,h_1\},\act{\act{\partial h_2}h_1h_2}l\}^{-1}\act{\{h_2,h_1\}}(\act{\act{\partial h_2}h_1}\left<h_2,l\right>)\{h_2|h_1,\partial l\}^{-1} \\
        \nonumber
        &= \left<h_2,\act{h_1}l\right>\act{\{h_2,\partial\act{h_1}l\}}\{\act{\partial h_2}(\act{h_1}l),\{h_2,h_1\}\}\{h_2|\partial\act{h_1}l,h_1\}^{-1},
      \end{align}
      where we have defined $\left<h,l\right> := \act{\act{h}l}\{h,l^{-1}\}^{-1}$.
    \item \label{ax:3CM-Pasting}Pasting:
      \begin{align}
        \label{eq:3CM-Pasting}
        \{ll',h\} &= \act{ll'}\{\partial l'^{-1},\partial l^{-1}|h\}\act{ll'}\{l'^{-1},\{\partial l^{-1},h\}\}^{-1}\{l,h\}\act{\act{h}l}\{l',h\} \\
        \label{eq:3CM-Pasting2}
        \{h,ll'\} &= \act{\act{h}(ll')}\{h|\partial l'^{-1},\partial l^{-1}\}\act{\act{h}l}\{h,l'\}\act{\act{h}l\act{\partial h}l'}\{\act{\partial h}l'^{-1},\{h,\partial l^{-1}\}\}^{-1}\{h,l\}
      \end{align}
    \item \label{ax:3CM-Cube}Cube:
      \begin{equation}
        \label{eq:3CM-Cube}
        \act{l'^{-1}}(\{\partial l',l^{-1}\}\{l,'l^{-1}\}^{-1}) = \act{l^{-1}}(\{l,l'^{-1}\}^{-1}\{\partial l,l'^{-1}\}')
      \end{equation}
  \item \label{ax:3CM-2on4}For each $h\in H$ and $m\in M$, we have
  \begin{align}
    \label{eq:3CM-HonM}
    \act{h}m &= m\{\partial m^{-1},h\}, \\
    \label{eq:3CM-pHonM}
    \act{\partial h}m &= \act{h}m\{h,\partial m^{-1}\}
  \end{align}
  \item \label{ax:3CM-SpSm}$S^+=S^-$ condition:
  \begin{align*}
    &\{h_z,h_y|h_x\}^{-1}L(\{h_z,h_y\},\act{\partial h_z}h_yh_z;h_x)^{-1}\act{\{h_z,h_y\}}\{\act{\partial h_z}h_y,h_z|h_x\} \\
    &= \act{\act{h_z}\{h_y,h_x\}}\{h_z|\act{\partial h_y}h_x,h_y\}^{-1}R(h_z;\{h_y,h_x\},\act{\partial h_y}h_xh_y)\{h_z|h_y,h_x\}
  \end{align*}
  where $L$ and $R$ are
  \begin{align}
    \label{eq:defL}
    L(l,h_2;h_1) &= \{l,\{h_2,h_1\}\}\act{\act{\partial l}\{h_2,h_1\}}\left<l,\act{\partial h_2}h_1\right>\{\partial l,h_2|h_1\}^{-1}, \\
    \label{eq:defR}
    R(h_2;l,h_1) &= \left<h_2,l\right>\act{\{h_2,\partial l\}}\{\act{\partial h_2}l,\{h_2,h_1\}\}\{h_2|\partial l,h_1\}^{-1}.
  \end{align}
\end{enumerate}
\end{definition}
The expressions $L(l,h_2;h_1)$ and $R(h_2;l,h_1)$, which we introduced as auxiliaries, correspond precisely to the twists of relations \eqref{eq:2CM-2-funct1} and \eqref{eq:2CM-2-funct2}, respectively:
\begin{align}
  \label{eq:3CM-twisted-2funct1}
  l\{h_2,h_1\} &= \partial_3L(l,h_2;h_1)\{\partial_2(l)h_2,h_1\}\act{\act{\partial_1(h_2)}h_1}l, \\
  \label{eq:3CM-twisted-2funct2}
  \act{h_2}l\{h_2,h_1\} &= \partial_3R(h_2;l,h_1)\{h_2,\partial_2(l)h_1\}\act{\partial_1(h_2)}l.
\end{align}
We will see later that these correspond to the diagrams for Gray 4-categories.

\begin{lemma}\label{thm:unit}
Liftings are normalized as follows:
\begin{align*}
  \{h_z|h_y,h_x\} &= \{h_z,h_y|h_x\} = 1_M, & &\text{where} & h_z,h_y \text{ or } h_x &= 1_H \\
  \{l,h\}_{32} &= \{h,l\}_{23} = 1_M, & &\text{where} & h=1_H\text{ or } l&=1_L, \\
  \{h_2,h_1\}_{22} &= 1_L, & &\text{where} & h_2\text{ or } h_1&=1_H.
\end{align*}
\end{lemma}
The proof of Lemma \ref{thm:unit} was given in \cite{Fukuda:2025}.

\begin{definition}\label{def:3CM-hom}
Let $\bG = M\to L\to H\to G$ and $\bG'=M'\to L'\to H'\to G'$ be 3-crossed modules. A {\bf{4-homomorphisms}} $F:~\bG\to\bG'$ is a collection of homomorphisms $F_i:~G_i\to G_i', i=1,\ldots,4$, preserving all the structures of 3-crossed modules, i.e., the boundary maps, group actions, and liftings. For example, the following diagram commutes:
\[
  \begin{tikzpicture}
    \node (A1) at (0,0) {$M$};
    \node (B1) at (1.5,0) {$L$};
    \node (C1) at (3,0) {$H$};
    \node (D1) at (4.5,0) {$G$};
    \node (A2) at (0,-1.5) {$M'$};
    \node (B2) at (1.5,-1.5) {$L'$};
    \node (C2) at (3,-1.5) {$H'$};
    \node (D2) at (4.5,-1.5) {$G'$};
    \draw [->] (A1) to node [anchor=south] {$\partial_3$} (B1);
    \draw [->] (B1)  to node [anchor=south] {$\partial_2$} (C1);
    \draw [->] (C1)  to node [anchor=south] {$\partial_1$} (D1);
    \draw [->] (A2)  to node [anchor=north] {$\partial_3'$} (B2);
    \draw [->] (B2)  to node [anchor=north] {$\partial_2'$} (C2);
    \draw [->] (C2)  to node [anchor=north] {$\partial_1'$} (D2);
    \draw [->] (A1)  to node [anchor=west] {$F_4$} (A2);
    \draw [->] (B1) to node [anchor=west] {$F_3$} (B2);
    \draw [->] (C1) to node [anchor=west] {$F_2$} (C2);
    \draw [->] (D1) to node [anchor=west] {$F_1$} (D2);
  \end{tikzpicture}
\]
\end{definition}

\subsection{Diagrammatic Representation}\label{sec:3CM-Diagram}

In this section, we give a diagrammatic representation of 3-crossed modules in order to compare them with Gray 4-categories. The details are given in our previous work \cite{Fukuda:2025}. The following diagrams help us to understand the structure of 3-crossed modules.

\begin{enumerate}
  \item Elements of groups:
  \item [$\bullet$] For any element $g\in G$ of the lowest group, we draw it as a segment:
  \begin{equation}
    \label{eq:3CM-Diagram-G-elem}
    \begin{tikzpicture}[baseline=(A.base)]
      \node (A) at (-0.75,0) {$\bullet$};
      \node (B) at (0.75,0) {$\bullet$};
      \draw (A.center) to node [fill=white] {$g$} (B.center);
    \end{tikzpicture},
  \end{equation}
  or a diagram given by compositions introduced later. 
  \item [$\bullet$] For any element $h\in H$ of the second group, we draw it as a square with two elements of $G$ on the upper and lower boundaries:
  \begin{equation}
    \label{eq:3CM-Diagram-H-elem}
    \begin{tikzpicture}[baseline=(T.base)]
      \node (A) at (-0.75,0.75) {$\bullet$};
      \node (B) at (0.75,0.75) {$\bullet$};
      \node (C) at (0.75,-0.75) {$\bullet$};
      \node (D) at (-0.75,-0.75) {$\bullet$};
      \draw [thick] (A.center) to node [fill=white] {$g$} (B.center) to (C.center) to node [fill=white] {$g'$} (D.center) to cycle;
      \node (T) at (0,0) {$h$};
    \end{tikzpicture}
  \end{equation}
  where $g,g':=\partial(h)g\in G$, or a diagram given by compositions introduced later. The blobs $\bullet$ and the boundary elements are often omitted. One should think of the identity element of $G$ as being placed at the top of the boundary. The lower boundary elements are then determined by the contents of the box, starting from the top. 
  The empty box denotes the identity element of $H$.
  \item [$\bullet$] For any element $l\in L$ of the third group, we draw it as an arrow connecting two elements of $H$:
  \begin{equation}
    \label{eq:3CM-Diagram-L-elem}
    \begin{tikzpicture}[baseline=(T.base)]
      \node (A) at (1,0) {\tikz{\strrect[thick,scale=.75]{1}{$h$}}};
      \node (B) at (-1,0) {\tikz{\strrect[thick,scale=.75]{1}{$h'$}}};
      \draw [thick,->] (A) to node [anchor=south] {$l$} (B);
    \end{tikzpicture}
  \end{equation}
  where $h,h':=\partial(l)h\in H$, or a diagram given by compositions introduced later. We often omit the elements attached to the arrow, since they can be recovered from the boundaries. In most cases, 22-Peiffer liftings are assigned to the arrow.  
  \item [$\bullet$] For any element $m\in M$ of the fourth group, we draw it as a double arrow connecting two elements of $L$:
  \begin{equation}
    \label{eq:3CM-Diagram-M-elem}
    \begin{tikzpicture}[baseline=(T.base)]
      \node (A) at (1.5,0) {\tikz{\strrect[thick,scale=.75]{1}{$h$}}};
      \node (B) at (-1.5,0) {\tikz{\strrect[thick,scale=.75]{1}{$h'$}}};
      \draw [thick,->,bend right=1cm] (A) to node (C) [fill=white] {$l$} (B);
      \draw [thick,->,bend left=1cm] (A) to node (D) [fill=white] {$l'$} (B);
      \path (C) to node {$\Downarrow m$} (D);
    \end{tikzpicture}
  \end{equation}
  where $l,l':=\partial(m)l\in L$, or a diagram given by compositions introduced later. In the diagrams \eqref{eq:3CM-Diagram-H-elem}, \eqref{eq:3CM-Diagram-L-elem} and \eqref{eq:3CM-Diagram-M-elem} above, we regard an element $x\in X$ of a higher group as a map from the lower group $Y$ to itself, given by $x~:~Y\ni y\mapsto \partial(x)y = y'\in Y$.
  \item Vertical compositions:
  \item [$\bullet$] We draw the multiplication of the group $G$ as the composition of two segments along the horizontal direction on the paper:
  \[
    \begin{tikzpicture}[baseline=(A.base)]
      \draw[thick] (-1,0) node (A) {$\bullet$} to node [fill=white] {$g_2$} (0,0) node {$\bullet$} to node [fill=white] {$g_1$} (1,0) node {$\bullet$};
    \end{tikzpicture}
    =
    \begin{tikzpicture}[baseline=(A.base)]
      \draw[thick] (-1,0) node (A) {$\bullet$} to node [fill=white] {$g_2g_1$} (1,0) node {$\bullet$};
    \end{tikzpicture},
  \]
  where $g_2,g_1\in G$. 
  \item [$\bullet$] We draw the multiplication of the group $H$ as the composition of two boxes along the vertical direction on the paper:
  \[
    \begin{tikzpicture}[baseline=(T.base)]
      \node (A) at (-0.5,1) {$\bullet$};
      \node (B) at (-0.5,0) {$\bullet$};
      \node (C) at (-0.5,-1) {$\bullet$};
      \node (D) at (0.5,1) {$\bullet$};
      \node (E) at (0.5,0) {$\bullet$};
      \node (F) at (0.5,-1) {$\bullet$};
      \draw [thick] (A.center) to (C.center);
      \draw [thick] (D.center) to (F.center);
      \draw [thick] (A.center) to node [fill=white,transform shape,scale=.75] {$g$} (D.center);
      \draw [thick] (B.center) to node [fill=white,transform shape,scale=.75] {$g'$} (E.center);
      \draw [thick] (C.center) to node [fill=white,transform shape,scale=.75] {$g''$} (F.center);
      \node at (0,1/2) {$h$};
      \node at (0,-1/2) {$h'$};
    \end{tikzpicture}
    =
    \begin{tikzpicture}[baseline=(T.base)]
      \node (A) at (-0.75,0.75) {$\bullet$};
      \node (B) at (0.75,0.75) {$\bullet$};
      \node (C) at (0.75,-0.75) {$\bullet$};
      \node (D) at (-0.75,-0.75) {$\bullet$};
      \draw [thick] (A.center) to node [fill=white] {$g$} (B.center) to (C.center) to node [fill=white] {$g''$} (D.center) to cycle;
      \node (T) at (0,0) {$h'h$};
    \end{tikzpicture},
  \]
  where $g,g'=\partial(h)g,g''=\partial(h')g'\in G$ and $h,h'\in H$. 
  \item [$\bullet$] We draw the multiplication of the group $L$ as a composition of two arrows:
  \[
    \begin{tikzpicture}[baseline=(T.base)]
      \node (A) at (2,0) {\tikz{\strrect[thick,scale=.5]{1}{$h$}}};
      \node (B) at (0,0) {\tikz{\strrect[thick,scale=.5]{1}{$h'$}}};
      \node (C) at (-2,0) {\tikz{\strrect[thick,scale=.5]{1}{$h''$}}};
      \draw [thick,->] (A) to node [anchor=south] {$l$} (B);
      \draw [thick,->] (B) to node [anchor=south] {$l'$} (C);
    \end{tikzpicture}
    =
    \begin{tikzpicture}[baseline=(T.base)]
      \node (A) at (1,0) {\tikz{\strrect[thick,scale=.5]{1}{$h$}}};
      \node (B) at (-1,0) {\tikz{\strrect[thick,scale=.5]{1}{$h''$}}};
      \draw [thick,->] (A) to node [anchor=south] {$l'l$} (B);
    \end{tikzpicture}
  \]
  where $h,h':=\partial(l)h,h''=\partial(l')h'\in H$ and $l,l'\in L$. 
  \item [$\bullet$] We draw the multiplication of the group $M$ as a composition of two double arrows:
  \[
    \begin{tikzpicture}[baseline=(T.base)]
      \node (A) at (1.5,0) {\tikz{\strrect[thick,scale=.5]{1}{$h$}}};
      \node (B) at (-1.5,0) {\tikz{\strrect[thick,scale=.5]{1}{$h'$}}};
      \draw [thick,->,bend right=1.5cm] (A) to node [fill=white] (C1) {$l$} (B);
      \draw [thick,->] (A) to node [fill=white] (C2) {$l'$} (B);
      \draw [thick,->,bend left=1.5cm] (A) to node [fill=white] (C3) {$l''$} (B);
      \path (C1) to node {$\Downarrow m$} (C2);
      \path (C2) to node {$\Downarrow m'$} (C3);
    \end{tikzpicture}
    =
    \begin{tikzpicture}[baseline=(T.base)]
      \node (A) at (1.5,0) {\tikz{\strrect[thick,scale=.5]{1}{$h$}}};
      \node (B) at (-1.5,0) {\tikz{\strrect[thick,scale=.5]{1}{$h'$}}};
      \draw [thick,->,bend right=1.5cm] (A) to node [fill=white] (C1) {$l$} (B);
      \draw [thick,->,bend left=1.5cm] (A) to node [fill=white] (C3) {$l''$} (B);
      \path (C1) to node {$\Downarrow m'm$} (C3);
    \end{tikzpicture}
  \]
  where $m,m'\in M$, $l,l':=\partial(m)l,l'':=\partial(m')l'\in L$ and $h,h':=\partial(l)h\in H$. 
  \item Actions: 
  \item [$\bullet$] We draw the action of $G$ on the higher groups $H,L$ and $M$ as a composition of a segment with the boxes constituting the diagram, along the horizontal direction of the paper. In particular, when the segment is attached to the left side of a box, we interpret it as the action; when it is attached to the right side, no action takes place. For the second group $H$,
  \begin{align}
    \label{eq:3CM-Diagram-G-on-H}
    \begin{tikzpicture}[baseline=(T.base)]
      \node (A) at (-1,1/2) {$\bullet$};
      \node (B) at (0,1/2) {$\bullet$};
      \node (C) at (1,1/2) {$\bullet$};
      \node (D) at (-1,-1/2) {$\bullet$};
      \node (E) at (0,-1/2) {$\bullet$};
      \node (F) at (1,-1/2) {$\bullet$};
      \draw [thick] (A.center) to (D.center);
      \draw [thick] (B.center) to (E.center);
      \draw [thick] (C.center) to (F.center);
      \draw [thick] (A.center) to node [fill=white,transform shape,scale=.75] {$g_2$} (B.center) to node [fill=white,transform shape,scale=.75] {$g_1$} (C.center);
      \draw [thick] (D.center) to node [fill=white,transform shape,scale=.75] {$g_2$} (E.center) to node [fill=white,transform shape,scale=.75] {$g_1'$} (F.center);
      \node (T) at (1/2,0) {$h_1$};
    \end{tikzpicture}
    &=
    \begin{tikzpicture}[baseline=(T.base)]
      \node (A) at (-0.75,0.75) {$\bullet$};
      \node (B) at (0.75,0.75) {$\bullet$};
      \node (C) at (0.75,-0.75) {$\bullet$};
      \node (D) at (-0.75,-0.75) {$\bullet$};
      \draw [thick] (A.center) to node [fill=white] {$g_2g_1$} (B.center) to (C.center) to node [fill=white] {$g_2g_1'$} (D.center) to cycle;
      \node (T) at (0,0) {$\act{g_2}h_1$};
    \end{tikzpicture},
    &
    \begin{tikzpicture}[baseline=(T.base)]
      \node (A) at (-1,1/2) {$\bullet$};
      \node (B) at (0,1/2) {$\bullet$};
      \node (C) at (1,1/2) {$\bullet$};
      \node (D) at (-1,-1/2) {$\bullet$};
      \node (E) at (0,-1/2) {$\bullet$};
      \node (F) at (1,-1/2) {$\bullet$};
      \draw [thick] (A.center) to (D.center);
      \draw [thick] (B.center) to (E.center);
      \draw [thick] (C.center) to (F.center);
      \draw [thick] (A.center) to node [fill=white,transform shape,scale=.75] {$g_2$} (B.center) to node [fill=white,transform shape,scale=.75] {$g_1$} (C.center);
      \draw [thick] (D.center) to node [fill=white,transform shape,scale=.75] {$g_2'$} (E.center) to node [fill=white,transform shape,scale=.75] {$g_1$} (F.center);
      \node (T) at (-1/2,0) {$h_2$};
    \end{tikzpicture}
    &=
    \begin{tikzpicture}[baseline=(T.base)]
      \node (A) at (-0.75,0.75) {$\bullet$};
      \node (B) at (0.75,0.75) {$\bullet$};
      \node (C) at (0.75,-0.75) {$\bullet$};
      \node (D) at (-0.75,-0.75) {$\bullet$};
      \draw [thick] (A.center) to node [fill=white] {$g_2g_1$} (B.center) to (C.center) to node [fill=white] {$g_2'g_1$} (D.center) to cycle;
      \node (T) at (0,0) {$h_2$};
    \end{tikzpicture},
  \end{align}
  where $g_1,g_1':=\partial(h_1)g_1,g_2,g_2':=\partial(h_2)g_2\in G$ and $h_1,h_2\in H$. We compose a segment with a box by using the identity element of $H$ and attaching the segment to the box. For the third group $L$,
  \begin{align}
    \nonumber
    &
    \begin{tikzpicture}[baseline=-2pt]
      \node (P) at (-2,0) {%
        \begin{tikzpicture}%
          \node (A) at (-1,1/2) {$\bullet$};%
          \node (B) at (0,1/2) {$\bullet$};%
          \node (C) at (1,1/2) {$\bullet$};%
          \node (D) at (-1,-1/2) {$\bullet$};%
          \node (E) at (0,-1/2) {$\bullet$};%
          \node (F) at (1,-1/2) {$\bullet$};%
          \draw [thick] (A.center) to (D.center);%
          \draw [thick] (B.center) to (E.center);%
          \draw [thick] (C.center) to (F.center);%
          \draw [thick] (A.center) to node [fill=white,transform shape,scale=.75] {$g_2$} (B.center) to node [fill=white,transform shape,scale=.75] {$g_1$} (C.center);%
          \draw [thick] (D.center) to node [fill=white,transform shape,scale=.75] {$g_2$} (E.center) to node [fill=white,transform shape,scale=.75] {$g_1'$} (F.center);%
          \node at (1/2,0) {$h_1'$};%
        \end{tikzpicture}%
        };
      \node (Q) at (2,0) {%
        \begin{tikzpicture}%
          \node (A) at (-1,1/2) {$\bullet$};%
          \node (B) at (0,1/2) {$\bullet$};%
          \node (C) at (1,1/2) {$\bullet$};%
          \node (D) at (-1,-1/2) {$\bullet$};%
          \node (E) at (0,-1/2) {$\bullet$};%
          \node (F) at (1,-1/2) {$\bullet$};%
          \draw [thick] (A.center) to (D.center);%
          \draw [thick] (B.center) to (E.center);%
          \draw [thick] (C.center) to (F.center);%
          \draw [thick] (A.center) to node [fill=white,transform shape,scale=.75] {$g_2$} (B.center) to node [fill=white,transform shape,scale=.75] {$g_1$} (C.center);%
          \draw [thick] (D.center) to node [fill=white,transform shape,scale=.75] {$g_2$} (E.center) to node [fill=white,transform shape,scale=.75] {$g_1'$} (F.center);%
          \node at (1/2,0) {$h_1$};%
        \end{tikzpicture}%
        };
      \draw [thick,->] (Q) to node [anchor=south] {$\act{g_2}l_1$} (P);
    \end{tikzpicture}, &
    &
    \begin{tikzpicture}[baseline=-2pt]
      \node (P) at (-2,0) {%
        \begin{tikzpicture}%
          \node (A) at (-1,1/2) {$\bullet$};%
          \node (B) at (0,1/2) {$\bullet$};%
          \node (C) at (1,1/2) {$\bullet$};%
          \node (D) at (-1,-1/2) {$\bullet$};%
          \node (E) at (0,-1/2) {$\bullet$};%
          \node (F) at (1,-1/2) {$\bullet$};%
          \draw [thick] (A.center) to (D.center);%
          \draw [thick] (B.center) to (E.center);%
          \draw [thick] (C.center) to (F.center);%
          \draw [thick] (A.center) to node [fill=white,transform shape,scale=.75] {$g_2$} (B.center) to node [fill=white,transform shape,scale=.75] {$g_1$} (C.center);%
          \draw [thick] (D.center) to node [fill=white,transform shape,scale=.75] {$g_2'$} (E.center) to node [fill=white,transform shape,scale=.75] {$g_1$} (F.center);%
          \node at (-1/2,0) {$h_2'$};%
        \end{tikzpicture}%
        };
      \node (Q) at (2,0) {%
        \begin{tikzpicture}%
          \node (A) at (-1,1/2) {$\bullet$};%
          \node (B) at (0,1/2) {$\bullet$};%
          \node (C) at (1,1/2) {$\bullet$};%
          \node (D) at (-1,-1/2) {$\bullet$};%
          \node (E) at (0,-1/2) {$\bullet$};%
          \node (F) at (1,-1/2) {$\bullet$};%
          \draw [thick] (A.center) to (D.center);%
          \draw [thick] (B.center) to (E.center);%
          \draw [thick] (C.center) to (F.center);%
          \draw [thick] (A.center) to node [fill=white,transform shape,scale=.75] {$g_2$} (B.center) to node [fill=white,transform shape,scale=.75] {$g_1$} (C.center);%
          \draw [thick] (D.center) to node [fill=white,transform shape,scale=.75] {$g_2'$} (E.center) to node [fill=white,transform shape,scale=.75] {$g_1$} (F.center);%
          \node at (-1/2,0) {$h_2$};%
        \end{tikzpicture}%
        };
      \draw [thick,->] (Q) to node [anchor=south] {$l_2$} (P);
    \end{tikzpicture}
  \end{align}
  where $l_1,l_2\in L$ and $h_1':=\partial(l_1)h_1,h_2':=\partial(l_2)h_2\in H$. For the fourth group,
  \begin{align}
    \nonumber
    &
    \begin{tikzpicture}[baseline=-2pt]
      \node (P) at (-2,0) {%
        \begin{tikzpicture}%
          \node (A) at (-1,1/2) {$\bullet$};%
          \node (B) at (0,1/2) {$\bullet$};%
          \node (C) at (1,1/2) {$\bullet$};%
          \node (D) at (-1,-1/2) {$\bullet$};%
          \node (E) at (0,-1/2) {$\bullet$};%
          \node (F) at (1,-1/2) {$\bullet$};%
          \draw [thick] (A.center) to (D.center);%
          \draw [thick] (B.center) to (E.center);%
          \draw [thick] (C.center) to (F.center);%
          \draw [thick] (A.center) to node [fill=white,transform shape,scale=.75] {$g_2$} (B.center) to node [fill=white,transform shape,scale=.75] {$g_1$} (C.center);%
          \draw [thick] (D.center) to node [fill=white,transform shape,scale=.75] {$g_2$} (E.center) to node [fill=white,transform shape,scale=.75] {$g_1'$} (F.center);%
          \node at (1/2,0) {$h_1'$};%
        \end{tikzpicture}%
        };
      \node (Q) at (2,0) {%
        \begin{tikzpicture}%
          \node (A) at (-1,1/2) {$\bullet$};%
          \node (B) at (0,1/2) {$\bullet$};%
          \node (C) at (1,1/2) {$\bullet$};%
          \node (D) at (-1,-1/2) {$\bullet$};%
          \node (E) at (0,-1/2) {$\bullet$};%
          \node (F) at (1,-1/2) {$\bullet$};%
          \draw [thick] (A.center) to (D.center);%
          \draw [thick] (B.center) to (E.center);%
          \draw [thick] (C.center) to (F.center);%
          \draw [thick] (A.center) to node [fill=white,transform shape,scale=.75] {$g_2$} (B.center) to node [fill=white,transform shape,scale=.75] {$g_1$} (C.center);%
          \draw [thick] (D.center) to node [fill=white,transform shape,scale=.75] {$g_2$} (E.center) to node [fill=white,transform shape,scale=.75] {$g_1'$} (F.center);%
          \node at (1/2,0) {$h_1$};%
        \end{tikzpicture}%
        };
      \draw [thick,->,bend right=1cm] (Q) to node [fill=white] (R) {$\act{g_2}l_1$} (P);
      \draw [thick,->,bend left=1cm] (Q) to node [fill=white] (S) {$\act{g_2}l_1'$} (P);
      \path (R) to node {$\Downarrow\act{g_2}m_1$} (S);
    \end{tikzpicture}, &
    &
    \begin{tikzpicture}[baseline=-2pt]
      \node (P) at (-2,0) {%
        \begin{tikzpicture}%
          \node (A) at (-1,1/2) {$\bullet$};%
          \node (B) at (0,1/2) {$\bullet$};%
          \node (C) at (1,1/2) {$\bullet$};%
          \node (D) at (-1,-1/2) {$\bullet$};%
          \node (E) at (0,-1/2) {$\bullet$};%
          \node (F) at (1,-1/2) {$\bullet$};%
          \draw [thick] (A.center) to (D.center);%
          \draw [thick] (B.center) to (E.center);%
          \draw [thick] (C.center) to (F.center);%
          \draw [thick] (A.center) to node [fill=white,transform shape,scale=.75] {$g_2$} (B.center) to node [fill=white,transform shape,scale=.75] {$g_1$} (C.center);%
          \draw [thick] (D.center) to node [fill=white,transform shape,scale=.75] {$g_2'$} (E.center) to node [fill=white,transform shape,scale=.75] {$g_1$} (F.center);%
          \node at (-1/2,0) {$h_2$};%
        \end{tikzpicture}%
        };
      \node (Q) at (2,0) {%
        \begin{tikzpicture}%
          \node (A) at (-1,1/2) {$\bullet$};%
          \node (B) at (0,1/2) {$\bullet$};%
          \node (C) at (1,1/2) {$\bullet$};%
          \node (D) at (-1,-1/2) {$\bullet$};%
          \node (E) at (0,-1/2) {$\bullet$};%
          \node (F) at (1,-1/2) {$\bullet$};%
          \draw [thick] (A.center) to (D.center);%
          \draw [thick] (B.center) to (E.center);%
          \draw [thick] (C.center) to (F.center);%
          \draw [thick] (A.center) to node [fill=white,transform shape,scale=.75] {$g_2$} (B.center) to node [fill=white,transform shape,scale=.75] {$g_1$} (C.center);%
          \draw [thick] (D.center) to node [fill=white,transform shape,scale=.75] {$g_2'$} (E.center) to node [fill=white,transform shape,scale=.75] {$g_1$} (F.center);%
          \node at (-1/2,0) {$h_2$};%
        \end{tikzpicture}%
        };
      \draw [thick,->,bend right=1cm] (Q) to node [fill=white] (R) {$l_2$} (P);
      \draw [thick,->,bend left=1cm] (Q) to node [fill=white] (S) {$l_2'$} (P);
      \path (R) to node {$\Downarrow m_2$} (S);
    \end{tikzpicture}
  \end{align}
  where $m_2,m_1\in M$ and $l_2':=\partial(m_2)l_2,l_1':=\partial(m_1)l_1\in L$. 
  \item [$\bullet$] We draw the action of the second group $H$ on the higher groups $L$ and $M$ by attaching the diagram \eqref{eq:3CM-Diagram-H-elem} of $H$ to the boxes constituting the diagram \eqref{eq:3CM-Diagram-L-elem} of $L$ From below, while it does not act from above. For the third group $L$,
  \begin{align}
    \nonumber
    &
    \begin{tikzpicture}[baseline=-2pt]
      \node (A) at (1,0) {\tikz{\strsquare[thick,scale=.5]{1}{2}{$h_1$,$h_2$}}};
      \node (B) at (-1,0) {\tikz{\strsquare[thick,scale=.5]{1}{2}{$h_1$,$h_2'$}}};
      \draw [thick,->] (A) to node [anchor=south] {$\act{h_2}l_1$} (B);
    \end{tikzpicture}, &
    &
    \begin{tikzpicture}[baseline=-2pt]
      \node (A) at (1,0) {\tikz{\strsquare[thick,scale=.5]{1}{2}{$h_1$,$h_2$}}};
      \node (B) at (-1,0) {\tikz{\strsquare[thick,scale=.5]{1}{2}{$h_1$,$h_2'$}}};
      \draw [thick,->] (A) to node [anchor=south] {$l_2$} (B);
    \end{tikzpicture},
  \end{align}
  where $l_2,l_1$ and $h_2,h_1,h_2'=\partial(l_2)h_2,h_1'=\partial(l_1)h_1\in H$. For the fourth group $M$,
  \begin{align}
    \nonumber
    &
    \begin{tikzpicture}[baseline=-2pt]
      \node (A) at (1.5,0) {\tikz{\strsquare[thick,scale=.5]{1}{2}{$h_1$,$h_2$}}};
      \node (B) at (-1.5,0) {\tikz{\strsquare[thick,scale=.5]{1}{2}{$h_1'$,$h_2$}}};
      \draw [thick,->,bend right=1.5cm] (A) to node [fill=white] (C) {$\act{h_2}l_1$} (B);
      \draw [thick,->,bend left=1.5cm] (A) to node [fill=white] (D) {$\act{h_2}l_1'$} (B);
      \path (C) to node {$\Downarrow \act{h_2}m_1$} (D);
    \end{tikzpicture}, &
    &
    \begin{tikzpicture}[baseline=-2pt]
      \node (A) at (1.5,0) {\tikz{\strsquare[thick,scale=.5]{1}{2}{$h_1$,$h_2$}}};
      \node (B) at (-1.5,0) {\tikz{\strsquare[thick,scale=.5]{1}{2}{$h_1$,$h_2'$}}};
      \draw [thick,->,bend right=1.5cm] (A) to node [fill=white] (C) {$l_2$} (B);
      \draw [thick,->,bend left=1.5cm] (A) to node [fill=white] (D) {$l_2'$} (B);
      \path (C) to node {$\Downarrow m_2$} (D);
    \end{tikzpicture},
  \end{align}
  where $m_2,m_1\in M$ and $l_2':=\partial(m_2)l_2, l_1':=\partial(m_1)l_1\in L$.
  \item [$\bullet$] We draw the action of the third group $L$ on the fourth group $M$ by attaching the diagram \eqref{eq:3CM-Diagram-L-elem} of $L$ to the tip of the arrow in the diagram \eqref{eq:3CM-Diagram-L-elem} of $M$, while it does not act if the diagram is attached to the base of the arrow:
  \begin{align}
    \nonumber
    \begin{tikzpicture}[baseline=(T.base)]
      \node (A) at (2,0) {\tikz{\strrect[thick,scale=.5]{1}{$h$}}};
      \node (B) at (0,0) {\tikz{\strrect[thick,scale=.5]{1}{$h'$}}};
      \node (C) at (-2,0) {\tikz{\strrect[thick,scale=.5]{1}{$h''$}}};
      \draw [thick,->] (B) to node [anchor=south] {$l_2$} (C);
      \draw [thick,->,bend right=1.5cm] (A) to node [fill=white] (C) {$l_1$} (B);
      \draw [thick,->,bend left=1.5cm] (A) to node [fill=white] (D) {$l_1'$} (B);
      \path (C) to node {$\Downarrow m_1$} (D);
    \end{tikzpicture}
    &=
    \begin{tikzpicture}[baseline=(T.base)]
      \node (A) at (1.5,0) {\tikz{\strrect[thick,scale=.5]{1}{$h$}}};
      \node (B) at (-1.5,0) {\tikz{\strrect[thick,scale=.5]{1}{$h''$}}};
      \draw [thick,->,bend right=1.5cm] (A) to node [fill=white] (C) {$l_2l_1$} (B);
      \draw [thick,->,bend left=1.5cm] (A) to node [fill=white] (D) {$l_2l_1'$} (B);
      \path (C) to node {$\Downarrow \act{l_2}m_1$} (D);
    \end{tikzpicture}, \\
    \begin{tikzpicture}[baseline=(T.base)]
      \node (A) at (2,0) {\tikz{\strrect[thick,scale=.5]{1}{$h$}}};
      \node (B) at (0,0) {\tikz{\strrect[thick,scale=.5]{1}{$h'$}}};
      \node (C) at (-2,0) {\tikz{\strrect[thick,scale=.5]{1}{$h''$}}};
      \draw [thick,->] (A) to node [anchor=south] {$l_1$} (B);
      \draw [thick,->,bend right=1.5cm] (B) to node [fill=white] (A) {$l_2$} (C);
      \draw [thick,->,bend left=1.5cm] (B) to node [fill=white] (D) {$l_2'$} (C);
      \path (A) to node {$\Downarrow m_2$} (D);
    \end{tikzpicture}
    &=
    \begin{tikzpicture}[baseline=(T.base)]
      \node (A) at (1.5,0) {\tikz{\strrect[thick,scale=.5]{1}{$h$}}};
      \node (B) at (-1.5,0) {\tikz{\strrect[thick,scale=.5]{1}{$h''$}}};
      \draw [thick,->,bend right=1.5cm] (A) to node [fill=white] (C) {$l_2l_1$} (B);
      \draw [thick,->,bend left=1.5cm] (A) to node [fill=white] (D) {$l_2'l_1$} (B);
      \path (C) to node {$\Downarrow m_2$} (D);
    \end{tikzpicture},
  \end{align}
  where $m_2,m_1\in M$, $l_2,l_1,l_2'=\partial(m_2)l_2,l_1'=\partial(m_1)l_1\in L$ and $h,h':=\partial(l_1)m,h'':=\partial(l_2)h'\in H$. 
  \item Liftings
  \item [$\bullet$] 22-Peiffer lifting:
  \[
    \begin{tikzpicture}[baseline=-2pt]
      \node (A) at (1.5,0) {\tikz[baseline=-2pt,scale=.5]{\strrect{2}{$h_2$,,,$h_1$}}};
      \node (B) at (-1.5,0) {\tikz[baseline=-2pt,scale=.5]{\strrect{2}{,$h_1$,$h_2$,}}};
      \draw [->] (A) to node [anchor=south] {$\{h_2,h_1\}$} (B);
    \end{tikzpicture}
  \]
  \item [$\bullet$] 32- and 23-Peiffer liftings:
  \begin{align}
    \nonumber
    &
    \begin{tikzpicture}[baseline=-2pt]
      \node (A) at (1.5,0) {\tikz[baseline=-2pt,scale=.4]{\strrect{2}{,,,$h$}}};
      \node (B) at (0,1.75) {\tikz[baseline=-2pt,scale=.4]{\strrect{2}{$\partial l$,,,$h$}}};
      \node (C) at (-1.5,0) {\tikz[baseline=-2pt,scale=.4]{\strrect{2}{,$h$,$\partial l$,}}};
      \draw [->] (A) to node [anchor=south,sloped] {$\act{h}l$} (B);
      \draw [->] (B) to node [anchor=south,sloped,transform shape,scale=.75] {$\{\partial l,h\}$} (C);
      \draw [->] (A) to node [fill=white] (D) {$l$} (C);
      \path (B) to node [transform shape,scale=.75] {$\Downarrow \left<l,h\right>$} (D);
    \end{tikzpicture}, &
    &
    \begin{tikzpicture}[baseline=-2pt]
      \node (A) at (1.5,0) {\tikz[baseline=-2pt,scale=.4]{\strrect{2}{$h$,,,}}};
      \node (B) at (0,1.75) {\tikz[baseline=-2pt,scale=.4]{\strrect{2}{$h$,,,$\partial l$}}};
      \node (C) at (-1.5,0) {\tikz[baseline=-2pt,scale=.4]{\strrect{2}{,$\partial l$,$h$,}}};
      \draw [->] (A) to node [anchor=south,sloped] {$\act{\partial h}l$} (B);
      \draw [->] (B) to node [anchor=south,sloped,transform shape,scale=.75] {$\{h,\partial l\}$} (C);
      \draw [->] (A) to node [fill=white] (D) {$\act{h}l$} (C);
      \path (B) to node [transform shape,scale=.75] {$\Downarrow \left<h,l\right>$} (D);
    \end{tikzpicture}
  \end{align}
  where $\left<h,l\right>:=\act{\act{h}l}\{h,l^{-1}\}^{-1}$ and $\left<l,h\right>:=\act{l}\{l^{-1},h\}^{-1}$.
  \item [$\bullet$] 33-Peiffer lifting:
  \[
    \begin{tikzpicture}[baseline=-2pt]
      \node (A) at (1.5,0) {\tikz[baseline=-2pt,scale=.4]{\strsquare{1}{2}{}}};
      \node (B1) at (0,1.5) {\tikz[baseline=-2pt,scale=.4]{\strsquare{1}{2}{,$\partial l_2$}}};
      \node (B2) at (0,-1.5) {\tikz[baseline=-2pt,scale=.4]{\strsquare{1}{2}{$\partial l_1$}}};
      \node (C) at (-1.5,0) {\tikz[baseline=-2pt,scale=.4]{\strsquare{1}{2}{$\partial l_1$,$\partial l_2$}}};
      \draw [->] (A) to node [anchor=south,sloped] {$l_2$} (B1);
      \draw [->] (B1) to node [anchor=south,sloped] {$\act{\partial l_2}l_1$} (C);
      \draw [->] (A) to node [anchor=north,sloped] {$l_1$} (B2);
      \draw [->] (B2) to node [anchor=north,sloped] {$l_2$} (C);
      \path (B1) to node {$\Downarrow\{l_2,l_1\}$} (B2);
    \end{tikzpicture}
  \]
  \item [$\bullet$] left- and right-Homanian:
  \begin{align}
    \nonumber
    &
    \begin{tikzpicture}[baseline=-2pt]
      \node (A) at (2,0) {\tikz[baseline=-2pt,scale=.4]{\strsquare{2}{3}{$h_y$,,$h_z$,,,$h_x$}}};
      \node (B) at (0,2) {\tikz[baseline=-2pt,scale=.4]{\strsquare{2}{3}{$h_y$,,,$h_x$,$h_z$,}}};
      \node (C) at (-2,0) {\tikz[baseline=-2pt,scale=.4]{\strsquare{2}{3}{,$h_x$,$h_y$,,$h_z$,}}};
      \draw [->] (A) to node [fill=white,transform shape,scale=.75] (D) {$\{h_zh_y,h_x\}$} (C);
      \draw [->] (A) to node [anchor=south,sloped,transform shape,scale=.75] {$\{h_z,\act{\partial h_y}h_x\}$} (B);
      \draw [->] (B) to node [anchor=south,sloped,transform shape,scale=.75] {$\act{h_z}\{h_y,h_x\}$} (C);
      \path (B) to node [transform shape,scale=.75] {$\Downarrow\{h_z,h_y|h_x\}$} (D);
    \end{tikzpicture}, &
    &
    \begin{tikzpicture}[baseline=-2pt]
      \node (A) at (2,0) {\tikz[baseline=-2pt,scale=.4]{\strsquare{2}{3}{$h_z$,,,$h_x$,,$h_y$}}};
      \node (B) at (0,2) {\tikz[baseline=-2pt,scale=.4]{\strsquare{2}{3}{,$h_x$,$h_z$,,,$h_y$}}};
      \node (C) at (-2,0) {\tikz[baseline=-2pt,scale=.4]{\strsquare{2}{3}{,$h_x$,,$h_y$,$h_z$,}}};
      \draw [->] (A) to node [fill=white,transform shape,scale=.75] (D) {$\{h_zh_y,h_x\}$} (C);
      \draw [->] (A) to node [anchor=south,sloped,transform shape,scale=.75] {$\act{\act{\partial h_z}h_y}\{h_z,h_x\}$} (B);
      \draw [->] (B) to node [anchor=south,sloped,transform shape,scale=.75] {$\{h_z,h_y\}$} (C);
      \path (B) to node [transform shape,scale=.75] {$\Downarrow\{h_z|h_y,h_x\}$} (D);
    \end{tikzpicture}.
  \end{align}
  \item The fundamental relations can be represented by diagrams. The complete set of diagrammatic representations of these relations is given in Appendix \ref{app:3CM-diagram-list}.
\end{enumerate}
Let us give some examples. The first example is the element $\act{\partial h_2}h_1h_2 \in H$. This can be represented as
\begin{equation}
  \label{eq:3-CM-Diagram-example1}
  \tikz[baseline=-8pt,scale=.4]{\strrect{2}{$h_2$,,,$h_1$}} = 
  \begin{tikzpicture}[baseline=(T.base)]
      \node (A1) at (-1.25,1.25) {$\bullet$};
      \node (A2) at (-1.25,0) {$\bullet$};
      \node (A3) at (-1.25,-1.25) {$\bullet$};
      \node (B1) at (0,1.25) {$\bullet$};
      \node (B2) at (0,0) {$\bullet$};
      \node (B3) at (0,-1.25) {$\bullet$};
      \node (C1) at (1.25,1.25) {$\bullet$};
      \node (C2) at (1.25,0) {$\bullet$};
      \node (C3) at (1.25,-1.25) {$\bullet$};
      \draw [thick] (A1.center) to (A3.center);
      \draw [thick] (B1.center) to (B3.center);
      \draw [thick] (C1.center) to (C3.center);
      \draw [thick] (A1.center) to node [fill=white,transform shape,scale=.75] {$e_G$} (B1.center) to node [fill=white,transform shape,scale=.75] {$e_G$} (C1.center);
      \draw [thick] (A2.center) to node [fill=white,transform shape,scale=.75] {$\partial(h_2)$} (B2.center) to node [fill=white,transform shape,scale=.75] {$e_G$} (C2.center);
      \draw [thick] (A3.center) to node [fill=white,transform shape,scale=.75] {$\partial(h_2)$} (B3.center) to node [fill=white,transform shape,scale=.75] {$\partial(h_1)$} (C3.center);
      \node at (-1.25/2,1.25/2) {$h_2$};
      \node at (1.25/2,1.25/2) {$e_H$};
      \node at (1.25/2,-1.25/2) {$h_1$};
      \node at (-1.25/2,-1.25/2) {$e_H$};
    \end{tikzpicture}.
  \end{equation}
Let us reconstruct the information of $H$ from the diagram \eqref{eq:3-CM-Diagram-example1}. First, we read the data of the lowest group $G$. Since the boundary elements are omitted, we should think of the identities $e_G$ as being assigned to the upper edge. Then the lower edges are determined by the contents of the boxes. The second and third line edges are $\partial(h_2)$ and $\partial(h_2h_1)$.

Next, we read the data of the second group $H$. We observe that the boxes in the first and second lines are precisely the right- and left-hand cases of Equation \eqref{eq:3CM-Diagram-G-on-H}, respectively. Hence the boxes in the first and second lines represent $h_2$ and $\act{\partial h_2}h_1$, respectively, and we can reconstruct the starting point $\act{\partial h_2}h_1h_2$. At first glance, this seems inconsistent with the boundary element. However, since $h_2h_1$ equals $\act{\partial h_2}h_1h_2$ up to the boundary element of $L$, i.e., up to a Peiffer lifting, there is no problem.

Another example is the relation for the left-Homanians \eqref{eq:3-CM-Fund-left-Hom}. This can be represented as
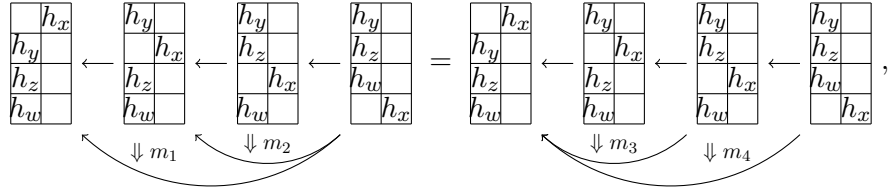
\[
  \begin{tikzpicture}[baseline=-2pt]
    \node (A) at (2.25,0) {\tikz[baseline=-2pt,scale=.4]{\strsquare{2}{4}{$h_y$,,$h_z$,,$h_w$,,,$h_x$}}};
    \node (B) at (.75,0) {\tikz[baseline=-2pt,scale=.4]{\strsquare{2}{4}{$h_y$,,$h_z$,,,$h_x$,$h_w$,}}};
    \node (C) at (-.75,0) {\tikz[baseline=-2pt,scale=.4]{\strsquare{2}{4}{$h_y$,,,$h_x$,$h_z$,,$h_w$,}}};
    \node (D) at (-2.25,0) {\tikz[baseline=-2pt,scale=.4]{\strsquare{2}{4}{,$h_x$,$h_y$,,$h_z$,,$h_w$,}}};
    \draw [->] (A) to (B);
    \draw [->] (B) to (C);
    \draw [->] (C) to (D);
    \draw [->,bend left=1.5cm] (A.south west) to (D.south east);
    \draw [->,bend left=1.5cm] (A.south west) to node (D) {} (C.south east);
    \path (B.north) to [pos=15/16] coordinate (A) (B.south);
    \path (B) to node [transform shape,scale=.75] {$\Downarrow m_2$} (D);
    \path (C.north) to [pos=5/4] coordinate (D) (C.south);
    \path (C) to node [transform shape,scale=.75] {$\Downarrow m_1$} (D);
  \end{tikzpicture}
  =
  \begin{tikzpicture}[baseline=-2pt]
    \node (A) at (2.25,0) {\tikz[baseline=-2pt,scale=.4]{\strsquare{2}{4}{$h_y$,,$h_z$,,$h_w$,,,$h_x$}}};
    \node (B) at (.75,0) {\tikz[baseline=-2pt,scale=.4]{\strsquare{2}{4}{$h_y$,,$h_z$,,,$h_x$,$h_w$,}}};
    \node (C) at (-.75,0) {\tikz[baseline=-2pt,scale=.4]{\strsquare{2}{4}{$h_y$,,,$h_x$,$h_z$,,$h_w$,}}};
    \node (D) at (-2.25,0) {\tikz[baseline=-2pt,scale=.4]{\strsquare{2}{4}{,$h_x$,$h_y$,,$h_z$,,$h_w$,}}};
    \draw [->] (A) to (B);
    \draw [->] (B) to (C);
    \draw [->] (C) to (D);
    \draw [->,bend left=1.5cm] (A.south west) to (D.south east);
    \draw [->,bend left=1.5cm] (B.south west) to node (A) {} (D.south east);
    \path (C) to node [transform shape,scale=.75] {$\Downarrow m_3$} (A);
    \path (B.north) to [pos=5/4] coordinate (D) (B.south);
    \path (B) to node [transform shape,scale=.75] {$\Downarrow m_4$} (D);
  \end{tikzpicture},
\]
where
\begin{align}
  \nonumber
  m_1 &= \{h_wh_z,h_y|h_x\}, &
  m_3 &= \act{h_w}\{h_z,h_y|h_x\}, \\
  \nonumber
  m_2 &= \{h_w,h_z|\act{\partial h_y}h_x\}, &
  m_4 &= \{h_w,h_zh_y|h_x\}.
\end{align}
For the right-Homanian, the analogous diagram holds. Since these are exactly the nets of tetrahedra, we call them tetrahedron relations. The names of the other relations are determined by the same reasoning.

\subsection{Gray 3-Categories} \label{sec:Gray}
In this section, we give the definition of Gray 3-categories. As in Section \ref{sec:CM}, we begin with Gray 3-categories and then generalize them to 4-categories in Section~\ref{sec:Gray4}. Our definition of a Gray 3-category differs slightly from the one introduced in \cite{Sarikaya:2024}. In particular, we use the term ``vertical'' for the operations corresponding to the ``ordinary'' composition of $n$-morphisms, such as the $\circ$ appearing in $f\circ g$, and we do not include ``horizontal composition'' in the following definitions, for reasons that will soon become clear.

To begin with, let us recall the definition of ordinary 2-categories.
\begin{definition}\label{def:strict-2-cat}
A strict 2-category $\cC_2$ consists of the following data:
\begin{enumerate}
  \item \label{ax:2-Cat-2} Collections of $i$-morphisms $\mor_i(\cC_2)$ for $i=0,1,2$. In particular, for $i=0$, $\mor_0(\cC_2)$ is also referred to as the collection of objects and denoted by $\obj(\cC_2)$.
  \begin{itemize}
    \item For $i=1,2$, each $i$-morphism $f$ has an $l$-source $s_l(f)\in \mor_l(\cC_2)$, $l=0,1,\ldots,i$, and an $l$-target $t_l(f)\in \mor_l(\cC_2)$. In particular, for an $i$-morphism $f$ with $(i-1)$-source $x$ and $(i-1)$-target $y$, we write $f:~x\to y$.
    \item The source and target maps $s_l,t_s:~\mor_i(\cC_2)\to\mor_l(\cC_2)$ satisfy the globular relations:
    \begin{align}
      \nonumber
      s_k\circ s_l &= s_k, & t_k\circ t_l &= t_k, \\
      \nonumber
      s_k\circ t_l &= s_k, & t_k\circ s_l &= t_k,
    \end{align}
    for all $k<l<i$. 
  \end{itemize}
  \item \label{ax:2-Cat-3} For each $i$-morphism $x\in\mor_i(\cC)$ ($i=0,1$), an identity $(i+1)$-morphism $\id_x:~x\to x$,
  \item \label{ax:2-Cat-4} Vertical compositions of $i$-morphisms $\#_i:~\mor_i(\cC)\times_{i-1} \mor_i(\cC) \to \mor_i(\cC)$ for $i=1,2$.  In other words, for $i$-morphisms $f:~x\to y,g:~y\to z$, we have an $i$-morphism $g\#_if:~x\to z$.
  \newcounter{temp}
  \setcounter{temp}{\arabic{enumi}}
\end{enumerate}
For any two objects $x$ and $y$, the data $(\{f|f\in\mor_1(\cC),f:~x\to y\},\#_2,s_1,t_1)$ is required to form a category, denoted by $\Hom(x,y)$. 
\begin{enumerate}
  \setcounter{enumi}{\thetemp}
  \item Whiskering by 1-morphisms: for any 1-morphism $f:~w\to x$ and $g:~y\to z$, there are functors\label{def:item:C-is-V}
  \begin{align}
    \nonumber
    -\#_1^{\text{r}}f&:~\Hom(x,y)\to \Hom(w,y) \\
    \nonumber
    g\#_1^{\text{l}}-&:~\Hom(x,y)\to \Hom(x,z)
  \end{align}
  satisfying the following identities (where sources and targets are assumed to be compatible):
  \begin{align}
    \nonumber
    (-\#_1^{\text{r}}f_2)\circ (-\#_1^{\text{r}}f_1) &= -\#_1^{\text{r}}(f_1\#_1f_2), &
    g\#_1^{\text{l}}f &= g\#_1^{\text{r}}f = g\#_1f, \\
    \nonumber
    (g_1\#_1^{\text{l}}-)\circ (g_2\#_1^{\text{l}}-) &= (g_1\#_1g_2)\#_1^{\text{l}}-, & 
    (-\#_1^{\text{r}}\id_x) &= (\id_y\#_1^{\text{l}}-) = \text{Id}_{\Hom(x,y)}. \\
    \nonumber
    (-\#_1^{\text{r}}f)\circ (g\#_1^{\text{l}}-) &= (g\#_1^{\text{l}}-)\circ (-\#_1^{\text{r}}f), 
  \end{align}
  When no confusion can arise, we simply denote these by $\#_1$. 
  \setcounter{temp}{\arabic{enumi}}
\end{enumerate}
In addition, the following axioms must hold:
\begin{enumerate}
  \setcounter{enumi}{\thetemp}
  \item Associativity: Vertical composition of 1-morphisms is associative.
  \item Unit laws: For any 1-morphism $f:x\to y$, the identities satisfy $f\#_1\id_x = \id_y\#_1f = f$. 
  \item \label{ax:2-Cat-8} Interchange law: for two 2-morphisms $(\beta,\alpha)\in\mor(\cC_2)\times_0\mor(\cC_2)$, the following holds:
  \[
    (\beta\#_1t_1(\alpha))\#_2(s_1(\beta)\#_1\alpha) = (t_1(\beta)\#_1\alpha)\#_2(\beta\#_1s_1(\alpha))
  \]
\end{enumerate}
\end{definition}

In preparation for the generalization to the Gray 3- and 4-categories, we restate \ref{ax:2-Cat-2}, \ref{ax:2-Cat-3} and \ref{ax:2-Cat-4} in a redundant manner.

In addition, the interchange law \ref{ax:2-Cat-8} looks different from the usual one because it is expressed using the ``horizontal composition'' of 2-morphisms, which we have not introduced. This condition says that the result is the same no matter which 2-morphism acts first in the direction of the vertical composition of the 1-morphisms. That is, it guarantees that the ``horizontal composition'' can be introduced without ambiguity.

However, in the next generalization, we consider a weakening of this relation by a 3-morphism, which makes the composition ill-defined. With this in mind, we do not include ``horizontal composition'' in the above and following definitions. That said, since it is convenient, we introduce a horizontal composition as follows. For two 2-morphisms $\alpha$ and $\beta$ as in \ref{ax:2-Cat-8}, the horizontal composition $\beta\#_1\alpha$ is defined as
\[
  \beta\#_1\alpha := (\beta\#_1t_1(\alpha))\#_2(s_1(\beta)\#_1\alpha).
\]
One can check that the ordinary interchange law holds with these definitions of the horizontal and vertical compositions. 

Now, let us turn to the definition of Gray 3-categories. 
\begin{definition}\label{def:Gray3}
A Gray 3-category $\cC_3$ consists of the following data:
\begin{enumerate}
  \item Collections of $i$-morphisms $\mor_i(\cC_3)$ for $i=0,1,2,3$. In particular, for $i=0$, $\mor_0(\cC_3)$ is also referred to as the collection of objects and denoted by $\obj(\cC_3)$.
  \begin{itemize}
    \item For $i=1,2,3$, each $i$-morphism $f$ has an $l$-source $s_l(f)\in \mor_l(\cC_3)$, $l=0,1,\ldots,i$, and an $l$-target $t_l(f)\in \mor_l(\cC_3)$. In particular, for an $i$-morphism $f$ with $(i-1)$-source $x$ and $(i-1)$-target $y$, we write $f:~x\to y$.
    \item The source and target maps $s_l,t_l:~\mor_i(\cC_3)\to\mor_l(\cC_3)$ satisfy the globular relations:
    \begin{align}
      \nonumber
      s_k\circ s_l &= s_k, & t_k\circ t_l &= t_k, \\
      \nonumber
      s_k\circ t_l &= s_k, & t_k\circ s_l &= t_k,
    \end{align}
    for all $k<l<i$. 
  \end{itemize}
  \item For each $i$-morphism $x$ ($i=0,1,2$), an identity $(i+1)$-morphism $\id_x:~x\to x$,
  \item Vertical compositions of $i$-morphisms $\#_i:~\mor_i(\cC)\times_{i-1} \mor_i(\cC) \to \mor_i(\cC)$ for $i=1,2,3$. For any two $i$-morphisms $f:~x\to y,g:~y\to z$, we have an $i$-morphism $g\#_if:~x\to z$. 
  \setcounter{temp}{\arabic{enumi}}
\end{enumerate}
For any two 1-morphisms $f$ and $g$, the data
\[
  \left(
    \begin{array}{l}
      \{\alpha|\alpha\in\mor_2(\cC_3),\alpha:~f\to g\}, \\
      \{\Phi|\Phi\in\mor_3(\cC_3),s_1(\Phi) = f,t_1(\Phi) = g\},
    \end{array}
    \#_3,s_3,t_3\right)
\]
is required to form a category, denoted by $\Hom_1(f,g)$.
\begin{enumerate}
  \setcounter{enumi}{\thetemp}
  \item Whiskering by 2-morphisms: for any 2-morphism $\alpha:~e\to f$ and $\beta:~g\to h$, there are 2-functors 
  \begin{align}
    \nonumber
    -\#_2^{\text{r}}\alpha&:~\Hom_1(f,g)\to \Hom_1(e,g), \\ 
    \nonumber
    \beta\#_2^{\text{l}}-&:~\Hom_1(f,g)\to \Hom_1(f,h),
  \end{align}
  satisfying identities analogous to those for whiskering by 1-morphisms in 2-categories \ref{def:item:C-is-V}.
  When no confusion can arise, we simply denote these by $\#_2$. 
  \setcounter{temp}{\arabic{enumi}}
\end{enumerate}
For any two objects $x$ and $y$, the data 
\[
  \left(
    \begin{array}{l}
      \{f|f\in\mor_1(\cC_3),f:~x\to y\}, \\
      \{\alpha|\alpha\in\mor_2(\cC_3),s_0(\alpha) = x,t_0(\alpha) = y\}, \\
      \{\Phi|\Phi\in\mor_3(\cC_3),s_0(\Phi) = y,t_0(\Phi) = y\},
    \end{array}
  \#_2,\#_3,s_1,t_1,s_2,t_2\right)
\]
is required to form a 2-category, denoted by $\Hom_2(x,y)$
\begin{enumerate}
  \setcounter{enumi}{\thetemp}
  \item Whiskering by 1-morphisms: for any 1-morphism $f:~w\to x$ and $g:~y\to z$, there are 2-functors 
  \begin{align}
    \nonumber
    -\#_1^{\text{r}}f&:~\Hom_2(x,y)\to \Hom_2(w,y) \\
    \nonumber
    g\#_1^{\text{l}}-&:~ \Hom_2(x,y)\to \Hom_2(x,z)
  \end{align}
  satisfying identities analogous to those for whiskering by 1-morphisms in 2-categories \ref{def:item:C-is-V}. When no confusion can arise, we simply denote these by $\#_1$. 
  \item Interchange 3-morphism (or Peiffer lifting): for any two 2-morphisms $(\beta,\alpha)\in\mor_2(\cC_3)\times_0\mor_2(\cC_3)$, there is a 3-isomorphism denoted by $\{\beta,\alpha\}$, whose 2-source and 2-target are
  \begin{align}
    \label{eq:Peiffer-Source}
    s_2(\{\beta,\alpha\}) &= (\beta\#_1t_1(\alpha))\#_2(s_1(\beta)\#_1\alpha), \\
    \label{eq:Peiffer-Target}
    t_2(\{\beta,\alpha\}) &= (t_1(\beta)\#_1\alpha)\#_2(\beta\#_1s_1(\alpha)).
  \end{align}
  \setcounter{temp}{\arabic{enumi}}
\end{enumerate}
In what follows, expressions such as the left-hand sides of equations \eqref{eq:Peiffer-Source} and \eqref{eq:Peiffer-Target} appear frequently and take up much space, so we use the following shorthand notation:
\begin{align}
  \nonumber
  \begin{bmatrix}& y \\ x &\end{bmatrix} &:= (x\#_1t_1(y))\#_2(s_1(x)\#_1y), &
  \begin{pmatrix}& w \\ z &\end{pmatrix} &:= (z\#_2t_2(w))\#_3(s_2(z)\#_2w), \\
  \nonumber
  \begin{bmatrix}x& \\ &y\end{bmatrix} &:= (t_1(x)\#_1y)\#_2(x\#_1s_1(y)), &
  \begin{pmatrix}z & \\ & w\end{pmatrix} &:= (t_2(z)\#_2w)\#_3(z\#_2s_2(w)),
\end{align}
where $(x,y)\in \mor_i(\cC_3)\times_0\mor_j(\cC_3),~(i,j)=(2,2),(2,3),(3,2)$, and $(z,w)\in \mor_3(\cC_3)\times_1\mor_3(\cC_3)$.

In addition, the following axioms must hold:
\begin{enumerate}
  \setcounter{enumi}{\thetemp}
  \item Associativity: Vertical composition of 1-morphisms is associative. 
  \item Unit laws: For any 1-morphism $f:~x\to y$, the identities satisfy $f\#_1\id_x = \id_y\#_1f = f$. 
  \item \label{ax:Gray3-1morPeiffer} 
  The compatibility conditions between the Peiffer lifting and 1-morphisms are
  \begin{itemize}
    \item for any $(h,\beta,\alpha) \in \mor_1(\cC_3)\times_0\mor_2(\cC_3)\times_0\mor_2(\cC_3)$,
    \[
      h\#_1\{\beta,\alpha\} = \{h\#_1\beta,\alpha\},
    \]
    \item for any $(\beta,g,\alpha) \in \mor_2(\cC_3)\times_0\mor_1(\cC_3)\times_0\mor_2(\cC_3)$,
    \[
      \{\beta\#_1g,\alpha\} = \{\beta,g\#_1\alpha\},
    \]
    \item for any $(\beta,\alpha,f) \in \mor_2(\cC_3)\times_0\mor_2(\cC_3)\times_0\mor_1(\cC_3)$,
    \[
      \{\beta,\alpha\}\#_1f = \{\beta,\alpha\#_1f\},
    \]
  \end{itemize}
  \item 
  The compatibility conditions between the Peiffer lifting and 2-morphisms are
  \begin{align}
    \label{eq:P-decom1}
    \{\beta_2\#_2\beta_1,\alpha\} &= [(\beta_2\#_1t_1(\alpha))\#_2\{\beta_1,\alpha\}]\#_3[\{\beta_2,\alpha\}\#_2(\beta_1\#_1s_1(\alpha))], \\
    \label{eq:P-decom2}
    \{\beta,\alpha_2\#_2\alpha_1\}
    &= [\{\beta,\alpha_2\}\#_2(s_1(\beta)\#_1\alpha_1)]\#_3[(t_1(\beta)\#_1\alpha_2)\#_2\{\beta,\alpha_1\}],
  \end{align}
  where
  \begin{align}
    \nonumber
    (\beta_2,\beta_1,\alpha) &\in (\mor_2(\cC_3)\times_1\mor_2(\cC_3))\times_0\mor_2(\cC_3), \\
    \nonumber
    (\beta,\alpha_2,\alpha_1) &\in \mor_2(\cC_3)\times_0(\mor_2(\cC_3)\times_1\mor_2(\cC_3)).
  \end{align}
  \item 
  The compatibility conditions between the Peiffer lifting and 3-morphisms are
  \begin{align}
    \label{eq:2-funct1}
    \{t_2(\Psi),\alpha\}\#_3\begin{bmatrix}\Psi & \\ & \alpha\end{bmatrix} &= \begin{bmatrix}&\alpha \\ \Psi &\end{bmatrix}\#_3\{s_2(\Psi),\alpha\} \\
    \label{eq:2-funct2}
    \{\beta,t_2(\Phi)\}\#_3\begin{bmatrix}\beta & \\ & \Phi\end{bmatrix} &= \begin{bmatrix}& \Phi \\ \beta &\end{bmatrix}\#_3\{\beta,s_2(\Phi)\}
  \end{align}
  where
  \begin{align}
    \nonumber
    (\Psi,\alpha) &\in \mor_3(\cC_3)\times_0\mor_2(\cC_3), \\
    \nonumber
    (\beta,\Phi) &\in \mor_2(\cC_3)\times_0\mor_3(\cC_3).
  \end{align}
\end{enumerate}
\end{definition}

Note that the associativity, unit and interchange laws for $\#_2$ and $\#_3$ follow directly from the requirement that $\Hom_2(x,y)$ is a 2-category. In contrast, the associativity and unit laws for the vertical composition of 1-morphisms must be specified explicitly.

The above definition differs slightly from the pioneering works, such as \cite{Martins:2009evc} and \cite{Sarikaya:2024}. In our definition, we have not introduced the composition of 3-morphisms in the 2-morphism direction, called ``vertical composition of 3-morphisms'' in \cite{Martins:2009evc}, which corresponds to the ``horizontal composition'' in the 2-category case. The interchange law would make this composition well-defined, but we do not include it, for the same reason as in the case of 2-categories. Correspondingly, the ``2-functoriality'' relation in \cite{Martins:2009evc} splits into two types of relations, \eqref{eq:2-funct1} and \eqref{eq:2-funct2}.

\begin{lemma}
The following normalization conditions hold.
\begin{align}
  \nonumber
  \{\id_g,\alpha\} &= \id_{g\#_1\alpha}, & \{\beta,\id_f\} &= \id_{\beta\#_1f}.
\end{align}
\end{lemma}
The proof is given, for example, in \cite{Sarikaya:2024}.

\begin{definition}\label{def:Gray3Functor}
Let $\cC$ and $\cD$ be Gray 3-categories. A Gray 3-functor $F:~\cC\to\cD$ is a collection of maps $F_i:~\mor_i(\cC)\to\mor_i(\cD), i=0,\ldots,3$, preserving all the structures of Gray 3-categories, i.e., the identity, source and target maps, vertical compositions, whiskerings, and the Peiffer lifting.
\end{definition}

\subsection{Diagrammatic Representation of Gray 3-categories}\label{sec:diag-Gray3}
Before proceeding to the definition of Gray 4-categories, we introduce the diagrammatic representation of Gray 3-categories. As in the case of crossed modules, these diagrams help us to understand not only the structure of Gray 3-categories but also that of Gray 4-categories, since the same diagrams can be used to represent the relations of Gray 4-categories.

\begin{enumerate}
  \item Morphisms:
  \item [$\bullet$] For any object $x\in \obj(\cC_4)$, we draw it as a blob $\bullet$,
  \item [$\bullet$] For any 1-morphism $f\in \mor_1(\cC_4)$, we draw it as a horizontal segment:
  \begin{equation}
    \label{eq:Gray4-Diagram-1-Mor}
    \begin{tikzpicture}[baseline=(A.base)]
      \node (A) at (-0.75,0) {$\bullet$};
      \node (B) at (0.75,0) {$\bullet$};
      \draw [thick] (A.center) to node [anchor=south] {$f$} (B.center);
      \node at (B.center) [anchor=west] {$s_0(f)$};
      \node at (A.center) [anchor=east] {$t_0(f)$};
    \end{tikzpicture},
  \end{equation}
  or a diagram given by compositions introduced later. The source and target objects are attached to the right and left sides, respectively; as introduced above, they are represented as blobs.
  \item [$\bullet$] For any 2-morphism $\alpha\in\mor_2(\cC_4)$, we draw it as a square:
  \begin{equation}
    \label{eq:Gray4-Diagram-2-Mor}
    \begin{tikzpicture}[baseline=(T.base)]
      \node (A) at (-0.75,0.75) {$\bullet$};
      \node (B) at (0.75,0.75) {$\bullet$};
      \node (C) at (0.75,-0.75) {$\bullet$};
      \node (D) at (-0.75,-0.75) {$\bullet$};
      \draw [thick] (A.center) node [anchor=east] {$t_0(\alpha)$} to node [anchor=south] {$s_1(\alpha)$} (B.center) node [anchor=west] {$s_0(\alpha)$} to (C.center) node [anchor=west] {$s_0(\alpha)$} to node [anchor=north] {$t_1(\alpha)$} (D.center) node [anchor=east] {$t_0(\alpha)$} to cycle;
      \node (T) at (0,0) {$\alpha$};
    \end{tikzpicture}
  \end{equation}
  or a diagram given by compositions introduced later.
  The source and target 1-morphisms are attached to the upper and lower boundaries. The boundary elements and blobs $\bullet$ are often omitted. The empty box denotes the identity 2-morphism $\id_f$.
  \item [$\bullet$] For any 3-morphism $\Phi\in\mor_3(\cC_4)$, we draw it as an arrow connecting two 2-morphisms:
  \begin{equation}
    \label{eq:Gray4-Diagram-3-Mor}
    \begin{tikzpicture}[baseline=(T.base)]
      \node (A) at (1,0) {\tikz{\strrect[thick,transform shape,scale=.75]{1}{$s_2(\Phi)$}}};
      \node (B) at (-1,0) {\tikz{\strrect[thick,transform shape,scale=.75]{1}{$t_2(\Phi)$}}};
      \draw [thick,->] (A) to node [anchor=south] {$\Phi$} (B);
    \end{tikzpicture}
  \end{equation}
  or a diagram given by compositions introduced later.  
  \item Vertical compositions:
  \item [$\bullet$] Vertical composition of 1-morphisms: We draw it as the composition of segments along the horizontal direction on the paper:
  \[
    \begin{tikzpicture}[baseline=(A.base)]
      \draw[thick] (-1,0) node (A) {$\bullet$} to node [anchor=south] {$g$} (0,0) node {$\bullet$} to node [anchor=south] {$f$} (1,0) node {$\bullet$};
    \end{tikzpicture}
    =
    \begin{tikzpicture}[baseline=(A.base)]
      \draw[thick] (-1,0) node (A) {$\bullet$} to node [anchor=south] {$g\#_1f$} (1,0) node {$\bullet$};
    \end{tikzpicture},
  \]
  where $(f,g)\in \mor_1(\cC_4)\times_0\mor_1(\cC_4)$. 
  \item [$\bullet$] Vertical composition of 2-morphisms: We draw it as the composition of boxes along the vertical direction on the paper:
  \[
    \begin{tikzpicture}[baseline=(T.base)]
      \node (A) at (-0.5,1) {$\bullet$};
      \node (B) at (-0.5,0) {$\bullet$};
      \node (C) at (-0.5,-1) {$\bullet$};
      \node (D) at (0.5,1) {$\bullet$};
      \node (E) at (0.5,0) {$\bullet$};
      \node (F) at (0.5,-1) {$\bullet$};
      \draw [thick] (A.center) to (C.center);
      \draw [thick] (D.center) to (F.center);
      \draw [thick] (A.center) to (D.center);
      \draw [thick] (B.center) to (E.center);
      \draw [thick] (C.center) to (F.center);
      \node at (0,1/2) {$\beta$};
      \node at (0,-1/2) {$\alpha$};
    \end{tikzpicture}
    =
    \begin{tikzpicture}[baseline=(T.base)]
      \node (A) at (-0.75,0.75) {$\bullet$};
      \node (B) at (0.75,0.75) {$\bullet$};
      \node (C) at (0.75,-0.75) {$\bullet$};
      \node (D) at (-0.75,-0.75) {$\bullet$};
      \draw [thick] (A.center) to (B.center) to (C.center) to (D.center) to cycle;
      \node (T) at (0,0) {$\beta\#_2\alpha$};
    \end{tikzpicture},
  \]
  where $(\beta,\alpha)\in\mor_2(\cC_4)\times_1\mor_2(\cC_4)$. 
  \item [$\bullet$] Vertical composition of 3-morphisms: We draw it as a composition of two arrows:
  \[
    \begin{tikzpicture}[baseline=(T.base)]
      \node (A) at (2,0) {\tikz{\strrect[thick,scale=.5]{1}{$\alpha$}}};
      \node (B) at (0,0) {\tikz{\strrect[thick,scale=.5]{1}{$\beta$}}};
      \node (C) at (-2,0) {\tikz{\strrect[thick,scale=.5]{1}{$\gamma$}}};
      \draw [thick,->] (A) to node [anchor=south] {$\Phi$} (B);
      \draw [thick,->] (B) to node [anchor=south] {$\Psi$} (C);
    \end{tikzpicture}
    =
    \begin{tikzpicture}[baseline=(T.base)]
      \node (A) at (1,0) {\tikz{\strrect[thick,scale=.5]{1}{$\alpha$}}};
      \node (B) at (-1,0) {\tikz{\strrect[thick,scale=.5]{1}{$\gamma$}}};
      \draw [thick,->] (A) to node [anchor=south] {$\Psi\#_3\Phi$} (B);
    \end{tikzpicture}
  \]
  where $\Psi:~\beta\to\gamma$ and $\Phi:~\alpha\to\beta$. 
  \item Whiskerings
  \item [$\bullet$] Whiskering by 1-morphisms: We draw it by attaching a segment representing the 1-morphism to the boxes of the diagrams, along the horizontal direction on the paper. For 2-morphisms,
  \begin{align}
    \nonumber
    \begin{tikzpicture}[baseline=(T.base)]
      \node (A) at (-1,1/2) {$\bullet$};
      \node (B) at (0,1/2) {$\bullet$};
      \node (C) at (1,1/2) {$\bullet$};
      \node (D) at (-1,-1/2) {$\bullet$};
      \node (E) at (0,-1/2) {$\bullet$};
      \node (F) at (1,-1/2) {$\bullet$};
      \draw [thick] (A.center) to (D.center);
      \draw [thick] (B.center) to (E.center);
      \draw [thick] (C.center) to (F.center);
      \draw [thick] (A.center) to node [fill=white,transform shape,scale=.75] {$g$} (B.center) to (C.center);
      \draw [thick] (D.center) to node [fill=white,transform shape,scale=.75] {$g$} (E.center) to (F.center);
      \node (T) at (1/2,0) {$\alpha$};
    \end{tikzpicture}
    &=
    \begin{tikzpicture}[baseline=(T.base)]
      \node (A) at (-0.75,0.75) {$\bullet$};
      \node (B) at (0.75,0.75) {$\bullet$};
      \node (C) at (0.75,-0.75) {$\bullet$};
      \node (D) at (-0.75,-0.75) {$\bullet$};
      \draw [thick] (A.center) to (B.center) to (C.center) to (D.center) to cycle;
      \node (T) at (0,0) {$g\#_1\alpha$};
    \end{tikzpicture},
    &
    \begin{tikzpicture}[baseline=(T.base)]
      \node (A) at (-1,1/2) {$\bullet$};
      \node (B) at (0,1/2) {$\bullet$};
      \node (C) at (1,1/2) {$\bullet$};
      \node (D) at (-1,-1/2) {$\bullet$};
      \node (E) at (0,-1/2) {$\bullet$};
      \node (F) at (1,-1/2) {$\bullet$};
      \draw [thick] (A.center) to (D.center);
      \draw [thick] (B.center) to (E.center);
      \draw [thick] (C.center) to (F.center);
      \draw [thick] (A.center) to (B.center) to node [fill=white,transform shape,scale=.75] {$f$} (C.center);
      \draw [thick] (D.center) to (E.center) to node [fill=white,transform shape,scale=.75] {$f$} (F.center);
      \node (T) at (-1/2,0) {$\beta$};
    \end{tikzpicture}
    &=
    \begin{tikzpicture}[baseline=(T.base)]
      \node (A) at (-0.75,0.75) {$\bullet$};
      \node (B) at (0.75,0.75) {$\bullet$};
      \node (C) at (0.75,-0.75) {$\bullet$};
      \node (D) at (-0.75,-0.75) {$\bullet$};
      \draw [thick] (A.center) to (B.center) to (C.center) to (D.center) to cycle;
      \node (T) at (0,0) {$\beta\#_1f$};
    \end{tikzpicture},
  \end{align}
  where $(g,\alpha)\in\mor_1(\cC_4)\times_0\mor_2(\cC_4)$ and $(\beta,f)\in\mor_2(\cC_4)\times_0\mor_1(\cC_4)$. We compose a segment with boxes by using the identity elements $\id_f$ and $\id_g$. For 3-morphisms,
  \begin{align}
    \nonumber
    &
    \begin{tikzpicture}[baseline=-2pt]
      \node (P) at (-2,0) {%
        \begin{tikzpicture}[baseline=(T.base)]%
          \node (A) at (-1,1/2) {$\bullet$};%
          \node (B) at (0,1/2) {$\bullet$};%
          \node (C) at (1,1/2) {$\bullet$};%
          \node (D) at (-1,-1/2) {$\bullet$};%
          \node (E) at (0,-1/2) {$\bullet$};%
          \node (F) at (1,-1/2) {$\bullet$};%
          \draw [thick] (A.center) to (D.center);%
          \draw [thick] (B.center) to (E.center);%
          \draw [thick] (C.center) to (F.center);%
          \draw [thick] (A.center) to node [fill=white,transform shape,scale=.75] {$g$} (B.center) to (C.center);%
          \draw [thick] (D.center) to node [fill=white,transform shape,scale=.75] {$g$} (E.center) to (F.center);%
          \node (T) at (1/2,0) {$t_2(\Phi)$};%
        \end{tikzpicture}%
        };
      \node (Q) at (2,0) {%
        \begin{tikzpicture}[baseline=(T.base)]%
          \node (A) at (-1,1/2) {$\bullet$};%
          \node (B) at (0,1/2) {$\bullet$};%
          \node (C) at (1,1/2) {$\bullet$};%
          \node (D) at (-1,-1/2) {$\bullet$};%
          \node (E) at (0,-1/2) {$\bullet$};%
          \node (F) at (1,-1/2) {$\bullet$};%
          \draw [thick] (A.center) to (D.center);%
          \draw [thick] (B.center) to (E.center);%
          \draw [thick] (C.center) to (F.center);%
          \draw [thick] (A.center) to node [fill=white,transform shape,scale=.75] {$g$} (B.center) to (C.center);%
          \draw [thick] (D.center) to node [fill=white,transform shape,scale=.75] {$g$} (E.center) to (F.center);%
          \node (T) at (1/2,0) {$s_2(\Phi)$};%
        \end{tikzpicture}%
        };
      \draw [thick,->] (Q) to node [anchor=south] {$g\#_1\Phi$} (P);
    \end{tikzpicture},
    &
    &
    \begin{tikzpicture}[baseline=-2pt]
      \node (P) at (-2,0) {%
        \begin{tikzpicture}[baseline=(T.base)]
          \node (A) at (-1,1/2) {$\bullet$};
          \node (B) at (0,1/2) {$\bullet$};
          \node (C) at (1,1/2) {$\bullet$};
          \node (D) at (-1,-1/2) {$\bullet$};
          \node (E) at (0,-1/2) {$\bullet$};
          \node (F) at (1,-1/2) {$\bullet$};
          \draw [thick] (A.center) to (D.center);
          \draw [thick] (B.center) to (E.center);
          \draw [thick] (C.center) to (F.center);
          \draw [thick] (A.center) to (B.center) to node [fill=white,transform shape,scale=.75] {$f$} (C.center);
          \draw [thick] (D.center) to (E.center) to node [fill=white,transform shape,scale=.75] {$f$} (F.center);
          \node (T) at (-1/2,0) {$t_2(\Psi)$};
        \end{tikzpicture}
        };
      \node (Q) at (2,0) {%
        \begin{tikzpicture}[baseline=(T.base)]%
          \node (A) at (-1,1/2) {$\bullet$};%
          \node (B) at (0,1/2) {$\bullet$};%
          \node (C) at (1,1/2) {$\bullet$};%
          \node (D) at (-1,-1/2) {$\bullet$};%
          \node (E) at (0,-1/2) {$\bullet$};%
          \node (F) at (1,-1/2) {$\bullet$};%
          \draw [thick] (A.center) to (D.center);%
          \draw [thick] (B.center) to (E.center);%
          \draw [thick] (C.center) to (F.center);%
          \draw [thick] (A.center) to (B.center) to node [fill=white,transform shape,scale=.75] {$f$} (C.center);%
          \draw [thick] (D.center) to (E.center) to node [fill=white,transform shape,scale=.75] {$f$} (F.center);%
          \node (T) at (-1/2,0) {$s_2(\Psi)$};%
        \end{tikzpicture}%
        };
      \draw [thick,->] (Q) to node [anchor=south] {$\Psi\#_1f$} (P);
    \end{tikzpicture}
  \end{align}
  where $(g,\Phi)\in\mor_1(\cC_4)\times_0\mor_3(\cC_4)$ and $(\Psi,f)\in\mor_3(\cC_4)\times_0\mor_1(\cC_4)$. 
  \item [$\bullet$] Whiskering by 2-morphisms: We draw it by attaching a box representing the 2-morphism to the boxes constituting the diagrams. For 3-morphisms, 
  \begin{align}
    \nonumber
    &
    \begin{tikzpicture}[baseline=-2pt]
      \node (A) at (1,0) {\tikz{\strsquare[thick,scale=.5]{1}{2}{$\alpha$,$\beta$}}};
      \node (B) at (-1,0) {\tikz{\strsquare[thick,scale=.5]{1}{2}{$\alpha'$,$\beta$}}};
      \draw [thick,->] (A) to node [anchor=south] {$\beta\#_2\Phi$} (B);
    \end{tikzpicture}, &
    &
    \begin{tikzpicture}[baseline=-2pt]
      \node (A) at (1,0) {\tikz{\strsquare[thick,scale=.5]{1}{2}{$\alpha$,$\beta'$}}};
      \node (B) at (-1,0) {\tikz{\strsquare[thick,scale=.5]{1}{2}{$\alpha$,$\beta$}}};
      \draw [thick,->] (A) to node [anchor=south] {$\Psi\#_2\alpha$} (B);
    \end{tikzpicture},
  \end{align}
  where $(\beta,\Phi)\in\mor_2(\cC_4)\times_1\mor_3(\cC_4),s_2(\Phi)=\alpha,t_2(\Phi)=\alpha'$ and $(\Psi,\alpha)\in\mor_3(\cC_4),s_2(\Psi)=\beta,t_2(\Psi)=\beta'$.
  \item  Peiffer lifting:\label{diag:Peiffer22}
  \[
    \begin{tikzpicture}[baseline=-2pt]
      \node (A) at (1.5,0) {\tikz[baseline=-2pt,scale=.5]{\strrect{2}{$\beta$,,,$\alpha$}}};
      \node (B) at (-1.5,0) {\tikz[baseline=-2pt,scale=.5]{\strrect{2}{,$\alpha$,$\beta$,}}};
      \draw [thick,->] (A) to node [anchor=south] {$\{\beta,\alpha\}$} (B);
    \end{tikzpicture}
  \]
  In the previously introduced simplified notation, the boxes above can be depicted as
  \begin{align*}
    \begin{bmatrix}\beta & \\ & \alpha \end{bmatrix} &= \tikz[baseline=-2pt]{\node (B) at (0,0) {\tikz[baseline=-2pt,scale=.5]{\strrect{2}{$\beta$,,,$\alpha$}}};}, & \begin{bmatrix} & \alpha \\ \beta & \end{bmatrix} &= \tikz[baseline=-2pt]{\node (B) at (0,0) {\tikz[baseline=-2pt,scale=.5]{\strrect{2}{,$\alpha$,$\beta$,}}};}.
  \end{align*}
  They represent the same thing, and we use whichever of the two is more convenient.
\end{enumerate}

These diagrams clarify the role of the relations \eqref{eq:P-decom1}--\eqref{eq:2-funct2}. For example, the relations \eqref{eq:P-decom1} and \eqref{eq:2-funct1} can be represented as the commutativity of the following diagrams, respectively:
\begin{align*}
  &
  \begin{tikzpicture}[baseline=(I.base)]
    \node (A) at (2,0) {\tikz[baseline=-2pt,scale=.33]{\strsquare{2}{3}{$\beta_1$,,$\beta_2$,,,$\alpha$}}};
    \node (B) at (0,-2) {\tikz[baseline=-2pt,scale=.33]{\strsquare{2}{3}{$\beta_1$,,,$\alpha$,$\beta_2$,}}};
    \node (C) at (-2,0) {\tikz[baseline=-2pt,scale=.33]{\strsquare{2}{3}{,$\alpha$,$\beta_1$,,$\beta_2$,}}};
    \draw [->] (A) to node [fill=white,transform shape,scale=.75] (D) {$\{\beta_2\#_2\beta_1,\alpha\}$} (C);
    \draw [->] (A) to node [anchor=north west,transform shape,scale=.66] {$\{\beta_1,\alpha\}\#_2(\beta_2\#_1s_1(\alpha))$} (B);
    \draw [->] (B) to node [anchor=north east,transform shape,scale=.66] {$(\beta_2\#_1t_1(\alpha))\#_2\{\beta_1,\alpha\}$} (C);
    \path (D) to node (I) {$=$} (B);
  \end{tikzpicture},
  &
  &
  \begin{tikzpicture}[baseline=-2pt]
    \node (A) at (1.5,0) {\tikz[baseline=-2pt,scale=.33]{\strrect{2}{$\beta$,,,$\alpha$}}};
    \node (B1) at (0,1.5) {\tikz[baseline=-2pt,scale=.33]{\strrect{2}{$\beta'$,,,$\alpha$}}};
    \node (B2) at (0,-1.5) {\tikz[baseline=-2pt,scale=.33]{\strrect{2}{,$\alpha$,$\beta$,}}};
    \node (C) at (-1.5,0) {\tikz[baseline=-2pt,scale=.33]{\strrect{2}{,$\alpha$,$\beta'$,}}};
    \draw [thick,->] (A) to node [anchor=south west,transform shape,scale=.75] {$\begin{bmatrix}\Psi & \\ & \alpha\end{bmatrix}$} (B1);
    \draw [thick,->] (A) to node [sloped,anchor=north] {$\{\beta,\alpha\}$} (B2);
    \draw [thick,->] (B1) to node [sloped,anchor=south] {$\{\beta',\alpha\}$} (C);
    \draw [thick,->] (B2) to node [anchor=north east,transform shape,scale=.75] {$\begin{bmatrix}& \alpha \\ \Psi &\end{bmatrix}$} (C);
    \path (B1) to node {$=$} (B2);
  \end{tikzpicture}
\end{align*}
The left diagram says that exchanging \tikz[baseline=-4pt,scale=.33]{\strsquare{2}{1}{,$\alpha$}} with \tikz[baseline=-8pt,scale=.33]{\strrect{2}{$\beta_1$,,$\beta_2$,}} all at once is the same as first exchanging \tikz[baseline=-4pt,scale=.33]{\strsquare{2}{1}{,$\alpha$}} with \tikz[baseline=-4pt,scale=.33]{\strsquare{2}{1}{$\beta_2$,}} and then with \tikz[baseline=-4pt,scale=.33]{\strsquare{2}{1}{$\beta_1$,}}. The right diagram says that the 3-morphisms $\Psi$ and $\{\beta,\alpha\}$ commute. These diagrams are ``twisted'' by the introduction of 4-morphisms in the Gray 4-categories below.

\subsection{Gray 4-Categories} \label{sec:Gray4}
In this section, we generalize the ordinary Gray 3-categories to Gray 4-categories. Along the way, we introduce their diagrammatic representation, as we did for Gray 3-categories. As mentioned in the Introduction, this generalization is suggested by the theory of semistrict braided monoidal 2-categories \cite{Kapranov:1994,Baez:1996,Crans:1998}. A Gray 4-category with a single object and a single 1-morphism can be thought of as a semistrict braided monoidal 2-category.

\begin{definition}
\label{def:Gray4}
A Gray 4-category $\cC_4$ consists of the following data:
\begin{enumerate}
  \item Collections of $i$-morphisms $\mor_i(\cC_4)$ for $i=0,1,2,3,4$. In particular, for $i=0$, $\mor_0(\cC_4)$ is also referred to as the collection of objects and denoted by $\obj(\cC_4)$.
  \begin{itemize}
    \item For $i=1,2,3,4$, each $i$-morphism $f$ has an $l$-source $s_l(f)\in \mor_l(\cC_4)$, $l=0,1,\ldots,i$, and an $l$-target $t_l(f)\in \mor_l(\cC_4)$. In particular, for an $i$-morphism $f$ with $(i-1)$-source $x$ and $(i-1)$-target $y$, we write $f:~x\to y$.
    \item The source and target maps $s_l,t_l:~\mor_i(\cC_4)\to\mor_l(\cC_4)$ satisfy the globular relations:
    \begin{align}
      \nonumber
      s_k\circ s_l &= s_k, & t_k\circ t_l &= t_k, \\
      \nonumber
      s_k\circ t_l &= s_k, & t_k\circ s_l &= t_k,
    \end{align}
    for all $k<l<i$. 
  \end{itemize}
  \item For each $i$-morphism $x$ ($i=0,1,2,3$), an identity $(i+1)$-morphism $\id_x:~x\to x$,
  \item Vertical compositions of $i$-morphisms $\#_i:~\mor_i(\cC)\times_{i-1} \mor_i(\cC) \to \mor_i(\cC)$ for $i=1,2,3,4$. For any two $i$-morphisms $f:~x\to y$ and $g:~y\to z$, we have an $i$-morphism $g\#_i f:~x\to z$. 
  \setcounter{temp}{\arabic{enumi}}
\end{enumerate}
For any two 2-morphisms $\alpha$ and $\beta$, the data 
\[
  \left(
    \begin{array}{l}
      \{\Phi|\Phi\in\mor_3(\cC_4),f:~\alpha\to \beta\}, \\
      \{A|A\in\mor_4(\cC_4),s_0(A) = \alpha,t_0(A) = \beta\},
    \end{array}
  \#_4,s_3,t_3\right)
\]
is required to form a category, denoted by $\Hom_1(\alpha,\beta)$. 
\begin{enumerate}
  \setcounter{enumi}{\thetemp}
  \item Whiskering by 3-morphisms: for any 3-morphism $\Phi:~\omega\to\alpha$ and $\Psi:~\beta\to\gamma$, there are functors 
  \begin{align}
    \nonumber
    -\#_3^{\text{r}}\Phi&:~\Hom_1(\alpha,\beta)\to \Hom_1(\omega,\beta) \\ 
    \nonumber
    \Psi\#_3^{\text{l}}-&:~\Hom_1(\alpha,\beta)\to \Hom_1(\alpha,\gamma)
  \end{align}
  satisfying identities analogous to those for whiskering by 1-morphisms in 2-categories \ref{def:item:C-is-V}.
  When no confusion can arise, we simply denote these by $\#_3$. 
  \setcounter{temp}{\arabic{enumi}}
\end{enumerate}
For any two 1-morphisms $f$ and $g$, the data 
\[
  \left(
    \begin{array}{l}
      \{\alpha|\alpha\in\mor_2(\cC_4),\alpha:~f\to g\}, \\
      \{\Phi|\Phi\in\mor_3(\cC_4),s_1(\Phi)=f,t_1(\Phi)=g\}, \\
      \{A|A\in\mor_4(\cC_4),s_1(A)=f,t_1(A)=g\},
    \end{array}
    \#_3,\#_4,s_2,t_2,s_3,t_3
  \right)
\]
is required to form a 2-category, denoted by $\Hom_2(f,g)$. 
\begin{enumerate}
  \setcounter{enumi}{\thetemp}
  \item Whiskering by 2-morphisms: for any 2-morphism $\alpha:~e\to f$ and $\beta:~g\to h$, there are 2-functors 
  \begin{align}
    \nonumber
    -\#_2^{\text{r}}\alpha&:~\Hom_2(f,g)\to \Hom_2(e,g) \\ 
    \beta\#_2^{\text{l}}-&:~\Hom_2(f,g)\to \Hom_2(f,h)
  \end{align}
  satisfying identities analogous to those for whiskering by 1-morphisms in 2-categories \ref{def:item:C-is-V}. When no confusion can arise, we simply denote these by $\#_2$. 
  \item Interchange 4-morphism (33-Peiffer lifting): for any two 3-morphisms $(\Psi,\Phi)\in\mor_3(\cC_4)\times_1\mor_3(\cC_4)$, there is a 4-isomorphism denoted by $\{\Psi,\Phi\}_{33}$, whose 2-source and 2-target are
  \begin{align}
    \nonumber
    s_3(\{\Psi,\Phi\}_{33}) &= (\Psi\#_2t_2(\Phi))\#_3(s_2(\Psi)\#_2\Phi), \\
    \nonumber
    t_3(\{\Psi,\Phi\}_{33}) &= (t_2(\Psi)\#_2\Phi)\#_3(\Psi\#_2s_2(\Phi)).
  \end{align}
  \setcounter{temp}{\arabic{enumi}}
\end{enumerate}
For any two objects $x$ and $y$, the data 
\[
  \left(
    \begin{array}{l}
      \{f|f\in\mor_1(\cC_4),f:~x\to y\}, \\
      \{\alpha|\alpha\in\mor_2(\cC_4),s_0(\alpha)=x,t_0(\alpha)=y\}, \\
      \{\Phi|\Phi\in\mor_3(\cC_4),s_0(\Phi)=x,t_0(\Phi)=y\}, \\
      \{A|A\in\mor_4(\cC_4),s_0(A)=x,t_0(A)=y\},
    \end{array}
    \#_2,\#_3,\#_4,s_i,t_i,\{-,-\}_{33}
  \right)
\]
is required to form a Gray 3-category, denoted by $\Hom_3(x,y)$. 
\begin{enumerate}
  \setcounter{enumi}{\thetemp}
  \item Whiskering by 1-morphisms: for any 1-morphism $f:~w\to x$ and $g:~y\to z$, there are Gray 3-functors 
  \begin{align}
    \nonumber
    -\#_1^{\text{r}}f&:~\Hom_3(x,y)\to \Hom_3(w,y) \\
    \nonumber
    g\#_1^{\text{l}}-&:~\Hom_3(x,y)\to \Hom_3(x,z)
  \end{align}
  satisfying identities analogous to those for whiskering by 1-morphisms in 2-categories \ref{def:item:C-is-V}. When no confusion can arise, we simply denote these by $\#_1$. 
  \item Interchange 3-morphism (or 22-Peiffer lifting): there is a 3-isomorphism $\{\beta,\alpha\}_{22}$ for any two 2-morphisms $(\beta,\alpha)\in\mor_2(\cC_4)\times_0\mor_2(\cC_4)$, whose 2-source and 2-target are
  \begin{align}
    \nonumber
    s_2(\{\beta,\alpha\}_{22}) &= (\beta\#_1t_1(\alpha))\#_2(s_1(\beta)\#_1\alpha), \\
    \nonumber
    t_2(\{\beta,\alpha\}_{22}) &= (t_1(\beta)\#_1\alpha)\#_2(\beta\#_1s_1(\alpha)).
  \end{align}
  \item Twist of $\{-\#_2-,-\}_{22}$ (left-Homanian): there is a 4-isomorphism $\{\beta_2,\beta_1|\alpha\}$ for morphisms $(\beta_2,\beta_1,\alpha) \in (\mor_2(\cC_4)\times_1\mor_2(\cC_4))\times_0\mor_2(\cC_4)$, whose 3-source and 3-target are
  \begin{align}
    \nonumber
    s_3(\{\beta_2,\beta_1|\alpha\}) &= \{\beta_2\#_2\beta_1,\alpha\}_{22} \\
    \nonumber
    t_3(\{\beta_2,\beta_1|\alpha\}) &= [(\beta_2\#_1t_1(\alpha))\#_2\{\beta_1,\alpha\}_{22}]\#_3[\{\beta_2,\alpha\}_{22}\#_2(\beta_1\#_1s_1(\alpha))]
  \end{align}
  \item Twist of $\{-,-\#_2-\}_{22}$ (right-Homanian): there is a 4-isomorphism $\{\beta|\alpha_2,\alpha_1\}$ for morphisms $(\beta,\alpha_2,\alpha_1) \in \mor_2(\cC_4)\times_0(\mor_2(\cC_4)\times_1\mor_2(\cC_4))$, whose 3-source and 3-target are
  \begin{align}
    \nonumber
    s_3(\{\beta|\alpha_2,\alpha_1\}) &= \{\beta,\alpha_2\#_2\alpha_1\}_{22} \\
    \nonumber
    t_3(\{\beta|\alpha_2,\alpha_1\}) &= [\{\beta,\alpha_2\}_{22}\#_2(s_1(\beta)\#_1\alpha_1)]\#_3[(t_1(\beta)\#_1\alpha_2)\#_2\{\beta,\alpha_1\}_{22}]
  \end{align}
  \item 32-lifting: there is a 4-isomorphism $\{\Psi,\alpha\}_{32}$ for morphisms $(\Psi,\alpha) \in \mor_3(\cC_4)\times_0\mor_2(\cC_4)$, whose 3-source and 3-target are
  \begin{align}
    \nonumber
    s_3(\{\Psi,\alpha\}_{32}) &= \{t_2(\Psi),\alpha\}_{22}\#_3[(t_1(\Psi)\#_1\alpha)\#_2(\Psi\#_1s_1(\alpha))] \\
    \nonumber
    t_3(\{\Psi,\alpha\}_{32}) &= [(\Psi\#_1t_1(\alpha))\#_2(s_1(\Psi)\#_1\alpha)]\#_3\{s_2(\Psi),\alpha\}_{22}
  \end{align}
  \item 23-lifting: there is a 4-isomorphism $\{\beta,\Phi\}_{23}$ for morphisms $(\beta,\Phi) \in \mor_2(\cC_4)\times_0\mor_3(\cC_4)$, whose 3-source and 3-target are
  \begin{align}
    \nonumber
    s_3(\{\beta,\Phi\}_{23}) &= \{\beta,t_2(\Phi)\}_{22}\#_3[(t_1(\beta)\#_1\Phi)\#_2(\beta\#_1s_1(\Phi))] \\
    \nonumber
    t_3(\{\beta,\Phi\}_{23}) &= [(\beta\#_1t_1(\Phi))\#_2(s_1(\beta)\#_1\Phi)]\#_3\{\beta,s_2(\Phi)\}_{22}
  \end{align}
  \setcounter{temp}{\arabic{enumi}}
\end{enumerate}
When it is clear which lifting is used from the context, subscripts are omitted. 

Before proceeding to the coherence conditions, it is helpful to introduce the diagrammatic representation of the morphisms. These are essentially the same as for Gray 3-categories, except that Gray 4-categories now have 4-morphisms. In addition to the diagrams introduced in Section \ref{sec:diag-Gray3}, we have
\begin{itemize}
  \item For any 4-morphism $A\in \mor_4$, we draw it as a double arrow connecting two 3-morphisms:
  \begin{equation}
    \label{eq:Gray4-Diagram-4-Mor}
    \begin{tikzpicture}[baseline=(T.base)]
      \node (A) at (1.5,0) {\tikz{\strrect[thick,transform shape,scale=.75]{1}{$s_2(A)$}}};
      \node (B) at (-1.5,0) {\tikz{\strrect[thick,transform shape,scale=.75]{1}{$t_2(A)$}}};
      \draw [thick,->,bend right=1cm] (A) to node (C) [fill=white,transform shape,scale=.75] {$s_3(A)$} (B);
      \draw [thick,->,bend left=1cm] (A) to node (D) [fill=white,transform shape,scale=.75] {$t_3(A)$} (B);
      \path (C) to node {$\Downarrow A$} (D);
    \end{tikzpicture}
  \end{equation}
  or a diagram given by compositions introduced later. 
  \item Vertical composition of 4-morphisms: We draw it as a composition of two double arrows:
  \[
    \begin{tikzpicture}[baseline=(T.base)]
      \node (A) at (1.5,0) {\tikz{\strrect[thick,scale=.5]{1}{$\alpha$}}};
      \node (B) at (-1.5,0) {\tikz{\strrect[thick,scale=.5]{1}{$\beta$}}};
      \draw [thick,->,bend right=1.5cm] (A) to node (C1) {} (B);
      \draw [thick,->] (A) to node (C2) {} (B);
      \draw [thick,->,bend left=1.5cm] (A) to node (C3) {} (B);
      \path (C1) to node {$\Downarrow A$} (C2);
      \path (C2) to node {$\Downarrow B$} (C3);
    \end{tikzpicture}
    =
    \begin{tikzpicture}[baseline=(T.base)]
      \node (A) at (1.5,0) {\tikz{\strrect[thick,scale=.5]{1}{$\alpha$}}};
      \node (B) at (-1.5,0) {\tikz{\strrect[thick,scale=.5]{1}{$\beta$}}};
      \draw [thick,->,bend right=1.5cm] (A) to node (C1) {} (B);
      \draw [thick,->,bend left=1.5cm] (A) to node (C3) {} (B);
      \path (C1) to node {$\Downarrow B\#_4A$} (C3);
    \end{tikzpicture}
  \]
  where $(B,A)\in\mor_4(\cC_4)\times_3\mor_3(\cC_4)$, $s_2(A) = s_2(B) = \alpha$ and $t_2(A) = t_2(B) = \beta$. 
  \item Whiskering by 1-morphisms: For 4-morphisms,
  \begin{align}
    \nonumber
    &
    \begin{tikzpicture}[baseline=-2pt]
      \node (P) at (-2,0) {%
        \begin{tikzpicture}%
          \node (A) at (-1,1/2) {$\bullet$};%
          \node (B) at (0,1/2) {$\bullet$};%
          \node (C) at (1,1/2) {$\bullet$};%
          \node (D) at (-1,-1/2) {$\bullet$};%
          \node (E) at (0,-1/2) {$\bullet$};%
          \node (F) at (1,-1/2) {$\bullet$};%
          \draw [thick] (A.center) to (D.center);%
          \draw [thick] (B.center) to (E.center);%
          \draw [thick] (C.center) to (F.center);%
          \draw [thick] (A.center) to node [fill=white,transform shape,scale=.75] {$g$} (B.center) to (C.center);%
          \draw [thick] (D.center) to node [fill=white,transform shape,scale=.75] {$g$} (E.center) to (F.center);%
          \node at (1/2,0) {$t_2(A)$};%
        \end{tikzpicture}%
        };
      \node (Q) at (2,0) {%
        \begin{tikzpicture}%
          \node (A) at (-1,1/2) {$\bullet$};%
          \node (B) at (0,1/2) {$\bullet$};%
          \node (C) at (1,1/2) {$\bullet$};%
          \node (D) at (-1,-1/2) {$\bullet$};%
          \node (E) at (0,-1/2) {$\bullet$};%
          \node (F) at (1,-1/2) {$\bullet$};%
          \draw [thick] (A.center) to (D.center);%
          \draw [thick] (B.center) to (E.center);%
          \draw [thick] (C.center) to (F.center);%
          \draw [thick] (A.center) to node [fill=white,transform shape,scale=.75] {$g$} (B.center) to (C.center);%
          \draw [thick] (D.center) to node [fill=white,transform shape,scale=.75] {$g$} (E.center) to (F.center);%
          \node at (1/2,0) {$s_2(A)$};%
        \end{tikzpicture}%
        };
      \draw [thick,->,bend right=1cm] (Q) to node (R) {} (P);
      \draw [thick,->,bend left=1cm] (Q) to node (S) {} (P);
      \path (R) to node {$\Downarrow g\#_1A$} (S);
    \end{tikzpicture}, &
    &
    \begin{tikzpicture}[baseline=-2pt]
      \node (P) at (-2,0) {%
        \begin{tikzpicture}%
          \node (A) at (-1,1/2) {$\bullet$};%
          \node (B) at (0,1/2) {$\bullet$};%
          \node (C) at (1,1/2) {$\bullet$};%
          \node (D) at (-1,-1/2) {$\bullet$};%
          \node (E) at (0,-1/2) {$\bullet$};%
          \node (F) at (1,-1/2) {$\bullet$};%
          \draw [thick] (A.center) to (D.center);%
          \draw [thick] (B.center) to (E.center);%
          \draw [thick] (C.center) to (F.center);%
          \draw [thick] (A.center) to (B.center) to node [fill=white,transform shape,scale=.75] {$f$} (C.center);%
          \draw [thick] (D.center) to (E.center) to node [fill=white,transform shape,scale=.75] {$f$} (F.center);%
          \node at (-1/2,0) {$t_2(B)$};%
        \end{tikzpicture}%
        };
      \node (Q) at (2,0) {%
        \begin{tikzpicture}%
          \node (A) at (-1,1/2) {$\bullet$};%
          \node (B) at (0,1/2) {$\bullet$};%
          \node (C) at (1,1/2) {$\bullet$};%
          \node (D) at (-1,-1/2) {$\bullet$};%
          \node (E) at (0,-1/2) {$\bullet$};%
          \node (F) at (1,-1/2) {$\bullet$};%
          \draw [thick] (A.center) to (D.center);%
          \draw [thick] (B.center) to (E.center);%
          \draw [thick] (C.center) to (F.center);%
          \draw [thick] (A.center) to (B.center) to node [fill=white,transform shape,scale=.75] {$f$} (C.center);%
          \draw [thick] (D.center) to (E.center) to node [fill=white,transform shape,scale=.75] {$f$} (F.center);%
          \node at (-1/2,0) {$s_2(B)$};%
        \end{tikzpicture}%
        };
      \draw [thick,->,bend right=1cm] (Q) to node (R) {} (P);
      \draw [thick,->,bend left=1cm] (Q) to node (S) {} (P);
      \path (R) to node {$\Downarrow B\#_1f$} (S);
    \end{tikzpicture}
  \end{align}
  where $(g,A)\in \mor_1(\cC_4)\times_0\mor_4(\cC_4)$ and $(B,f)\in\mor_4(\cC_4)\times_0\mor_1(\cC_4)$. 
  \item [$\bullet$] Whiskering by 2-morphisms: For 4-morphisms, 
  \begin{align}
    \nonumber
    &
    \begin{tikzpicture}[baseline=-2pt]
      \node (A) at (1.5,0) {\tikz{\strsquare[thick,scale=.5]{1}{2}{$\alpha$,$\beta$}}};
      \node (B) at (-1.5,0) {\tikz{\strsquare[thick,scale=.5]{1}{2}{$\alpha'$,$\beta$}}};
      \draw [thick,->,bend right=1.5cm] (A) to node (C) {} (B);
      \draw [thick,->,bend left=1.5cm] (A) to node (D) {} (B);
      \path (C) to node {$\Downarrow \beta\#_2A$} (D);
    \end{tikzpicture}, &
    &
    \begin{tikzpicture}[baseline=-2pt]
      \node (A) at (1.5,0) {\tikz{\strsquare[thick,scale=.5]{1}{2}{$\alpha$,$\beta$}}};
      \node (B) at (-1.5,0) {\tikz{\strsquare[thick,scale=.5]{1}{2}{$\alpha$,$\beta'$}}};
      \draw [thick,->,bend right=1.5cm] (A) to node (C) {} (B);
      \draw [thick,->,bend left=1.5cm] (A) to node (D) {} (B);
      \path (C) to node {$\Downarrow B\#_2\alpha$} (D);
    \end{tikzpicture},
  \end{align}
  where $(\beta,A)\in \mor_2(\cC_4)\times_1\mor_4(\cC_4),s_2(A)=\alpha,t_2(A)=\alpha'$ and $(B,\alpha)\in \mor_4(\cC_4)\times_1\mor_2(\cC_4),s_2(B)=\beta,t_2(B)=\beta'$. 
  \item [$\bullet$] Whiskering by 3-morphisms: We draw it by attaching an arrow representing the 3-morphism to the ends of arrows in the diagram:
  \begin{align}
    \nonumber
    \begin{tikzpicture}[baseline=(T.base)]
      \node (A) at (2,0) {\tikz{\strrect[thick,scale=.5]{1}{$\alpha$}}};
      \node (B) at (0,0) {\tikz{\strrect[thick,scale=.5]{1}{$\beta$}}};
      \node (C) at (-2,0) {\tikz{\strrect[thick,scale=.5]{1}{$\gamma$}}};
      \draw [thick,->] (B) to node [anchor=south] {$\Psi$} (C);
      \draw [thick,->,bend right=1.5cm] (A) to node (C) {} (B);
      \draw [thick,->,bend left=1.5cm] (A) to node (D) {} (B);
      \path (C) to node {$\Downarrow A$} (D);
    \end{tikzpicture}
    &=
    \begin{tikzpicture}[baseline=(T.base)]
      \node (A) at (1.5,0) {\tikz{\strrect[thick,scale=.5]{1}{$\alpha$}}};
      \node (B) at (-1.5,0) {\tikz{\strrect[thick,scale=.5]{1}{$\gamma$}}};
      \draw [thick,->,bend right=1.5cm] (A) to node (C) {} (B);
      \draw [thick,->,bend left=1.5cm] (A) to node (D) {} (B);
      \path (C) to node {$\Downarrow\Psi\#_3A$} (D);
    \end{tikzpicture}, \\
    \nonumber
    \begin{tikzpicture}[baseline=(T.base)]
      \node (A) at (2,0) {\tikz{\strrect[thick,scale=.5]{1}{$\alpha$}}};
      \node (B) at (0,0) {\tikz{\strrect[thick,scale=.5]{1}{$\beta$}}};
      \node (C) at (-2,0) {\tikz{\strrect[thick,scale=.5]{1}{$\gamma$}}};
      \draw [thick,->] (A) to node [anchor=south] {$\Phi$} (B);
      \draw [thick,->,bend right=1.5cm] (B) to node (A) {} (C);
      \draw [thick,->,bend left=1.5cm] (B) to node (D) {} (C);
      \path (A) to node {$\Downarrow B$} (D);
    \end{tikzpicture}
    &=
    \begin{tikzpicture}[baseline=(T.base)]
      \node (A) at (1.5,0) {\tikz{\strrect[thick,scale=.5]{1}{$\alpha$}}};
      \node (B) at (-1.5,0) {\tikz{\strrect[thick,scale=.5]{1}{$\gamma$}}};
      \draw [thick,->,bend right=1.5cm] (A) to node (C) {} (B);
      \draw [thick,->,bend left=1.5cm] (A) to node (D) {} (B);
      \path (C) to node {$\Downarrow B\#_3\Phi$} (D);
    \end{tikzpicture},
  \end{align}
  where $(\Psi,A)\in\mor_3(\cC_4)\times_2\mor_4(\cC_4),\Psi:~\beta\to\gamma,s_2(A)=\alpha,t_2(A)=\beta$ and $(B,\Phi)\in\mor_4(\cC_4)\times_2\mor_3(\cC_4),s_2(B)=\beta,t_2(B)=\gamma,\Phi:~\alpha\to\beta$. 
\end{itemize}
Using these diagrams, the above liftings can be represented as follows:
\begin{itemize}
  \item [$\bullet$] 32-Peiffer lifting and 23-Peiffer liftings:
  \begin{align}
    \nonumber
    &
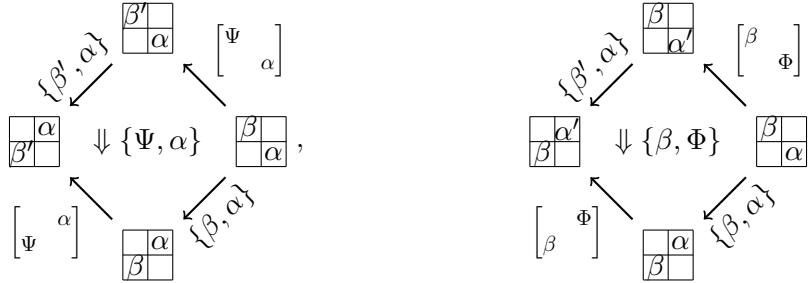
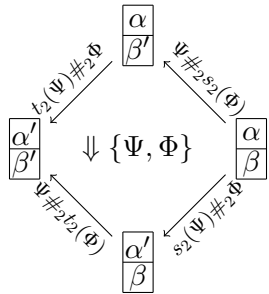
    \begin{tikzpicture}[baseline=-2pt]
      \node (A) at (1.5,0) {\tikz[baseline=-2pt,scale=.33]{\strrect{2}{$\beta$,,,$\alpha$}}};
      \node (B1) at (0,1.5) {\tikz[baseline=-2pt,scale=.33]{\strrect{2}{$\beta'$,,,$\alpha$}}};
      \node (B2) at (0,-1.5) {\tikz[baseline=-2pt,scale=.33]{\strrect{2}{,$\alpha$,$\beta$,}}};
      \node (C) at (-1.5,0) {\tikz[baseline=-2pt,scale=.33]{\strrect{2}{,$\alpha$,$\beta'$,}}};
      \draw [thick,->] (A) to node [anchor=south west,transform shape,scale=.75] {$\begin{bmatrix}\Psi & \\ & \alpha\end{bmatrix}$} (B1);
      \draw [thick,->] (A) to node [sloped,anchor=north] {$\{\beta,\alpha\}$} (B2);
      \draw [thick,->] (B1) to node [sloped,anchor=south] {$\{\beta',\alpha\}$} (C);
      \draw [thick,->] (B2) to node [anchor=north east,transform shape,scale=.75] {$\begin{bmatrix} & \alpha \\ \Psi &\end{bmatrix}$} (C);
      \path (B1) to node {$\Downarrow\{\Psi,\alpha\}$} (B2);
    \end{tikzpicture}, &
    &
    \begin{tikzpicture}[baseline=-2pt]
      \node (A) at (1.5,0) {\tikz[baseline=-2pt,scale=.33]{\strrect{2}{$\beta$,,,$\alpha$}}};
      \node (B1) at (0,1.5) {\tikz[baseline=-2pt,scale=.33]{\strrect{2}{$\beta$,,,$\alpha'$}}};
      \node (B2) at (0,-1.5) {\tikz[baseline=-2pt,scale=.33]{\strrect{2}{,$\alpha$,$\beta$,}}};
      \node (C) at (-1.5,0) {\tikz[baseline=-2pt,scale=.33]{\strrect{2}{,$\alpha'$,$\beta$,}}};
      \draw [thick,->] (A) to node [anchor=south west,transform shape,scale=.75] {$\begin{bmatrix}\beta & \\ & \Phi\end{bmatrix}$} (B1);
      \draw [thick,->] (A) to node [sloped,anchor=north] {$\{\beta,\alpha\}$} (B2);
      \draw [thick,->] (B1) to node [sloped,anchor=south] {$\{\beta',\alpha\}$} (C);
      \draw [thick,->] (B2) to node [anchor=north east,transform shape,scale=.75] {$\begin{bmatrix} & \Phi\\\beta & \end{bmatrix}$} (C);
      \path (B1) to node {$\Downarrow \{\beta,\Phi\}$} (B2);
    \end{tikzpicture}
  \end{align}
  \item [$\bullet$] 33-Peiffer lifting:
  \[
    \begin{tikzpicture}[baseline=-2pt]
      \node (A) at (1.5,0) {\tikz[baseline=-2pt,scale=.4]{\strsquare{1}{2}{$\alpha$,$\beta$}}};
      \node (B1) at (0,1.5) {\tikz[baseline=-2pt,scale=.4]{\strsquare{1}{2}{$\alpha$,$\beta'$}}};
      \node (B2) at (0,-1.5) {\tikz[baseline=-2pt,scale=.4]{\strsquare{1}{2}{$\alpha'$,$\beta$}}};
      \node (C) at (-1.5,0) {\tikz[baseline=-2pt,scale=.4]{\strsquare{1}{2}{$\alpha'$,$\beta'$}}};
      \draw [->] (A) to node [anchor=south,sloped,transform shape,scale=.75] {$\Psi\#_2s_2(\Phi)$} (B1);
      \draw [->] (B1) to node [anchor=south,sloped,transform shape,scale=.75] {$t_2(\Psi)\#_2\Phi$} (C);
      \draw [->] (A) to node [anchor=north,sloped,transform shape,scale=.75] {$s_2(\Psi)\#_2\Phi$} (B2);
      \draw [->] (B2) to node [anchor=north,sloped,transform shape,scale=.75] {$\Psi\#_2t_2(\Phi)$} (C);
      \path (B1) to node {$\Downarrow \{\Psi,\Phi\}$} (B2);
    \end{tikzpicture}
  \]
  \item [$\bullet$] left- and right-Homanian:
  \begin{align}
    \nonumber
    &
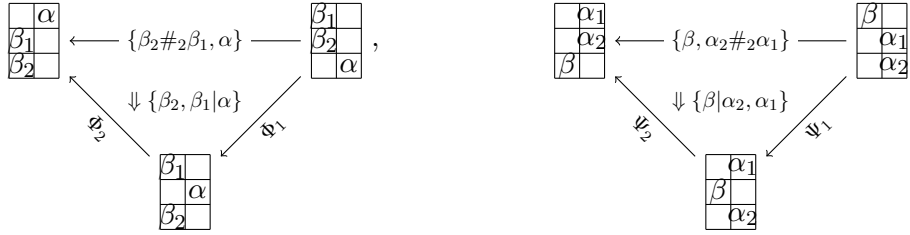
    \begin{tikzpicture}[baseline=-2pt]
      \node (A) at (2,0) {\tikz[baseline=-2pt,scale=.33]{\strsquare{2}{3}{$\beta_1$,,$\beta_2$,,,$\alpha$}}};
      \node (B) at (0,-2) {\tikz[baseline=-2pt,scale=.33]{\strsquare{2}{3}{$\beta_1$,,,$\alpha$,$\beta_2$,}}};
      \node (C) at (-2,0) {\tikz[baseline=-2pt,scale=.33]{\strsquare{2}{3}{,$\alpha$,$\beta_1$,,$\beta_2$,}}};
      \draw [->] (A) to node [fill=white,transform shape,scale=.75] (D) {$\{\beta_2\#_2\beta_1,\alpha\}$} (C);
      \draw [->] (A) to node [anchor=north,sloped,transform shape,scale=.75] {$\Phi_1$} (B);
      \draw [->] (B) to node [anchor=north,sloped,transform shape,scale=.75] {$\Phi_2$} (C);
      \path (D) to node [transform shape,scale=.75] {$\Downarrow\{\beta_2,\beta_1|\alpha\}$} (B);
    \end{tikzpicture}, &
    &
    \begin{tikzpicture}[baseline=-2pt]
      \node (A) at (2,0) {\tikz[baseline=-2pt,scale=.33]{\strsquare{2}{3}{$\beta$,,,$\alpha_1$,,$\alpha_2$}}};
      \node (B) at (0,-2) {\tikz[baseline=-2pt,scale=.33]{\strsquare{2}{3}{,$\alpha_1$,$\beta$,,,$\alpha_2$}}};
      \node (C) at (-2,0) {\tikz[baseline=-2pt,scale=.33]{\strsquare{2}{3}{,$\alpha_1$,,$\alpha_2$,$\beta$,}}};
      \draw [->] (A) to node [fill=white,transform shape,scale=.75] (D) {$\{\beta,\alpha_2\#_2\alpha_1\}$} (C);
      \draw [->] (A) to node [anchor=north,sloped,transform shape,scale=.75] {$\Psi_1$} (B);
      \draw [->] (B) to node [anchor=north,sloped,transform shape,scale=.75] {$\Psi_2$} (C);
      \path (D) to node [transform shape,scale=.75] {$\Downarrow\{\beta|\alpha_2,\alpha_1\}$} (B);
    \end{tikzpicture}
  \end{align}
  where
  \begin{align}
    \nonumber
    \Phi_1 &= \{\beta_1,\alpha\}\#_2(\beta_2\#_1s_1(\alpha)), & \Psi_1 &= (t_1(\beta)\#_1\alpha_2)\#_2\{\beta,\alpha_1\}, \\
    \nonumber
    \Phi_2 &= (\beta_2\#_1t_1(\alpha))\#_2\{\beta_1,\alpha\}, & \Psi_2 &= \{\beta,\alpha_2\}\#_2(s_1(\beta)\#_1\alpha_2).
  \end{align}
\end{itemize}

Returning to the definition, we require the following axioms:
\begin{enumerate}
  \setcounter{enumi}{\thetemp}
  \item Associativity: Vertical composition of 1-morphisms is associative, 
  \item Unit laws: For any 1-morphism $f:~x\to y$, the identities satisfy $f\#_1\id_x = \id_y\#_1f = f$, 
  \item \label{ax:Gray4-unit}Compatibility with unit:
    \begin{align}
      \{y,\id_f\} &= y\#_1f, &
      \{\id_g,x\} &= g\#_1x, \\
      \{\beta|\alpha,\id_{s_1(\alpha)}\} &= \id_{\{\beta,\alpha\}}, &
      \{\id_{t_1(\beta)},\beta|\alpha\} &= \id_{\{\beta,\alpha\}}, \\
      \{\beta|\id_{t_1(\alpha)},\alpha\} &= \id_{\{\beta,\alpha\}}, &
      \{\beta,\id_{s_1(\beta)}|\alpha\} &= \id_{\{\beta,\alpha\}}, \\
      \{\id_g|\alpha_2,\alpha_1\} &= \id_{\id_{g\#_1(\alpha_2\#_2\alpha_1)}}, &
      \{\beta_2,\beta_1|\id_f\} &= \id_{\id_{(\beta_2\#_2\beta_1)\#_1f}},
    \end{align}
  where $x$ and $y$ are 2-morphisms or 3-morphisms.
  \item \label{ax:Gray4-LiftingWhiskering}Compatibility with $\#_1$: The compatibility conditions between the liftings and 1-morphisms are\footnote{Since whiskering by 1-morphisms is a Gray functor, the compatibility with 33-Peiffer lifting and $\#_1$ is determined, i.e., $g\#_1\{\Psi,\Phi\}_{33} = \{g\#_1\Psi,g\#_1\Phi\}_{33}$ and $\{\Psi,\Phi\}_{33}\#_1f = \{\Psi\#_1f,\Phi\#_1f\}_{33}$.}
  \begin{align}
    \nonumber
    \{h\#_1\beta,\alpha\}_{22} &= h\#_1\{\beta,\alpha\}_{22}, \\
    \nonumber
    \{\beta\#_1g,\alpha\}_{22} &= \{\beta,g\#_1\alpha\}_{22}, \\
    \nonumber
    \{\beta,\alpha\#_1f\}_{22} &= \{\beta,\alpha\}_{22}\#_1f,
  \end{align}
  \begin{align}
    \nonumber
    \{f\#_1\beta',f\#_1\beta|\alpha\} &= f\#_1\{\beta',\beta|\alpha\}, &
    \{f\#_1\beta'|\alpha'\alpha\} &= f\#_1\{\beta|\alpha',\alpha\}, \\
    \nonumber
    \{\beta'\#_1f,\beta\#_1f|\alpha\} &= \{\beta',\beta|f\#_1\alpha\}, &
    \{\beta'\#_1f|\alpha',\alpha\} &= \{\beta'|f\#_1\alpha',f\#_1\alpha\}, \\
    \nonumber
    \{\beta',\beta|\alpha\#_1f\} &= \{\beta',\beta|\alpha\}\#_1f, & 
    \{\beta'|\alpha'\#_1f,\alpha\#_1f\} &= \{\beta'|\alpha',\alpha\}\#_1f,
  \end{align}
  \begin{align}
    \nonumber
    \{f\#_1\Psi,\alpha\}_{32} &= f\#_1\{\Psi,\alpha\}_{32}, &
    \{f\#_1\beta,\Phi\}_{23} &= f\#_1\{\beta,\Phi\}_{23}, \\
    \nonumber
    \{\Psi\#_1f,\alpha\}_{32} &= \{\Psi,f\#_1\alpha\}_{32}, &
    \{\beta\#_1f,\Phi\}_{23} &= \{\beta,f\#_1\Phi\}_{23} \\
    \nonumber
    \{\Psi,\alpha\#_1f\}_{32} &= \{\Psi,\alpha\}_{32}\#_1f, &
    \{\beta,\Phi\#_1f\}_{23} &= \{\beta,\Phi\}_{23}\#_1f,
  \end{align}
  where sources and targets are assumed to be compatible.
  \item \label{ax:Gray4-HomanRels}left-Homanian relation: 
  
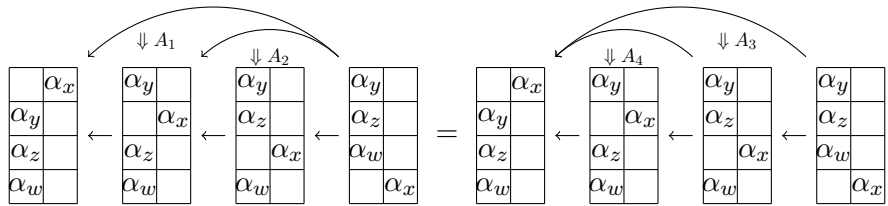
\[
    \begin{tikzpicture}[baseline=-2pt]
      \node (P) at (9/4,0) {\tikz[scale=0.45]{\strsquare{2}{4}{$\alpha_y$,,$\alpha_z$,,$\alpha_w$,,,$\alpha_x$}}};
      \node (Q) at (3/4,0) {\tikz[scale=0.45]{\strsquare{2}{4}{$\alpha_y$,,$\alpha_z$,,,$\alpha_x$,$\alpha_w$,}}};
      \node (R) at (-3/4,0) {\tikz[scale=0.45]{\strsquare{2}{4}{$\alpha_y$,,,$\alpha_x$,$\alpha_z$,,$\alpha_w$,}}};
      \node (S) at (-9/4,0) {\tikz[scale=0.45]{\strsquare{2}{4}{,$\alpha_x$,$\alpha_y$,,$\alpha_z$,,$\alpha_w$,}}};
      \draw [->] (P) to (Q);
      \draw [->] (Q) to (R);
      \draw [->] (R) to (S);
      \draw [->,bend right=1.5cm] (P.north west) to (R.north east);
      \draw [->,bend right=1.5cm] (P.north west) to (S.north east);
      \node at (R.north) [anchor=south,transform shape,scale=.66] {$\Downarrow A_1$};
      \node at (Q.north) [transform shape,scale=.66] {$\Downarrow A_2$};
    \end{tikzpicture}
    =
    \begin{tikzpicture}[baseline=-2pt]
      \node (P) at (9/4,0) {\tikz[scale=0.45]{\strsquare{2}{4}{$\alpha_y$,,$\alpha_z$,,$\alpha_w$,,,$\alpha_x$}}};
      \node (Q) at (3/4,0) {\tikz[scale=0.45]{\strsquare{2}{4}{$\alpha_y$,,$\alpha_z$,,,$\alpha_x$,$\alpha_w$,}}};
      \node (R) at (-3/4,0) {\tikz[scale=0.45]{\strsquare{2}{4}{$\alpha_y$,,,$\alpha_x$,$\alpha_z$,,$\alpha_w$,}}};
      \node (S) at (-9/4,0) {\tikz[scale=0.45]{\strsquare{2}{4}{,$\alpha_x$,$\alpha_y$,,$\alpha_z$,,$\alpha_w$,}}};
      \draw [->] (P) to (Q);
      \draw [->] (Q) to (R);
      \draw [->] (R) to (S);
      \draw [->,bend right=1.5cm] (Q.north west) to (S.north east);
      \draw [->,bend right=1.5cm] (P.north west) to (S.north east);
      \node at (R.north) [transform shape,scale=.66] {$\Downarrow A_4$};
      \node at (Q.north) [anchor=south,transform shape,scale=.66] {$\Downarrow A_3$};
    \end{tikzpicture}
  \]
  where
  \begin{align*}
    A_1 &= \{\alpha_w\#_2\alpha_z,\alpha_y|\alpha_x\}, &
    A_3 &= \{\alpha_w,\alpha_z\#_2\alpha_y|\alpha_x\}, \\
    A_2 &= \{\alpha_w,\alpha_z|\alpha_x\}\#_2(\alpha_y\#_1s_1(\alpha_x)), &
    A_4 &= (\alpha_w\#_1t_1(\alpha))\#_2\{\alpha_z,\alpha_y|\alpha_x\}.
  \end{align*}
  In terms of the formula, this can be written as \eqref{eq:left-Hom}.
  \item right-Homanian relation: The analogous relation to the diagram above holds; see \eqref{eq:right-Hom}.
  \item \label{ax:Gray4-KV}KV-polytope: 
  
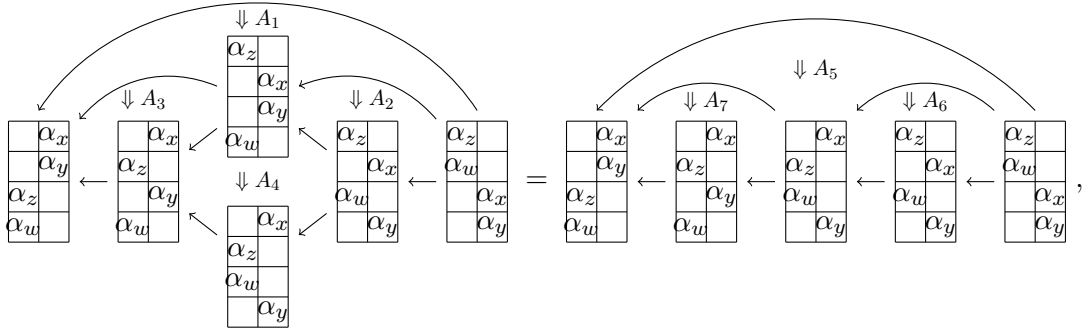
\[
    \begin{tikzpicture}[baseline=-2pt]
      \node (A) at (2.9,0) {\tikz[baseline=-2pt,scale=.4]{\strsquare{2}{4}{$\alpha_z$,,$\alpha_w$,,,$\alpha_x$,,$\alpha_y$}}};
      \node (B) at (1.45,0) {\tikz[baseline=-2pt,scale=.4]{\strsquare{2}{4}{$\alpha_z$,,,$\alpha_x$,$\alpha_w$,,,$\alpha_y$}}};
      \node (C1) at (0,1.125) {\tikz[baseline=-2pt,scale=.4]{\strsquare{2}{4}{$\alpha_z$,,,$\alpha_x$,,$\alpha_y$,$\alpha_w$,}}};
      \node (C2) at (0,-1.125) {\tikz[baseline=-2pt,scale=.4]{\strsquare{2}{4}{,$\alpha_x$,$\alpha_z$,,$\alpha_w$,,,$\alpha_y$}}};
      \node (D) at (-1.45,0) {\tikz[baseline=-2pt,scale=.4]{\strsquare{2}{4}{,$\alpha_x$,$\alpha_z$,,,$\alpha_y$,$\alpha_w$,}}};
      \node (E) at (-2.9,0) {\tikz[baseline=-2pt,scale=.4]{\strsquare{2}{4}{,$\alpha_x$,,$\alpha_y$,$\alpha_z$,,$\alpha_w$,}}};
      \draw [->] (A) to (B);
      \draw [->] (B) to (C1);
      \draw [->] (C1) to (D);
      \draw [->] (B) to (C2);
      \draw [->] (C2) to (D);
      \draw [->] (D) to (E);
      \draw [->,bend right=1.25cm] (A) to node (P) {} (C1);
      \draw [->,bend right=1.25cm] (C1) to node (Q) {} (E);
      \draw [->,bend right=2cm] (A.north) to node (R) {} (E.north);
      \path (C1.south) to node [transform shape,scale=.75] {$\Downarrow A_4$} (C2.north);
      \path (Q) to node [transform shape,scale=.75] {$\Downarrow A_3$} (D);
      \path (P) to node [transform shape,scale=.75] {$\Downarrow A_2$} (B);
      \path (R) to node [transform shape,scale=.75] {$\Downarrow A_1$} (C1);
    \end{tikzpicture}
    =
    \begin{tikzpicture}[baseline=-2pt]
      \node (A) at (2.9,0) {\tikz[baseline=-2pt,scale=.4]{\strsquare{2}{4}{$\alpha_z$,,$\alpha_w$,,,$\alpha_x$,,$\alpha_y$}}};
      \node (B) at (1.45,0) {\tikz[baseline=-2pt,scale=.4]{\strsquare{2}{4}{$\alpha_z$,,,$\alpha_x$,$\alpha_w$,,,$\alpha_y$}}};
      \node (C) at (0,0) {\tikz[baseline=-2pt,scale=.4]{\strsquare{2}{4}{,$\alpha_x$,$\alpha_z$,,$\alpha_w$,,,$\alpha_y$}}};
      \node (D) at (-1.45,0) {\tikz[baseline=-2pt,scale=.4]{\strsquare{2}{4}{,$\alpha_x$,$\alpha_z$,,,$\alpha_y$,$\alpha_w$,}}};
      \node (E) at (-2.9,0) {\tikz[baseline=-2pt,scale=.4]{\strsquare{2}{4}{,$\alpha_x$,,$\alpha_y$,$\alpha_z$,,$\alpha_w$,}}};
      \draw [->] (A) to (B);
      \draw [->] (B) to (C);
      \draw [->] (C) to (D);
      \draw [->] (D) to (E);
      \draw [->,bend right=1.6cm] (A.north) to node (F) {} (E.north);
      \draw [->,bend right=1.5cm] (A.north west) to node (A) {} (C.north east);
      \draw [->,bend right=1.5cm] (C.north west) to node (E) {} (E.north east);
      \path (A) to node [transform shape,scale=.75] {$\Downarrow A_6$} (B);
      \path (E) to node [transform shape,scale=.75] {$\Downarrow A_7$} (D);
      \path (F) to node [transform shape,scale=.75] {$\Downarrow A_5$} (C);
    \end{tikzpicture},
  \]
  where 
  \begin{align}
    \nonumber
    A_1 &= \{\alpha_w,\alpha_z|\alpha_y\#_2\alpha_x\}, &
    A_5 &= \{\alpha_w\#_2\alpha_z|\alpha_y,\alpha_x\}, \\
    \nonumber
    A_2 &= \{\alpha_w|\alpha_y,\alpha_x\}\#_2(\alpha_z\#_1s_1(\alpha_x)), &
    A_6 &= (t_1(\alpha_w)\#_1\alpha_y)\#_2\{\alpha_w\#_2\alpha_z|\alpha_x\}, \\
    \nonumber
    A_3 &= (\alpha_w\#_1t_1(\alpha_y))\#_2\{\alpha_z|\alpha_y,\alpha_x\}, &
    A_7 &= \{\alpha_w,\alpha_z|\alpha_y\}\#_2(s_1(\alpha_z)\#_1\alpha_x), \\
    \nonumber
    A_4 &= \{\{\alpha_w,\alpha_y\},\{\alpha_z,\alpha_x\}\}.
  \end{align}
  See \eqref{eq:KV-poly}.
  \item \label{ax:Gray4-Prism}Prism 3-22: 
  
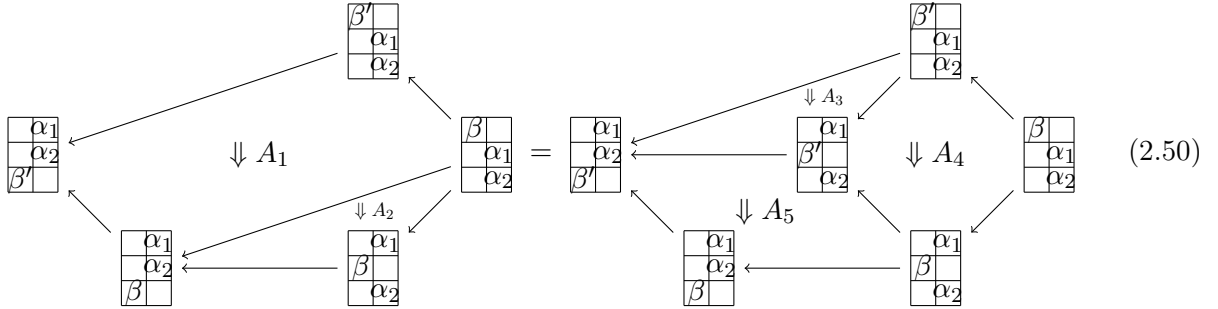
\begin{equation}
    \begin{tikzpicture}[baseline=-2pt]
      \node (T) at (-3,0) {\tikz{\strsquare[scale=0.33]{2}{3}{,$\alpha_1$,,$\alpha_2$,$\beta'$,}}};
      \node (S) at (-1.5,-1.5) {\tikz{\strsquare[scale=0.33]{2}{3}{,$\alpha_1$,,$\alpha_2$,$\beta$,}}};
      \node (Q1) at (1.5,1.5) {\tikz{\strsquare[scale=0.33]{2}{3}{$\beta'$,,,$\alpha_1$,,$\alpha_2$}}};
      \node (Q2) at (1.5,-1.5) {\tikz{\strsquare[scale=0.33]{2}{3}{,$\alpha_1$,$\beta$,,,$\alpha_2$}}};
      \node (P) at (3,0) {\tikz{\strsquare[scale=0.33]{2}{3}{$\beta$,,,$\alpha_1$,,$\alpha_2$}}};
      \draw [->] (P) to (Q1);
      \draw [->] (P) to (Q2);
      \draw [->] (P) to (S);
      \draw [->] (Q1) to (T);
      \draw [->] (Q2) to (S);
      \draw [->] (S) to (T);
      \path (Q1) to [pos=45/160] node (Q1T) {} (T);
      \path (P) to [pos=115/160] node (PS1) {} (S);
      \path (P) to [pos=45/160] node (PS2) {} (S);
      \path (Q1T) to node {$\Downarrow A_1$} (PS1);
      \path (PS2) to node [transform shape,scale=.66] {$\Downarrow A_2$} (Q2);
    \end{tikzpicture}
    =
    \begin{tikzpicture}[baseline=-2pt]
      \node (T) at (-3,0) {\tikz{\strsquare[scale=0.33]{2}{3}{,$\alpha_1$,,$\alpha_2$,$\beta'$,}}};
      \node (S) at (-1.5,-1.5) {\tikz{\strsquare[scale=0.33]{2}{3}{,$\alpha_1$,,$\alpha_2$,$\beta$,}}};
      \node (R) at (0,0) {\tikz{\strsquare[scale=0.33]{2}{3}{,$\alpha_1$,$\beta'$,,,$\alpha_2$}}};
      \node (Q1) at (1.5,1.5) {\tikz{\strsquare[scale=0.33]{2}{3}{$\beta'$,,,$\alpha_1$,,$\alpha_2$}}};
      \node (Q2) at (1.5,-1.5) {\tikz{\strsquare[scale=0.33]{2}{3}{,$\alpha_1$,$\beta$,,,$\alpha_2$}}};
      \node (P) at (3,0) {\tikz{\strsquare[scale=0.33]{2}{3}{$\beta$,,,$\alpha_1$,,$\alpha_2$}}};
      \draw [->] (P) to (Q1);
      \draw [->] (P) to (Q2);
      \draw [->] (Q1) to (R);
      \draw [->] (Q1) to (T);
      \draw [->] (Q2) to (R);
      \draw [->] (Q2) to (S);
      \draw [->] (R) to (T);
      \draw [->] (S) to (T);
      \path (Q1) to [pos=45/160] node (Q1T) {} (T);
      \path (R) to [pos=18/160] node (RT) {} (T);
      \path (Q2) to [pos=142/160] node (Q2S) {} (S);
      \path (Q1T) to node [transform shape,scale=.66] {$\Downarrow A_3$} (R);
      \path (Q1) to node {$\Downarrow A_4$} (Q2);
      \path (RT) to node {$\Downarrow A_5$} (Q2S);
    \end{tikzpicture}
  \end{equation}
  where
  \begin{align*}
    A_1 &= \{\Psi,\alpha_2\#_2\alpha_1\}, &
    A_3 &= \{t_2(\Psi)|\alpha_2,\alpha_1\}, \\
    A_2 &= \{s_2(\Psi)|\alpha_2,\alpha_1\}, &
    A_4 &= (t_1(\Psi)\#_1\alpha_2)\#_2\{\Psi,\alpha_1\}, \\
    &&
    A_5 &= \{\Psi,\alpha_2\}\#_2(s_1(\Psi)\#_1\alpha_1).
  \end{align*}
  See \eqref{eq:Prism3|22}.
  The classification number 3-22 encodes the structure of the Peiffer lifting relation $\{\Psi,\alpha_2\#_2\alpha_1\}$: the first digit 3 indicates that the first slot is occupied by a 3-morphism, while 22 indicates that the second slot is occupied by the vertical composition of two 2-morphisms.
  For 3-22, 23-2, 22-3, 2-32 and 2-23, similar relations hold; see \eqref{eq:Prism32|2}--\eqref{eq:Prism2|23}.
  \item \label{ax:Gray4-Cube}Cube: 
  \[
    \begin{tikzpicture}[baseline=-2pt]
      \node (A) at (3,0) {\tikz{\strrect[scale=0.45]{2}{$\beta$,,,$\alpha$}}};
      \node (B1) at (1,1) {\tikz{\strrect[scale=0.45]{2}{$\beta$,,,$\alpha'$}}};
      \node (B2) at (1,-1) {\tikz{\strrect[scale=0.45]{2}{,$\alpha$,$\beta$,}}};
      \node (C1) at (-1,1) {\tikz{\strrect[scale=0.45]{2}{$\beta'$,,,$\alpha'$}}};
      \node (C2) at (-1,-1) {\tikz{\strrect[scale=0.45]{2}{,$\alpha'$,$\beta$,}}};
      \node (D) at (-3,0) {\tikz{\strrect[scale=0.45]{2}{,$\alpha'$,$\beta'$,}}};
      \draw [->] (A) to (B1);
      \draw [->] (A) to (B2);
      \draw [->] (B1) to (C1);
      \draw [->] (B2) to (C2);
      \draw [->] (C1) to (D);
      \draw [->] (C2) to (D);
      \draw [->] (B1) to (C2);
      \path (B1) to node {$\Downarrow A_2$} (B2);
      \path (C1) to node {$\Downarrow A_1$} (C2);
    \end{tikzpicture}
    = 
    \begin{tikzpicture}[baseline=-2pt]
      \node (A) at (3,0) {\tikz{\strrect[scale=0.45]{2}{$\beta$,,,$\alpha$}}};
      \node (B1) at (1,1.5) {\tikz{\strrect[scale=0.45]{2}{$\beta$,,,$\alpha'$}}};
      \node (B2) at (1,0) {\tikz{\strrect[scale=0.45]{2}{$\beta'$,,,$\alpha$}}};
      \node (B3) at (1,-1.5) {\tikz{\strrect[scale=0.45]{2}{,$\alpha$,$\beta$,}}};
      \node (C1) at (-1,1.5) {\tikz{\strrect[scale=0.45]{2}{$\beta'$,,,$\alpha'$}}};
      \node (C2) at (-1,0) {\tikz{\strrect[scale=0.45]{2}{,$\alpha$,$\beta'$,}}};
      \node (C3) at (-1,-1.5) {\tikz{\strrect[scale=0.45]{2}{,$\alpha'$,$\beta$,}}};
      \node (D) at (-3,0) {\tikz{\strrect[scale=0.45]{2}{,$\alpha'$,$\beta'$,}}};
      \draw [->] (A) to (B1);
      \draw [->] (A) to (B2);
      \draw [->] (A) to (B3);
      \draw [->] (B1) to (C1);
      \draw [->] (B2) to (C2);
      \draw [->] (B3) to (C3);
      \draw [->] (C1) to (D);
      \draw [->] (C2) to (D);
      \draw [->] (C3) to (D);
      \draw [->] (B2) to (C1);
      \draw [->] (B3) to (C2);
      \path (B1) to node {$\Downarrow A_3$} (B2);
      \path (C1) to node {$\Downarrow A_4$} (C2);
      \path (B2) to node {$\Downarrow A_5$} (B3);
      \path (C2) to node {$\Downarrow A_6$} (C3);
    \end{tikzpicture}
  \]
  where
  \begin{align*}
    A_1 &= \{\Psi,t_2(\Phi)\}, &
    A_3 &= \{t_1(\Psi)\#_1\Phi,\Psi\#_1s_1(\Phi)\},\\
    A_2 &= \{s_2(\Psi),\Phi\}, &
    A_4 &= \{t_2(\Psi),\Phi\}, \\
    &&
    A_5 &= \{\Psi,s_2(\Phi)\}, \\
    &&
    A_6 &= \{\Psi\#_1t_1(\Phi),s_1(\Psi)\#_1\Phi\}.
  \end{align*}
  See \eqref{eq:Cube}.
  \item \label{ax:Gray4-Pasting}Pasting 33-2:
  \begin{equation}
    \label{fig:pasting-33-2}
    \begin{tikzpicture}[baseline=-2pt]
      \node (A) at (2,0) {\tikz{\strrect[scale=0.45]{2}{$\beta$,,,$\alpha$}}};
      \node (B1) at (0,1) {\tikz{\strrect[scale=0.45]{2}{$\beta''$,,,$\alpha$}}};
      \node (B2) at (0,-1) {\tikz{\strrect[scale=0.45]{2}{,$\alpha$,$\beta$,}}};
      \node (C) at (-2,0) {\tikz{\strrect[scale=0.45]{2}{,$\alpha$,$\beta''$,}}};
      \draw [->] (A) to (B1);
      \draw [->] (A) to (B2);
      \draw [->] (B1) to (C);
      \draw [->] (B2) to (C);
      \path (B1) to node {$\Downarrow\{\Psi'\#_3\Psi,\alpha\}$} (B2);
    \end{tikzpicture}
    =
    \begin{tikzpicture}[baseline=-2pt]
      \node (A) at (3,0) {\tikz{\strrect[scale=0.45]{2}{$\beta$,,,$\alpha$}}};
      \node (B1) at (1,1) {\tikz{\strrect[scale=0.45]{2}{$\beta'$,,,$\alpha$}}};
      \node (B2) at (1,-1) {\tikz{\strrect[scale=0.45]{2}{,$\alpha$,$\beta$,}}};
      \node (C1) at (-1,1) {\tikz{\strrect[scale=0.45]{2}{$\beta''$,,,$\alpha$}}};
      \node (C2) at (-1,-1) {\tikz{\strrect[scale=0.45]{2}{,$\alpha$,$\beta'$,}}};
      \node (D) at (-3,0) {\tikz{\strrect[scale=0.45]{2}{,$\alpha$,$\beta''$,}}};
      \draw [->] (A) to (B1);
      \draw [->] (A) to (B2);
      \draw [->] (B1) to (C1);
      \draw [->] (B2) to (C2);
      \draw [->] (C1) to (D);
      \draw [->] (C2) to (D);
      \draw [->] (B1) to (C2);
      \path (B1) to node {$\Downarrow\{\Psi,\alpha\}$} (B2);
      \path (C1) to node {$\Downarrow\{\Psi',\alpha\}$} (C2);
    \end{tikzpicture}
  \end{equation}
  See \eqref{eq:Pasting332}. For 2-33, a similar relation holds; see \eqref{eq:Pasting233}.
  \item \label{ax:Gray4-Tube}Tube 1: 
  \[
    \begin{tikzpicture}[baseline=-2pt]
      \node (A) at (2,0) {\tikz{\strrect[scale=0.45]{2}{$\beta$,,,$\alpha$}}};
      \node (B1) at (0,1) {\tikz{\strrect[scale=0.45]{2}{$\beta'$,,,$\alpha$}}};
      \node (B2) at (0,-1) {\tikz{\strrect[scale=0.45]{2}{,$\alpha$,$\beta$,}}};
      \node (C) at (-2,0) {\tikz{\strrect[scale=0.45]{2}{,$\alpha$,$\beta'$,}}};
      \draw [->,bend left=0.5cm] (A) to node (P) {} (B1);
      \draw [->,bend right=0.5cm] (A) to node (Q) {} (B1);
      \draw [->] (A) to (B2);
      \draw [->] (B1) to (C);
      \draw [->] (B2) to (C);
      \path (B1) to node {$\Downarrow\{t_3(B),\alpha\}$} (B2);
      \path (Q) to node [transform shape,scale=.66] (A) {$\Downarrow$} (P);
      \node [transform shape,scale=.66] (D) at (5/4,5/4) {$\begin{bmatrix}B & \\ &\alpha\end{bmatrix}$};
    \end{tikzpicture}
    =
    \begin{tikzpicture}[baseline=-2pt]
      \node (A) at (2,0) {\tikz{\strrect[scale=0.45]{2}{$\beta$,,,$\alpha$}}};
      \node (B1) at (0,1) {\tikz{\strrect[scale=0.45]{2}{$\beta'$,,,$\alpha$}}};
      \node (B2) at (0,-1) {\tikz{\strrect[scale=0.45]{2}{,$\alpha$,$\beta$,}}};
      \node (C) at (-2,0) {\tikz{\strrect[scale=0.45]{2}{,$\alpha$,$\beta'$,}}};
      \draw [->] (A) to (B1);
      \draw [->] (A) to (B2);
      \draw [->] (B1) to (C);
      \draw [->,bend right=0.5cm] (B2) to node (P) {} (C);
      \draw [->,bend left=0.5cm] (B2) to node (Q) {} (C);
      \path (B1) to node {$\Downarrow \{s_3(B),\alpha\}$} (B2);
      \path (P) to node [transform shape,scale=.66] (A) {$\Downarrow$} (Q);
      \node [transform shape,scale=.66] (D) at (-5/4,-5/4) {$\begin{bmatrix}& \alpha \\ B &\end{bmatrix}$};
    \end{tikzpicture}
  \]
  See \eqref{eq:Tube1}. A similar relation holds for $\begin{pmatrix}\beta & \\ & A\end{pmatrix}$; see \eqref{eq:Tube2}.
  \item \label{ax:Gray4-SpSm}$S^+=S^-$ condition:
  \[
    \begin{tikzpicture}[baseline=-2pt]
      \node (P) at (3,0) {\tikz[scale=0.45]{\strrect{3}{$\alpha_z$,,,,$\alpha_y$,,,,$\alpha_x$}}};
      \node (Q1) at (1,1.5) {\tikz[scale=0.45]{\strrect{3}{$\alpha_z$,,,,,$\alpha_x$,,$\alpha_y$,}}};
      \node (Q2) at (1,-1.5) {\tikz[scale=0.45]{\strrect{3}{,$\alpha_y$,,$\alpha_z$,,,,,$\alpha_x$}}};
      \node (R1) at (-1,1.5) {\tikz[scale=0.45]{\strrect{3}{,,$\alpha_x$,$\alpha_z$,,,,$\alpha_y$,}}};
      \node (R2) at (-1,-1.5) {\tikz[scale=0.45]{\strrect{3}{,$\alpha_y$,,,,$\alpha_x$,$\alpha_z$,,}}};
      \node (S) at (-3,0) {\tikz[scale=0.45]{\strrect{3}{,,$\alpha_x$,,$\alpha_y$,,$\alpha_z$,,}}};
      \draw [->] (P) to (Q1);
      \draw [->] (P) to (Q2);
      \draw [->] (Q1) to coordinate (S1) (R1);
      \draw [->] (Q2) to coordinate (S2) (R2);
      \draw [->] (R1) to (S);
      \draw [->] (R2) to (S);
      \draw [->,bend left=1cm] (P) to (R1);
      \draw [->,bend right=1cm] (Q2) to (S);
      \path (Q1.north) to [pos=1/2] coordinate (A) (Q2.south);
      \path (Q1) to node {$\Downarrow A_1$} (A);
      \path (R1.north) to [pos=1/2] coordinate (B) (R2.south);
      \path (B) to node {$\Downarrow A_3$} (R2);
      \path (S1) to [pos=2/5] coordinate (Sn) (S2);
      \path (S1) to [pos=3/5] coordinate (Ss) (S2);
      \path (Sn) to node {$\Downarrow A_2$} (Ss);
    \end{tikzpicture}
    =
    \begin{tikzpicture}[baseline=-2pt]
      \node (P) at (3,0) {\tikz[scale=0.45]{\strrect{3}{$\alpha_z$,,,,$\alpha_y$,,,,$\alpha_x$}}};
      \node (Q1) at (1,1.5) {\tikz[scale=0.45]{\strrect{3}{$\alpha_z$,,,,,$\alpha_x$,,$\alpha_y$,}}};
      \node (Q2) at (1,-1.5) {\tikz[scale=0.45]{\strrect{3}{,$\alpha_y$,,$\alpha_z$,,,,,$\alpha_x$}}};
      \node (R1) at (-1,1.5) {\tikz[scale=0.45]{\strrect{3}{,,$\alpha_x$,$\alpha_z$,,,,$\alpha_y$,}}};
      \node (R2) at (-1,-1.5) {\tikz[scale=0.45]{\strrect{3}{,$\alpha_y$,,,,$\alpha_x$,$\alpha_z$,,}}};
      \node (S) at (-3,0) {\tikz[scale=0.45]{\strrect{3}{,,$\alpha_x$,,$\alpha_y$,,$\alpha_z$,,}}};
      \draw [->] (P) to (Q1);
      \draw [->] (P) to (Q2);
      \draw [->] (Q1) to coordinate (S1) (R1);
      \draw [->] (Q2) to coordinate (S2) (R2);
      \draw [->] (R1) to (S);
      \draw [->] (R2) to (S);
      \draw [->,bend left=1cm] (Q1) to (S);
      \draw [->,bend right=1cm] (P) to (R2);
      \path (Q1.north) to [pos=1/2] coordinate (A) (Q2.south);
      \path (A) to node {$\Downarrow A_6$} (Q2);
      \path (R1.north) to [pos=1/2] coordinate (B) (R2.south);
      \path (R1) to node {$\Downarrow A_4$} (B);
      \path (S1) to [pos=2/5] coordinate (Sn) (S2);
      \path (S1) to [pos=3/5] coordinate (Ss) (S2);
      \path (Sn) to node {$\Downarrow A_5$} (Ss);
    \end{tikzpicture}
  \]
  where
  \begin{align*}
    A_1 &= \{t_1(\alpha_z)\#_1\alpha_y,\alpha_z\#_1s_1(\alpha_y)|\alpha_x\}^{-1_4}, &
    A_4 &= \{\alpha_z|\alpha_y\#_1t_1(\alpha_x),s_1(\alpha_y)\#_1\alpha_x\}^{-1_4},\\
    A_2 &= \{\{\alpha_z,\alpha_y\},\alpha_x\}^{-1_4}, &
    A_5 &= \{\alpha_z,\{\alpha_y,\alpha_x\}\}, \\
    A_3 &= \{\alpha_z\#_1t_1(\alpha_y),s_1(\alpha_z)\#_1\alpha_y|\alpha_x\}, &
    A_6 &= \{\alpha_z|t_1(\alpha_y)\#_1\alpha_x,\alpha_y\#_1s_1(\alpha_x)\}.
  \end{align*}
  See \eqref{eq:SpSm}. 
\end{enumerate}
\end{definition}
In particular, when all 1-, 2-, 3- and 4-morphisms of a Gray 4-category with a single object are invertible, we call the resulting 4-category a \textbf{Gray 4-group}.

\begin{lemma}
The following normalization conditions hold.
\[
  \{\id_\beta,\alpha\}_{32} = \{\beta,\id_\alpha\}_{23} = \id_{\{\beta,\alpha\}_{22}},
\]
\end{lemma}
This can be shown by setting the appropriate 3-morphisms in the Pasting relations \ref{ax:Gray4-Pasting} to identity 3-morphisms.

\begin{definition}\label{def:Gray4Functor}
Let $\cC$ and $\cD$ be Gray 4-categories. A Gray 4-functor $F:~\cC\to\cD$ is a collection of maps $F_i:~\mor_i(\cC)\to\mor_i(\cD), i=0,\ldots,4$, preserving all the structures of Gray 4-categories, i.e., the identity, source and target maps, vertical compositions, whiskerings, and liftings.
\end{definition}

\section{Equivalence between 3-crossed modules and Gray 4-groups}\label{sec:main}
Let $\Gray$ be the category of Gray 4-groups with Gray 4-functors as morphisms, and let $\TXMod$ be the category of 3-crossed modules with 4-homomorphisms between them as morphisms. In this section, we show the equivalence of categories $\Gray \cong \TXMod$. First, we define a functor $\Delta:~\Gray\to \TXMod$.

\subsection{From Gray 4-groups to 3-crossed modules}\label{sec:Gray323CM}
Given a Gray 4-group ${\bf A_*}\in \Gray$, let us define the sets $G,H,L$ and $M$ as follows:
\begin{align*}
  G &:= \mor_1({\bf A_*}), &
  L &:= \{l\in\mor_3({\bf A_*})|s_2(l) = \id_{\id_*}\}, \\
  H &:= \{h\in\mor_2({\bf A_*})|s_1(h) = \id_*\}, &
  M &:= \{m\in\mor_4({\bf A_*})|s_3(m) = \id_{\id_{\id_*}}\}.
\end{align*}
We also denote them by $G_i,\ (i=1,2,3,4)$, where the subscript indicates the degree of the corresponding morphisms. These sets are groups with respect to the following operations. For the set $G$, 
\begin{align*}
  g_2g_1 &= g_2\#_1g_1,  & \forall g_1,g_2\in G, \\
  g^{-1} &= g^{-1_1}, & \forall g\in G.
\end{align*}
For the sets $H,L$ and $M$, 
\begin{align*}
  yx &= (y\#_{i-1}t_{i-1}(x))\#_ix & \forall x,y\in G_i, \\
  x^{-1} &= x^{-1_i}\#_{i-1}t_{i-1}(x)^{-1_{i-1}}, & \forall x\in G_i,
\end{align*}
where the subscript attached to $-1$ indicates the degree of the morphism. Indeed, one can prove associativity, the existence of identities and the inverse law. For the set $G$, the proof is well known. For the sets $G_i,\ i=2,3,4$, the proof is straightforward. For any $x,y,z\in G_i,\ i=2,3,4$, 
\begin{align*}
  (zy)x
  &= ([(z\#_{i-1}t_{i-1}(y))\#_iy]\#_{i-1}t_{i-1}(x))\#_ix, & 
  &\text{by definition}, \\
  &= ([z\#_{i-1}(t_{i-1}(y)\#_{i-1}t_{i-1}(x))]\#_i[y\#_{i-1}t_{i-1}(x)])\#_ix, &
  &\text{whiskering}, \\
  &= [z\#_{i-1}(t_{i-1}(y\#_{i-1}t_{i-1}(x)))]\#_i([y\#_{i-1}t_{i-1}(x)]\#_ix), &
  &\text{associativity of $\#_1$}, \\
  &= z(yx), &
  &\text{by definition}.
\end{align*}
The identity elements are the identity morphisms $e_H = \id_{\id_*}, e_L = \id_{\id_{\id_*}}$ and $e_M = \id_{\id_{\id_{\id_*}}}$. The inverse law is
\begin{align*}
  x^{-1}x 
  &= ([x^{-1_i}\#_{i-1}t_{i-1}(x)^{-1_{i-1}}]\#_{i-1}t_{i-1}(x))\#_ix \\
  &= (x^{-1_i}\#_{i-1}[t_{i-1}(x)^{-1_{i-1}}\#_{i-1}t_{i-1}(x)])\#_ix \\
  &= (x^{-1_i}\#_{i-1}\id_{s_{i-2}(x)})\#_ix \\
  &= x^{-1_i}\#_ix \\
  &= \id_{s_{i-1}(x)} = e_{G_i}
\end{align*}
The proof for $xx^{-1}$ is the same.

For $x\in G_{i+1},\ i=1,2,3$, let us define the boundary maps as $\partial_i:~G_{i+1}\ni x \mapsto t_i(x)\in G_i$. Applying the map to $x$ twice, it maps to the identity:
\begin{align}
  \label{eq:p1p2=e}
  \partial_1\circ\partial_2(l) &= t_1(l) = \id_* = e_G, &\forall l&\in L, \\
  \label{eq:p2p3=e}
  \partial_2\circ\partial_3(m) &= t_2(m) = \id_{\id_*} = e_H, &\forall m&\in M.
\end{align}
In addition, one can show that the maps are group isomorphisms. For any $x,y\in G_{i+1}$, 
\begin{align*}
  \partial_i(yx) 
  &= t_i((y\#_it_i(x))\#_{i+1}x), \\
  &= t_i(y\#_it_i(x)), \\
  &= t_i(y)\#_it_i(x), \\
  &= (t_i(y)\#_{i-1}t_{i-1}(x))\#_it_i(x), \\
  &= (\partial_iy)(\partial_ix).
\end{align*}

Let us define the operations $x\triangleright y$ as follows:
\begin{align*}
  \act{x}y &= x\#_1y\#_1x^{-1}, & \forall x\in G_1 &\text{ and }\forall y\in G_j,~j=2,3,4, \\
  \act{x}y &= (x\#_iy\#_ix^{-1_i})\#_{i-1}t_{i-1}(x)^{-1_{i-1}}, & \forall x\in G_i &\text{ and }\forall y\in G_j,~i<j,i=2,3.
\end{align*}
These are elements of the group $G_j$. For example, by a straightforward calculation, one can show that $s_2(\act{h}l)=\id_{\id_*}$ and $t_1(\act{h}l)=\id_*$ for $\forall h\in H$ and $\forall l\in L$:
\begin{align*}
  s_2(\act{h}l) 
  &= s_2((h\#_2l\#_2h^{-1_2})\#_1t_1(h)^{-1_1}), &
  t_1(\act{h}l)
  &= t_1(h\#_2t_2(l)\#_2(h^{-1_{\#_2}}))\#_1t_1(h)^{-1_1}, \\
  &= (h\#_2s_2(l)\#_2h^{-1_2})\#_1t_1(h)^{-1_1}, &
  &= t_1(h)\#_1t_1(h)^{-1_1}, \\
  &= (h\#_2\id_{\id_*}\#_2h^{-1_2})\#_1t_1(h)^{-1_1}, &
  &= \id_*. \\
  &= \id_{t_1(h)}\#_1t_1(h)^{-1_1}, \\
  &= \id_{\id_*},
\end{align*}
The calculation is the same for the other operations. In addition, one can show that
\begin{align}
  \label{eq:GtoT-action}
  \act{e_i}y &= y, & \act{x}(y_2y_1) &= \act{x}(y_2)\act{x}(y_1), & \act{x_2x_1}y &= \act{x_2}(\act{x_1}y).
\end{align}
Indeed, for $\forall y\in G_j$,
\begin{align*}
  \act{e_G}y &= \id_*\#_1y\#_1\id_* = y, &
  \act{e_i}y &= (\id_i\#_1y\#_1\id_i)\#_{i-1}\id_{i-1} = y, & i&=2,3,4,
\end{align*}
where we denote the identity $\id_{\id_{\ldots_{\id_*}}}$ by $\id_i$. The proof of the second equation in \eqref{eq:GtoT-action} differs slightly depending on the degrees $i$ and $j$, but the various cases are essentially the same. The most complicated cases are $(i,j)=(2,3)$ and $(3,4)$: 
\begin{align*}
  \act{x}(y_2y_1) 
  &= (x\#_i[(y_2\#_it_i(y_1))\#_{i+1}y_1]\#_ix^{-1_i})\#_{i-1}t_{i-1}(x)^{-1_{i-1}} \\
  &= [(x\#_iy_2\#_it_i(y_1\#_ix^{-1_i}))\#_{i+1}(x\#_iy_1\#_ix^{-1_i})]\#_{i-1}t_{i-1}(x)^{-1_{i-1}} \\
  &= [(x\#_iy_2\#_ix^{-1_i}\#_ix\#_it_i(y_1\#_ix^{-1_i}))\#_{i+1}(x\#_iy_1\#_ix^{-1_i})]\#_{i-1}t_{i-1}(x)^{-1_{i-1}} \\
  &= 
  [
    ((x\#_iy_2\#_ix^{-1_i})\#_{i-1}t_{i-1}(x)^{-1_{i-1}})
    \#_i
    t_i((x\#_iy_1\#_ix^{-1_i})\#_{i-1}t_{i-1}(x)^{-1_{i-1}})
  ] \\
  &\hspace{1em} 
  \#_{i+1}
  [
    (x\#_iy_1\#_ix^{-1_i})\#_{i-1}t_{i-1}(x)^{-1_{i-1}}
  ] \\
  &= \act{x}y_2\act{x}y_1
\end{align*}
For $i=1$ case, the third equation in \eqref{eq:GtoT-action} is
\begin{align*}
  \act{x_2x_1}y &= (x_2x_1)\#_1y\#_1(x_2x_1)^{-1}, \\
  &= x_2\#_1(x_1\#_1y\#_1x_1^{-1})\#_1x_2^{-1}, \\
  &= \act{x_2}(\act{x_1}y).
\end{align*}
For the cases $i=2,3$ and $4$,
\begin{align*} 
  \act{x_2x_1}y &= 
  [
    ((x_2\#_{i-1}t_{i-1}(x_1))\#_ix_1)
    \#_i
    y
    \#_i
    ((x_2\#_{i-1}t_{i-1}(x_1))\#_ix_1)^{-1_i}
  ] \\
  &\qquad 
  \#_it_{i-1}((x_2\#_{i-1}t_{i-1}(x_1))\#_ix_1)^{-1_{i-1}}, \\
  &= 
  [
    ((x_2\#_{i-1}t_{i-1}(x_1))
    \#_i
    (x_1\#_iy\#_ix_1^{-1_i})
    \#_i
    (x_2^{-1_i}\#_{i-1}t_{i-1}(x_1)))
  ] \\
  &\qquad 
  \#_i[t_{i-1}(x_1)^{-1_{i-1}}\#_{i-1}t_{i-1}(x_2)^{-1_{i-1}}], \\
  &= 
  [
    x_2
    \#_i
    [(x_1\#_iy\#_ix_1^{-1_i})\#_{i-1}t_{i-1}(x_1)^{-1_{i-1}}]
    \#_i
    x_2^{-1_i}
  ]\#_{i-1}t_{i-1}(x_2)^{-1_{i-1}}, \\
  &= \act{x_2}(\act{x_1}y).
\end{align*}
Thus, the operations $x\triangleright y$ are left actions. 

As can be seen from these calculations, we obtain a useful formula that will be used throughout. Let $F:~\mor_k({\bf A}_*)\times \mor_l({\bf A}_*)\times \ldots \to \mor_i({\bf A}_*),~k,l,\cdots \geq i$ be an $i$-morphism $G_i,i=2,3,4$ built from the operations $s_k,t_k$ and $\#_k,(-)^{-1_k},k=i,i+1,\ldots$. Then, for any $i$-morphism $f\in G_i$, the following formula holds:
\[
  f\#_iF(-,-,\ldots)\#_if^{-1_i} = F(f\#_i(-)\#_if^{-1_i},f\#_i(-)\#_if^{-1_i},\ldots).
\]
This follows from functoriality when $k>i$. Even when $F$ involves $k=i$, the formula still holds, since we now have inverses:
\begin{align*}
  x\#_i(y\#_iy')\#x^{-1_i} &= (x\#_iy\#_ix^{-1_i})\#_i(x\#_iy'\#x^{-1_i}), \\
  x\#_iy^{-1_i}\#_ix^{-1_i} &= (x\#_iy\#_ix^{-1_i})^{-1_i},
\end{align*}
where $y$ and $y'$ are $i$-morphisms in $G_i$, and we have used the associativity of the vertical composition $\#_i$. In what follows, this formula will be used frequently without further mention.

For our purposes, we define the group liftings in terms of the categorical liftings. To avoid confusion, we attach the subscript $c$ to the categorical liftings.
\begin{center}
\begin{minipage}{.495\columnwidth}
\begin{itemize}
  \item 22-Peiffer lifting:
  \[
    \{h_2,h_1\} := \begin{bmatrix}& (\act{\partial h_2}h_1h_2)^{-1} \\ \{h_2,h_1\}_c & \end{bmatrix}
  \]
  \item 32-Peiffer lifting:
  \[
    \{l,h\} := \begin{pmatrix} & \begin{bmatrix} & h^{-1} \\ \{l^{-1},h\}_c^{-1_4} & \end{bmatrix}\\ l &\end{pmatrix}
  \]
\end{itemize}
\end{minipage}
\begin{minipage}{.495\columnwidth}
\begin{itemize}
  \item 33-Peiffer lifting:
  \[
    \{l_2,l_1\} := \begin{pmatrix} & (\act{\partial l_2}l_1l_2)^{-1} \\ \{l_2,l_1\}_c & \end{pmatrix}
  \]
  \item 23-Peiffer lifting:
  \[
    \{h,l\} := \begin{pmatrix} & \begin{bmatrix} & h^{-1} \\ \{h,l^{-1}\}_c^{-1_4} & \end{bmatrix}\\\act{h}l &\end{pmatrix}
  \]
\end{itemize}
\end{minipage}
\begin{itemize}
  \item left-Homanian:
  \[
    \{h_z,h_y|h_x\} := \begin{pmatrix} & (\act{h_z}\{h_y,h_x\}\{h_z,\act{\partial h_y}h_x\})^{-1} \\ \begin{bmatrix} & (\act{\partial(h_zh_y)}h_xh_zh_y)^{-1} \\ \{h_z\#_1t_1(h_y),h_y|h_x\}_c^{-1_4} & \end{bmatrix} & \end{pmatrix}
  \]
  \item right-Homanian:
  \[
    \{h_z|h_y,h_x\} := \begin{pmatrix} & (\{h_z,h_y\}\act{\act{\partial h_z}h_y}\{h_z,h_x\})^{-1} \\ \begin{bmatrix} & (\act{\partial h_z}(h_xh_y)h_z)^{-1} \\ \{h_z|h_y\#_1t_1(h_x),h_x\}_c^{-1_4} & \end{bmatrix} & \end{pmatrix}
  \]
\end{itemize}
\end{center}
The 22-Peiffer lifting is an element of $L$, while the 32-, 23- and 33-Peiffer liftings and the Homanians are elements of $M$. The proof is straightforward. For example, the 3-source of the left-Homanian is
\begin{align*}
  &s_3(\{h_z,h_y|h_x\}) \\
  &= \begin{pmatrix} & (\act{h_z}\{h_y,h_x\}\{h_z,\act{\partial h_y}h_x\})^{-1} \\ \begin{bmatrix} & (\act{\partial(h_zh_y)}h_xh_zh_y)^{-1} \\ t_3(\{h_z\#_1t_1(h_y),h_y|h_x\}_c) & \end{bmatrix} & \end{pmatrix}, \\
  &= \begin{pmatrix} & (\act{h_z}\{h_y,h_x\}\{h_z,\act{\partial h_y}h_x\})^{-1} \\ \begin{bmatrix} & (\act{\partial(h_zh_y)}h_xh_zh_y)^{-1} \\ \begin{bmatrix}&\act{\partial(h_zh_y)}h_xh_zh_y \\ \act{h_z}\{h_y,h_x\}\{h_z,\act{\partial h_y}h_x\} & \end{bmatrix} & \end{bmatrix} & \end{pmatrix}, \\
  &= e_M.
\end{align*}
where we use
\begin{align*}
  &t_3(\{h_z\#_1t_1(h_y),h_y|h_x\}_c) \\
  &= [((h_z\#_1t_1(h_y))\#_1t_1(h_x))\#_2\{h_y,h_x\}_c]\#_3[\{h_z\#_1t_1(h_y),h_x\}_c\#_2h_y], \displaybreak[1] \\
  &= \begin{bmatrix}& \begin{bmatrix}& \act{\partial h_y}h_xh_y \\ \{h_y,h_x\} & \end{bmatrix}\\ h_z\end{bmatrix}\#_3\begin{bmatrix}& h_y \\ \begin{bmatrix}&\act{\partial (h_zh_y)}h_xh_z \\ \{h_z,\act{\partial h_y}h_x\} &\end{bmatrix} &\end{bmatrix}, \\
  &= \begin{bmatrix}& h_z\act{\partial h_y}h_xh_y \\ \act{h_z}\{h_y,h_x\} & \end{bmatrix}\#_3\begin{bmatrix}&\act{\partial(h_zh_y)}h_xh_zh_y \\ \{h_z,\act{\partial h_y}h_x\} &\end{bmatrix}, \displaybreak[1] \\
  &= \begin{bmatrix}&\act{\partial(h_zh_y)}h_xh_zh_y \\ \act{h_z}\{h_y,h_x\}\{h_z,\act{\partial h_y}h_x\} & \end{bmatrix}.
\end{align*}
The 2-target is 
\begin{align*}
  &t_2(\{h_z,h_y|h_x\}) \\
  &= t_2\left(\begin{pmatrix} & (\act{h_z}\{h_y,h_x\}\{h_z,\act{\partial h_y}h_x\})^{-1} \\ \begin{bmatrix} & (\act{\partial(h_zh_y)}h_xh_zh_y)^{-1} \\ s_3(\{h_z\#_1t_1(h_y),h_y|h_x\}_c) & \end{bmatrix} & \end{pmatrix}\right), \\
  &= \begin{bmatrix} & (\act{\partial(h_zh_y)}h_xh_zh_y)^{-1} \\ t_2(\{h_zh_y,h_x\}_c) & \end{bmatrix}\#_2t_2((\act{h_z}\{h_y,h_x\}\{h_z,\act{\partial h_y}h_x\})^{-1}), \\
  &= \begin{bmatrix} & (\act{\partial(h_zh_y)}h_xh_zh_y)^{-1} \\ \begin{bmatrix}& h_x \\ h_zh_y &\end{bmatrix} & \end{bmatrix}\#_2t_2((\act{h_z}\{h_y,h_x\}\{h_z,\act{\partial h_y}h_x\})^{-1}), \\
  &= e_H.
\end{align*}
Thus, the left-Homanian is an element of the group $M$. By the same procedure, we can show that the other liftings are elements of the groups $L$ or $M$. 

\begin{theorem}
The above sequence of groups $\bG := M\xrightarrow{\partial_3}L \xrightarrow{\partial_2}H \xrightarrow{\partial_1} G$, together with the left actions and the liftings defined above, constitutes a 3-crossed module.
\end{theorem}
\begin{proof}
The proof essentially follows the approach outlined below. First, we compare the diagrams of the fundamental relations of Gray 4-categories with those of 3-crossed modules. This tells us which relation in a category ${\bf A_*}$ corresponds to which relation in a 3-crossed module. Now choose one of the relations that we wish to prove on the 3-crossed module side, and consider the corresponding relation on the 4-category side. Using the invertibility of the morphisms and making all sources equal to the identity in the diagram, we obtain the relation on the 3-crossed module side.

First, we need to show that the data $(M\xrightarrow{\partial_3}L\xrightarrow{\partial_2}H,\triangleright,\{-,-\}_{33})$ for the higher levels constitutes a 2-crossed module. In other words, we must check that this data satisfies the seven axioms of Definition \ref{def:2CM}. By construction, these components already satisfy axioms \ref{ax:2CM-exact}, \ref{ax:2CM-BAction} and \ref{ax:2CM-BoundaryPeiffer} of Definition \ref{def:2CM}, so it remains to verify the remaining axioms. The proof is straightforward:
\begin{enumerate}
  \setcounter{enumi}{2}
  \item $\act{h}\{l,l'\} = \{\act{h}l,\act{h}l'\}$
  \begin{align*}
    (\text{L.H.S}) &= \left[h\#_2\begin{pmatrix}&(\act{\partial l}l'l)^{-1}\\\{l,l'\}_c&\end{pmatrix}\#_2h^{-1_2}\right] \displaybreak[1] \\
    &= \begin{pmatrix}&[h\#_2(\act{\partial l}l'l)^{-1}\#_2h^{-1_2}]\#_1t_1(h)^{-1_1}\\ [h\#_2\{l,l'\}_c\#_2h^{-1_2}]\#_1t_1(h)^{-1_1}&\end{pmatrix} \displaybreak[1] \\
    &= \begin{pmatrix}&(\act{\partial (\act{h}l)}(\act{h}l')(\act{h}l))^{-1}\\ \{h\#_2l\#_2h^{-1_2},h\#_2l'\#_2h^{-1_2}\}_c\#_1t_1(h)^{-1_1}&\end{pmatrix} \displaybreak[1] \\
    &= \begin{pmatrix}&(\act{\partial (\act{h}l)}(\act{h}l')(\act{h}l))^{-1}\\ \{\act{h}l,\act{h}l'\}_c&\end{pmatrix} = (\text{R.H.S}),
  \end{align*}
  where, in the third equality, we have used the relation \ref{ax:Gray3-1morPeiffer} of Definition \ref{def:Gray3} and, in the fourth equality, the functoriality of whiskering by 1-morphisms.
  \stepcounter{enumi}
  \item Peiffer identity $\act{\partial m}m' = mm'm^{-1}$: This follows from the interchange law of 4-morphisms. 
  \begin{align*}
    (\text{L.H.S}) 
    &= t_3(m)\#_3m'\#_3t_3(m)^{-1_3} \\
    &= [(t_3(m)\#_3m')\#_4(m\#_3s_3(m))\#_4m^{-1_4}]\#_3t_3(m)^{-1_3} \displaybreak[1] \\
    &= [(m\#_3t_3(m'))\#_4(s_3(m)\#_3m')\#_4m^{-1_4}]\#_3t_3(m)^{-1_3} \\
    &= [(mm')\#_3t_3(m)^{-1_3}]\#_4m^{-1} \displaybreak[1] \\
    &= mm'm^{-1} = (\text{R.H.S}).
  \end{align*}
  \item Equations \eqref{eq:HH_H-2CM} and \eqref{eq:H_HH-2CM}: These equations can be shown by the approach outlined above. For example, let us focus on \eqref{eq:HH_H-2CM}. The corresponding relation on the category side is \eqref{eq:P-decom1} (with the degrees of the morphisms shifted). From this, we obtain
  \begin{align*}
    \{l_zl_y,l_x\}_c &= [([l_z\#_2t_2(l_y)]\#_2t_2(l_x))\#_3\{l_y,l_x\}_c]\#_4[\{l_z\#_2t_2(l_y),l_x\}_c\#_3(l_y\#_2s_2(l_x))] \\
    &= [(l_z\#_2t_2(l_yl_x))\#_3\{l_y,l_x\}_c]\#_4[(\{l_z,\act{\partial l_y}l_x\}_c\#_2t_2(l_y))\#_3l_y].
  \end{align*}
  In terms of the Peiffer lifting on the group side,
  \[
    \begin{pmatrix}& \act{\partial (l_zl_y)}l_xl_zl_y \\ \{l_zl_y,l_x\} &\end{pmatrix} 
    = 
    \begin{pmatrix}&\begin{pmatrix} & \act{\partial l_y}l_xl_y \\ \{l_y,l_x\} &\end{pmatrix} \\ l_z &\end{pmatrix}
    \#_4
    \begin{pmatrix}& \act{\partial (l_zl_y)}l_xl_zl_y \\ \{l_z,\act{\partial l_y}l_x\}\end{pmatrix}.
  \]
  Multiplying both sides of this equation by $(-\#_2t_2((\act{\partial(l_zl_y)}l_xl_zl_y)^{-1}))\#_3(\act{\partial(l_zl_y)}l_xl_zl_y)^{-1}$ gives
  \begin{align*}
    \{l_zl_y,l_x\}
    &= 
    \begin{pmatrix} & l_z\act{\partial l_y}l_xl_y(\act{\partial (l_zl_y)}l_xl_zl_y)^{-1}\\ \act{l_z}\{l_y,l_x\}&\end{pmatrix}
    \#_4 \{l_z,\act{\partial l_y}l_x\} \\
    &= \act{l_z}\{l_y,l_x\}\{l_z,\act{\partial l_y}l_x\}.
  \end{align*}
  The proof of Equation \eqref{eq:H_HH-2CM} is the same.
  \item Equations \eqref{eq:LH-2CM} and \eqref{eq:HL-2CM}: 
  These equations can also be shown by the same approach. Let us focus on \eqref{eq:LH-2CM}. The corresponding relation is \eqref{eq:2-funct1}. Using this equation, we have
  \[
    \{t_3(m),l\}_c\#_4\begin{pmatrix}m & \\ & l\end{pmatrix} = \begin{pmatrix}&l \\ m&\end{pmatrix}
  \]
  Here we have used $\{\id_3,Phi\}_{33,c} = \id_{\Phi}$. 
  In terms of the group-theoretic liftings,
  \begin{align*}
    \begin{pmatrix} &l\partial m \\ \{\partial m,l\} & \end{pmatrix}\#_4\begin{pmatrix} & m\\l & \end{pmatrix} &= \begin{pmatrix}&l \\ m&\end{pmatrix} \\
    \{\partial m,l\}\act{l}m &= m.
  \end{align*}
  Same as for Equation \eqref{eq:HL-2CM}.
\end{enumerate}
Therefore, the data $(M\to L\to H\to G,\triangleright,\{-,-\}_{33})$ form a 2-crossed module. 

Next, we have to show that the remaining axioms of Definition \ref{def:3CM} are satisfied. By construction, the conditions \ref{ax:3CM-exact}, \ref{ax:3CM-BAction} and \ref{ax:3CM-BoundaryLiftings} already hold. The proof is again straightforward:
\begin{enumerate}
  \setcounter{enumi}{3}
  \item The action of $g\in G$ on the liftings: these can be verified in exactly the same way as in the case of the 2-crossed module, using the compatibility conditions between the liftings and $\#_1$ from \ref{ax:Gray4-LiftingWhiskering}. For example, let us consider the left-Homanian. By almost the same calculation as in the 2-crossed module case, one can show that $\act{g}\{h,h'\} = \{\act{g}h,\act{g}h'\}$. Since $g\#_1$ and $\#_1g^{-1}$ are whiskerings by 1-morphisms, we have
  \begin{align*}
    &\act{g}\{h_z,h_y|h_x\} \\
    &= \begin{pmatrix}&(\act{h'_z}\{h'_y,h'_x\}\{h'_z,\act{\partial(h'_y)}h_x'\})^{-1} \\ \begin{bmatrix}&(\act{\partial(h'_zh'_y)}h'_xh'_zh'_y)^{-1} \\ (g\#_1\{h_z\#_1t_1(h_y),h_y|h_x\}_c\#_1g^{-1})^{-1_4} & \end{bmatrix}&\end{pmatrix},
  \end{align*}
  where $h'$ stands for $\act{g}h$. Using \ref{ax:Gray4-LiftingWhiskering} for the categorical lifting part, we get
  \begin{align*}
    &g\#_1\{h_z\#_1t_1(h_y),h_y|h_x\}_c\#_1g^{-1} \\
    &= \{g\#_1h_z\#_1t_1(h_y),g\#_1h_y|h_x\#_1g^{-1}\}_c \\
    &= \{g\#_1h_z\#_1(t_1(h_y)g^{-1}),g\#_1h_y\#_1g^{-1}|g\#_1h_x\#_1g^{-1}\}_c \\
    &= \{\act{g}h_z\#_1(t_1(\act{g}h_y)),\act{g}h_y|\act{g}h_x\}_c.
  \end{align*}
  Thus, we obtain $\act{g}\{h_z,h_y|h_x\} = \{\act{g}h_z,\act{g}h_y|\act{g}h_x\}$. 
  \setcounter{enumi}{5}
  \item, (vii), (viii), (x), (xi) and (xii): As mentioned above, these equations can be verified by eliminating sources from the corresponding relations on the category side. The equation to be used is found by comparing the diagrammatic representations of the 3-crossed module with the categorical ones; the correspondence is given in Table \ref{tab:diagram-rel}. Along the way, we frequently use the normalization of the liftings \ref{ax:Gray4-unit} and the compatibility of the liftings with whiskering by $\#_1$ \ref{ax:Gray4-LiftingWhiskering}.
  \begin{center}
    \begin{table}[htb]
      \caption{Corresponding diagrams}
      \label{tab:diagram-rel}
      \begin{tabular}{ccc}
        \hline
        Axioms of Def.~\ref{def:3CM} & name & corresponding diagrams in Def.~\ref{def:Gray4} \\ \hline
        \ref{ax:3CM-Tetrahedrons} & Tetrahedrons & \ref{ax:Gray4-HomanRels} \\
        \ref{ax:3CM-KV} & KV-Polytope & \ref{ax:Gray4-KV} \\
        \ref{ax:3CM-Prisms} & Prisms & \ref{ax:Gray4-Prism} \\
        \ref{ax:3CM-Pasting} & Pastings & \ref{ax:Gray4-Prism} and \ref{ax:Gray4-Pasting} \\
        \ref{ax:3CM-Cube} & Cube & \ref{ax:Gray4-Cube} \\
        \ref{ax:3CM-2on4} & Action of $G_2$ on $G_4$ & \ref{ax:Gray4-Tube} \\
        \ref{ax:3CM-SpSm} & $S^+=S^-$ & \ref{ax:Gray4-SpSm}\\\hline
      \end{tabular}
    \end{table}
  \end{center}
  
  Here, we focus on the left-Homanian relation \eqref{eq:3-CM-Fund-left-Hom} and verify it. As its name indicates, the corresponding diagram is \ref{ax:Gray4-HomanRels} in Definition \ref{def:Gray4}. In terms of formulas, this equation is given by Equation \eqref{eq:left-Hom}. Let us set the $\alpha$s to
  \begin{align*}
    \alpha_x &= h_x, &
    \alpha_y &= h_y, \\
    \alpha_z &= h_z\#_1t_1(h_y), &
    \alpha_w &= h_w\#_1t_1(h_zh_y).
  \end{align*}
  Substituting them into \eqref{eq:left-Hom}, we get
  \begin{align*}
    &
    [
      (((h_wh_z)\#_1t_1(h_yh_x))\#_2\{h_y,h_x\}_c)
      \#_3
      (\{h_w\#_1t_1(h_zh_y),h_z\#_1t_1(h_y)|h_x\}_c\#_2h_y)
    ] \\
    &\#_4
    \{(h_wh_z)\#_1t_1(h_y),h_y|h_x\}_c \\
    &= \\
    &
    [
      ((h_w\#_1t_1(h_zh_yh_x))\#_2\{h_z\#_1t_1(h_y),h_y|h_x\}_c)
      \#_3
      (\{h_w\#_1t_1(h_zh_y),h_x\}_c\#_2((h_zh_y)))
    ] \\
    &\#_4
    \{h_w\#_1t_1(h_zh_y),h_zh_y|h_x\}_c.
  \end{align*}
  Taking $\#_1t_1$ to the right side of the expressions, and using the simplified notation, the above equation can be rewritten as
  \begin{align*}
    &\quad 
    \{(h_wh_z)\#_1t_1(h_y),h_y|h_x\}_c^{-1_4} \\
    & \#_4
    [
      \begin{bmatrix} & \{h_y,h_x\}_c\\ h_wh_z&\end{bmatrix}
      \#_3
      \begin{bmatrix}& h_y\\ \{h_w\#_1t_1(h_z),h_z|\act{\partial h_y}h_x\}_c^{-1_4} & \end{bmatrix}
    ] \\
    &= \\
    &
    \{h_w\#_1t_1(h_zh_y),h_zh_y|h_x\}_c^{-1_4} \\
    &\quad \#_4
    [
      \begin{bmatrix} & \{h_z\#_1t_1(h_y),h_y|h_x\}_c^{-1_4}\\ h_w&\end{bmatrix}
      \#_3
      \begin{bmatrix}& h_zh_y\\ \{h_w,\act{\partial(h_zh_y)}h_x\}_c&\end{bmatrix}
    ],
  \end{align*}
  where we have taken inverses of both sides. In terms of the group liftings,
  \begin{align*}
    &\quad 
    \begin{bmatrix}& \act{\partial(h_wh_zh_y)}h_xh_wh_zh_y\\ \begin{pmatrix}& \act{h_wh_z}\{h_y,h_x\}\{h_wh_z,\act{\partial h_y}h_x\}\\ \{h_wh_z,h_y|h_x\}& \end{pmatrix}& \end{bmatrix} \\
    & \#_4
    [
      \begin{bmatrix} &\begin{bmatrix}& \act{\partial h_y}h_xh_y\\ \{h_y,h_x\}&\end{bmatrix}\\ h_wh_z&\end{bmatrix} \\
      & 
      \#_3
      \begin{bmatrix}& h_y\\ \begin{bmatrix}& \act{\partial(h_wh_zh_y)}h_xh_wh_z\\\begin{pmatrix} & \act{h_w}\{h_z,\act{\partial h_y}h_x\}\{h_w,\act{\partial (h_zh_y)}h_x\}\\ \{h_w,h_z|\act{\partial h_y}h_x\}&\end{pmatrix} &\end{bmatrix} & \end{bmatrix}
    ] \\
    &= \\
    &
    \begin{bmatrix}& \act{\partial (h_wh_zh_y)}h_xh_wh_zh_y\\ \begin{pmatrix} & \act{h_w}\{h_zh_y,h_x\}\{h_w,\act{\partial(h_zh_y)}h_x\}\\\{h_w,h_zh_y|h_x\} &\end{pmatrix} &\end{bmatrix} \\
    &\quad \#_4
    [
      \begin{bmatrix} & \begin{bmatrix}& \act{\partial (h_zh_y)}h_xh_zh_y\\ \begin{pmatrix}& \act{h_z}\{h_y,h_x\}\{h_z,\act{\partial h_y}h_x\}\\ \{h_z,h_y|h_x\} &\end{pmatrix}& \end{bmatrix}\\ h_w&\end{bmatrix} \\
      & \#_3
      \begin{bmatrix}& h_zh_y\\ \begin{bmatrix} & \act{\partial (h_wh_zh_y)}h_xh_w\\ \{h_w,\act{\partial (h_zh_y)}h_x\} &\end{bmatrix}&\end{bmatrix}
    ].
  \end{align*}
  Rearranging the $h$ factors, we obtain
  \begin{align*}
    &\quad 
    \begin{bmatrix}& \act{\partial(h_wh_zh_y)}h_xh_wh_zh_y\\ \begin{pmatrix}& \act{h_wh_z}\{h_y,h_x\}\{h_wh_z,\act{\partial h_y}h_x\}\\ \{h_wh_z,h_y|h_x\}& \end{pmatrix}& \end{bmatrix}
    \\
    & \#_4
    [
      \begin{bmatrix}& h_wh_z\act{\partial h_y}h_xh_y\\ \act{h_wh_z}\{h_y,h_x\}&\end{bmatrix} \\
      & 
      \#_3
      \begin{bmatrix}& \act{\partial(h_wh_zh_y)}h_xh_wh_zh_y\\\begin{pmatrix} & \act{h_w}\{h_z,\act{\partial h_y}h_x\}\{h_w,\act{\partial (h_zh_y)}h_x\}\\ \{h_w,h_z|\act{\partial h_y}h_x\}&\end{pmatrix} &\end{bmatrix}
    ] \\
    &= \\
    &
    \begin{bmatrix}& \act{\partial (h_wh_zh_y)}h_xh_wh_zh_y\\ \begin{pmatrix} & \act{h_w}\{h_zh_y,h_x\}\{h_w,\act{\partial(h_zh_y)}h_x\}\\\{h_w,h_zh_y|h_x\} &\end{pmatrix} &\end{bmatrix} \\
    &\quad \#_4
    [
      \begin{bmatrix}& h_w\act{\partial (h_zh_y)}h_xh_zh_y\\ \begin{pmatrix}& \act{h_wh_z}\{h_y,h_x\}\act{h_w}\{h_z,\act{\partial h_y}h_x\}\\ \act{h_w}\{h_z,h_y|h_x\} &\end{pmatrix}& \end{bmatrix} \\
      &
      \#_3
      \begin{bmatrix} & \act{\partial (h_wh_zh_y)}h_xh_wh_zh_y\\ \{h_w,\act{\partial (h_zh_y)}h_x\} &\end{bmatrix}
    ].
  \end{align*}
  Then, we can cancel the rightmost factors $\act{\partial (h_wh_zh_y)}h_xh_wh_zh_y$ from both sides
  \begin{align*}
    &
    \begin{pmatrix}& \act{h_wh_z}\{h_y,h_x\}\{h_wh_z,\act{\partial h_y}h_x\}\\ \{h_wh_z,h_y|h_x\}& \end{pmatrix}
    \\
    & \#_4
    [
      \begin{bmatrix}& t_2(\act{h_w}\{h_z,\act{\partial h_y}h_x\}\{h_w,\act{\partial (h_zh_y)}h_x\})\\ \act{h_wh_z}\{h_y,h_x\}&\end{bmatrix} \\
      & 
      \#_3
      \begin{pmatrix} & \act{h_w}\{h_z,\act{\partial h_y}h_x\}\{h_w,\act{\partial (h_zh_y)}h_x\}\\ \{h_w,h_z|\act{\partial h_y}h_x\}&\end{pmatrix}
    ] \\
    &= \\
    &
    \begin{pmatrix} & \act{h_w}\{h_zh_y,h_x\}\{h_w,\act{\partial(h_zh_y)}h_x\}\\\{h_w,h_zh_y|h_x\} &\end{pmatrix} \\
    & \#_4
    [
      \begin{bmatrix}& t_2(\{h_w,\act{\partial (h_zh_y)}h_x\})\\ \begin{pmatrix}& \act{h_wh_z}\{h_y,h_x\}\act{h_w}\{h_z,\act{\partial h_y}h_x\}\\ \act{h_w}\{h_z,h_y|h_x\} &\end{pmatrix}& \end{bmatrix} \\
      &
      \#_3
      \{h_w,\act{\partial (h_zh_y)}h_x\}
    ],
  \end{align*}
  where we used $t_2(\{h_2,h_1\}) = h_2h_1h_2^{-1}\act{\partial h_2}h_1^{-1}$ in the second and sixth lines. By a simple calculation, we can show
  \begin{align*}
    \begin{bmatrix}&t_2(l')\\ l&\end{bmatrix}\#_3\begin{pmatrix}& l'\\ m&\end{pmatrix} &= \begin{pmatrix}& ll'\\ \act{l}m&\end{pmatrix}, \displaybreak[1] \\
    \begin{bmatrix}& t_2(l')\\\begin{pmatrix}& l\\ m &\end{pmatrix} &\end{bmatrix}\#_3l &= \begin{pmatrix}& ll'\\ m &\end{pmatrix}.
  \end{align*}
  Then, the second factors on the left- and right-hand sides simplify to the following:
  \begin{align*}
    &\quad 
    \begin{pmatrix}& \act{h_wh_z}\{h_y,h_x\}\{h_wh_z,\act{\partial h_y}h_x\}\\ \{h_wh_z,h_y|h_x\}& \end{pmatrix}
    \\
    & \#_4
    \begin{pmatrix} & \act{h_wh_z}\{h_y,h_x\}\act{h_w}\{h_z,\act{\partial h_y}h_x\}\{h_w,\act{\partial (h_zh_y)}h_x\}\\ \act{\act{h_wh_z}\{h_y,h_x\}}\{h_w,h_z|\act{\partial h_y}h_x\}&\end{pmatrix} \\
    &= \\
    &
    \begin{pmatrix} & \act{h_w}\{h_zh_y,h_x\}\{h_w,\act{\partial(h_zh_y)}h_x\}\\\{h_w,h_zh_y|h_x\} &\end{pmatrix} \\
    &\quad \#_4
    \begin{pmatrix}& \act{h_wh_z}\{h_y,h_x\}\act{h_w}\{h_z,\act{\partial h_y}h_x\}\{h_w,\act{\partial (h_zh_y)}h_x\}\\ \act{h_w}\{h_z,h_y|h_x\} &\end{pmatrix}.
  \end{align*}
  We can cancel the rightmost factors $\act{h_wh_z}\{h_y,h_x\}\act{h_w}\{h_z,\act{\partial h_y}h_x\}\{h_w,\act{\partial (h_zh_y)}h_x\}$ from both sides
  \begin{align*}
    &\quad 
    \begin{pmatrix}& t_3(\act{\act{h_wh_z}\{h_y,h_x\}}\{h_w,h_z|\act{\partial h_y}h_x\})\\ \{h_wh_z,h_y|h_x\}& \end{pmatrix}
    \#_4
    \act{\act{h_wh_z}\{h_y,h_x\}}\{h_w,h_z|\act{\partial h_y}h_x\} \\
    &= 
    \begin{pmatrix} & t_3(\act{h_w}\{h_z,h_y|h_x\})\\\{h_w,h_zh_y|h_x\} &\end{pmatrix}
    \#_4
    \act{h_w}\{h_z,h_y|h_x\},
  \end{align*}
  where we have used $t_3(\{h_3,h_2|h_1\}) = \{h_3h_2,h_1\}\{h_3,\act{\partial h_2}h_1\}^{-1}\act{h_3}\{h_2,h_1\}^{-1}$ in the parentheses. Since both sides are of the form $(m\#_3t_3(m'))\#_4m' = mm'$, we therefore obtain
  \[
    \{h_wh_z,h_y|h_x\}\act{\act{h_wh_z}\{h_y,h_x\}}\{h_w,h_z|\act{\partial h_y}h_x\} = \{h_w,h_zh_y|h_x\}\act{h_w}\{h_z,h_y|h_x\}.
  \]
  This is exactly Equation \eqref{eq:3-CM-Fund-left-Hom}.
  \setcounter{enumi}{8}
  \item Pastings: The proof is essentially the same as above. In Equation \eqref{eq:Pasting332}, set the variables as follows:
  \begin{align*}
    \Psi' &= l_1^{-1}\#_1t_2(l_2^{-1}), &
    \Psi &= l_2^{-1}, &
    \alpha &= h,
  \end{align*}
  then, the Equation \eqref{eq:Pasting332} becomes
  \begin{align}
    \nonumber
    \{(l_2l_1)^{-1},h\}_c^{-1_4}
    &= 
    \bigl(
      \{l_1^{-1}\#_2t_2(l_2^{-1}),h\}_c^{-1_4}
      \#_3
      \begin{bmatrix}l_2^{-1} & \\ & h\end{bmatrix}
    \bigr) \\
    \label{eq:piece0}
    & \#_4
    \bigl(
      \begin{bmatrix}& h \\ l_1^{-1}\#_2t_2(l_2^{-1}) &\end{bmatrix}
      \#_3
      \{l_2^{-1},h\}_c^{-1_4}
    \bigr).
  \end{align}
  The left-hand side can be rewritten in terms of the group liftings:
  \begin{equation}
    \label{eq:piece1}
    \{(l_2l_1)^{-1},h\}_c^{-1_4} = \begin{bmatrix}& h \\ \begin{pmatrix} & \{l_2l_1,h\} \\ l_1^{-1}l_2^{-1} & \end{pmatrix} & \end{bmatrix}.
  \end{equation}
  The second factor of the right-hand side can also be written
  \begin{align}
    \nonumber
    \begin{bmatrix}& h \\ l_1^{-1}\#_2t_2(l_2^{-1}) &\end{bmatrix}
    \#_3
    \{l_2^{-1},h\}_c^{-1_4}
    &= 
    \begin{bmatrix}& h \\ l_1^{-1}\#_2t_2(l_2^{-1}) &\end{bmatrix}
    \#_3
    \begin{bmatrix} & h\\ \begin{pmatrix} & \{l_2,h\} \\ l_2^{-1} & \end{pmatrix} & \end{bmatrix}, \\
    \label{eq:piece2}
    &= \begin{bmatrix} & h\\ \begin{pmatrix} & \{l_2,h\} \\ l_1^{-1}l_2^{-1} & \end{pmatrix} & \end{bmatrix}.
  \end{align}
  To proceed further, unlike the other relations proved above, we need to use another categorical relation. In Equation \eqref{eq:Prism32|2}, setting $\Psi = l_1^{-1},\gamma = t_2(l_2^{-1}),\alpha=h$, we obtain
  \begin{align*}
    &
    \{l_1^{-1}\#_2t_2(l_2^{-1}),h\}_c^{-1_4} \\
    &= 
    \Bigl(
      \{t_2(l_1^{-1}),t_2(l_2^{-1})|h\}_c^{-1_4}
      \#_3
      \begin{bmatrix}l_1^{-1}\#_2t_2(l_2^{-1}) & \\ & h\end{bmatrix}
    \Bigr) \\
    &\quad \#_4
    \Bigl(
      \begin{bmatrix} & \{t_2(l_2^{-1}),h\}_c \\ t_2(l_1^{-1}) & \end{bmatrix}
      \#_3
      [\{l_1^{-1},h\}_c^{-1_4}\#_2(t_2(l_2^{-1}))]
    \Bigr) \\
    &\quad \#_4
    \{l_1^{-1}\#_1t_1(h),\{t_2(l_2^{-1}),h\}_c\}_c^{-1_4}.
  \end{align*}
  By careful calculations analogous to those for Equation \eqref{eq:piece2}, we obtain
  \begin{align}
    &
    \nonumber
    \{l_1^{-1}\#_1t_1(h),\{t_2(l_2^{-1}),h\}_c\}_c^{-1_4}
    \#_3
    \begin{bmatrix}l_2^{-1} & \\ & h\end{bmatrix} \\
    \label{eq:piece3}
    &= 
    \begin{bmatrix} & h \\ \begin{pmatrix} & \begin{pmatrix} & t_3(\{l_2,h\}) \\ \act{l_2l_1}\{l_1^{-1},\{\partial l_2^{-1},h\}\}^{-1} &\end{pmatrix} \\ l_1^{-1}l_2^{-1} &\end{pmatrix} & \end{bmatrix}.
  \end{align}
  \begin{align}
    \nonumber
    &
    \begin{bmatrix} & \{t_2(l_2^{-1}),h\}_c \\ t_2(l_1^{-1}) & \end{bmatrix}
    \#_3
    [\{l_1^{-1},h\}_c^{-1_4}\#_2t_2(l_2^{-1})]
    \#_3
    \begin{bmatrix}l_2^{-1} & \\ & h\end{bmatrix} \\
    \label{eq:piece4}
    &= 
    \begin{bmatrix} & h \\ \begin{pmatrix} & \partial(\act{l_2l_1}\{l_1^{-1},\{\partial l_2^{-1},h\}\}^{-1}\{l_2,h\}) \\ \begin{pmatrix} & \act{\partial\act{l_2l_1}\{l_1^{-1},\{\partial l_2^{-1},h\}\}^{-1}l_2\{\partial l_2^{-1},h\}}\{l_1,h\}\\ l_1^{-1}l_2^{-1} & \end{pmatrix} &\end{pmatrix} & \end{bmatrix}
  \end{align}
  \begin{align}
    &\nonumber
    \{t_2(l_1^{-1}),t_2(l_2^{-1})|h\}_c^{-1_4}
    \#_3
    \begin{bmatrix}l_1^{-1}\#_2t_2(l_2^{-1}) & \\ & h\end{bmatrix} 
    \#_3
    \begin{bmatrix}l_2^{-1} & \\ & h\end{bmatrix} \\
    \label{eq:piece5}
    &= 
    \begin{bmatrix} & h \\ \begin{pmatrix} & \begin{pmatrix} & \partial(\act{l_2l_1}\{l_1^{-1},\{\partial l_2^{-1},h\}\}^{-1}\{l_2,h\}\act{\act{h}l_2}\{l_1,h\}) \\ \act{l_2l_1}\{\partial l_1^{-1},\partial l_2^{-1}|h\} & \end{pmatrix} \\ l_1^{-1}l_2^{-1} & \end{pmatrix} &\end{bmatrix}
  \end{align}
  Substituting all the pieces \eqref{eq:piece1}--\eqref{eq:piece5} into \eqref{eq:piece0} and eliminating $h$ and $l_1^{-1}l_2^{-1}$, we get
  \begin{align*}
    &
    \{l_2l_1,h\} \\
    &= 
    \begin{pmatrix} & \partial(\act{l_2l_1}\{l_1^{-1},\{\partial l_2^{-1},h\}\}^{-1}\{l_2,h\}\act{\act{h}l_2}\{l_1,h\}) \\ \act{l_2l_1}\{\partial l_1^{-1},\partial l_2^{-1}|h\} & \end{pmatrix} \\
    &\#_4
    \begin{pmatrix} & \partial(\act{l_2l_1}\{l_1^{-1},\{\partial l_2^{-1},h\}\}^{-1}\{l_2,h\}) \\ \act{\partial\act{l_2l_1}\{l_1^{-1},\{\partial l_2^{-1},h\}\}^{-1}l_2\{\partial l_2^{-1},h\}}\{l_1,h\} &\end{pmatrix} \\
    &\#_4
    \begin{pmatrix} & t_3(\{l_2,h\}) \\ \act{l_2l_1}\{l_1^{-1},\{\partial l_2^{-1},h\}\}^{-1} &\end{pmatrix}
    \#_4
    \{l_2,h\}.
  \end{align*}
  From the definition of multiplication and Peiffer identities, 
  \begin{align*}
    &
    \{l_2l_1,h\} \\
    &= 
    \begin{pmatrix} & \partial(\act{l_2l_1}\{l_1^{-1},\{\partial l_2^{-1},h\}\}^{-1}\{l_2,h\}\act{\act{h}l_2}\{l_1,h\}) \\ \act{l_2l_1}\{\partial l_1^{-1},\partial l_2^{-1}|h\} & \end{pmatrix} \\
    &\#_4
    \act{\partial\act{l_2l_1}\{l_1^{-1},\{\partial l_2^{-1},h\}\}^{-1}l_2\{\partial l_2^{-1},h\}}\{l_1,h\}\act{l_2l_1}\{l_1^{-1},\{\partial l_2^{-1},h\}\}^{-1}\{l_2,h\}, \\
    &= \begin{pmatrix} & \partial(\act{l_2l_1}\{l_1^{-1},\{\partial l_2^{-1},h\}\}^{-1}\{l_2,h\}\act{\act{h}l_2}\{l_1,h\}) \\ \act{l_2l_1}\{\partial l_1^{-1},\partial l_2^{-1}|h\} & \end{pmatrix} \\
    &\#_4
    \act{l_2l_1}\{l_1^{-1},\{\partial l_2^{-1},h\}\}^{-1}\{l_2,h\}\act{\partial \{l_2,h\}l_2\{\partial l_2^{-1},h\}}\{l_1,h\}, \\
    &= \begin{pmatrix} & \partial(\act{l_2l_1}\{l_1^{-1},\{\partial l_2^{-1},h\}\}^{-1}\{l_2,h\}\act{\act{h}l_2}\{l_1,h\}) \\ \act{l_2l_1}\{\partial l_1^{-1},\partial l_2^{-1}|h\} & \end{pmatrix} \\
    &\#_4
    \act{l_2l_1}\{l_1^{-1},\{\partial l_2^{-1},h\}\}^{-1}\{l_2,h\}\act{\partial \{l_2,h\}l_2\{\partial l_2^{-1},h\}}\{l_1,h\}, \\
    &= \act{l_2l_1}\{\partial l_1^{-1},\partial l_2^{-1}|h\}\act{l_2l_1}\{l_1^{-1},\{\partial l_2^{-1},h\}\}^{-1}\{l_2,h\}\act{\act{h}l_2}\{l_1,h\}.
  \end{align*}
  Therefore, Equation \eqref{eq:3CM-Pasting} holds.
\end{enumerate}
Let us define a functor $\Delta:~\Gray\to \TXMod$ by $\Delta({\bf A_*}) = \bG$. 

\end{proof}

\subsection{From 3-crossed modules to Gray 4-groups}
\label{3CM_to_Gray4}

We construct a functor from the category of 3-crossed modules $\TXMod$ to the category of Gray 4-groups ${\bf Gray_4Group}$.

Let $\bG \coloneqq ( M\overset{\partial_3}{\longrightarrow}L\overset{\partial_2}{\longrightarrow}H\overset{\partial_1}{\longrightarrow} G)$ be a 3-crossed module. We construct the Gray 4-group ${\bf A_*}$ from $\bG$, and this construction defines the functor $\Theta : \TXMod \rightarrow {\bf Gray_4Group}$.

The structure of ${\bf A_*}$ is given by the following data:

\vspace{.5em}
\noindent
{\bf Objects and $i$-morphisms $\mor_i({\bf A_*})$ for $i=1,2,3$.}
The object $A_0$ of ${\bf A_*}$ is trivial, and the morphisms $\mor_i({\bf A_*})$ for $i=1,2,3$ are given as follows:
\begin{equation*}
\begin{aligned}
\mor_1({\bf A_*}) &\coloneqq G \\
\mor_2({\bf A_*}) &\coloneqq H \times G \\
\mor_3({\bf A_*}) &\coloneqq L \times H \times G \\
\mor_4({\bf A_*}) &\coloneqq M \times L \times H \times G.
\end{aligned}
\end{equation*}

\vspace{.5em}
\noindent
{\bf Source and target maps.}
 For $i = 1,2,3$, each $i$-morphism has $l$-source maps for $l = 0, \dots ,i-1$ and, similarly, $l$-target maps. For $\mor_4({\bf A_*})$, there are four types of source and target maps; for an element $(m,l,h,g) \in \mor_4({\bf A_*})$, these are constructed from the following data:
\begin{equation*}
\begin{aligned}
t_3(m, l, h, g) &= (\partial_3(m) \cdot l, h, g ), &\qquad s_3(m, l, h, g) &= (l, h, g ), \\
t_2(m, l, h, g) &= (\partial_2(l) \cdot h, g), &\qquad s_2(m, l, h, g) &= (h,g), \\
t_1(m, l, h, g) &= \partial_1(h) \cdot g, &\qquad s_1(m, l, h, g) &= g, \\
t_0(m, l, h, g) &= *, &\qquad s_0(m, l, h, g) &= *.
\end{aligned}
\end{equation*}

For $(l, h, g) \in \mor_3(\mathbf{A}_*)$, we define:
\begin{equation*}
\begin{aligned}
t_2(l, h, g) &= (\partial_2(l) \cdot h,g), &\qquad s_2(l, h, g) &= (h,g), \\
t_1(l, h, g) &= \partial_1(h) \cdot g, &\qquad s_1(l, h, g) &= g, \\
t_0(l, h, g) &= *, &\qquad s_0(l, h, g) &= *.
\end{aligned}
\end{equation*}

For $(h, g) \in \mor_2(\mathbf{A}_*)$, we define:
\begin{equation*}
\begin{aligned}
t_1(h, g) &= \partial_1(h) \cdot g, &\qquad s_1(h, g) &= g, \\
t_0(h, g) &= *, &\qquad s_0(h, g) &= *.
\end{aligned}
\end{equation*}

For $g \in \mor_1(\mathbf{A}_*)$, we define:
\begin{equation*}
t_1(g) = g, \quad s_1(g) = g, \quad t_0(g) = *, \quad s_0(g) = *.
\end{equation*}

\vspace{.5em}
\noindent
{\bf Vertical composition of $i$-morphisms $\#_i$.}
 A vertical composition of $i$-morphisms $\#_i : \mor_i(\mathbf{A}_*)\allowbreak \times_{i-1} \mor_i(\mathbf{A}_*)  \rightarrow \mor_i(\mathbf{A}_*) $ is based on the group product.

For $ \left( (m', \partial_3(m) l,h,g) , (m,l,h,g) \right) \in \mor_4({\bf A_*}) \times_3 \mor_4({\bf A_*})$, the vertical composition of 4-morphisms is constructed from the following data:
\[
(m', \partial_3(m) l,h,g) \#_4 (m,l,h,g) = (m' \cdot m,l,h,g),
\]
where ``$\cdot$'' denotes the group multiplication in $M$.

For $ \left( ( l',\partial_2(l)h,g) , (l,h,g) \right) \in \mor_3({\bf A_*}) \times_2 \mor_3({\bf A_*})$, we define the vertical composition of 3-morphisms:
\[
( l',\partial_2(l)h,g) \#_3 ( l,h,g) = (l'\cdot l,h,g),
\]
where ``$\cdot$'' denotes the group multiplication in $L$.

For $\big((h', \partial_1(h)g), (h, g)\big) \in \mor_2(\mathbf{A_*}) \times_1 \mor_2(\mathbf{A_*})$, we define
\[
(h', \partial_1(h)g) \#_2 (h, g) = (h' \cdot h, g),
\]
where ``$\cdot$'' denotes the group multiplication in $H$.

For $(g', g) \in \mor_1(\mathbf{A_*}) \times_0 \mor_1(\mathbf{A_*})$, we define
\[
g' \#_1 g = g' \cdot g,
\]
where ``$\cdot$'' denotes the group multiplication in $G$.

\vspace{.5em}
\noindent
{\bf Whiskering by 1-morphisms.}
Whiskering by 1-morphisms consists of the following data:
For $i = 2,3,4$,
\begin{equation*}
\begin{aligned}
\#^l_1: \mor_1 \times_0 \mor_i \rightarrow \mor_i,\\
\#^r_1: \mor_i \times_0 \mor_1 \rightarrow \mor_i.
\end{aligned}
\end{equation*}

For $g' \in \mor_1(\mathbf{A_*}) = G$ and an $i$-morphism, the whiskering $g' \#^l_1 (-)$ is defined as follows:

For 2-morphisms $(h, g) \in \mor_2(\mathbf{A_*})$:
$$
g' \#^l_1 (h, g) = (\act{g'}{}{h},\ g'g).
$$
For 3-morphisms $(l, h, g) \in \mor_3(\mathbf{A_*})$:
$$
g' \#^l_1 (l, h, g) = (\act{g'}{}{l},\ \act{g'}{}{h},\ g'g).
$$
For 4-morphisms $(m, l, h, g) \in \mor_4(\mathbf{A_*})$:
$$
g' \#^l_1 (m, l, h, g) = (\act{g'}{}{m},\ \act{g'}{}{l},\ \act{g'}{}{h},\ g'g).
$$
For an $i$-morphism and $g' \in \mor_1(\mathbf{A_*}) = G$, the whiskering $(-) \#^r_1 g'$ is defined as follows:
For 2-morphisms $(h, g) \in \mor_2(\mathbf{A_*})$:
$$
(h, g) \#^r_1 g' = (h,\ gg').
$$
For 3-morphisms $(l, h, g) \in \mor_3(\mathbf{A_*})$:
$$
(l, h, g) \#^r_1 g' = (l,\ h,\ gg').
$$
For 4-morphisms $(m, l, h, g) \in \mor_4(\mathbf{A_*})$:
$$
(m, l, h, g) \#^r_1 g' = (m,\ l,\ h,\ gg').
$$

\vspace{.5em}
\noindent
{\bf Whiskering by 2-morphisms.} Whiskering by 2-morphisms is analogous to whiskering by 1-morphisms, except that it acts only on 3- and 4-morphisms. Thus, for $i = 3,4$, it is given by the following data:
\begin{equation*}
\begin{aligned}
\#^l_2: \mor_2 \times_1 \mor_i \rightarrow \mor_i,\\
\#^r_2: \mor_i \times_1 \mor_2 \rightarrow \mor_i.
\end{aligned}
\end{equation*}

First, we construct the case of the left whiskering.
For $\left( (h', \del_1(h)g) , (l,h,g) \right)\in \mor_2 \times_1 \mor_3$, the whiskering $(h', \del_1(h)g)  \#^l_2  (l,h,g)$ is defined as follows:
\[
(h', \del_1(h)g)  \#^l_2  (l,h,g) = (\act{h'}l,h' h, g).
\]

For $\left( (h', \del_1(h)g) , (m,l,h,g) \right)\in \mor_2 \times_1 \mor_4$, the whiskering $(h', \del_1(h)g)  \#^l_2  (m,l,h,g)$ is defined as follows:
\[
(h', \del_1(h)g)  \#^l_2  (m,l,h,g) = (\act{h'}m, \act{h'}l,h' h, g).
\]

The case of the right whiskering is simpler. 
For $\left( (l',h', \del_1(h)g) , (h,g) \right)\in \mor_3 \times_1 \mor_2$, the whiskering $(l',h', \del_1(h)g) \#^r_2  (h,g)$ is defined as follows:
\[
(l',h', \del_1(h)g)   \#^r_2  (h,g) = (l', h' h, g).
\]

For $\left( (m', l',h', \del_1(h)g) , (h,g) \right)\in \mor_4 \times_1 \mor_2$, the whiskering $(m', l',h', \del_1(h)g) \#^r_2  (h,g)$ is defined as follows:
\[
(m', l',h', \del_1(h)g)   \#^r_2  (h,g) = (m', l', h' h, g).
\]

\vspace{.5em}
\noindent
{\bf Whiskering by 3-morphisms.}
 Whiskering by 3-morphisms acts only on 4-morphisms, so we only need to give the following data:
\begin{align*}
  \#^l_3&:~\mor_3 \times_2 \mor_4 \rightarrow \mor_4,\\
  \#^r_3&:~\mor_4 \times_2 \mor_3 \rightarrow \mor_4.
\end{align*}

The construction is almost the same as in the 2-morphism case, the only difference being that we now use a left $L$-action $\triangleright^3_4$ instead of a left $H$-action $\triangleright^2_3$. For $\left( (l',\del_2(l)h, g) , (m,l,h,g) \right)\in \mor_3 \times_2 \mor_4$, the whiskering $(l',\del_2(l)h, g)  \#^l_3  (m,l,h,g)$ is defined as follows:
\[
(l',\del_2(l)h, g)  \#^l_3  (m,l,h,g) = (\act{l'}m,l'l, h, g).
\]

The construction on the other side is done in a similar manner to the 2-morphism case. 
For $\left( (m',l',\del_2(l)h,g),  (l,h, g)  \right)\in \mor_4 \times_2 \mor_3$, the whiskering $(m',l',\del_2(l)h,g) \#^r_3  (l,h, g)$ is defined as follows:
\[
(m',l',\del_2(l)h,g) \#^r_3  (l,h, g) = (m', l'l,h,g).
\]

We now construct the six types of liftings. Most of them can be defined directly in a Gray 4-group by using the corresponding liftings of the 3-crossed module; the 32- and 23-Peiffer liftings differ slightly, while the remaining ones are essentially unchanged. We therefore first construct the 22- and 33-Peiffer liftings together with the left- and right-Homanian, and then handle the 32- and 23-Peiffer liftings.

\vspace{.5em}
\noindent
{\bf 22-Peiffer lifting.} A 22-Peiffer lifting in a Gray 4-group is constructed directly from the 22-Peiffer lifting in a 3-crossed module. For $((h_2, g_2),(h_1, g_1))\in \mor_2({\bf A_*}) \times_0 \mor_2({\bf A_*}) $, the 22-Peiffer lifting $\{(h_2, g_2), (h_1, g_1)\}_{22} \in \mor_3({\bf A_*})$ is defined as follows:
\[
\{(h_2, g_2), (h_1, g_1)\}_{22} = ( \{h_2,\act{g_2}h_1\}_{22},\act{\del_2 h_2 g_2}h_1h_2,g_2g_1).
\]

\vspace{.5em}
\noindent
{\bf 33-Peiffer lifting.} A 33-Peiffer lifting in a Gray 4-group is constructed in a manner analogous to the 22-Peiffer lifting. For $((l_2,h_2, g_2),(l_1,h_1, g_1))\in \mor_3({\bf A_*}) \times_1 \mor_3({\bf A_*}) $, the 33-Peiffer lifting $\{(l_2, h_2, \del_1h_1 g_1), (l_1, h_1, g_1)\}_{33} \in \mor_4({\bf A_*})$ is defined as follows:
\[
\{(l_2, h_2, g_2), (l_1, h_1, g_1)\}_{33} = (\{l_2, \act{h_2}l_1 \}_{33} ,\act{\del_2 l_2h_2}l_1l_2,h_2h_1, g_1).
\]

\vspace{.5em}
\noindent
{\bf left-Homanian.} A left-Homanian is constructed in the same way. For $(\beta_2,\beta_1,\alpha)\in \mor_2({\bf A_*}) \times_1 \mor_2({\bf A_*}) \times_0 \mor_2({\bf A_*}) $, a left-Homanian $\{\beta_2,\beta_1|\alpha\} \in \mor_4({\bf A_*})$ is defined as follows:
\begin{align*}
\{\beta_2,\beta_1|\alpha\}
&= (\{h_3,h_2 | h_1\}^{-1},\{h_3h_2, h_1\}_{22},\act{\del_1(h_3h_2)g_2}h_1h_3h_2,g_2g_1),
\end{align*}
where $\beta_2,\beta_1$ and $\alpha$ are $(h_3, \del_1h_2 g_2),(h_2, g_2) $ and $ (h_1, g_1)$, respectively.

\vspace{.5em}
\noindent
{\bf right-Homanian.} A right-Homanian is constructed in the same way. For $(\beta,\alpha_2,\alpha_2)\in \mor_2({\bf A_*}) \times_0 \mor_2({\bf A_*}) \times_{s_1,t_1} \mor_2({\bf A_*}) $, a right-Homanian $\{(h_3, g_2)|(h_2, \del_1 h_1 g_1), (h_1, g_1)\} \in \mor_4({\bf A_*})$ is defined as follows:
\[
\{\beta|\alpha_2, \alpha_1\} = (\{h_3| h_2, h_1\}^{-1},\{h_3, h_2h_1\}_{22},\act{\del_1(h_3)g_2}(h_2h_1)h_3,g_2g_1),
\]
where $\beta,\alpha_2$ and $\beta_1$ are $(h_3, g_2),(h_2, \del_1 h_1 g_1)$ and $(h_1, g_1)$, respectively.

\vspace{.5em}
\noindent
{\bf 32-Peiﬀer lifting.} Unlike the previous cases, the 32-Peiﬀer lifting in a Gray 4-group differs from that in a 3-crossed module. To construct it, we must use not only the 32-Peiﬀer lifting but also a 33-Peiﬀer lifting and a left-Homanian in the 3-crossed module. For $((l_2, h_2, g_2),(h_1,g_1))\in \mor_3({\bf A_*}) \times_0\mor_2({\bf A_*})$, the 32-Peiffer lifting $\{(l_2, h_2, g_2),(h_1,g_1)\}_{32} \in \mor_4({\bf A_*})$ is defined as follows:
\[
\{(l_2, h_2, g_2),(h_1,g_1)\}_{32} = \left(L(l_2, h_2; h'_1), \{\del_2(l_2)h_2,h'_1\}_{22}\act{(\act{\del_1 h_2}h'_1)}l_2, \act{\del_1 h_2}h'_1 h_2, g_2g_1 \right).
\]
Here, we write $\act{g_2}h_1$ as $h'_1$ and $L(l,h_2;h_1) := \{l,\{h_2,h_1\}\}\act{\act{\partial l}\{h_2,h_1\}}\left<l,\act{\partial h_2}h_1\right>\{\partial l,h_2|h_1\}^{-1}$ for convenience.

\vspace{.5em}
\noindent
{\bf 23-Peiﬀer lifting.} The 23-Peiﬀer lifting is constructed in a manner analogous to the 32-Peiffer lifting. To construct it in a Gray 4-group, we use a 32-Peiﬀer lifting, a 33-Peiﬀer lifting, and a right-Homanian in the 3-crossed module. For $((h_2, g_2),(l_1,h_1,g_1))\in \mor_2({\bf A_*}) \times_0\mor_3({\bf A_*})$, the 23-Peiffer lifting $\{( h_2, g_2),(l_1,h_1,g_1)\}_{32} \in \mor_4({\bf A_*})$ is defined as follows:
\[
\{(h_2, g_2),(l_1, h_1,g_1)\}_{23} = \left(R(h_2; l'_1, h'_1), \{h_2,\del_2(l'_1)h'_1\}_{22}\act{\del_1 h_2}l'_1, \act{\del_1 h_2}h'_1 h_2, g_2g_1  \right).
\]
Here, we write $\act{g_2}h_1$ as $h'_1$ and $\act{g_2}l_1$ as $l'_1$ and $R(h_2;l,h_1) := \left<h_2,l\right>\act{\{h_2,\partial l\}}\{\act{\partial h_2}l,\{h_2,h_1\}\}\{h_2|\partial l,h_1\}^{-1}$ for convenience. 

We have now constructed all the structure of a Gray 4-group. To confirm that $A_{*}$ is indeed a Gray 4-group, it remains to check that this structure satisfies the axioms of a Gray 4-category. 
\begin{theorem}
\label{3-cm_to_gray}
Let $\bG \coloneqq ( M\overset{\partial_3}{\longrightarrow}L\overset{\partial_2}{\longrightarrow}H\overset{\partial_1}{\longrightarrow} G)$ be a 3-crossed module. Then the 4-category ${A_{*}}$ constructed from the 3-crossed module $\bG$ as above is a Gray 4-group.
\end{theorem}
\begin{proof}
It is trivial to verify that the 4-category ${A_{*}}$ satisfies the axioms \ref{ax:Gray4-unit}--\ref{ax:Gray4-KV} of Gray 4-categories. Indeed, between a 3-crossed module and a Gray 4-category only the 32- and 23-Peiffer liftings differ; the remaining structure is described by the same diagrams. Consequently, the verification that ${A_{*}}$ satisfies these axioms is entirely straightforward.

In contrast, the axioms \ref{ax:Gray4-Prism}--\ref{ax:Gray4-SpSm} require computation: the verification that the 4-category ${A_{*}}$ satisfies them is not immediate, because the passage from the 3-crossed module $\bG$ to the 4-category ${\bf A_*}$ is not simple for the 32- and 23-Peiffer liftings, which use not only the 32- and 23-Peiffer liftings but also other liftings. We prove that ${\bf A_*}$ satisfies these axioms by direct calculation; the details are collected in Appendix \ref{proof:3-cm_to_Gray4-group}, and here we only carry out the calculation for prism 3-22.

For $(l,h_z, g_2) \in \mor_3({\bf A_*})$ and $(h_y, \del_1 h_x g_1),  (h_x, g_1)\in \mor_2({\bf A_*})$, we can write the LHS of prism 3-22 as follows:
\begin{align}
  \nonumber
  &\left( (l,h_z, g_2)  \#_3\{(h_z, g_2)|(h_y, \del h_x g_1), (h_x, g_1)\} \right) \#_4 \{(l,h_z, g_2), (h_yh_x, g_1)\}_{32} \\
  \nonumber
  &= 
  \left(
    \act{l}\{h_3| h_2, h_1\}^{-1},l\{h_z, h_yh_z\}_{22},\act{\del_1(h_z)}(h'_yh'_x)h_z,g_2g_1
  \right) \\
  \nonumber
  &\quad \#_4
  \left(
    L(l,h_z;h'_yh'_x), \{\del(l)h_z,(h'_yh'_x)\}_{22}\act{(\act{\del_1 h_z}(h'_yh'_x))}l, \act{\del_1 h_z}(h'_yh'_x) h_z, g_2g_1 
  \right) \\
  \label{eq:LHSprism3-22}
  &= 
  \left(
    \act{l}\{h_z| h'_y, h'_x\}^{-1}L(l,h_z;h'_yh'_x), 
    \{\del_2(l)h_z,(h'_yh'_x)\}_{22}\act{(\act{\del_1 h_z}(h'_yh'_x))}l, 
    \act{\del_1 h_z}(h'_yh'_x) h_z, 
    g_2g_1 
  \right)
\end{align}
Here, we write $\act{g_2}h_y$ as $h'_y$ and $\act{g_2}h_x$ as $h'_x$ for convenience.

The RHS of prism 3-22 can be written as follows:
\begin{align}
  \nonumber
  & 
  \Bigl[
    \big(\{ (l, h_z,g_2), \left(h_y, \del_1(h_x) g_1\right)\}_{32} \#_2 (\act{g_2}h_1, g_1 ) \big) \\
    \nonumber
    &\quad \#_3 
    \left( (\act{\del_1(h_z)g_2}h_y, \del_1(h_z)g_2\del_1(h_x)g_1) \#_2\{ (h_z,g_2), (h_x, g_1)\} \right)
  \Bigr] \\
  \nonumber
  &\quad \#_4
  \Bigl[
    \left(\{(\del_2(l)h_z,g_2), (h_y,\del_1(h_x)g_1) \} \#_2 (\act{g_2}h_x, g_2g_1) \right) \\
    \nonumber
    &\quad \#_3
    \left( (\act{\del_1(h_z)g_2}h_y, \del_1(h_z)g_2\del_1(h_x)g_1)\#_2\{(l,h_z,g_2), (h_x,g_1)\}_{32} \right)
  \Bigr] \\
  \nonumber
  &\quad \#_4 
  \left[ 
    \{(\del_2 h_z, g_2)|(h_y, \del h_x g_1), (h_x, g_1) \} 
    \#_3 
    (\act{\del_2(h_z)g_2}(h_yh_x), \del_2(h_z)g_2g_1) 
  \right] \\
  \nonumber
  &= 
  \Bigl(
    L(l,h_z;h'_y)\act{\{\partial_2 lh_z,h'_y\}}\left(\act{\act{\partial_1 h_z}h'_y}L(l,h_z;h'_x)\right)\{\partial_2 lh_z|h'_y,h'_x\}^{-1}, \\
    &\qquad\{\del_2(l)h_z,(h'_yh'_x)\}_{22}\act{(\act{\del_1 h_z}(h'_yh'_x))}l, \quad
    \act{\del_1 h_z}(h'_yh'_x) h_z, \quad
    g_2g_1 
  \Bigr)
\end{align}
Here, we write $\act{g_2}h_y$ as $h'_y$ and $\act{g_2}h_x$ as $h'_x$ for convenience. 

As the two sides of prism 3-22 differ only in their first term, which lies in $M$, it remains to establish the following identity.
\[
  \act{l}\{h_z|h_y,h_x\}^{-1}L(l,h_z;h_yh_x) = L(l,h_z;h_y)\act{\{\partial lh_z,h_y\}}\left(\act{\act{\partial h_z}h_y}L(l,h_z;h_x)\right)\{\partial lh_z|h_y,h_x\}^{-1}.
\]
For simplicity, we omit the apostrophe from here on and write them as $h_x$ and $h_y$.

We begin with the left-hand side. First, by the definition of $L$,
\[
  \text{(L.H.S)} = \red{\act{l}\{h_z|h_y,h_x\}^{-1}\{l,\{h_z,h_yh_x\}\}}\blue{\act{\act{\partial l}\{h_z,h_yh_x\}}\left<l,\act{\partial h_z}(h_yh_x)\right>}\{\partial l,h_z|h_yh_x\}^{-1}
\]
Focusing on the two red factors and applying the decomposition formula \ref{ax:3CM-BoundaryLiftings} for the 22-Peiffer lifting from Definition~\ref{def:3CM} to the interior of the 33-Peiffer lifting, we obtain
\[
  \red{\act{l}\{h_z|h_y,h_x\}^{-1}\{l,\{h_z,h_yh_x\}\}} = \red{\act{l}\{h_z|h_y,h_x\}^{-1}\{l,\partial\{h_z|h_y,h_x\}^{-1}\{h_z,h_y\}\act{\act{\partial h_z}h_x}\{h_z,h_x\}\}}
\]
Furthermore, from the 2-functoriality relation \eqref{eq:2CM-2-funct2} of the 2-crossed module,
\[
  \red{\act{l}\{h_z|h_y,h_x\}^{-1}\{l,\{h_z,h_yh_x\}\}} = \red{\{l,\{h_z,h_y\}\act{\act{\partial h_z}h_x}\{h_z,h_x\}\}\act{\partial l}\{h_z|h_y,h_x\}^{-1}}
\]
Also, using the decomposition formula for the 33-Peiffer lifting \eqref{eq:H_HH-2CM},
\[
  \red{\act{l}\{h_z|h_y,h_x\}^{-1}\{l,\{h_z,h_yh_x\}\}} = \red{\{l,\{h_z,h_y\}\}\act{\act{\partial l}\{h_z,h_y\}}\{l,\act{\act{\partial h_z}h_x}\{h_z,h_x\}\}\act{\partial l}\{h_z|h_y,h_x\}^{-1}}
\]
Next, we turn to the blue-colored part. Rewriting the Prism relation \ref{ax:3CM-Prisms} of Definition~\ref{def:3CM} in terms of the bracket $\left<,\right>$ yields
\begin{equation}
  \label{eq:braket322}
  \left<l,hh'\right> = \left<l,h\right>\act{\{\partial l,h\}}(\act{h}\left<l,h'\right>)\{\partial l|h,h'\}^{-1}
\end{equation}
Using this, we obtain
\[
  \blue{\act{\act{\partial l}\{h_z,h_yh_x\}}\left(\left<l,\act{\partial h_z}h_y\right>\act{\{\partial l,\act{\partial h_z}h_y\}}(\act{\act{\partial h_z}h_y}\left<l,\act{\partial h_z}h_x\right>)\{\partial l|\act{\partial h_z}h_y,\act{\partial h_z}h_x\}^{-1}\right)}
\]
From the above, the left-hand side becomes
\begin{align*}
  \text{(L.H.S)} &= 
  \{l,\{h_z,h_y\}\}\red{\act{\act{\partial l}\{h_z,h_y\}}\{l,\act{\act{\partial h_z}h_x}\{h_z,h_x\}\}}\blue{\act{\partial l}\{h_z|h_y,h_x\}^{-1}} \\
  &\qquad \act{\act{\partial l}\{h_z,h_yh_x\}}\left(\left<l,\act{\partial h_z}h_y\right>\act{\{\partial l,\act{\partial h_z}h_y\}}(\act{\act{\partial h_z}h_y}\left<l,\act{\partial h_z}h_x\right>)\{\partial l|\act{\partial h_z}h_y,\act{\partial h_z}h_x\}^{-1}\right) \\
  &\qquad \{\partial l,h_z|h_yh_x\}^{-1} \\
  &= 
  \{l,\{h_z,h_y\}\}\green{\act{\act{\partial l}\{h_z,h_y\}l\act{\act{\partial h_z}h_y}\{h_z,h_x\}l^{-1}}\left<l,\act{\partial h_z}h_y\right>}\red{\act{\act{\partial l}\{h_z,h_y\}}\{l,\act{\act{\partial h_z}h_y}\{h_z,h_x\}\}} \\
  &\qquad \act{\act{\partial l}\{h_z,h_y\}\act{\partial l\act{\partial h_z}h_y}\{h_z,h_x\}\{\partial l,\act{\partial h_z}h_y\}}(\act{\act{\partial h_z}h_y}\left<l,\act{\partial h_z}h_x\right>) \\
  &\qquad \blue{\act{\partial l}\{h_z|h_y,h_x\}^{-1}}\act{\act{\partial l}\{h_z,h_yh_x\}}\{\partial l|\act{\partial h_z}h_y,\act{\partial h_z}h_x\}^{-1}\{\partial l,h_z|h_yh_x\}^{-1}
\end{align*}
We reorder the colored factors using the Peiffer identity (recoloring as we go, so that the correspondence between colors and expressions differs from the previous step). Focusing on the green part, we apply equation \eqref{eq:LH-2CM} for the action of $H$ on $L$ in a 2-crossed module from Definition~\ref{def:2CM} and obtain
\begin{align*}
  &\green{\act{\act{\partial l}\{h_z,h_y\}l\act{\act{\partial h_z}h_y}\{h_z,h_x\}l^{-1}}\left<l,\act{\partial h_z}h_y\right>}\red{\act{\act{\partial l}\{h_z,h_y\}}\{l,\act{\act{\partial h_z}h_y}\{h_z,h_x\}\}} \\
  &= \green{\act{\act{\partial l}\{h_z,h_y\}}\left<l,\act{\partial h_z}h_y\right>\act{\act{\partial l}\{h_z,h_y\}}\{\partial \left<l,\act{\partial h_z}h_y\right>^{-1},l\act{\act{\partial h_z}h_y}\{h_z,h_x\}l^{-1}\}}\red{\act{\act{\partial l}\{h_z,h_y\}}\{l,\act{\act{\partial h_z}h_y}\{h_z,h_x\}\}}
\end{align*}
Focusing on the second green factor together with the red factor, and combining the 2-functoriality with the other relations of the 2-crossed module, one can show
\[
  \{\partial m,ll'l^{-1}\}\{l,l'\} = \{\partial m,\partial\{l,l'\}\act{\partial l}l'\}\act{\partial\partial m}\{l,l'\} = \act{\partial m}\{l,l'\}\{\partial m,\act{\partial l}l'\} = \{\partial ml,l'\}
\]
Using this, we obtain
\begin{align*}
  &\green{\act{\act{\partial l}\{h_z,h_y\}l\act{\act{\partial h_z}h_y}\{h_z,h_x\}l^{-1}}\left<l,\act{\partial h_z}h_y\right>}\red{\act{\act{\partial l}\{h_z,h_y\}}\{l,\act{\act{\partial h_z}h_y}\{h_z,h_x\}\}} \\
  &= \green{\act{\act{\partial l}\{h_z,h_y\}}\left<l,\act{\partial h_z}h_y\right>\act{\act{\partial l}\{h_z,h_y\}}\{\partial \left<l,\act{\partial h_z}h_y\right>^{-1}l,\act{\act{\partial h_z}h_y}\{h_z,h_x\}\}} \\
  &= \green{\act{\act{\partial l}\{h_z,h_y\}}\left<l,\act{\partial h_z}h_y\right>\act{\act{\partial l}\{h_z,h_y\}}\{\{\partial l,\act{\partial h_z}h_y\}\act{\act{\partial h_z}h_y}l,\act{\act{\partial h_z}h_y}\{h_z,h_x\}\}}
\end{align*}
Here, in the last equality we used the boundary formula \ref{ax:3CM-BoundaryLiftings} of Definition~\ref{def:3CM}. Furthermore, using the decomposition relation \eqref{eq:HH_H-2CM} of the Peiffer lifting of the 2-crossed module from Definition~\ref{def:2CM},
\begin{align*}
  &\green{\act{\act{\partial l}\{h_z,h_y\}l\act{\act{\partial h_z}h_y}\{h_z,h_x\}l^{-1}}\left<l,\act{\partial h_z}h_y\right>}\red{\act{\act{\partial l}\{h_z,h_y\}}\{l,\act{\act{\partial h_z}h_y}\{h_z,h_x\}\}} \\
  &= \green{\act{\act{\partial l}\{h_z,h_y\}}\left<l,\act{\partial h_z}h_y\right>\act{\act{\partial l}\{h_z,h_y\}\{\partial l,\act{\partial h_z}h_y\}}\{\act{\act{\partial h_z}h_y}l,\act{\act{\partial h_z}h_y}\{h_z,h_x\}\}\act{\act{\partial l}\{h_z,h_y\}}\{\{\partial l,\act{\partial h_z}h_y\},\act{\act{\partial h_z}h_y\partial l}\{h_z,h_x\}\}}
\end{align*}
Writing out the entire left-hand side once more,
\begin{align*}
  (\text{L.H.S}) &= \{l,\{h_z,h_y\}\}\act{\act{\partial l}\{h_z,h_y\}}\left<l,\act{\partial h_z}h_y\right> \\
  &\qquad \act{\act{\partial l}\{h_z,h_y\}\{\partial l,\act{\partial h_z}h_y\}}\{\act{\act{\partial h_z}h_y}l,\act{\act{\partial h_z}h_y}\{h_z,h_x\}\}\red{\act{\act{\partial l}\{h_z,h_y\}}\{\{\partial l,\act{\partial h_z}h_y\},\act{\act{\partial h_z}h_y\partial l}\{h_z,h_x\}\}} \\
  &\qquad \act{\act{\partial l}\{h_z,h_y\}\act{\partial l\act{\partial h_z}h_y}\{h_z,h_x\}\{\partial l,\act{\partial h_z}h_y\}}(\act{\act{\partial h_z}h_y}\left<l,\act{\partial h_z}h_x\right>) \\
  &\qquad \act{\partial l}\{h_z|h_y,h_x\}^{-1}\act{\act{\partial l}\{h_z,h_yh_x\}}\{\partial l|\act{\partial h_z}h_y,\act{\partial h_z}h_x\}^{-1}\{\partial l,h_z|h_yh_x\}^{-1}
\end{align*}
(We have recolored.) Focusing on the red factor and reordering using the Peiffer identity,
\begin{align*}
  (\text{L.H.S}) &= \{l,\{h_z,h_y\}\}\act{\act{\partial l}\{h_z,h_y\}}\left<l,\act{\partial h_z}h_y\right> \\
  &\qquad \act{\blue{\act{\partial l}\{h_z,h_y\}\{\partial l,\act{\partial h_z}h_y\}}}\left(\{\act{\act{\partial h_z}h_y}l,\act{\act{\partial h_z}h_y}\{h_z,h_x\}\}\act{\act{\act{\partial h_z}h_y\partial l}\{h_z,h_x\}}(\act{\act{\partial h_z}h_y}\left<l,\act{\partial h_z}h_x\right>)\right) \\
  &\qquad \red{\act{\act{\partial l}\{h_z,h_y\}}\{\{\partial l,\act{\partial h_z}h_y\},\act{\act{\partial h_z}h_y\partial l}\{h_z,h_x\}\}} \\
  &\qquad \act{\partial l}\{h_z|h_y,h_x\}^{-1}\act{\act{\partial l}\{h_z,h_yh_x\}}\{\partial l|\act{\partial h_z}h_y,\act{\partial h_z}h_x\}^{-1}\{\partial l,h_z|h_yh_x\}^{-1}
\end{align*}
Focus on the blue part. Using $\act{\partial l}\{h_z,h_y\}\{\partial l,\act{\partial h_z}h_y\} = \partial\{\partial l|h_z,h_y\}^{-1}\{\partial lh_z,h_y\}$ together with the Peiffer identity,
\begin{align*}
  (\text{L.H.S}) &= \{l,\{h_z,h_y\}\}\act{\act{\partial l}\{h_z,h_y\}}\left<l,\act{\partial h_z}h_y\right>\{\partial l,h_z|h_y\}^{-1} \\
  &\qquad \act{\{\partial lh_z,h_y\}}\left(\act{\act{\partial h_z}h_y}\left[\{l,\{h_z,h_x\}\}\act{\act{\partial l}\{h_z,h_x\}}\left<l,\act{\partial h_z}h_x\right>\right]\right) \\
  &\qquad \red{\{\partial l,h_z|h_y\}\act{\act{\partial l}\{h_z,h_y\}}\left(\{\{\partial l,\act{\partial h_z}h_y\},\act{\act{\partial h_z}h_y\partial l}\{h_z,h_x\}\}\right)} \\
  &\qquad \red{\act{\partial l}\{h_z|h_y,h_x\}^{-1}\act{\act{\partial l}\{h_z,h_yh_x\}}\{\partial l|\act{\partial h_z}h_y,\act{\partial h_z}h_x\}^{-1}\{\partial l,h_z|h_yh_x\}^{-1}}
\end{align*}
(We have recolored.) Focus on the red part. Using the KV-polytope relation \ref{ax:3CM-KV} of Definition~\ref{def:3CM} with $h_w = \partial l$, 
\begin{align*}
  (\text{L.H.S}) &= \{l,\{h_z,h_y\}\}\act{\act{\partial l}\{h_z,h_y\}}\left<l,\act{\partial h_z}h_y\right>\{\partial l,h_z|h_y\}^{-1} \\
  &\qquad \act{\{\partial lh_z,h_y\}}\left(\act{\act{\partial h_z}h_y}\left[\{l,\{h_z,h_x\}\}\act{\act{\partial l}\{h_z,h_x\}}\left<l,\act{\partial h_z}h_x\right>\right]\right) \\
  &\qquad \act{\{\partial lh_z,h_y\}}(\act{\act{\partial (\partial lh_z)}h_y}\{\partial l,h_z|h_x\})^{-1}\{\partial lh_z|h_y,h_x\}^{-1} \\
  &= L(l,h_z;h_y)\act{\{\partial lh_z,h_y\}}(\act{\act{\partial h_z}h_y}L(l,h_z;h_x))\{\partial lh_z|h_y,h_x\}^{-1} \\
  &= \text{(R.H.S)}
\end{align*}

The other relations can be proved in a similar manner. We give the calculation in Appendix \ref{proof:3-cm_to_Gray4-group}.

\end{proof}

\subsection{Equivalence of 3-crossed modules and Gray 4-groups}\label{sec:thm}

In Section \ref{sec:Gray323CM} we constructed the functor $\Delta: \Gray \rightarrow \TXMod $, and above we constructed the functor $\Theta : \TXMod \rightarrow \Gray$. We now show that these two functors establish a categorical equivalence between $\TXMod$ and $ \Gray$.

\begin{theorem}
  \label{thm:equivalence_3cm_gray4group}
Let $\TXMod$ be the category of 3-crossed modules and let $ \Gray$ be the category of Gray 4-groups. Then $\TXMod$ and $ \Gray$ are categorically equivalent.
\end{theorem}
\begin{proof}
To prove that $\TXMod$ and $\Gray$ are categorically equivalent, we will show that the composite $\Theta \circ \Delta:~\Gray\to\Gray$ is naturally isomorphic to ${\rm Id}_{\text{Gray}_4} : \Gray \to \Gray$ and that $\Delta \circ \Theta:~\TXMod \to \TXMod$ is naturally isomorphic to ${\rm Id}_{\text{3CM}} : \TXMod \to \TXMod$, where ${\rm Id}_{\text{Gray}_4}$ and ${\rm Id}_{\text{3CM}}$ denote the relevant identity functors.

\vspace{.5em}
\noindent
$\rm(\hspace{.18em}i\hspace{.18em})$ Natural isomorphism $\alpha : {\rm Id}_{\text{Gray}_4} \to \Theta \circ \Delta$.
For each object $\cC \in \obj(\Gray)$, the component $\alpha_{\cC}$ of this natural isomorphism is constructed as follows.
\begin{itemize}
  \item For $* \in \obj(\cC)$, 
  \[
  \alpha_\cC (*) = *.
  \]
  \item For $ g \in \mor_1(\cC)$, 
  \[
  \alpha_\cC (g) = g.
  \]
  \item For $ h \in \mor_2(\cC)$, 
  \[
  \alpha_\cC (h) = \left(h \#_1 s_1(h)^{-1_1}, s_1(h)\right).
  \]
  \item For $ l \in \mor_3(\cC)$, 
  \[
  \alpha_\cC (l) = \left(\left(l \#_2 s_2(l)^{-1_2}\right)\#_1 t_1(l)^{-1_1},s_2(l) \#_1 s_1(l)^{-1_1},s_1(l)\right)
  \]
  \item For $ m \in \mor_4(\cC)$, 
  \[
  \alpha_\cC (m) = (x,y,z,w),
  \]
  where
  \begin{align*}
    x &= \left(\left(m \#_3 s_3(m)^{-1_3} \right) \#_2 t_2(m)^{-1_2}\right) \#_1 t_1(m)^{-1_1}, & z &= s_2(m) \#_1 s_1(m)^{-1_1}, \\ y &= \left(s_3(m) \#_2 s_2(m)^{-1_2}\right)\#_1 t_1(m)^{-1_1}, & w &= s_1(m).
  \end{align*}
\end{itemize}

It is clear that $\alpha$ is a natural transformation and that the morphism $\alpha_\cC$ is an isomorphism. The morphism $\alpha'_\cC : \Theta\circ \Delta(\cC) \rightarrow \cC$, which is an inverse of the morphism $\alpha_\cC$, is constructed as follows.
\begin{itemize}
  \item For $* \in \obj\left(\Theta\circ \Delta(\cC)\right)$, 
  \[
  \alpha'_\cC (*) = *.
  \]
  \item For $g \in \mor_1\left(\Theta\circ \Delta(\cC)\right)$, 
  \[
  \alpha'_\cC (g) = g.
  \]
  \item For $(h,g) \in \mor_2\left(\Theta\circ \Delta(\cC)\right)$, 
  \[
  \alpha'_\cC ((h,g)) = h \#_1 g .
  \]
  \item For $(l,h,g) \in \mor_3\left(\Theta\circ \Delta(\cC)\right)$, 
  \[
  \alpha'_\cC ((l,h,g)) = \left(l \#_1 \left(h({\rm id_1}) \#_1 g\right)  \right) \#_2 (h \#_1 g) .
  \]
  \item For $(m,l,h,g) \in \mor_4\left(\Theta\circ \Delta(\cC)\right)$, 
  \[
  \alpha'_\cC \left( (m,l,h,g) \right) = \left(\left(\left(\left( m \#_2 l({\rm id_2}) \right) \#_1 \left(h({\rm id_1}) \#_1 g\right) \right) \right)  \#_3 l \right)  \#_2 (h \#_1 g).
  \]
\end{itemize}

\vspace{.5em}
\noindent
$\rm(\hspace{.08em}ii\hspace{.08em})$ Natural isomorphism $\beta : {\rm Id}_{\text{3CM}} \rightarrow \Delta \circ \Theta$. The natural transformation $\beta$ is defined trivially. For any 3-crossed module $\bG=(M\rightarrow L\rightarrow H\rightarrow G) \in \obj(\TXMod)$, the component $\beta_\bG$ is given as follows.
\begin{itemize}
  \item For $g \in G$, 
  \[
  \beta_\bG (g) = g.
  \]
  \item For $h \in H$, 
  \[
  \beta_\bG (h) = (h,1).
  \]
  \item For $l \in L$, 
  \[
  \beta_\bG (l) = (l,1,1).
  \]
  \item For $m \in M$, 
  \[
  \beta_\bG (m) = (m,1,1,1).
  \]
\end{itemize}

Each component $\beta_\bG$ is clearly an isomorphism, which completes the proof.
\end{proof}

We have just related 3-crossed modules to Gray 4-groups. On the other hand, a Gray 4-category is also related to a braided monoidal 2-category (see \cite{Baez:1996} or \cite{Crans:1998} for the definition): in fact, the category of Gray 4-categories with trivial object and trivial 1-morphisms is isomorphic to the category of braided monoidal 2-categories. In this context, a 3-crossed module may be regarded as the group version of a braided monoidal 2-category.

\begin{lemma}
Let ${\bf Gray_{4*}}$ be the category of Gray 4-categories whose object is trivial and whose 1-morphisms are trivial, and let ${\bf BrdMonCat}$ be the category of braided monoidal 2-categories. ${\bf Gray_{4*}}$ and ${\bf BrdMonCat}$ are isomorphic as categories.
\end{lemma}
\begin{proof}
We now spell out the correspondence explicitly. For a braided monoidal 2-category:
\[
  (\cC, \bigotimes, I, R, \tilde{R}_{(-|-,-)}, \tilde{R}_{(-,-|-)} ),
\]
we construct an object of ${\bf Gray_{4*}}$ as follows:

\vspace{.5em}
\noindent
\textbf{Object}: The object is trivial $\{*\}$. 

\vspace{.5em}
\noindent
\textbf{1-morphism}: The 1-morphism is trivial $\{*\}$.

\vspace{.5em}
\noindent
\textbf{2-morphism}: Each object $A \in \cC$ becomes a 2-morphism $A : * \rightarrow * $ in the Gray 4-category, and for two such 2-morphisms $A,B$ the vertical composition $A\#_2 B$ is defined via the tensor $A\bigotimes B$ of the braided monoidal 2-category. 

\vspace{.5em}
\noindent
\textbf{3-morphism}: Each 1-morphism $f$ becomes a 3-morphism $f$ in the Gray 4-category, and for two such 3-morphisms $f,g$ the vertical composition $f\#_3 g$ is defined via their composition $f \circ g$ in the braided monoidal 2-category. 

\vspace{.5em}
\noindent
\textbf{4-morphism}: Each 2-morphism $\alpha$ becomes a 4-morphism $\alpha$ in the Gray 4-category, and for two such morphisms $\alpha, \beta$ the vertical composition $\alpha \#_4 \beta$ is defined via their composition $\alpha \circ \beta$ in the braided monoidal 2-category. 

\vspace{.5em}
\noindent
\textbf{Whiskering}: Whiskering is likewise defined using the structure of the braided monoidal 2-category: for any 2-morphism $A$ and any i-morphism $x$ in the Gray 4-category with $i = 3,4$, the whiskering $A \#_2 x$ is given by $A \bigotimes x$, using that $\bigotimes$ is a 2-functor of the braided monoidal 2-category. Whiskering by 3-morphisms, in turn, is defined through horizontal composition of 1-morphisms in the braided monoidal 2-category.

\vspace{.5em}
\noindent
\textbf{Liftings}: The structure of the 22-, 32-, and 23-liftings in the Gray 4-category is expressed using the pseudonatural equivalence $R$ of the braided monoidal 2-category. Thus, for two 2-morphisms $(A, B) \in\mor_2({\bf Gray_{4*}})\times_1\mor_2({\bf Gray_{4*}})$, the 33-Peiffer lifting $\{A,B\}_{22}$ is given by $R_{A,B}$; for morphisms $(f,A) \in \mor_3({\bf Gray_{4*}})\times_0\mor_2({\bf Gray_{4*}})$, the 32-lifting $\{f,A\}_{32}$ is given by $R_{f,A}$. The 23-lifting is obtained in a similar manner.

\vspace{.5em}
\noindent
\textbf{33-Lifting}: The 33-lifting structure in the Gray 4-category is constructed using the 2-functoriality of $\bigotimes$. Indeed, for two 3-morphisms $(f,g)\in\mor_3({\bf Gray_{4*}})\times_1\mor_3({\bf Gray_{4*}})$, the 33-lifting $\{f,g\}_{33}$ is defined via $\bigotimes_{f,g}$.

\vspace{.5em}
\noindent
\textbf{Homanians}: The twists given by the left-Homanian $\{-,-|-\}$ and the right-Homanian $\{-|-,-\}$ are described directly via $\tilde{R}_{(-|-,-)}$ and $\tilde{R}_{(-,-|-)}$. For 2-morphisms $(B_2,B_1,A) \in {\lparen}\mor_2({\bf Gray_{4*}})\times_1\mor_2({\bf Gray_{4*}}){\rparen}\times_0\mor_2({\bf Gray_{4*}})$, the left-Homanian $\{B_2,B_1|A\}$ is defined via $\tilde{R}_{(B_2,B_1|A)}$; the right-Homanian is handled in a similar manner.

\end{proof}

\begin{remark}
We have seen that the 3-crossed modules introduced in this paper are equivalent to Gray 4-groups (Theorem \ref{thm:equivalence_3cm_gray4group}). The definition of a 3-crossed module given in our previous work \cite{Fukuda:2025} is more economical: it omits the two prism relations (\ref{eq:3cm_prism23-2}) and (\ref{eq:3cm_prism2-23}). By the same argument as in Theorem \ref{thm:equivalence_3cm_gray4group}, the 3-crossed modules of \cite{Fukuda:2025} are equivalent to a variant of Gray 4-groups in which these two prism relations, Prism 23-2 and Prism 2-23, are not imposed. The addition of these two relations in the present paper is precisely what is needed to obtain an equivalence with the full notion of Gray 4-groups.
\end{remark}

\section{Conclusion and Discussion}\label{sec:con}
In this paper, we proved a categorical equivalence between the category of 3-crossed modules $\TXMod$ and the category of Gray 4-groups $\Gray$, the groupoid version of Gray 4-categories corresponding to semistrict braided monoidal 2-categories. This result determines the place of the definition of 3-crossed modules given in our previous work \cite{Fukuda:2025} within the literature of higher category theory.

Just as crossed modules and 2-crossed modules are closely related to the construction of ordinary knot invariants and topological invariants of three- and four-dimensional manifolds, the 3-crossed modules defined in this work are strongly expected to provide a framework for constructing algebraic invariants of 2-knots (surface knots) and topological invariants of five dimensional manifolds. In particular, the fact that a semistrict braided monoidal 2-category gives a solution of the Zamolodchikov tetrahedron equation is immediately translated, via the equivalence established in this paper, into a corresponding statement for 3-crossed modules.

However, the essential point is that the algebraic structure underlying the tetrahedron equation can be described purely in terms of a 3-crossed module, independently of categorical language. In actual computations, we can directly associate some of the elementary moves in a motion-picture (movie) description of a 2-knot with Peiffer liftings of a 3-crossed module (and their compositions).\footnote{We can establish this correspondence directly for most moves, with the exception of the generalization of Reidemeister move 1; with a little ingenuity, the latter can also be accommodated.} Under this correspondence, each Roseman move -- a relation among movies built from the collection of fundamental moves \cite{Roseman1998}, and one of the key ingredients in the construction of 2-knot invariants -- can be interpreted in terms of an algebraic relation among elements of the 3-crossed module. Moreover, we can show that these relations hold as a consequence of the axioms of a 3-crossed module. The resolution of the Zamolodchikov tetrahedron equation can therefore be recast in a self-contained form, as a concrete algebraic computation within a 3-crossed module, independently of the Kapranov--Voevodsky approach.

At first glance, this correspondence might seem to complete the construction of 2-knot invariants. However, a movie description of closed knots alone is insufficient for applying 3-crossed modules; instead, it is essential to treat them as movie descriptions of tangle diagrams. In this tangle setting, the precise set of Roseman-type moves required to guarantee topological invariance remains to be fully determined.

\addtocontents{toc}{\protect\setcounter{tocdepth}{1}}
\appendix
\section{Diagrams of the fundamental relations of 3-crossed modules}\label{app:3CM-diagram-list}

In the diagrammatic representation of 3-crossed modules, since the auxiliary brackets $\left<-,-\right>$ are more convenient than $\{-,-\}$, we use them. However, these auxiliary brackets are obtained by dropping part of the information of the 2-sources from \eqref{eq:defL} and \eqref{eq:defR}, which directly correspond to the categorical language. Consequently, the diagrams representing the fundamental relations in terms of these brackets may look different from their categorical counterparts expressed in terms of \eqref{eq:defL} and \eqref{eq:defR}. Nevertheless, in order to make the correspondence between the algebraic and categorical descriptions transparent, we use the same names for the diagrams in the following as for their categorical counterparts.
\subsection{right-Homanian relation}
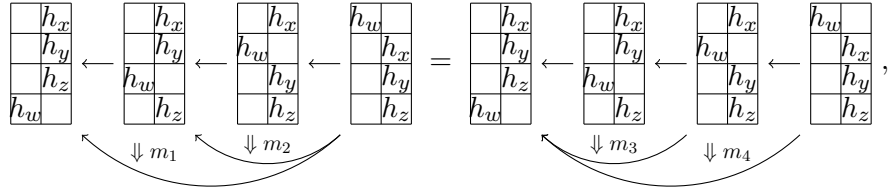
\[
  \begin{tikzpicture}[baseline=-2pt]
    \node (A) at (2.25,0) {\tikz[baseline=-2pt,scale=.4]{\strsquare{2}{4}{$h_w$,,,$h_x$,,$h_y$,,$h_z$}}};
    \node (B) at (.75,0) {\tikz[baseline=-2pt,scale=.4]{\strsquare{2}{4}{,$h_x$,$h_w$,,,$h_y$,,$h_z$}}};
    \node (C) at (-.75,0) {\tikz[baseline=-2pt,scale=.4]{\strsquare{2}{4}{,$h_x$,,$h_y$,$h_w$,,,$h_z$}}};
    \node (D) at (-2.25,0) {\tikz[baseline=-2pt,scale=.4]{\strsquare{2}{4}{,$h_x$,,$h_y$,,$h_z$,$h_w$,}}};
    \draw [->] (A) to (B);
    \draw [->] (B) to (C);
    \draw [->] (C) to (D);
    \draw [->,bend left=1.5cm] (A.south west) to (D.south east);
    \draw [->,bend left=1.5cm] (A.south west) to node (D) {} (C.south east);
    \path (B.north) to [pos=15/16] coordinate (A) (B.south);
    \path (B) to node [transform shape,scale=.75] {$\Downarrow m_2$} (D);
    \path (C.north) to [pos=5/4] coordinate (D) (C.south);
    \path (C) to node [transform shape,scale=.75] {$\Downarrow m_1$} (D);
  \end{tikzpicture}
  =
  \begin{tikzpicture}[baseline=-2pt]
    \node (A) at (2.25,0) {\tikz[baseline=-2pt,scale=.4]{\strsquare{2}{4}{$h_w$,,,$h_x$,,$h_y$,,$h_z$}}};
    \node (B) at (.75,0) {\tikz[baseline=-2pt,scale=.4]{\strsquare{2}{4}{,$h_x$,$h_w$,,,$h_y$,,$h_z$}}};
    \node (C) at (-.75,0) {\tikz[baseline=-2pt,scale=.4]{\strsquare{2}{4}{,$h_x$,,$h_y$,$h_w$,,,$h_z$}}};
    \node (D) at (-2.25,0) {\tikz[baseline=-2pt,scale=.4]{\strsquare{2}{4}{,$h_x$,,$h_y$,,$h_z$,$h_w$,}}};
    \draw [->] (A) to (B);
    \draw [->] (B) to (C);
    \draw [->] (C) to (D);
    \draw [->,bend left=1.5cm] (A.south west) to (D.south east);
    \draw [->,bend left=1.5cm] (B.south west) to node (A) {} (D.south east);
    \path (C) to node [transform shape,scale=.75] {$\Downarrow m_3$} (A);
    \path (B.north) to [pos=5/4] coordinate (D) (B.south);
    \path (B) to node [transform shape,scale=.75] {$\Downarrow m_4$} (D);
  \end{tikzpicture},
\]
where
\begin{align}
  \nonumber
  m_1 &= \{h_w|h_z,h_yh_x\}, &
  m_3 &= \{h_w|h_z,h_y\}, \\
  \nonumber
  m_2 &= \act{\act{\partial h_w}h_z}\{h_w|h_y,h_x\}, &
  m_4 &= \{h_w|h_zh_y,h_x\}.
\end{align}

\subsection{left-Homanian relation}
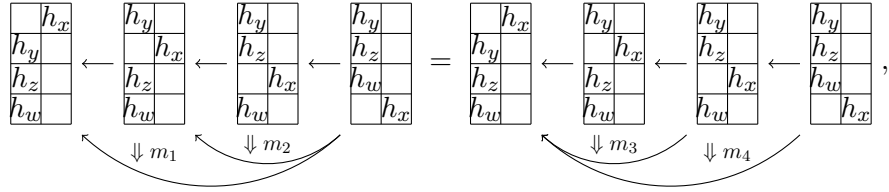
\[
  \begin{tikzpicture}[baseline=-2pt]
    \node (A) at (2.25,0) {\tikz[baseline=-2pt,scale=.4]{\strsquare{2}{4}{$h_y$,,$h_z$,,$h_w$,,,$h_x$}}};
    \node (B) at (.75,0) {\tikz[baseline=-2pt,scale=.4]{\strsquare{2}{4}{$h_y$,,$h_z$,,,$h_x$,$h_w$,}}};
    \node (C) at (-.75,0) {\tikz[baseline=-2pt,scale=.4]{\strsquare{2}{4}{$h_y$,,,$h_x$,$h_z$,,$h_w$,}}};
    \node (D) at (-2.25,0) {\tikz[baseline=-2pt,scale=.4]{\strsquare{2}{4}{,$h_x$,$h_y$,,$h_z$,,$h_w$,}}};
    \draw [->] (A) to (B);
    \draw [->] (B) to (C);
    \draw [->] (C) to (D);
    \draw [->,bend left=1.5cm] (A.south west) to (D.south east);
    \draw [->,bend left=1.5cm] (A.south west) to node (D) {} (C.south east);
    \path (B.north) to [pos=15/16] coordinate (A) (B.south);
    \path (B) to node [transform shape,scale=.75] {$\Downarrow m_2$} (D);
    \path (C.north) to [pos=5/4] coordinate (D) (C.south);
    \path (C) to node [transform shape,scale=.75] {$\Downarrow m_1$} (D);
  \end{tikzpicture}
  =
  \begin{tikzpicture}[baseline=-2pt]
    \node (A) at (2.25,0) {\tikz[baseline=-2pt,scale=.4]{\strsquare{2}{4}{$h_y$,,$h_z$,,$h_w$,,,$h_x$}}};
    \node (B) at (.75,0) {\tikz[baseline=-2pt,scale=.4]{\strsquare{2}{4}{$h_y$,,$h_z$,,,$h_x$,$h_w$,}}};
    \node (C) at (-.75,0) {\tikz[baseline=-2pt,scale=.4]{\strsquare{2}{4}{$h_y$,,,$h_x$,$h_z$,,$h_w$,}}};
    \node (D) at (-2.25,0) {\tikz[baseline=-2pt,scale=.4]{\strsquare{2}{4}{,$h_x$,$h_y$,,$h_z$,,$h_w$,}}};
    \draw [->] (A) to (B);
    \draw [->] (B) to (C);
    \draw [->] (C) to (D);
    \draw [->,bend left=1.5cm] (A.south west) to (D.south east);
    \draw [->,bend left=1.5cm] (B.south west) to node (A) {} (D.south east);
    \path (C) to node [transform shape,scale=.75] {$\Downarrow m_3$} (A);
    \path (B.north) to [pos=5/4] coordinate (D) (B.south);
    \path (B) to node [transform shape,scale=.75] {$\Downarrow m_4$} (D);
  \end{tikzpicture},
\]
where
\begin{align}
  \nonumber
  m_1 &= \{h_wh_z,h_y|h_x\}, &
  m_3 &= \act{h_w}\{h_z,h_y|h_x\}, \\
  \nonumber
  m_2 &= \{h_w,h_z|\act{\partial h_y}h_x\}, &
  m_4 &= \{h_w,h_zh_y|h_x\}.
\end{align}

\subsection{KV-Polytope}

\[
  \begin{tikzpicture}[baseline=-2pt]
    \node (A) at (2.9,0) {\tikz[baseline=-2pt,scale=.4]{\strsquare{2}{4}{$h_z$,,$h_w$,,,$h_x$,,$h_y$}}};
    \node (B) at (1.45,0) {\tikz[baseline=-2pt,scale=.4]{\strsquare{2}{4}{$h_z$,,,$h_x$,$h_w$,,,$h_y$}}};
    \node (C1) at (0,-1.125) {\tikz[baseline=-2pt,scale=.4]{\strsquare{2}{4}{$h_z$,,,$h_x$,,$h_y$,$h_w$,}}};
    \node (C2) at (0,1.125) {\tikz[baseline=-2pt,scale=.4]{\strsquare{2}{4}{,$h_x$,$h_z$,,$h_w$,,,$h_y$}}};
    \node (D) at (-1.45,0) {\tikz[baseline=-2pt,scale=.4]{\strsquare{2}{4}{,$h_x$,$h_z$,,,$h_y$,$h_w$,}}};
    \node (E) at (-2.9,0) {\tikz[baseline=-2pt,scale=.4]{\strsquare{2}{4}{,$h_x$,,$h_y$,$h_z$,,$h_w$,}}};
    \draw [->] (A) to (B);
    \draw [->] (B) to (C1);
    \draw [->] (C1) to (D);
    \draw [->] (B) to (C2);
    \draw [->] (C2) to (D);
    \draw [->] (D) to (E);
    \draw [->,bend left=1.25cm] (A) to node (P) {} (C1);
    \draw [->,bend left=1.25cm] (C1) to node (Q) {} (E);
    \draw [->,bend left=2cm] (A.south) to node (R) {} (E.south);
    \path (C2.south) to node [transform shape,scale=.75] {$\Downarrow m_1$} (C1.north);
    \path (Q) to node [transform shape,scale=.75] {$\Downarrow m_3$} (D);
    \path (P) to node [transform shape,scale=.75] {$\Downarrow m_2$} (B);
    \path (R) to node [transform shape,scale=.75] {$\Downarrow m_4$} (C1);
  \end{tikzpicture}
  =
  \begin{tikzpicture}[baseline=-2pt]
    \node (A) at (2.9,0) {\tikz[baseline=-2pt,scale=.4]{\strsquare{2}{4}{$h_z$,,$h_w$,,,$h_x$,,$h_y$}}};
    \node (B) at (1.45,0) {\tikz[baseline=-2pt,scale=.4]{\strsquare{2}{4}{$h_z$,,,$h_x$,$h_w$,,,$h_y$}}};
    \node (C) at (0,0) {\tikz[baseline=-2pt,scale=.4]{\strsquare{2}{4}{,$h_x$,$h_z$,,$h_w$,,,$h_y$}}};
    \node (D) at (-1.45,0) {\tikz[baseline=-2pt,scale=.4]{\strsquare{2}{4}{,$h_x$,$h_z$,,,$h_y$,$h_w$,}}};
    \node (E) at (-2.9,0) {\tikz[baseline=-2pt,scale=.4]{\strsquare{2}{4}{,$h_x$,,$h_y$,$h_z$,,$h_w$,}}};
    \draw [->] (A) to (B);
    \draw [->] (B) to (C);
    \draw [->] (C) to (D);
    \draw [->] (D) to (E);
    \draw [->,bend left=1.6cm] (A.south) to node (F) {} (E.south);
    \draw [->,bend left=1.5cm] (A.south west) to node (A) {} (C.south east);
    \draw [->,bend left=1.5cm] (C.south west) to node (E) {} (E.south east);
    \path (A) to node [transform shape,scale=.75] {$\Downarrow m_5$} (B);
    \path (E) to node [transform shape,scale=.75] {$\Downarrow m_6$} (D);
    \path (F) to node [transform shape,scale=.75] {$\Downarrow m_7$} (C);
  \end{tikzpicture},
\]
\begin{align*}
  m_1 &= \{\{h_w,\act{\partial h_z}h_y\},\act{\act{\partial (h_wh_z)}h_yh_w}\{h_z,h_x\}\}^{-1}, & 
  m_5 &= \act{\act{\partial (h_wh_z)}h_y}\{h_w,h_z|h_x\} \\
  m_2 &= \{h_w|\act{\partial h_z}h_y,\act{\partial h_z}h_x\}, & 
  m_6 &= \{h_w,h_z|h_y\} \\
  m_3 &= \act{h_w}\{h_z|h_y,h_x\}, & 
  m_7 &= \{h_wh_z|h_y,h_x\} \\
  m_4 &= \{h_w,h_z|h_yh_x\}.
\end{align*}

\subsection{prisms}
Rewriting Equations \eqref{eq:3CM-L-HH} and \eqref{eq:3CM-HH-L} in terms of the auxiliary brackets $\left<,\right>$, we obtain
\begin{align}
  \left<l,hh'\right> &= \left<l,h\right>\act{\{\partial l,h\}}(\act{h}\left<l,h'\right>)\{\partial l|h,h'\}^{-1}, \\
  \label{eq:3CM-aux-HH-L}
  \left<hh',l\right> &= \act{h}\left<h',l\right>\act{\act{h}\{h',\partial l\}}\left<h,\act{\partial h'}l\right>\{h,h'|\partial l\}^{-1}.
\end{align}
These relations, Equations \eqref{eq:3cm_prism23-2} and \eqref{eq:3cm_prism2-23} can be depicted as follows:
\[
  \begin{tikzpicture}[baseline=-2pt]
    \node (A) at (1.5,0) {\tikz[baseline=-2pt,scale=.4]{\strsquare{1}{2}{$h'$,$h$}}};
    \node (B) at (0,1.75) {\tikz[baseline=-2pt,scale=.4]{\strsquare{2}{3}{$\partial l$,,,$h'$,,$h$}}};
    \node (C) at (-1.5,0) {\tikz[baseline=-2pt,scale=.4]{\strsquare{2}{3}{,$h'$,,$h$,$\partial l$,}}};
    \draw [->] (A) to (B);
    \draw [->] (B) to (C);
    \draw [->] (A) to coordinate (D) (C);
    \path (B) to node [transform shape,scale=.75] {$\Downarrow \left<l,hh'\right>$} (D);
  \end{tikzpicture}
  =
  \begin{tikzpicture}[baseline=-2pt]
    \node (A) at (3,0) {\tikz[baseline=-2pt,scale=.4]{\strsquare{1}{2}{$h'$,$h$}}};
    \node (B) at (0,3.5) {\tikz[baseline=-2pt,scale=.4]{\strsquare{2}{3}{$\partial l$,,,$h'$,,$h$}}};
    \node (C) at (-1.5,1.75) {\tikz[baseline=-2pt,scale=.4]{\strsquare{2}{3}{,$h'$,$\partial l$,,,$h$}}};
    \node (D) at (-3,0) {\tikz[baseline=-2pt,scale=.4]{\strsquare{2}{3}{,$h'$,,$h$,$\partial l$,}}};
    \draw [->] (A) to (B);
    \draw [->,bend right=1cm] (B.west) to node [transform shape,scale=.75,sloped,anchor=north] {$\Downarrow \{\partial l|h,h'\}^{-1}$} (D.north);
    \draw [->] (B) to (C);
    \draw [->] (C) to (D);
    \draw [->] (A) to coordinate (E) (D);
    \draw [->] (A) to coordinate (F) (C);
    \path (B) to node [transform shape,scale=.75] {$\Downarrow \act{h}\left<l,h'\right>$} (F);
    \path (C) to node [transform shape,scale=.75] {$\Downarrow \left<l,h\right>$} (E);
  \end{tikzpicture}
\]
\[
  \begin{tikzpicture}[baseline=-2pt]
    \node (A) at (1.5,0) {\tikz[baseline=-2pt,scale=.4]{\strsquare{1}{2}{$h'$,$h$}}};
    \node (B) at (0,1.75) {\tikz[baseline=-2pt,scale=.4]{\strsquare{2}{3}{$h'$,,$h$,,,$\partial l$}}};
    \node (C) at (-1.5,0) {\tikz[baseline=-2pt,scale=.4]{\strsquare{2}{3}{,$\partial l$,$h'$,,$h$,}}};
    \draw [->] (A) to (B);
    \draw [->] (B) to (C);
    \draw [->] (A) to coordinate (D) (C);
    \path (B) to node [transform shape,scale=.75] {$\Downarrow \left<hh',l\right>$} (D);
  \end{tikzpicture}
  =
  \begin{tikzpicture}[baseline=-2pt]
    \node (A) at (3,0) {\tikz[baseline=-2pt,scale=.4]{\strsquare{1}{2}{$h'$,$h$}}};
    \node (B) at (0,3.5) {\tikz[baseline=-2pt,scale=.4]{\strsquare{2}{3}{$h'$,,$h$,,,$\partial l$}}};
    \node (C) at (-1.5,1.75) {\tikz[baseline=-2pt,scale=.4]{\strsquare{2}{3}{$h'$,,,$\partial l$,$h$,}}};
    \node (D) at (-3,0) {\tikz[baseline=-2pt,scale=.4]{\strsquare{2}{3}{,$\partial l$,$h'$,,$h$,}}};
    \draw [->] (A) to (B);
    \draw [->,bend right=1cm] (B.west) to node [transform shape,scale=.75,sloped,anchor=north] {$\Downarrow \{h,h'|\partial l\}^{-1}$} (D.north);
    \draw [->] (B) to (C);
    \draw [->] (C) to (D);
    \draw [->] (A) to coordinate (E) (D);
    \draw [->] (A) to coordinate (F) (C);
    \path (B) to node [transform shape,scale=.75] {$\Downarrow \left<h,\act{\partial h'}l\right>$} (F);
    \path (C) to node [transform shape,scale=.75] {$\Downarrow \act{h}\left<h',l\right>$} (E);
  \end{tikzpicture}
\]
\[
  \begin{tikzpicture}[baseline=-2pt]
    \node (A) at (1.5,0) {\tikz[baseline=-2pt,scale=.4]{\strrect{2}{$h_2$,,,$h_1$}}};
    \node (B1) at (0,1.5) {\tikz[baseline=-2pt,scale=.4]{\strsquare{2}{3}{$\partial l$,,$h_2$,,,$h_1$}}};
    \node (B2) at (0,-1.5) {\tikz[baseline=-2pt,scale=.4]{\strrect{2}{,$h_1$,$h_2$,}}};
    \node (C) at (-1.5,0) {\tikz[baseline=-2pt,scale=.4]{\strsquare{2}{3}{$\partial l$,,,$h_1$,$h_2$,}}};
    \node (D) at (-4.5,0) {\tikz[baseline=-2pt,scale=.4]{\strsquare{2}{3}{,$h_1$,$\partial l$,,$h_2$,}}};
    \draw [->] (A) to (B1);
    \draw [->] (A) to (B2);
    \draw [->] (B1) to node [transform shape,scale=.75,sloped,anchor=north] {$\Downarrow \{h_2,\partial l|h_1\}^{-1}$} (D);
    \draw [->] (B2) to node [transform shape,scale=.75,sloped,anchor=south] {$\Downarrow \act{h_2}\left<l,h_1\right>$} (D);
    \draw [->] (B1) to (C);
    \draw [->] (C) to (D);
    \draw [->] (B2) to (C);
    \path (B1) to node [transform shape,scale=.6] {$\Downarrow \{\{h_2,h_1\},\act{\act{\partial h_2}h_1h_2}l\}^{-1}$} (B2);
  \end{tikzpicture}
  =
  \begin{tikzpicture}[baseline=-2pt]
    \node (A) at (1.5,0) {\tikz[baseline=-2pt,scale=.4]{\strrect{2}{$h_2$,,,$h_1$}}};
    \node (B1) at (0,1.5) {\tikz[baseline=-2pt,scale=.4]{\strsquare{2}{3}{$\partial l$,,$h_2$,,,$h_1$}}};
    \node (B2) at (0,-1.5) {\tikz[baseline=-2pt,scale=.4]{\strrect{2}{,$h_1$,$h_2$,}}};
    \node (C) at (-1.5,0) {\tikz[baseline=-2pt,scale=.4]{\strsquare{2}{3}{$h_2$,,,$h_1$,$\partial \act{h_2}l$,}}};
    \node (D) at (-4.5,0) {\tikz[baseline=-2pt,scale=.4]{\strsquare{2}{3}{,$h_1$,$\partial l$,,$h_2$,}}};
    \draw [->] (A) to (B1);
    \draw [->] (A) to (B2);
    \draw [->] (B1) to node [transform shape,scale=.75,sloped,anchor=north] {$\Downarrow \{\partial \act{h_2}l,h_2|h_1\}^{-1}$} (D);
    \draw [->] (B2) to (D);
    \draw [->] (A) to coordinate (E) (C);
    \draw [->] (B1) to (C);
    \draw [->] (C) to (D);
    \path (B2) to node [transform shape,scale=.75] {$\Downarrow \{\act{h_2}l,\{h_2,h_1\}\}^{-1}$} (C);
    \path (B1) to node [transform shape,scale=.75] {$\Downarrow \left<\act{h_2}l,\act{\partial h_2}h_1\right>$} (E);
  \end{tikzpicture}
\]
\[
  \begin{tikzpicture}[baseline=-2pt]
    \node (A) at (1.5,0) {\tikz[baseline=-2pt,scale=.4]{\strrect{2}{$h_2$,,,$h_1$}}};
    \node (B1) at (0,1.5) {\tikz[baseline=-2pt,scale=.4]{\strsquare{2}{3}{$\partial l$,,$h_2$,,,$h_1$}}};
    \node (B2) at (0,-1.5) {\tikz[baseline=-2pt,scale=.4]{\strrect{2}{,$h_1$,$h_2$,}}};
    \node (C) at (-1.5,0) {\tikz[baseline=-2pt,scale=.4]{\strsquare{2}{3}{$h_2$,,,$h_1$,$\partial \act{h_2}l$,}}};
    \node (D) at (-4.5,0) {\tikz[baseline=-2pt,scale=.4]{\strsquare{2}{3}{,$h_1$,$\partial l$,,$h_2$,}}};
    \draw [->] (A) to (B1);
    \draw [->] (A) to (B2);
    \draw [->] (B1) to node [transform shape,scale=.75,sloped,anchor=north] {$\Downarrow \{h_2|h_1,\partial l,\}^{-1}$} (D);
    \draw [->] (B2) to (D);
    \draw [->] (A) to coordinate (E) (C);
    \draw [->] (B1) to (C);
    \draw [->] (C) to (D);
    \path (B2) to node [transform shape,scale=.75] {$\Downarrow \{\{h_2,h_1\},\act{\act{\partial h_2}h_1h_2}l\}^{-1}$} (C);
    \path (B1) to node [transform shape,scale=.75] {$\Downarrow \act{\act{\partial h_2}h_1}\left<h_2,l\right>$} (E);
  \end{tikzpicture}
  =
  \begin{tikzpicture}[baseline=-2pt]
    \node (A) at (1.5,0) {\tikz[baseline=-2pt,scale=.4]{\strrect{2}{$h_2$,,,$h_1$}}};
    \node (B1) at (0,1.5) {\tikz[baseline=-2pt,scale=.4]{\strsquare{2}{3}{$\partial l$,,$h_2$,,,$h_1$}}};
    \node (B2) at (0,-1.5) {\tikz[baseline=-2pt,scale=.4]{\strrect{2}{,$h_1$,$h_2$,}}};
    \node (C) at (-1.5,0) {\tikz[baseline=-2pt,scale=.4]{\strsquare{2}{3}{$\partial l$,,,$h_1$,$h_2$,}}};
    \node (D) at (-4.5,0) {\tikz[baseline=-2pt,scale=.4]{\strsquare{2}{3}{,$h_1$,$\partial l$,,$h_2$,}}};
    \draw [->] (A) to (B1);
    \draw [->] (A) to (B2);
    \draw [->] (B1) to node [transform shape,scale=.75,sloped,anchor=north] {$\Downarrow \{h_2|\partial \act{h_1}l,h_1\}^{-1}$} (D);
    \draw [->] (B2) to node [transform shape,scale=.75,sloped,anchor=south] {$\Downarrow \left<h_2,\act{h_1}l\right>$} (D);
    \draw [->] (B1) to (C);
    \draw [->] (C) to (D);
    \draw [->] (B2) to (C);
    \path (B1) to node [transform shape,scale=.6] {$\Downarrow \{\act{\partial h_2}(\act{h_1}l),\{h_2,h_1\}\}$} (B2);
  \end{tikzpicture}
\]

\subsection{pastings}
Rewriting Equations \eqref{eq:3CM-Pasting} and \eqref{eq:3CM-Pasting2} in terms of the auxiliary brackets $\left<,\right>$, we obtain
\begin{align}
  \label{eq:3CM-aux-Pasting1}
  \left<l_1l_2,h\right> &= \act{l_1}\left<l_2,h\right>\{l_1,\{\partial l_2,h\}\}\act{\act{\partial l_1}\{\partial l_2,h\}}\left<l_1,h\right>\{\partial l_1,\partial l_2|h\}^{-1} \\
  \label{eq:3CM-aux-Pasting2}
  \left<h,l'l\right> &= \act{\act{h}l'}\left<h,l\right>\left<h,l'\right>\act{\{h,\partial l'\}}\{\act{\partial h}l',\{h,\partial l\}\}\{h|\partial l',\partial l\}^{-1}
\end{align}
These relations can be depicted as follows:
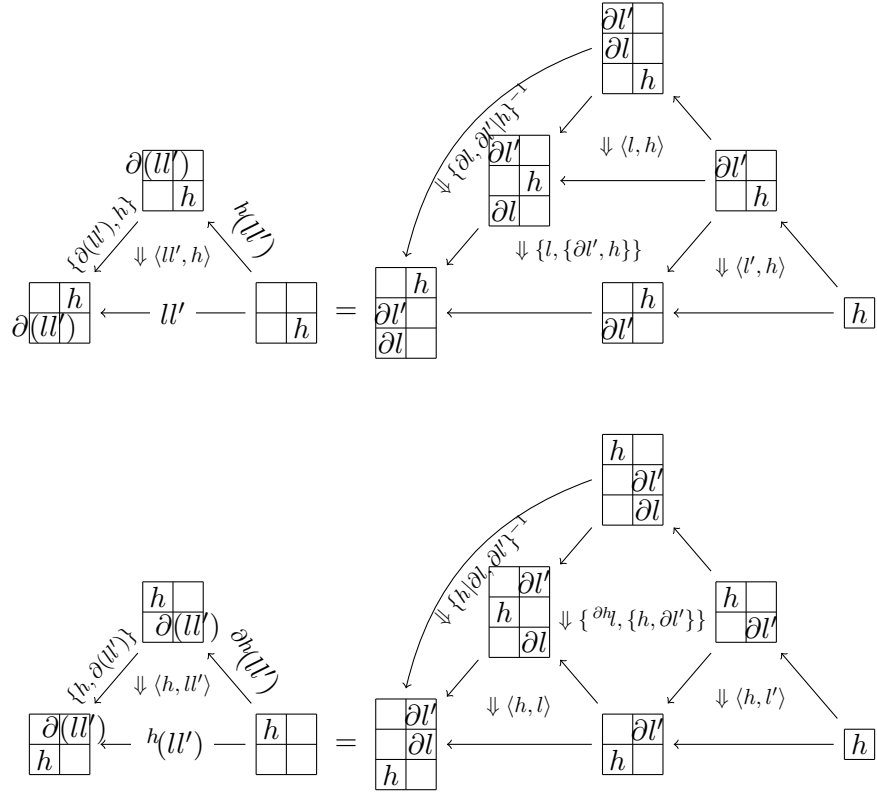
\begin{align*}
  \begin{tikzpicture}[baseline=-2pt]
    \node (A) at (1.5,0) {\tikz[baseline=-2pt,scale=.4]{\strrect{2}{,,,$h$}}};
    \node (B) at (0,1.75) {\tikz[baseline=-2pt,scale=.4]{\strrect{2}{$\partial (ll')$,,,$h$}}};
    \node (C) at (-1.5,0) {\tikz[baseline=-2pt,scale=.4]{\strrect{2}{,$h$,$\partial (ll')$,}}};
    \draw [->] (A) to node [anchor=south,sloped] {$\act{h}(ll')$} (B);
    \draw [->] (B) to node [anchor=south,sloped,transform shape,scale=.75] {$\{\partial (ll'),h\}$} (C);
    \draw [->] (A) to node [fill=white] (D) {$ll'$} (C);
    \path (B) to node [transform shape,scale=.75] {$\Downarrow \left<ll',h\right>$} (D);
  \end{tikzpicture}
  =
  \begin{tikzpicture}[baseline=-2pt]
    \node (A) at (3,0) {\tikz[baseline=-2pt,scale=.4]{\strrect{1}{$h$}}};
    \node (B) at (1.5,1.75) {\tikz[baseline=-2pt,scale=.4]{\strrect{2}{$\partial l'$,,,$h$}}};
    \node (C1) at (0,3.5) {\tikz[baseline=-2pt,scale=.4]{\strsquare{2}{3}{$\partial l'$,,$\partial l$,,,$h$}}};
    \node (C2) at (0,0) {\tikz[baseline=-2pt,scale=.4]{\strrect{2}{,$h$,$\partial l'$,}}};
    \node (D) at (-1.5,1.75) {\tikz[baseline=-2pt,scale=.4]{\strsquare{2}{3}{$\partial l'$,,,$h$,$\partial l$,}}};
    \node (E) at (-3,0) {\tikz[baseline=-2pt,scale=.4]{\strsquare{2}{3}{,$h$,$\partial l'$,,$\partial l$,}}};
    \draw [->] (A) to (B);
    \draw [->] (B) to (C1);
    \draw [->] (C1) to (D);
    \draw [->] (D) to (E);
    \draw [->] (A) to coordinate (F) (C2);
    \draw [->] (C2) to (E);
    \draw [->] (B) to (D);
    \draw [->] (B) to (C2);
    \draw [->,bend right=1cm] (C1.west) to node [transform shape,scale=.75,sloped,anchor=north] {$\Downarrow \{\partial l,\partial l'|h\}^{-1}$} (E.north);
    \path (D) to node [transform shape,scale=.75] {$\Downarrow \{l,\{\partial l',h\}\}$} (C2);
    \path (C1) to node [pos=1/4,transform shape,scale=.75] {$\Downarrow \left<l,h\right>$} (C2);
    \path (B) to node [transform shape,scale=.75] {$\Downarrow \left<l',h\right>$} (F);
  \end{tikzpicture}
\end{align*}
\begin{align*}
  \begin{tikzpicture}[baseline=-2pt]
    \node (A) at (1.5,0) {\tikz[baseline=-2pt,scale=.4]{\strrect{2}{$h$,,,}}};
    \node (B) at (0,1.75) {\tikz[baseline=-2pt,scale=.4]{\strrect{2}{$h$,,,$\partial (ll')$}}};
    \node (C) at (-1.5,0) {\tikz[baseline=-2pt,scale=.4]{\strrect{2}{,$\partial (ll')$,$h$,}}};
    \draw [->] (A) to node [anchor=south,sloped] {$\act{\partial h}(ll')$} (B);
    \draw [->] (B) to node [anchor=south,sloped,transform shape,scale=.75] {$\{h,\partial (ll')\}$} (C);
    \draw [->] (A) to node [fill=white] (D) {$\act{h}(ll')$} (C);
    \path (B) to node [transform shape,scale=.75] {$\Downarrow \left<h,ll'\right>$} (D);
  \end{tikzpicture}
  =
  \begin{tikzpicture}[baseline=-2pt]
    \node (A) at (3,0) {\tikz[baseline=-2pt,scale=.4]{\strrect{1}{$h$}}};
    \node (B) at (1.5,1.75) {\tikz[baseline=-2pt,scale=.4]{\strrect{2}{$h$,,,$\partial l'$}}};
    \node (C1) at (0,3.5) {\tikz[baseline=-2pt,scale=.4]{\strsquare{2}{3}{$h$,,,$\partial l'$,,$\partial l$}}};
    \node (C2) at (0,0) {\tikz[baseline=-2pt,scale=.4]{\strrect{2}{,$\partial l'$,$h$,}}};
    \node (D) at (-1.5,1.75) {\tikz[baseline=-2pt,scale=.4]{\strsquare{2}{3}{,$\partial l'$,$h$,,,$\partial l$}}};
    \node (E) at (-3,0) {\tikz[baseline=-2pt,scale=.4]{\strsquare{2}{3}{,$\partial l'$,,$\partial l$,$h$,}}};
    \draw [->] (A) to (B);
    \draw [->] (B) to (C1);
    \draw [->] (C1) to (D);
    \draw [->] (D) to (E);
    \draw [->] (A) to coordinate (F1) (C2);
    \draw [->] (C2) to coordinate (F2) (E);
    \draw [->] (B) to (C2);
    \draw [->] (C2) to (D);
    \draw [->,bend right=1cm] (C1.west) to node [transform shape,scale=.75,sloped,anchor=north] {$\Downarrow \{h|\partial l,\partial l'\}^{-1}$} (E.north);
    \path (C1) to node [transform shape,scale=.75] {$\Downarrow \{\act{\partial h}l,\{h,\partial l'\}\}$} (C2);
    \path (B) to node [transform shape,scale=.75] {$\Downarrow \left<h,l'\right>$} (F1);
    \path (D) to node [transform shape,scale=.75] {$\Downarrow \left<h,l\right>$} (F2);
  \end{tikzpicture}
\end{align*}

\subsection{Cube}
Rewriting Equation \eqref{eq:3CM-Cube} in terms of the auxiliary brackets $\left<,\right>$, we obtain
\begin{equation}
  \label{eq:3-CM-aux-Cube}
  \left<l_2,\partial l_1\right>
  = 
  \{l_2,l_1\}
  \left<\partial l_2,l_1\right>
  \act{\{\partial l_2,\partial l_1\}}\{l_1,l_2\}.
\end{equation}
This can be depicted as follows:
\[
  \begin{tikzpicture}[baseline=-2pt]
    \node (A) at (3,0) {\tikz[baseline=-2pt,scale=.4]{\strrect{1}{}}};
    \node (B) at (1.5,0) {\tikz[baseline=-2pt,scale=.4]{\strrect{1}{$\partial l_1$}}};
    \node (C) at (0,1.75) {\tikz[baseline=-2pt,scale=.4]{\strrect{2}{$\partial l_2$,,,$\partial l_1$}}};
    \node (D) at (-1.5,0) {\tikz[baseline=-2pt,scale=.4]{\strrect{2}{,$\partial l_1$,$\partial l_2$,}}};
    \draw [->] (A) to node [transform shape,scale=.75,anchor=south] {$l_1$} (B);
    \draw [->] (B) to coordinate (E) (D);
    \draw [->] (B) to (C);
    \draw [->] (C) to (D);
    \path (C) to node [transform shape,scale=.75] {$\Downarrow \left<l_2,\partial l_1\right>$} (E); 
  \end{tikzpicture}
  =
  \begin{tikzpicture}[baseline=-2pt]
    \node (A) at (3,0) {\tikz[baseline=-2pt,scale=.4]{\strrect{1}{}}};
    \node (B1) at (0,2.5) {\tikz[baseline=-2pt,scale=.4]{\strrect{1}{$\partial l_1$}}};
    \node (B2) at (0,0) {\tikz[baseline=-2pt,scale=.4]{\strrect{1}{$\partial l_1$}}};
    \node (B3) at (0,-2.5) {\tikz[baseline=-2pt,scale=.4]{\strrect{1}{$\partial l_1$}}};
    \node (C) at (-1.5,1.25) {\tikz[baseline=-2pt,scale=.4]{\strrect{2}{$\partial l_2$,,,$\partial l_1$}}};
    \node (D) at (-3,0) {\tikz[baseline=-2pt,scale=.4]{\strrect{2}{,$\partial l_1$,$\partial l_2$,}}};
    \draw [->] (A) to (B1);
    \draw [->] (A) to (B2);
    \draw [->] (A) to (B3);
    \draw [->] (B1) to (C);
    \draw [->] (B2) to (C);
    \draw [->] (B2) to coordinate (E) (D);
    \draw [->] (B3) to (D);
    \draw [->] (C) to (D);
    \path (C) to node [transform shape,scale=.75] {$\Downarrow \left<\partial l_2,l_1\right>$} (E);
    \path (B1.east) to node [transform shape,scale=.75] {$\Downarrow \{l_1,l_2\}$} (B2.east);
    \path (B2) to node [transform shape,scale=.75] {$\Downarrow \{l_2,l_1\}$} (B3);
  \end{tikzpicture}
\]

\subsection{Tube}
Rewriting Equations \eqref{eq:3CM-HonM} and \eqref{eq:3CM-pHonM} in terms of the auxiliary brackets $\left<,\right>$, we obtain
\begin{align}
  \label{eq:3CM-aux-HonM}
  \left<\partial m,h\right>\act{h}m &= m, \\
  \label{eq:3CM-aux-pHonM}
  \left<h,\partial m\right>\act{\partial h}m &= \act{h}m.
\end{align}
These relations can be depicted as follows:
\[
  \begin{tikzpicture}[baseline=-2pt]
    \node (A) at (1.5,0) {\tikz[baseline=-2pt,scale=.4]{\strrect{1}{$h$}}};
    \node (B) at (-1.5,0) {\tikz[baseline=-2pt,scale=.4]{\strrect{1}{$h$}}};
    \draw[->,bend right=1.5cm] (A.north west) to node [anchor=south] (C) {$e_L$} (B.north east);
    \draw[->,bend left=1.5cm] (A.south west) to node [anchor=north] (D) {$\partial m$} (B.south east);
    \path (C) to node {$\Downarrow m$} (D);
  \end{tikzpicture}
  =
  \begin{tikzpicture}[baseline=-2pt]
    \node (A) at (1.5,0) {\tikz[baseline=-2pt,scale=.4]{\strrect{1}{$h$}}};
    \node (B) at (0,1) {\tikz[baseline=-2pt,scale=.4]{\strrect{1}{$h$}}};
    \node (C) at (-1.5,0) {\tikz[baseline=-2pt,scale=.4]{\strrect{1}{$h$}}};
    \draw[->,bend right=1.5cm] (A) to coordinate (D) (B);
    \draw[->,bend left=1.5cm] (A) to coordinate (E) (B);
    \node [anchor=south west] at (D) {$e_L$};
    \path (D) to node {$\Downarrow \act{h}m$} (E);
    \draw [->] (B) to node [anchor=south east] {$e_L$} (C);
    \draw[->,bend left=1.5cm] (A.south west) to node [anchor=north] (D) {$\partial m$} (C.south east);
    \path (B.west) to node {$\Downarrow \left<\partial m,h\right>$} (D.west);
  \end{tikzpicture}
\]
\[
  \begin{tikzpicture}[baseline=-2pt]
    \node (A) at (1.5,0) {\tikz[baseline=-2pt,scale=.4]{\strrect{1}{$h$}}};
    \node (B) at (-1.5,0) {\tikz[baseline=-2pt,scale=.4]{\strrect{1}{$h$}}};
    \draw[->,bend right=1.5cm] (A.north west) to node [anchor=south] (C) {$e_L$} (B.north east);
    \draw[->,bend left=1.5cm] (A.south west) to node [anchor=north] (D) {$\partial \act{h}m$} (B.south east);
    \path (C) to node {$\Downarrow \act{h}m$} (D);
  \end{tikzpicture}
  =
  \begin{tikzpicture}[baseline=-2pt]
    \node (A) at (1.5,0) {\tikz[baseline=-2pt,scale=.4]{\strrect{1}{$h$}}};
    \node (B) at (0,1) {\tikz[baseline=-2pt,scale=.4]{\strrect{1}{$h$}}};
    \node (C) at (-1.5,0) {\tikz[baseline=-2pt,scale=.4]{\strrect{1}{$h$}}};
    \draw[->,bend right=1.5cm] (A) to coordinate (D) (B);
    \draw[->,bend left=1.5cm] (A) to coordinate (E) (B);
    \node [anchor=south west] at (D) {$e_L$}; 
    \path (D) to node {$\Downarrow \act{\partial h}m$} (E);
    \draw [->] (B) to node [anchor=south east] {$e_L$} (C);
    \draw[->,bend left=1.5cm] (A.south west) to node [anchor=north] (D) {$\partial \act{h}m$} (C.south east);
    \path (B.west) to node {$\Downarrow \left<h,\partial m\right>$} (D.west);
  \end{tikzpicture}
\]

\subsection{S condition}
\[
  \begin{tikzpicture}[baseline=-2pt]
    \node (P) at (3,0) {\tikz[scale=0.45]{\strrect{3}{$h_z$,,,,$h_y$,,,,$h_x$}}};
    \node (Q1) at (1,1.5) {\tikz[scale=0.45]{\strrect{3}{$h_z$,,,,,$h_x$,,$h_y$,}}};
    \node (Q2) at (1,-1.5) {\tikz[scale=0.45]{\strrect{3}{,$h_y$,,$h_z$,,,,,$h_x$}}};
    \node (R1) at (-1,1.5) {\tikz[scale=0.45]{\strrect{3}{,,$h_x$,$h_z$,,,,$h_y$,}}};
    \node (R2) at (-1,-1.5) {\tikz[scale=0.45]{\strrect{3}{,$h_y$,,,,$h_x$,$h_z$,,}}};
    \node (S) at (-3,0) {\tikz[scale=0.45]{\strrect{3}{,,$h_x$,,$h_y$,,$h_z$,,}}};
    \draw [->] (P) to (Q1);
    \draw [->] (P) to (Q2);
    \draw [->] (Q1) to coordinate (S1) (R1);
    \draw [->] (Q2) to coordinate (S2) (R2);
    \draw [->] (R1) to (S);
    \draw [->] (R2) to (S);
    \draw [->,bend left=1cm] (P) to (R1);
    \draw [->,bend right=1cm] (Q2) to (S);
    \path (Q1.north) to [pos=1/2] coordinate (A) (Q2.south);
    \path (Q1) to node {$\Downarrow m_1$} (A);
    \path (R1.north) to [pos=1/2] coordinate (B) (R2.south);
    \path (B) to node {$\Downarrow m_3$} (R2);
    \path (S1) to [pos=2/5] coordinate (Sn) (S2);
    \path (S1) to [pos=3/5] coordinate (Ss) (S2);
    \path (Sn) to node {$\Downarrow m_2$} (Ss);
  \end{tikzpicture}
  =
  \begin{tikzpicture}[baseline=-2pt]
    \node (P) at (3,0) {\tikz[scale=0.45]{\strrect{3}{$h_z$,,,,$h_y$,,,,$h_x$}}};
    \node (Q1) at (1,1.5) {\tikz[scale=0.45]{\strrect{3}{$h_z$,,,,,$h_x$,,$h_y$,}}};
    \node (Q2) at (1,-1.5) {\tikz[scale=0.45]{\strrect{3}{,$h_y$,,$h_z$,,,,,$h_x$}}};
    \node (R1) at (-1,1.5) {\tikz[scale=0.45]{\strrect{3}{,,$h_x$,$h_z$,,,,$h_y$,}}};
    \node (R2) at (-1,-1.5) {\tikz[scale=0.45]{\strrect{3}{,$h_y$,,,,$h_x$,$h_z$,,}}};
    \node (S) at (-3,0) {\tikz[scale=0.45]{\strrect{3}{,,$h_x$,,$h_y$,,$h_z$,,}}};
    \draw [->] (P) to (Q1);
    \draw [->] (P) to (Q2);
    \draw [->] (Q1) to coordinate (S1) (R1);
    \draw [->] (Q2) to coordinate (S2) (R2);
    \draw [->] (R1) to (S);
    \draw [->] (R2) to (S);
    \draw [->,bend left=1cm] (Q1) to (S);
    \draw [->,bend right=1cm] (P) to (R2);
    \path (Q1.north) to [pos=1/2] coordinate (A) (Q2.south);
    \path (A) to node {$\Downarrow m_6$} (Q2);
    \path (R1.north) to [pos=1/2] coordinate (B) (R2.south);
    \path (R1) to node {$\Downarrow m_4$} (B);
    \path (S1) to [pos=2/5] coordinate (Sn) (S2);
    \path (S1) to [pos=3/5] coordinate (Ss) (S2);
    \path (Sn) to node {$\Downarrow m_5$} (Ss);
  \end{tikzpicture}
\]
where
\begin{align*}
  m_1 &= \{\act{\partial h_z}h_y,h_z|h_x\}, &
  m_4 &= \{h_z|h_y,h_x\},\\
  m_2 &= L(\{h_z,h_y\},\act{\partial h_z}h_yh_z;h_x)^{-1}, &
  m_5 &= R(h_z;\{h_y,h_x\},\act{\partial h_y}h_xh_y), \\
  m_3 &= \{h_z,h_y|h_x\}^{-1}, &
  m_6 &= \{h_z|\act{\partial h_y}h_x,h_y\}^{-1}.
\end{align*}
$L$ and $R$ are defined in \eqref{eq:defL} and \eqref{eq:defR}, respectively. They can be depicted as:
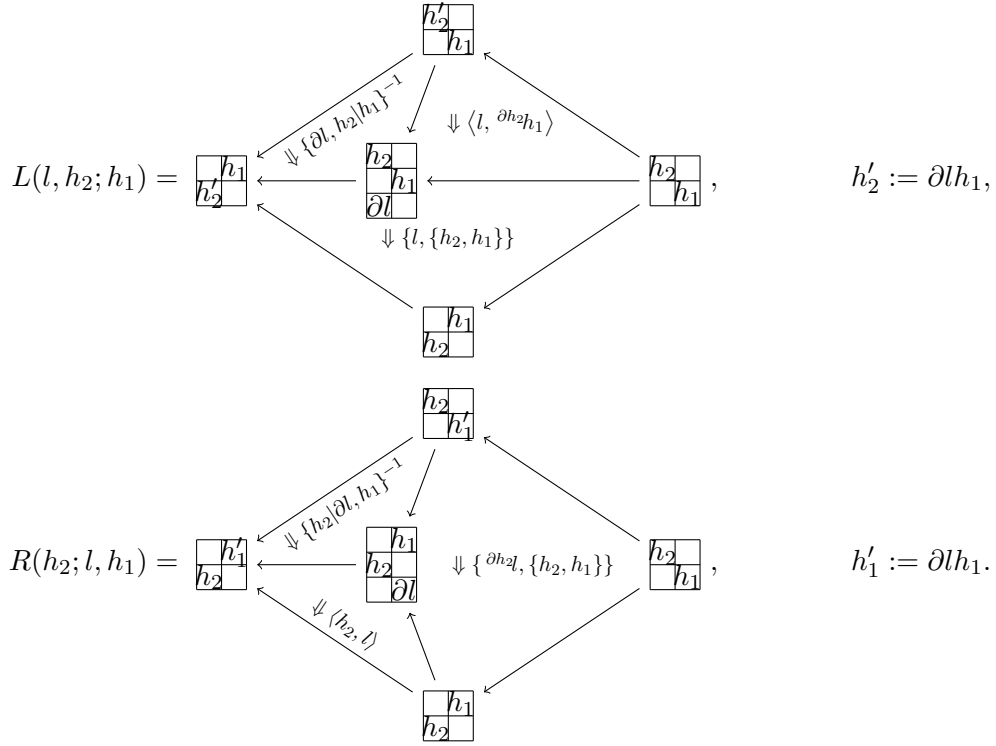
\begin{align*}
  L(l,h_2;h_1) &= 
  \begin{tikzpicture}[baseline=-2pt]
    \node (A) at (3,0) {\tikz[baseline=-2pt,scale=.33]{\strrect{2}{$h_2$,,,$h_1$}}};
    \node (B1) at (0,2) {\tikz[baseline=-2pt,scale=.33]{\strrect{2}{$h_2'$,,,$h_1$}}};
    \node (B2) at (0,-2) {\tikz[baseline=-2pt,scale=.33]{\strrect{2}{,$h_1$,$h_2$,}}};
    \node (C) at (-3,0) {\tikz[baseline=-2pt,scale=.33]{\strrect{2}{,$h_1$,$h_2'$,}}};
    \node (D) at (-.75,0) {\tikz[baseline=-2pt,scale=.33]{\strsquare{2}{3}{$h_2$,,,$h_1$,$\partial l$,}}};
    \draw [->] (A) to (B1);
    \draw [->] (A) to (B2);
    \draw [->] (B1) to node [sloped,transform shape,scale=.75,anchor=north] {$\Downarrow \{\partial l,h_2|h_1\}^{-1}$} (C);
    \draw [->] (B2) to (C);
    \draw [->] (A) to coordinate (E) (D);
    \draw [->] (B1) to (D);
    \draw [->] (D) to (C);
    \path (B1) to node [pos=3/4,transform shape,scale=.75] {$\Downarrow \{l,\{h_2,h_1\}\}$} (B2);
    \path (B1) to node [transform shape,scale=.75] {$\Downarrow \left<l,\act{\partial h_2}h_1\right>$} (E);
  \end{tikzpicture}, &
  h_2' &:= \partial lh_1, \\
  R(h_2;l,h_1) &= 
  \begin{tikzpicture}[baseline=-2pt]
    \node (A) at (3,0) {\tikz[baseline=-2pt,scale=.33]{\strrect{2}{$h_2$,,,$h_1$}}};
    \node (B1) at (0,2) {\tikz[baseline=-2pt,scale=.33]{\strrect{2}{$h_2$,,,$h_1'$}}};
    \node (B2) at (0,-2) {\tikz[baseline=-2pt,scale=.33]{\strrect{2}{,$h_1$,$h_2$,}}};
    \node (C) at (-3,0) {\tikz[baseline=-2pt,scale=.33]{\strrect{2}{,$h_1'$,$h_2$,}}};
    \node (D) at (-.75,0) {\tikz[baseline=-2pt,scale=.33]{\strsquare{2}{3}{,$h_1$,$h_2$,,,$\partial l$}}};
    \draw [->] (A) to (B1);
    \draw [->] (A) to (B2);
    \draw [->] (B1) to node [sloped,transform shape,scale=.75,anchor=north] {$\Downarrow \{h_2|\partial l,h_1\}^{-1}$} (C);
    \draw [->] (B2) to node [sloped,transform shape,scale=.75,anchor=south] {$\Downarrow \left<h_2,l\right>$} (C);
    \draw [->] (B1) to (D);
    \draw [->] (B2) to (D);
    \draw [->] (D) to (C);
    \path [->] (A) to node [transform shape,scale=.75] {$\Downarrow \{\act{\partial h_2}l,\{h_2,h_1\}\}$} (D);
  \end{tikzpicture}, &
  h_1' &:= \partial lh_1.
\end{align*}

\section{Fundamental relations of Gray 4-categories}
In this appendix, we provide the fundamental relations of Gray 4-categories in explicit algebraic form.

\subsection{Left-Homanian relation}
\begin{align}
  \nonumber
  &\quad 
  [
    (((\alpha_w\#_2\alpha_z)\#_1t_1(\alpha_x))\#_2\{\alpha_y,\alpha_x\})
    \#_3
    (\{\alpha_w,\alpha_z|\alpha_x\}\#_2(\alpha_y\#_1s_1(\alpha_x)))
  ] \\
  \nonumber
  &\quad \#_4
  \{\alpha_w\#_2\alpha_z,\alpha_y|\alpha_x\} \\
  \nonumber
  &= 
  [
    ((\alpha_w\#_1t_1(\alpha_x))\#_2\{\alpha_z,\alpha_y|\alpha_x\})
    \#_3
    (\{\alpha_w,\alpha_x\}\#_2((\alpha_z\#_2\alpha_y)\#_1s_1(\alpha_x)))
  ] \\
  \label{eq:left-Hom}
  &\quad \#_4
  \{\alpha_w,\alpha_z\#_2\alpha_y|\alpha_x\}
\end{align}

\subsection{Right-Homanian relation}
\begin{align}
  \nonumber
  &\Bigl(
    [\{\alpha_w,\alpha_x\}\#_2(s_1(\alpha_w)\#_1(\alpha_y\#_2\alpha_z))]
    \#_3
    [(t_1(\alpha_w)\#_1\alpha_x)\#_2\{\alpha_w|\alpha_y,\alpha_z\}]
  \Bigr) 
  \#_4
  \{\alpha_w|\alpha_x,\alpha_y\#_2\alpha_z\} \\
  \nonumber
  &= 
  \Bigl(
    [\{\alpha_w|\alpha_x,\alpha_y\}\#_2(s_1(\alpha_w)\#_1\alpha_z)]
    \#_3
    [(t_1(\alpha_w)\#_1(\alpha_x\#_2\alpha_y))\#_2\{\alpha_w,\alpha_z\}]
  \Bigr) 
  \\
  \label{eq:right-Hom}
  &\quad 
  \#_4
  \{\alpha_w|\alpha_x\#_2\alpha_y,\alpha_z\}
\end{align}

\subsection{KV-polytope}
\begin{align}
  \nonumber
  &
  \Bigl(
    [(\alpha_w\#_1t_1(\alpha_y))\#_2\{\alpha_z,\alpha_y\}\#_2(s_1(\alpha_z)\#_1\alpha_x)] \\
    \nonumber
    &\qquad \#_3
    \{\{\alpha_w,\alpha_y\},\{\alpha_z,\alpha_x\}\} \\
    \nonumber
    &\qquad \#_3
    [(t_1(\alpha_w)\#_1\alpha_y)\#_2\{\alpha_w,\alpha_x\}\#_2(\alpha_z\#_1s_1(\alpha_x))]
  \Bigr) \\
  \nonumber
  & \quad \#_4
  \Bigl(
    [(\alpha_w\#_1t_1(\alpha_y))\#_2\{\alpha_z|\alpha_y,\alpha_x\}]
    \#^4_3
    [\{\alpha_w|\alpha_y,\alpha_x\}\#_2(\alpha_z\#_1s_1(\alpha_x))]
  \Bigr) \\
  \nonumber
  & \quad \#_4 
  \{\alpha_w,\alpha_z|\alpha_y\#_2\alpha_x\} \\
  \nonumber
  &= 
  \Bigl(
    [\{\alpha_w,\alpha_z|\alpha_y\}\#_2(s_1(\alpha_z)\#_1\alpha_x)]
    \#^4_3
    [(t_1(\alpha_w)\#_1\alpha_y)\#_2\{\alpha_w,\alpha_z|\alpha_x\}]
  \Bigr) \\
  \label{eq:KV-poly}
  &\quad \#_4 
  \{\alpha_w\#_2\alpha_z|\alpha_y,\alpha_x\}
\end{align}

\subsection{Prism 3-22}
\begin{align}
  \nonumber
  &
  \Bigl(
    \begin{bmatrix} & \alpha_2\#_2\alpha_1 \\ \Psi &\end{bmatrix}
    \#_3
    \{s_2(\Psi)|\alpha_2,\alpha_1\}
  \Bigr)
  \#_4
  \{\Psi,\alpha_2\#_2\alpha_1\} \\
  \nonumber
  &= 
  \Bigl(
    [\{\Psi,\alpha_2\}\#_2(s_1(\Psi)\#_1\alpha_1)]
    \#_3
    [(t_1(\Psi)\#_1\alpha_2)\#_2\{s_2(\Psi),\alpha_1\}]
  \Bigr) \\
  \nonumber
  &\qquad \#_4
  \Bigl(
  [\{t_2(\Psi),\alpha_2\}\#_2(s_1(\Psi)\#_1\alpha_1)]
  \#_3
  [(t_1(\Psi)\#_1\alpha_2)\#_2\{\Psi,\alpha_1\}]
  \Bigr) \\
  \label{eq:Prism3|22}
  &\qquad \#_4
  \Bigl(
    \{t_2(\Psi)|\alpha_2,\alpha_1\}
    \#_3
    \begin{bmatrix}\Psi & \\ & \alpha_2\#_2\alpha_1\end{bmatrix}
  \Bigr)
\end{align}

\subsection{Prism 32-2} 
\begin{align}
  \nonumber
  &
  \Bigl(
    \begin{bmatrix}& \alpha \\ \Psi\#_2\gamma &\end{bmatrix}
    \#_3
    \{s_2(\Psi),\gamma|\alpha\}
  \Bigr)
  \#_4
  \{\Psi\#_2\gamma,\alpha\} \\
  \nonumber
  &= 
  \Bigl(
    \{\Psi\#_1t_1(\alpha),\{\gamma,\alpha\}\}
    \#_3
    [\{s_2(\Psi),\alpha\}\#_2(\gamma\#_1s_1(\alpha))]
  \Bigr) \\
  \nonumber
  &\quad \#_4
  \Bigl(
    [(t_2(\Psi)\#_1t_1(\alpha))\#_2\{\gamma,\alpha\}]\#_3[\{\Psi,\alpha\}\#_2(\gamma\#_1s_1(\alpha))]
  \Bigr) \\
  \label{eq:Prism32|2}
  &\quad \#_4
  \Bigl(
    \{t_2(\Psi),\gamma|\alpha\}
    \#_3
    \begin{bmatrix}\Psi\#_2\gamma & \\ & \alpha\end{bmatrix}
  \Bigr)
\end{align}

\subsection{Prism 23-2}
\begin{align}
  \nonumber
  &
  \bigl(
    \begin{bmatrix} & \alpha \\ \gamma\#_2\Psi &\end{bmatrix}
    \#_3
    \{\gamma,s_2(\Psi)|\alpha\}
  \bigr)
  \#_4
  \{\gamma\#_2\Psi,\alpha\} \\
  \nonumber
  &= 
  \bigl(
    ((\gamma\#_1t_1(\alpha))\#_2\{\Psi,\alpha\})
    \#_3
    [\{\gamma,\alpha\}\#_2s_2(\Psi\#_1s_1(\alpha))]
  \bigr) \\
  \nonumber
  &\quad \#_4
  \bigl(
    [(\gamma\#_1t_1(\alpha))\#_2\{t_2(\Psi),\alpha\}]
    \#_3
    \{\{\gamma,\alpha\},\Psi\#_1s_1(\alpha)\}^{-1_4}
  \bigr) \\
  \label{eq:Prism23|2}
  &\quad \#_4
  \bigl(
    \{\gamma,t_2(\Psi)|\alpha\}
    \#_3
    \begin{bmatrix}\gamma\#_2\Psi & \\ & \alpha\end{bmatrix}
  \bigr)
\end{align}

\subsection{Prism 22-3}
\begin{align}
  \nonumber
  &\Bigl(
    \begin{bmatrix} & \Phi \\ \beta_2\#_2\beta_1 &\end{bmatrix}
    \#_3
    \{\beta_2,\beta_1|s_2(\Phi)\}
  \Bigr)
  \#_4
  \{\beta_2\#_2\beta_1,\Phi\} \\
  \nonumber
  &= 
  \bigl(
    [(\beta_2\#_1t_1(\Phi))\#_2\{\beta_1,\Phi\}]
    \#_3
    [\{\beta_2,s_2(\Phi)\}\#_2(\beta_1\#_1s_1(\Phi))]
  \bigr) \\
  \nonumber
  &\quad \#_4
  \bigl(
    [(\beta_2\#_1t_1(\Phi))\#_2\{\beta_1,t_2(\Phi)\}]
    \#_3
    [\{\beta_2,\Phi\}\#_2(\beta_1\#_1s_1(\Phi))]
  \bigr) \\
  \label{eq:Prism22|3}
  &\qquad \#_4
  \bigl(
    \{\beta_2,\beta_1|t_2(\Phi)\}
    \#_3
    \begin{bmatrix} \beta_2\#_2\beta_1 & \\ & \Phi \end{bmatrix}
  \bigr)
\end{align}

\subsection{Prism 2-32}
\begin{align}
  \nonumber
  &
  \bigl(
    \begin{bmatrix}& \Phi\#_2\gamma \\ \beta &\end{bmatrix}
    \#_3
    \{\beta|s_2(\Phi),\gamma\}
  \bigr)
  \#_4
  \{\beta,\Phi\#_2\gamma\} \\
  \nonumber
  &= 
  \bigl(
    [\{\beta,\Phi\}\#_2(s_1(\beta)\#_1\gamma)]
    \#_3
    [s_2(t_1(\beta)\#_1\Phi)\#_2\{\beta,\gamma\}]
  \bigr) \\
  \nonumber
  &\quad 
  \#_4\bigl(
    [\{\beta,t_2(\Phi)\}\#_2(s_1(\beta)\#_1\gamma)]
    \#_3
    \{t_1(\beta)\#_1\Phi,\{\beta,\gamma\}\}
  \bigr) \\
  \label{eq:Prism2|3|2}
  &\quad \#_4
  \bigl(
    \{\beta|t_2(\Phi),\gamma\}
    \#_3
    \begin{bmatrix} \beta & \\ & \Phi\#_2\gamma\end{bmatrix}
  \bigr)
\end{align}

\subsection{Prism 2-23}
\begin{align}
  \nonumber
  &
  \bigl(
    \begin{bmatrix}& \gamma\#_2\Phi \\ \beta &\end{bmatrix}
    \#_3
    \{\beta|\gamma,s_2(\Phi)\}
  \bigr) 
  \#_4
  \{\beta,\gamma\#_2\Phi\} \\
  \nonumber
  &=
  \bigl(
    \{\{\beta,\gamma\},s_1(\beta)\#_1\Phi\}^{-1_4}
    \#_3
    [(t_1(\beta)\#_1\gamma)\#_2\{\beta,s_2(\Phi)\}]
  \bigr) \\
  \nonumber
  &\quad \#_4
  \bigl(
    [\{\beta,\gamma\}\#_2(s_1(\beta)\#_1t_2(\Phi))]
    \#_3
    [(t_1(\beta)\#_1\gamma)\#_2\{\beta,\Phi\}]
  \bigr) \\
  \label{eq:Prism2|23}
  &\quad \#_4
  \bigl(
    \{\beta|\gamma,t_2(\Phi)\}
    \#_3
    \begin{bmatrix}\beta & \\ & \gamma\#_2\Phi\end{bmatrix}
  \bigr)
\end{align}

\subsection{Cube} 
\begin{align}
  \nonumber
  &
  \bigl(
    \begin{bmatrix}& t_2(\Phi) \\ \Psi &\end{bmatrix}
    \#_3
    \{s_2(\Psi),\Phi\}
  \bigr)
  \#_4
  \bigl(
    \{\Psi,t_2(\Phi)\}
    \#_3
    \begin{bmatrix}s_2(\Psi) & \\ & \Phi\end{bmatrix}
  \bigr) \\
  \nonumber
  &= 
  [
    \{\Psi\#_1t_1(\Phi),s_1(\Psi)\#_1\Phi\}
    \#_3
    \{s_2(\Psi),s_2(\Phi)\}
  ] \\
  \nonumber
  &\qquad \#_4
  [
    \begin{bmatrix}& \Phi \\ t_2(\Psi) &\end{bmatrix}
    \#_3
    \{\Psi,s_2(\Phi)\}
  ] 
  \#_4
  [
    \{t_2(\Psi),\Phi\}
    \#_3
    \begin{bmatrix}\Psi & \\ & s_2(\Phi)\end{bmatrix}
  ] \\
  \label{eq:Cube}
  &\qquad \#_4
  [
    \{t_2(\Psi),t_2(\Phi)\}
    \#_3
    \{t_1(\Psi)\#_1\Phi,\Psi\#_1s_1(\Phi)\}
  ]
\end{align}

\subsection{Pasting 33-2}
\begin{align}
  \label{eq:Pasting332}
  \{\Psi'\#_3\Psi,\alpha\}
  &= 
  \bigl(
    \begin{bmatrix}& \alpha \\ \Psi' &\end{bmatrix}
    \#_3
    \{\Psi,\alpha\}
  \bigr) 
  \#_4
  \bigl(
    \{\Psi',\alpha\}
    \#_3
    \begin{bmatrix}\Psi & \\ & \alpha\end{bmatrix}
  \bigr)
\end{align}

\subsection{Pasting 2-33}
\begin{align}
  \label{eq:Pasting233}
  \{\beta,\Phi'\#_3\Phi\} &= 
  \bigl(
    \begin{bmatrix}& \Phi' \\ \beta &\end{bmatrix}
    \#_3
    \{\beta,\Phi\}
  \bigr)
  \#_4
  \bigl(
    \{\beta,\Phi'\}
    \#_3
    \begin{bmatrix}\beta & \\ & \Phi\end{bmatrix}
  \bigr)
\end{align}

\subsection{Tube 1}
\begin{equation}
  \label{eq:Tube1}
  \{t_3(B),\alpha\}\#_4\bigl(\{t_2(B),\alpha\}\#_3\begin{bmatrix}B & \\ &\alpha\end{bmatrix}\bigr) 
  = \bigl(\begin{bmatrix}& \alpha \\ B &\end{bmatrix}\#_3\{s_2(B),\alpha\}\bigr)\#_4\{s_3(B),\alpha\} 
\end{equation}

\subsection{Tube 2}
\begin{equation}
  \label{eq:Tube2}
  \{\beta,t_3(A)\}\#_4\bigl(\{\beta,t_2(A)\}\#_3\begin{bmatrix}\beta & \\ & A\end{bmatrix}\bigr) 
  = \bigl(\begin{bmatrix}& A \\ \beta &\end{bmatrix}\#_3\{\beta,s_2(A)\}\bigr)\#_4\{\beta,s_3(A)\} 
\end{equation}

\noindent
\subsection{S-condition}
\begin{align}
  \nonumber
  &
  \bigl(
    \{\alpha_z\#_1t_1(\alpha_y),s_1(\alpha_z)\#_1\alpha_y|\alpha_x\}
    \#_3
    \begin{bmatrix}\{\alpha_z,\alpha_y\} & \\ & \alpha_x\end{bmatrix}
  \bigr) \\
  \nonumber
  &\quad \#_4
  \{\{\alpha_z,\alpha_y\},\alpha_x\}^{-1_4} \\
  \nonumber
  &\quad \#_4
  \bigl(
    \begin{bmatrix}& \alpha_x \\ \{\alpha_z,\alpha_y\} &\end{bmatrix}
    \#_3
    \{t_1(\alpha_z)\#_1\alpha_y,\alpha_z\#_1s_1(\alpha_y)|\alpha_x\}^{-1_4}
  \bigr) \\
  \nonumber
  &= 
  \bigl(
    \begin{bmatrix}& \{\alpha_y,\alpha_x\} \\ \alpha_z &\end{bmatrix}
    \#_3
    \{\alpha_z|t_1(\alpha_y)\#_1\alpha_x,\alpha_y\#_1s_1(\alpha_x)\}
  \bigr) \\
  \nonumber
  &\quad \#_4
  \{\alpha_z,\{\alpha_y,\alpha_x\}\} \\
  \label{eq:SpSm}
  &\quad \#_4
  \bigl(
    \{\alpha_z|\alpha_y\#_1t_1(\alpha_x),s_1(\alpha_y)\#_1\alpha_x\}^{-1_4}
    \#_3
    \begin{bmatrix}\alpha_z & \\ & \{\alpha_y,\alpha_x\}\end{bmatrix}
  \bigr)
\end{align}

\section{Specific Calculations for the Proof of Theorem~\ref{3-cm_to_gray}}
\label{proof:3-cm_to_Gray4-group}
This appendix provides the calculations omitted in the proof of Theorem \ref{3-cm_to_gray}. As seen in the case of prism 3-22, the formula that remains to be proved concerns only the first entry of the parentheses. The desired formula is easily obtained by reading off the expressions for the 3-crossed module from the corresponding diagrams of the 4-category.
\subsection{Prism 32-2}
\[
  \act{l}\{h_z,h_y|h_x\}^{-1}L(l,h_zh_y;h_x) = \{l,\act{h_z}\{h_y,h_x\}\}\act{\act{\partial lh_z}\{h_y,h_x\}}L(l,h_z;\act{\partial h_y}h_x)\{\partial lh_z,h_y|h_x\}^{-1}
\]
Starting from the left-hand side, we have
\begin{align*}
  &(\text{L.H.S}) \\
  &= 
  \act{l}\{h_z,h_y|h_x\}^{-1}\{l,\{h_zh_y,h_x\}\}\act{\act{\partial l}\{h_zh_y,h_x\}}\left<l,\act{\partial (h_zh_y)}h_x\right>\{\partial l,h_zh_y|h_x\}^{-1}
  & &\text{\tiny \eqref{eq:defL}} \\
  &= \{l,\act{h_z}\{h_y,h_x\}\{h_z,\act{\partial h_y}h_x\}\}\act{\partial l}\{h_z,h_y|h_x\}^{-1}\act{\act{\partial l}\{h_zh_y,h_x\}}\left<l,\act{\partial (h_zh_y)}h_x\right>\{\partial l,h_zh_y|h_x\}^{-1}
  & &\text{\tiny \eqref{eq:2CM-2-funct2}} \\
  &= \{l,\act{h_z}\{h_y,h_x\}\}\act{\act{\partial lh_z}\{h_y,h_x\}}\{l,\{h_z,\act{\partial h_y}h_x\}\} & &\text{\tiny \eqref{eq:HH_H-2CM}} \\
  &\qquad \act{\act{\partial lh_z}\{h_y,h_x\}\act{\partial l}\{h_z,\act{\partial h_y}h_x\}}\left<l,\act{\partial (h_zh_y)}h_x\right>\act{\partial l}\{h_z,h_y|h_x\}^{-1}\{\partial l,h_zh_y|h_x\}^{-1} & &\text{\tiny } \\
  &= \{l,\act{h_z}\{h_y,h_x\}\}\act{\act{\partial lh_z}\{h_y,h_x\}}\left(\{l,\{h_z,\act{\partial h_y}h_x\}\}\act{\act{\partial l}\{h_z,\act{\partial h_y}h_x\}}\left<l,\act{\partial (h_zh_y)}h_x\right>\right) \\
  &\qquad \act{\act{\partial lh_z}\{h_y,h_x\}}\{\partial l,h_z|\act{\partial h_y}h_x\}^{-1}\{\partial lh_z,h_y|h_x\}^{-1}& &\text{\tiny \eqref{eq:3-CM-Fund-left-Hom}} \\
  &= \{l,\act{h_z}\{h_y,h_x\}\}\act{\act{\partial lh_z}\{h_y,h_x\}}L(l,h_z;\act{\partial h_y}h_x)\{\partial lh_z,h_y|h_x\}^{-1} & &\text{\tiny \eqref{eq:defL}} \\
  &= (\text{R.H.S})
\end{align*}
On the right, we roughly indicate which relations are used in each step.

\subsection{Prism 23-2}
\[
  \act{\act{h_z}l}\{h_z,h_y|h_x\}^{-1}L(\act{h_z}l,h_zh_y;h_x) = \act{h_z}L(l,h_y;h_x)\act{\act{h_z}\{h_y',h_x\}}\{\{h_z,\act{\partial h_y}h_x\},\act{\act{\partial (h_zh_y)}h_xh_z}l\}^{-1}\{h_z,h_y'|h_x\}^{-1}
\]
Starting from the left-hand side, we have
\begin{align*}
  &(\text{L.H.S}) \\
  &= 
  \act{\act{h_z}l}\{h_z,h_y|h_x\}^{-1}\{\act{h_z}l,\{h_zh_y,h_x\}\}\act{\act{\partial \act{h_z}l}\{h_zh_y,h_x\}}\left<\act{h_z}l,\act{\partial (h_zh_y)}h_x\right>\{\partial \act{h_z}l,h_zh_y|h_x\}^{-1} & &\text{\tiny \eqref{eq:defL}} \\
  &= \{\act{h_z}l,\act{h_z}\{h_y,h_x\}\{h_z,\act{\partial h_y}h_x\}\} \\
  &\qquad \act{\partial\act{h_z}l}\{h_z,h_y|h_x\}^{-1}\act{\act{\partial \act{h_z}l}\{h_zh_y,h_x\}}\left<\act{h_z}l,\act{\partial (h_zh_y)}h_x\right>\{\partial \act{h_z}l,h_zh_y|h_x\}^{-1} & &\text{\tiny \eqref{eq:2CM-2-funct2}} \\
  &= \{\act{h_z}l,\act{h_z}\{h_y,h_x\}\}\act{\act{h_z\partial l}\{h_y,h_x\}}\{\act{h_z}l,\{h_z,\act{\partial h_y}h_x\}\} & &\text{\tiny \eqref{eq:H_HH-2CM}} \\
  &\qquad \act{\act{h_z\partial l}\{h_y,h_x\}\act{\partial\act{h_z}l}\{h_z,\act{\partial h_y}h_x\}}\left<\act{h_z}l,\act{\partial (h_zh_y)}h_x\right>\act{\partial\act{h_z}l}\{h_z,h_y|h_x\}^{-1}\{\partial \act{h_z}l,h_zh_y|h_x\}^{-1} & &\text{\tiny \eqref{eq:2CM-Peifferid}} \\
  &= \{\act{h_z}l,\act{h_z}\{h_y,h_x\}\}\act{\act{h_z\partial l}\{h_y,h_x\}}\left(\{\act{h_z}l,\{h_z,\act{\partial h_y}h_x\}\}\act{\act{\partial\act{h_z}l}\{h_z,\act{\partial h_y}h_x\}}\left<\act{h_z}l,\act{\partial (h_zh_y)}h_x\right>\right) & &\text{\tiny \eqref{eq:H_HH-2CM}} \\
  &\qquad 
  \act{\act{h_z\partial}\{h_y,h_z\}}\{\partial \act{h_z}l,h_z|\act{\partial h_y}h_x\}^{-1}\{h_z\partial l,h_y|h_x\}^{-1} & &\text{\tiny \eqref{eq:3-CM-Fund-left-Hom}} \\
  &= \{\act{h_z}l,\act{h_z}\{h_y,h_x\}\} \\
  &\qquad \act{\act{h_z\partial l}\{h_y,h_x\}}\left(\act{h_z}\left<l,\act{\partial h_y}h_x\right>\act{\act{h_z}\{\partial l,\act{\partial h_y}h_x\}}\{\{h_z,\act{\partial h_y}h_x\},\act{\act{\partial (h_zh_y)}h_xh_z}l\}^{-1}\{h_z,\partial l|\act{\partial h_y}h_x\}^{-1}\right) & &\text{\tiny \eqref{eq:H_HH-2CM}} \\
  &\qquad 
  \{h_z\partial l,h_y|h_x\}^{-1} \\
  &= \act{h_z}\left(\{l,\{h_y,h_x\}\}\act{\act{\partial l}\{h_y,h_x\}}\left<l,\act{\partial h_y}h_x\right>\right) \\
  &\qquad \act{\act{h_z\partial l}\{h_y,h_x\}\act{h_z}\{\partial l,\act{\partial h_y}h_x\}}\{\{h_z,\act{\partial h_y}h_x\},\act{\act{\partial (h_zh_y)}h_xh_z}l\}^{-1} \\
  &\qquad \act{h_z}\{\partial l,h_y|h_x\}^{-1}\{h_z,\partial lh_y|h_x\}^{-1} & &\text{\tiny \eqref{eq:3-CM-Fund-left-Hom}} \\
  &= \act{h_z}\left(\{l,\{h_y,h_x\}\}\act{\act{\partial l}\{h_y,h_x\}}\left<l,\act{\partial h_y}h_x\right>\{\partial l,h_y|h_x\}^{-1}\right) \\
  &\qquad \act{\act{h_z}\{\partial lh_y,h_x\}}\{\{h_z,\act{\partial h_y}h_x\},\act{\act{\partial (h_zh_y)}h_xh_z}l\}^{-1}\{h_z,\partial lh_y|h_x\}^{-1} & &\text{\tiny \eqref{eq:2CM-Peifferid}} \\
  &= \act{h_z}L(l,h_y;h_x)\act{\act{h_z}\{\partial lh_y,h_x\}}\{\{h_z,\act{\partial h_y}h_x\},\act{\act{\partial (h_zh_y)}h_xh_z}l\}^{-1}\{h_z,\partial lh_y|h_x\}^{-1} & &\text{\tiny \eqref{eq:defL}} \\
  &= (\text{R.H.S})
\end{align*}

\subsection{Prism 22-3.}
\[
  \act{\act{h_zh_y}l}\{h_z,h_y|h_x\}^{-1}R(h_zh_y;l,h_x)
  = \act{h_z}R(h_y;l,h_x)\act{\act{h_z}\{h_y,\partial lh_x\}}R(h_z;\act{\partial h_y}l,\act{\partial h_y}h_x)\{h_z,h_y|\partial lh_x\}^{-1}
\]
Starting from the left-hand side, we have
\begin{align*}
  &(\text{L.H.S}) \\
  &= 
  \act{\act{h_zh_y}l}\{h_z,h_y|h_x\}^{-1}
  \left<h_zh_y,l\right>\act{\{h_zh_y,\partial l\}}\{\act{\partial (h_zh_y)}l,\{h_zh_y,h_x\}\}\{h_zh_y|\partial l,h_x\}^{-1} & &\text{\tiny \eqref{eq:defR}} \\
  &= 
  \left<h_zh_y,l\right>  &&\text{\tiny \eqref{eq:2CM-Peifferid}} \\
  &\qquad 
  \act{\{h_zh_y,\partial l\}}(\act{\act{\partial (h_zh_y)}l}\{h_z,h_y|h_x\}^{-1}\{\act{\partial (h_zh_y)}l,\{h_zh_y,h_x\}\}) \\
  &\qquad 
  \{h_zh_y|\partial l,h_x\}^{-1} \\
  &= 
  \act{h_z}\left<h_y,l\right>\act{\act{h_z}\{h_y,\partial l\}}\left<h_z,\act{\partial h_y}l\right>\{h_z,h_y|\partial l\}^{-1}  &&\text{\tiny \eqref{eq:3CM-aux-HH-L}} \\
  &\qquad 
  \act{\{h_zh_y,\partial l\}}(\{\act{\partial (h_zh_y)}l,\act{h_z}\{h_y,h_x\}\{h_z,\act{\partial h_y}h_x\}\}\act{\partial\act{\partial (h_zh_y)}l}\{h_z,h_y|h_x\}^{-1}) & &\text{\tiny \eqref{eq:2CM-2-funct2}} \\
  &\qquad 
  \{h_zh_y|\partial l,h_x\}^{-1} \\
  &= 
  \act{h_z}\left<h_y,l\right>\red{\act{\act{h_z}\{h_y,\partial l\}}\left<h_z,\act{\partial h_y}l\right>} \\
  &\qquad 
  \act{\act{h_z}\{h_y,\partial l\}\{h_z,\act{\partial h_y}\partial l\}}(\red{\{\act{\partial (h_zh_y)}l,\act{h_z}\{h_y,h_x\}\}}\act{\act{\partial \act{\partial (h_zh_y)}lh_z}\{h_y,h_x\}}\{\act{\partial (h_zh_y)}l,\{h_z,\act{\partial h_y}h_x\}\}) & &\text{\tiny \eqref{eq:H_HH-2CM}} \\
  &\qquad 
  \{h_z,h_y|\partial l\}^{-1}
  \act{\{h_zh_y,\partial l\}}(\act{\partial\act{\partial (h_zh_y)}l}\{h_z,h_y|h_x\}^{-1})
  \{h_zh_y|\partial l,h_x\}^{-1} & &\text{\tiny \eqref{eq:2CM-Peifferid}}
\end{align*}
Here, we focus on the two factors in the red-colored part. 
\begin{align*}
  &\act{\act{h_z}\{h_y,\partial l\}\{h_z,\act{\partial h_y}\partial l\}}(\act{\{h_z,\act{\partial h_y}\partial l\}^{-1}}\left<h_z,\act{\partial h_y}l\right>\{\act{\partial (h_zh_y)}l,\act{h_z}\{h_y,h_x\}\}) \\
  &= \act{\act{h_z}\{h_y,\partial l\}\{h_z,\act{\partial h_y}\partial l\}}\Bigl(\{\partial\act{\{h_z,\act{\partial h_y}\partial l\}^{-1}}\left<h_z,\act{\partial h_y}l\right>\act{\partial (h_zh_y)}l,\act{h_z}\{h_y,h_x\}\} \\
  &\hspace{5cm} \act{\act{\partial \act{\partial (h_zh_y)}lh_z}\{h_y,h_x\}\{h_z,\act{\partial h_y}\partial l\}^{-1}}\left<h_z,\act{\partial h_y}l\right>\Bigr) &&\text{\tiny \eqref{eq:2CM-2-funct1}} \\
  &= \act{\act{h_z}\{h_y,\partial l\}\{h_z,\act{\partial h_y}\partial l\}}\{\{h_z,\act{\partial h_y}\partial l\}^{-1}\act{h_z}(\act{\partial h_y}l),\act{h_z}\{h_y,h_x\}\} &&\text{\tiny \eqref{eq:3CM-bHL}} \\
  &\qquad  \act{\act{h_z}\{h_y,\partial l\}\{h_z,\act{\partial h_y}\partial l\}\act{\partial \act{\partial (h_zh_y)}lh_z}\{h_y,h_x\}\{h_z,\act{\partial h_y}\partial l\}^{-1}}\left<h_z,\act{\partial h_y}l\right> \\
  &= \act{\act{h_z}\{h_y,\partial l\}}\{\act{h_z}(\act{\partial h_y}l),\act{h_z}\{h_y,h_x\}\}\act{\act{h_z}\{h_y,\partial l\}\{h_z,\act{\partial h_y}\partial l\}}\{\{h_z,\act{\partial h_y}\partial l\}^{-1},\act{h_z\partial \act{\partial h_y}l}\{h_y,h_x\}\} &&\text{\tiny \eqref{eq:HH_H-2CM}} \\
  &\qquad  \act{\act{h_z}\{h_y,\partial l\}\partial(\{\{h_z,\act{\partial h_y}\partial l\},\act{\partial \act{\partial (h_zh_y)}lh_z}\{h_y,h_x\}\})\act{\partial \{h_z,\act{\partial h_y}\partial l\}\partial \act{\partial (h_zh_y)}lh_z}\{h_y,h_x\}}\left<h_z,\act{\partial h_y}l\right> & &\text{\tiny \eqref{eq:2CM-bPeiffer}} \\
  &= \act{\act{h_z}\{h_y,\partial l\}}\{\act{h_z}(\act{\partial h_y}l),\act{h_z}\{h_y,h_x\}\}\act{\act{h_z}\{h_y,\partial l\}}\{\{h_z,\act{\partial h_y}\partial l\},\act{h_z\partial \act{\partial h_y}l}\{h_y,h_x\}\}^{-1} &&\text{\tiny (*)} \\
  &\qquad  \act{\act{h_z}\{h_y,\partial l\}\partial(\{\{h_z,\act{\partial h_y}\partial l\},\act{\partial \act{\partial (h_zh_y)}lh_z}\{h_y,h_x\}\})\act{h_z\partial \act{\partial h_y}l}\{h_y,h_x\}}\left<h_z,\act{\partial h_y}l\right> \\
  &= \act{h_z}(\act{\{h_y,\partial l\}}\{\act{\partial h_y}l,\{h_y,h_x\}\}) \\
  &\qquad  \act{\act{h_z}\{h_y,\partial l\}\act{h_z\partial \act{\partial h_y}l}\{h_y,h_x\}}\left<h_z,\act{\partial h_y}l\right>
  \act{\act{h_z}\{h_y,\partial l\}}\{\{h_z,\act{\partial h_y}\partial l\},\act{h_z\partial \act{\partial h_y}l}\{h_y,h_x\}\}^{-1} & &\text{\tiny \eqref{eq:2CM-Peifferid}}
\end{align*}
Here, in (*) we used $\act{l}\{l^{-1},l'\} = \{l,\act{\partial l^{-1}}l'\}^{-1}$, which follows from Equation \eqref{eq:HH_H-2CM}. Using the above,
\begin{align*}
  (\text{L.H.S}) &= 
  \act{h_z}\left<h_y,l\right>\act{h_z}(\act{\{h_y,\partial l\}}\{\act{\partial h_y}l,\{h_y,h_x\}\}) \\
  &\qquad  \act{\act{h_z}\{h_y,\partial l\}\act{h_z\partial \act{\partial h_y}l}\{h_y,h_x\}}\left<h_z,\act{\partial h_y}l\right>
  \act{\act{h_z}\{h_y,\partial l\}}\{\{h_z,\act{\partial h_y}\partial l\},\act{h_z\partial \act{\partial h_y}l}\{h_y,h_x\}\}^{-1} \\
  &\qquad 
  \act{\act{h_z}\{h_y,\partial l\}\{h_z,\act{\partial h_y}\partial l\}\act{\partial \act{\partial (h_zh_y)}lh_z}\{h_y,h_x\}}\{\act{\partial (h_zh_y)}l,\{h_z,\act{\partial h_y}h_x\}\} \\
  &\qquad 
  \{h_z,h_y|\partial l\}^{-1}
  \act{\{h_zh_y,\partial l\}}(\act{\partial\act{\partial (h_zh_y)}l}\{h_z,h_y|h_x\}^{-1})
  \{h_zh_y|\partial l,h_x\}^{-1} \\
  &= \act{h_z}\left<h_y,l\right>\act{h_z}(\act{\{h_y,\partial l\}}\{\act{\partial h_y}l,\{h_y,h_x\}\}) \\
  &\qquad  \act{\act{h_z}\{h_y,\partial l\}\act{h_z\partial \act{\partial h_y}l}\{h_y,h_x\}}\left<h_z,\act{\partial h_y}l\right> \\
  &\qquad 
  \act{\act{h_z}\{h_y,\partial l\}\act{\partial \{h_z,\act{\partial h_y}\partial l\}\partial \act{\partial (h_zh_y)}lh_z}\{h_y,h_x\}\{h_z,\act{\partial h_y}\partial l\}}\{\act{\partial (h_zh_y)}l,\{h_z,\act{\partial h_y}h_x\}\} \\
  &\qquad \act{\act{h_z}\{h_y,\partial l\}}\{\{h_z,\act{\partial h_y}\partial l\},\act{h_z\partial \act{\partial h_y}l}\{h_y,h_x\}\}^{-1} & &\text{\tiny \eqref{eq:2CM-Peifferid}} \\
  &\qquad 
  \{h_z,h_y|\partial l\}^{-1}
  \act{\{h_zh_y,\partial l\}}(\act{\partial\act{\partial (h_zh_y)}l}\{h_z,h_y|h_x\}^{-1})
  \{h_zh_y|\partial l,h_x\}^{-1} \\
  &= \act{h_z}\left<h_y,l\right>\act{h_z}(\act{\{h_y,\partial l\}}\{\act{\partial h_y}l,\{h_y,h_x\}\}) \\
  &\qquad  \act{\act{h_z}\{h_y,\partial l\}\act{h_z\partial \act{\partial h_y}l}\{h_y,h_x\}}(\left<h_z,\act{\partial h_y}l\right>\act{\{h_z,\act{\partial h_y}\partial l\}}\{\act{\partial (h_zh_y)}l,\{h_z,\act{\partial h_y}h_x\}\}) \\
  &\qquad 
  \act{h_z}\{h_y|\partial l,h_x\}^{-1}\act{\act{h_z}\{h_y,\partial lh_x\}}\{h_z|\act{\partial h_y}\partial l,\act{\partial h_y}h_x\}^{-1}\{h_z,h_y|\partial lh_x\}^{-1} & &\text{\tiny \eqref{eq:3CM-KV}} \\
  &= \act{h_z}\left<h_y,l\right>\act{h_z}(\act{\{h_y,\partial l\}}\{\act{\partial h_y}l,\{h_y,h_x\}\})\act{h_z}\{h_y|\partial l,h_x\}^{-1} & &\text{\tiny \eqref{eq:2CM-Peifferid}} \\
  &\qquad  \act{\act{h_z}\{h_y,\partial lh_x\}}(\left<h_z,\act{\partial h_y}l\right>\act{\{h_z,\act{\partial h_y}\partial l\}}\{\act{\partial (h_zh_y)}l,\{h_z,\act{\partial h_y}h_x\}\}) \\
  &\qquad 
  \act{\act{h_z}\{h_y,\partial lh_x\}}\{h_z|\act{\partial h_y}\partial l,\act{\partial h_y}h_x\}^{-1}\{h_z,h_y|\partial lh_x\}^{-1} \\
  &= \act{h_z}R(h_y;l,h_x)\act{\act{h_z}\{h_y,\partial lh_x\}}R(h_z;\act{\partial h_y}l,\act{\partial h_y}h_x)\{h_z,h_y|\partial lh_x\}^{-1} & &\text{\tiny \eqref{eq:defR}} \\
  &= (\text{L.H.S})
\end{align*}

\subsection{Prism 2-32}
\begin{align*}
  \act{\act{h_z}l}\{h_z|h_y,h_x\}^{-1}R(h_z;l,h_yh_x) = R(h_z;l,h_y)\act{\{h_z,\partial lh_y\}}\{\act{\partial h_z}l,\act{\act{\partial h_z}h_y}\{h_z,h_x\}\}\{h_z|\partial lh_y,h_x\}^{-1}
\end{align*}
Starting from the left-hand side, we have
\begin{align*}
  \text{(L.H.S)} &= \act{\act{h_z}l}\{h_z|h_y,h_x\}^{-1}\left<h_z,l\right>\act{\{h_z,\partial l\}}\{\act{\partial h_z}l,\{h_z,h_yh_x\}\}\{h_z|\partial l,h_yh_x\}^{-1} & &\text{\tiny \eqref{eq:defR}} \\
  &= \left<h_z,l\right>\act{\{h_z,\partial l\}}(\act{\act{\partial h_z}l}\{h_z|h_y,h_x\}^{-1}\{\act{\partial h_z}l,\{h_z,h_yh_x\}\})\{h_z|\partial l,h_yh_x\}^{-1} & &\text{\tiny \eqref{eq:2CM-Peifferid}} \\
  &= \left<h_z,l\right>\act{\{h_z,\partial l\}}(\{\act{\partial h_z}l,\{h_z,h_y\}\act{\act{\partial h_z}h_y}\{h_z,h_x\}\}\act{\partial\act{\partial h_z}l}\{h_z|h_y,h_x\}^{-1})\{h_z|\partial l,h_yh_x\}^{-1} & &\text{\tiny \eqref{eq:2CM-2-funct2}} \\
  &= \left<h_z,l\right>\act{\{h_z,\partial l\}}(\{\act{\partial h_z}l,\{h_z,h_y\}\}\act{\act{\partial \act{\partial h_z}l}\{h_z,h_y\}}\{\act{\partial h_z}l,\act{\act{\partial h_z}h_y}\{h_z,h_x\}\}) & &\text{\tiny \eqref{eq:H_HH-2CM}} \\
  &\qquad 
  \{h_z|\partial l,h_y\}^{-1}\{h_z|\partial lh_y,h_x\}^{-1} & &\text{\tiny \eqref{eq:3-CM-Fund-right-Hom}} \\
  &= \left<h_z,l\right>\act{\{h_z,\partial l\}}\{\act{\partial h_z}l,\{h_z,h_y\}\}\{h_z|\partial l,h_y\}^{-1}\act{\{h_z,\partial lh_y\}}\{\act{\partial h_z}l,\act{\act{\partial h_z}h_y}\{h_z,h_x\}\} & &\text{\tiny \eqref{eq:2CM-Peifferid}} \\
  &\qquad 
  \{h_z|\partial lh_y,h_x\}^{-1} & &\text{\tiny \eqref{eq:3-CM-Fund-right-Hom}} \\
  &= R(h_z;l,h_y)\act{\{h_z,\partial lh_y\}}\{\act{\partial h_z}l,\act{\act{\partial h_z}h_y}\{h_z,h_x\}\}\{h_z|\partial lh_y,h_x\}^{-1} & &\text{\tiny \eqref{eq:defR}} \\
  &= (\text{R.H.S})
\end{align*}

\subsection{Prism 2-23}
\[
  \act{\act{h_zh_y}l}\{h_z|h_y,h_x\}^{-1}R(h_z;\act{h_y}l,h_yh_x) = \{\{h_z,h_y\},\act{\act{\partial h_z}h_yh_x}l\}^{-1}\act{\{h_z,h_y\}}(\act{\act{\partial h_z}h_y}R(h_z;l,h_x))\{h_z|h_y,h_x'\}
\]
Starting from the left-hand side, we have
\begin{align*}
  &(\text{L.H.S}) \\
  &= \act{\act{h_zh_y}l}\{h_z|h_y,h_x\}^{-1}\left<h_z,\act{h_y}l\right>\act{\{h_z,\partial\act{h_y}l\}}\{\act{\partial h_z}(\act{h_y}l),\{h_z,h_yh_x\}\}\{h_z,\partial \act{h_y}l|h_yh_x\}^{-1} & &\text{\tiny \eqref{eq:defR}} \\
  &= \left<h_z,\act{h_y}l\right>\act{\{h_z,\partial\act{h_y}l\}}(\act{\act{\partial h_z}(\act{h_y}l)}\{h_z|h_y,h_x\}^{-1}\{\act{\partial h_z}(\act{h_y}l),\{h_z,h_yh_x\}\})\{h_z,\partial \act{h_y}l|h_yh_x\}^{-1} & & \text{\tiny \eqref{eq:2CM-Peifferid}} \\
  &= \left<h_z,\act{h_y}l\right>\act{\{h_z,\partial\act{h_y}l\}}(\{\act{\partial h_z}(\act{h_y}l),\{h_z,h_y\}\act{\act{\partial h_z}h_y}\{h_z,h_x\}\}\act{\partial\act{\partial h_z}(\act{h_y}l)}\{h_z|h_y,h_x\}^{-1}) & &\text{\tiny \eqref{eq:2CM-2-funct2}} \\
  &\qquad \{h_z,\partial \act{h_y}l|h_yh_x\}^{-1} \\
  &= \left<h_z,\act{h_y}l\right>\act{\{h_z,\partial\act{h_y}l\}}(\{\act{\partial h_z}(\act{h_y}l),\{h_z,h_y\}\}\act{\act{\partial \act{\partial h_z}(\act{h_y}l)}\{h_z,h_y\}}\{\act{\partial h_z}(\act{h_y}l),\act{\act{\partial h_z}h_y}\{h_z,h_x\}\}) & &\text{\tiny \eqref{eq:H_HH-2CM}} \\
  &\qquad 
  \{h_z|\partial \act{h_y}l,h_y\}^{-1}\{h_z|\partial \act{h_y}lh_y,h_x\}^{-1} & &\text{\tiny \eqref{eq:3-CM-Fund-right-Hom}} \\
  &= \left<h_z,\act{h_y}l\right>\act{\{h_z,\partial\act{h_y}l\}}\{\act{\partial h_z}(\act{h_y}l),\{h_z,h_y\}\}\{h_z|\partial \act{h_y}l,h_y\}^{-1} & & \text{\tiny \eqref{eq:2CM-Peifferid}} \\
  &\qquad \act{\{h_z,\partial\act{h_y}lh_y\}}\{\act{\partial h_z}(\act{h_y}l),\act{\act{\partial h_z}h_y}\{h_z,h_x\}\}\{h_z|\partial \act{h_y}lh_y,h_x\}^{-1} \\
  &= \{\{h_z,h_y\},\act{\act{\partial h_z}h_yh_z}l\}^{-1}\act{\{h_z,h_y\}}(\act{\act{\partial h_z}h_y}\left<h_z,l\right>)\{h_z|h_y,\partial l\}^{-1} & & \text{\tiny \eqref{eq:3cm_prism2-23}} \\
  &\qquad \act{\{h_z,h_y\partial l\}}\{\act{\partial h_z}(\act{h_y}l),\act{\act{\partial h_z}h_y}\{h_z,h_x\}\}\{h_z|h_y\partial l,h_x\}^{-1} \\
  &= \{\{h_z,h_y\},\act{\act{\partial h_z}h_yh_z}l\}^{-1}\act{\{h_z,h_y\}}(\act{\act{\partial h_z}h_y}\left<h_z,l\right>)\act{\{h_z,h_y\}\act{\act{\partial h_z}h_y}\{h_z,\partial l\}}\{\act{\partial h_z}(\act{h_y}l),\act{\act{\partial h_z}h_y}\{h_z,h_x\}\} & &\text{\tiny \eqref{eq:2CM-Peifferid}} \\
  &\qquad \act{\{h_z,h_y\}}(\act{\act{\partial h_z}h_y}\{h_z|\partial l,h_x\}^{-1})\{h_z|h_y,\partial lh_x\}^{-1} & & \text{\tiny \eqref{eq:3-CM-Fund-right-Hom}} \\
  &= \{\{h_z,h_y\},\act{\act{\partial h_z}h_yh_z}l\}^{-1}\act{\{h_z,h_y\}}[\act{\act{\partial h_z}h_y}(\left<h_z,l\right>\act{\{h_z,\partial l\}}\{\act{\partial h_z}l,\{h_z,h_x\}\}\{h_z|\partial l,h_x\}^{-1})] \\
  &\qquad \{h_z|h_y,\partial lh_x\}^{-1} \\
  &= \{\{h_z,h_y\},\act{\act{\partial h_z}h_yh_z}l\}^{-1}\act{\{h_z,h_y\}}(\act{\act{\partial h_z}h_y}R(h_z;l,h_x))\{h_z|h_y,\partial lh_x\}^{-1} & &\text{\tiny \eqref{eq:defR}} \\
  &= (\text{R.H.S})
\end{align*}

\subsection{Cube}
\[
  \act{l_2}R(h_2;l_1,h_1)L(l_2,h_2,\partial l_1h_1) = \{l_2,\act{h_2}l_1\}\act{\act{\partial l_2h_2}l_1}L(l_2,h_2;h_1)R(\partial l_2h_2;l_1,h_1)\act{\{\partial l_2h_2,\partial l_1h_1\}}\{\act{\partial h_2}l_1,\act{\act{\partial h_2}h_1}l_2\}
\]
Starting from the left-hand side, we have
\begin{align*}
  &(\text{L.H.S}) \\
  &= 
  \act{l_2}R(h_2;l_1,h_1) \\
  &\qquad \{l_2,\{h_2,\partial l_1h_1\}\}\act{\act{\partial l_2}\{h_2,\partial l_1h_1\}}\left<l_2,\act{\partial h_2}(\partial l_1h_1)\right>\{\partial l_2,h_2|\partial l_1h_1\}^{-1} & &\text{\tiny \eqref{eq:defL}} \\
  &= \act{l_2}R(h_2;l_1,h_1)\{l_2,\partial R(h_2;l_1,h_1)^{-1}\act{h_2}l_1\{h_2,h_1\}\act{\partial h_2}l_1^{-1}\} & &\text{\tiny \eqref{eq:3CM-twisted-2funct2}} \\
  &\qquad \act{\act{\partial l_2}\{h_2,\partial l_1h_1\}}\left<l_2,\act{\partial h_2}(\partial l_1h_1)\right>\{\partial l_2,h_2|\partial l_1h_1\}^{-1} \\
  &= \{l_2,\act{h_2}l_1\{h_2,h_1\}\act{\partial h_2}l_1^{-1}\}\act{\partial l_2}R(h_2;l_1,h_1) & &\text{\tiny \eqref{eq:2CM-2-funct2}} \\
  &\qquad \act{\act{\partial l_2}\{h_2,\partial l_1h_1\}}\left<l_2,\act{\partial h_2}(\partial l_1h_1)\right>\{\partial l_2,h_2|\partial l_1h_1\}^{-1} \\
  &= \{l_2,\act{h_2}l_1\}\act{\act{\partial l_2h_2}l_1}\{l_2,\{h_2,h_1\}\}\act{\act{\partial l_2h_2}l_1\act{\partial l_2}\{h_2,h_1\}}\{l_2,\act{\partial h_2}l_1^{-1}\} & &\text{\tiny \eqref{eq:H_HH-2CM}} \\
  &\qquad \act{\partial l_2}R(h_2;l_1,h_1)\act{\act{\partial l_2}\{h_2,\partial l_1h_1\}}\left<l_2,\act{\partial h_2}(\partial l_1h_1)\right>\{\partial l_2,h_2|\partial l_1h_1\}^{-1}
\end{align*}
Here, applying Equation~\eqref{eq:braket322} to the $\left<-,-\right>$ part,
\begin{align*}
  &(\text{L.H.S}) \\
  &= \{l_2,\act{h_2}l_1\}\act{\act{\partial l_2h_2}l_1}\{l_2,\{h_2,h_1\}\}\act{\partial l_2}R(h_2;l_1,h_1)\red{\act{\act{\partial l_2}(\{h_2,\partial l_1h_1\}\act{\partial h_2}l_1)}\{l_2,\act{\partial h_2}l_1^{-1}\}} & &\text{\tiny \eqref{eq:2CM-Peifferid}} \\
  &\qquad \act{\act{\partial l_2}\{h_2,\partial l_1h_1\}}(\red{\left<l_2,\act{\partial h_2}\partial l_1\right>}\act{\{\partial l_2,\act{\partial h_2}\partial l_1\}}(\act{\act{\partial h_2}\partial l_1}\left<l_2,\act{\partial h_2}h_1\right>)\{\partial l_2|\act{\partial h_2}\partial l_1,\act{\partial h_2}h_1\}^{-1}) \\
  &\qquad \{\partial l_2,h_2|\partial l_1h_1\}^{-1}
\end{align*}
Furthermore, we transform the red-colored part using Equation \eqref{eq:3-CM-aux-Cube} 
and obtain
\begin{align*}
  \red{\act{\act{\partial l_2}(\act{\partial h_2}l_1)}\{l_2,\act{\partial h_2}l_1^{-1}\}\left<l_2,\act{\partial h_2}\partial l_1\right>} &= \{l_2,\act{\partial h_2}l_1\}^{-1}\left<l_2,\act{\partial h_2}\partial l_1\right> \\
  &= \left<\partial l_2,\act{\partial h_2}l_1\right>\act{\{\partial l_2,\act{\partial h_2}\partial l_1\}}\{\act{\partial h_2}l_1,l_2\}
\end{align*}
Hence,
\begin{align*}
  &(\text{L.H.S}) \\
  &= \{l_2,\act{h_2}l_1\}\act{\act{\partial l_2h_2}l_1}\{l_2,\{h_2,h_1\}\}\act{\partial l_2}R(h_2;l_1,h_1) & &\text{\tiny \eqref{eq:2CM-Peifferid}} \\
  &\qquad \act{\act{\partial l_2}\{h_2,\partial l_1h_1\}}(\left<\partial l_2,\act{\partial h_2}l_1\right>\act{\{\partial l_2,\act{\partial h_2}\partial l_1\}}(\red{\{\act{\partial h_2}l_1,l_2\}\act{\act{\partial h_2}\partial l_1}\left<l_2,\act{\partial h_2}h_1\right>})\{\partial l_2|\act{\partial h_2}\partial l_1,\act{\partial h_2}h_1\}^{-1}) \\
  &\qquad \{\partial l_2,h_2|\partial l_1h_1\}^{-1}
\end{align*}
(We have recolored.) Focusing on the red part,
\begin{align*}
  &\red{\{\act{\partial h_2}l_1,l_2\}\act{\act{\partial h_2}\partial l_1}\left<l_2,\act{\partial h_2}h_1\right>} \\
  &= \act{\act{\partial h_2}l_1}\left<l_2,\act{\partial h_2}h_1\right>\{\act{\partial h_2}l_1,\partial \left<l_2,\act{\partial h_2}h_1\right>^{-1}l_2\} & &\text{\tiny \eqref{eq:2CM-2-funct2}} \\
  &= \act{\act{\partial h_2}l_1}\left<l_2,\act{\partial h_2}h_1\right>\{\act{\partial h_2}l_1,\{\partial l_2,\act{\partial h_2}h_1\}\act{\act{\partial h_2}h_1}l_2\} & &\text{\tiny \eqref{eq:3CM-bLH}} \\
  &= \act{\act{\partial h_2}l_1}\left<l_2,\act{\partial h_2}h_1\right>\{\act{\partial h_2}l_1,\{\partial l_2,\act{\partial h_2}h_1\}\}\act{\act{\partial \act{\partial h_2}l_1}\{\partial l_2,\act{\partial h_2}h_1\}}\{\act{\partial h_2}l_1,\act{\act{\partial h_2}h_1}l_2\} & &\text{\tiny \eqref{eq:H_HH-2CM}}
\end{align*}
Using this,
\begin{align*}
  &(\text{L.H.S}) \\
  &= \{l_2,\act{h_2}l_1\}\act{\act{\partial l_2h_2}l_1}\{l_2,\{h_2,h_1\}\}\act{\partial l_2}R(h_2;l_1,h_1) \\
  &\qquad \act{\act{\partial l_2}\{h_2,\partial l_1h_1\}}\Bigl(\left<\partial l_2,\act{\partial h_2}l_1\right> \\
  &\hspace{2cm} \act{\{\partial l_2,\act{\partial h_2}\partial l_1\}}(\act{\act{\partial h_2}l_1}\left<l_2,\act{\partial h_2}h_1\right>\{\act{\partial h_2}l_1,\{\partial l_2,\act{\partial h_2}h_1\}\}\act{\act{\partial \act{\partial h_2}l_1}\{\partial l_2,\act{\partial h_2}h_1\}}\{\act{\partial h_2}l_1,\act{\act{\partial h_2}h_1}l_2\}) \\
  &\hspace{2cm} \{\partial l_2|\act{\partial h_2}\partial l_1,\act{\partial h_2}h_1\}^{-1}\Bigr)\{\partial l_2,h_2|\partial l_1h_1\}^{-1} \\
  &= \{l_2,\act{h_2}l_1\}\act{\act{\partial l_2h_2}l_1}\{l_2,\{h_2,h_1\}\}\act{\partial l_2}R(h_2;l_1,h_1) \\
  &\qquad \act{\act{\partial l_2}\{h_2,\partial l_1h_1\}}\Bigl(\left<\partial l_2,\act{\partial h_2}l_1\right> \\
  &\hspace{2cm} \act{\{\partial l_2,\act{\partial h_2}\partial l_1\}}(\act{\act{\partial h_2}l_1}\left<l_2,\act{\partial h_2}h_1\right>\{\act{\partial h_2}l_1,\{\partial l_2,\act{\partial h_2}h_1\}\}\act{\act{\partial \act{\partial h_2}l_1}\{\partial l_2,\act{\partial h_2}h_1\}}\{\act{\partial h_2}l_1,\act{\act{\partial h_2}h_1}l_2\}) \\
  &\hspace{2cm} \{\partial l_2|\act{\partial h_2}\partial l_1,\act{\partial h_2}h_1\}^{-1}\Bigr)\{\partial l_2,h_2|\partial l_1h_1\}^{-1} \\
  &= \{l_2,\act{h_2}l_1\}\act{\act{\partial l_2h_2}l_1}\{l_2,\{h_2,h_1\}\}\act{\partial l_2}R(h_2;l_1,h_1)\act{\act{\partial l_2}\{h_2,\partial l_1h_1\}\act{\partial l_2}(\act{\partial h_2}l_1)}\left<l_2,\act{\partial h_2}h_1\right> & &\text{\tiny \eqref{eq:2CM-Peifferid}} \\
  &\qquad \act{\act{\partial l_2}\{h_2,\partial l_1h_1\}}\Bigl(\left<\partial l_2,\act{\partial h_2}l_1\right>\act{\{\partial l_2,\act{\partial h_2}\partial l_1\}}(\{\act{\partial h_2}l_1,\{\partial l_2,\act{\partial h_2}h_1\}\}) \\
  &\hspace{2cm} \{\partial l_2|\act{\partial h_2}\partial l_1,\act{\partial h_2}h_1\}^{-1}
  \act{\{\partial l_2,\act{\partial h_2}(\partial l_1h_1)\}}\{\act{\partial h_2}l_1,\act{\act{\partial h_2}h_1}l_2\}
  \Bigr)\{\partial l_2,h_2|\partial l_1h_1\}^{-1} & &\text{\tiny \eqref{eq:2CM-Peifferid}} \\
  &= \{l_2,\act{h_2}l_1\}\act{\act{\partial l_2h_2}l_1}\{l_2,\{h_2,h_1\}\}\act{\act{\partial l_2}(\act{h_2}l_1\{h_2,h_1\})}\left<l_2,\act{\partial h_2}h_1\right>\red{\act{\partial l_2}R(h_2;l_1,h_1)} & &\text{\tiny \eqref{eq:2CM-Peifferid}} \\
  &\qquad \red{\act{\act{\partial l_2}\{h_2,\partial l_1h_1\}}R(\partial l_2;\act{\partial h_2}l_1,\act{\partial h_2}h_1)} & &\text{\tiny \eqref{eq:defR}} \\
  &\hspace{2cm} 
  \red{\{\partial l_2,h_2|\partial l_1h_1\}^{-1}}\act{\{\partial l_2h_2,\partial l_1h_1\}}\{\act{\partial h_2}l_1,\act{\act{\partial h_2}h_1}l_2\} & &\text{\tiny \eqref{eq:2CM-Peifferid}}
\end{align*}
For the red-colored part, using the 22--3 relation already established,
\begin{align*}
  &(\text{L.H.S}) \\
  &= \{l_2,\act{h_2}l_1\}\act{\act{\partial l_2h_2}l_1}\{l_2,\{h_2,h_1\}\}\act{\act{\partial l_2}(\act{h_2}l_1\{h_2,h_1\})}\left<l_2,\act{\partial h_2}h_1\right> \\
  &\qquad \act{\act{\partial l_2h_2}l_1}\{\partial l_2,h_2|h_1\}^{-1}R(\partial l_2h_2;l_1,h_1)\act{\{\partial l_2h_2,\partial l_1h_1\}}\{\act{\partial h_2}l_1,\act{\act{\partial h_2}h_1}l_2\} \\
  &= \{l_2,\act{h_2}l_1\}\act{\act{\partial l_2h_2}l_1}L(l_2,h_2;h_1)\\
  &\qquad R(\partial l_2h_2;l_1,h_1)\act{\{\partial l_2h_2,\partial l_1h_1\}}\{\act{\partial h_2}l_1,\act{\act{\partial h_2}h_1}l_2\} \\
  &= (\text{R.H.S})
\end{align*}

\subsection{Pasting 33-2}
\[
  L(l'l;h_2,h_1) = \act{l'}L(l;h_2,h_1)L(l';\partial lh_2,h_1)
\]
Starting from the left-hand side, we have
\begin{align*}
  &L(l'l;h_2,h_1) \\
  &= \{l'l,\{h_2,h_1\}\}\act{\act{\partial (l'l)}\{h_2,h_1\}}\left<l'l,\act{\partial h_2}h_1\right>\{\partial (l'l),h_2|h_1\}^{-1} & &\text{\tiny \eqref{eq:defL}} \\
  &= \act{l'}\{l,\{h_2,h_1\}\}\{l',\act{\partial l}\{h_2,h_1\}\} & &\text{\tiny \eqref{eq:HH_H-2CM}} \\
  &\qquad 
  \act{\act{\partial (l'l)}\{h_2,h_1\}}(\act{l'}\left<l,\act{\partial h_2}h_1\right>\{l',\{\partial l,\act{\partial h_2}h_1\}\}\act{\act{\partial l'}\{\partial l,\act{\partial h_2}h_1\}}\left<l',\act{\partial h_2}h_1\right>\{\partial l',\partial l|\act{\partial h_2}h_1\}^{-1}) & &\text{\tiny \eqref{eq:3CM-aux-Pasting1}} \\
  &\qquad \{\partial (l'l),h_2|h_1\}^{-1} \\
  &= \act{l'}(\{l,\{h_2,h_1\}\}\act{\act{\partial l}\{h_2,h_1\}}\left<l,\act{\partial h_2}h_1\right>) & &\text{\tiny \eqref{eq:2CM-Peifferid}} \\
  &\qquad 
  \{l',\act{\partial l}\{h_2,h_1\}\{\partial l,\act{\partial h_2}h_1\}\}\act{\act{\partial (l'l)}\{h_2,h_1\}\act{\partial l'}\{\partial l,\act{\partial h_2}h_1\}}\left<l',\act{\partial h_2}h_1\right> \\
  &\qquad \act{\partial l'}\{\partial l,h_2|h_1\}^{-1}\{\partial l',\partial lh_2|h_1\}^{-1} & &\text{\tiny \eqref{eq:3-CM-Fund-left-Hom}} \\
  &= \act{l'}(\{l,\{h_2,h_1\}\}\act{\act{\partial l}\{h_2,h_1\}}\left<l,\act{\partial h_2}h_1\right>) \\
  &\qquad 
  \{l',\act{\partial l}\{h_2,h_1\}\{\partial l,\act{\partial h_2}h_1\}\}\act{\partial l'}\{\partial l,h_2|h_1\}^{-1} & &\text{\tiny \eqref{eq:2CM-Peifferid}} \\
  &\qquad \act{\act{\partial l'}\{\partial lh_2,h_1\}}\left<l',\act{\partial h_2}h_1\right>\{\partial l',\partial lh_2|h_1\}^{-1} \\
  &= \act{l'}(\{l,\{h_2,h_1\}\}\act{\act{\partial l}\{h_2,h_1\}}\left<l,\act{\partial h_2}h_1\right>\{\partial l,h_2|h_1\}^{-1}) \\
  &\qquad 
  \{l',\{\partial lh_2,h_1\}\}\act{\act{\partial l'}\{\partial lh_2,h_1\}}\left<l',\act{\partial h_2}h_1\right>\{\partial l',\partial lh_2|h_1\}^{-1} & &\text{\tiny \eqref{eq:2CM-2-funct2}} \\
  &= \act{l'}L(l;h_2,h_1)L(l';\partial lh_2,h_1).
\end{align*}

\subsection{Pasting 2-33}
\[
  R(h_2,l'l,h_1) = \act{\act{h_2}l'}R(h_2,l,h_1)R(h_2,l',\partial lh_1),
\]
Starting from the left-hand side, we have
\begin{align*}
  &(\text{L.H.S}) \\
  &= \left<h_2,l'l\right>\act{\{h_2,\partial (l'l)\}}\{\act{\partial h_2}(l'l),\{h_2,h_1\}\}\{h_2|\partial(l'l),h_1\}^{-1} & &\text{\tiny \eqref{eq:defR}} \\
  &= \act{\act{h_2}l'}\left<h_2,l\right>\left<h_2,l'\right>\act{\{h_2,\partial l'\}}\{\act{\partial h_2}l',\{h_2,\partial l\}\}\{h_2|\partial l',\partial l\}^{-1} & &\text{\tiny \eqref{eq:3CM-aux-Pasting2}} \\
  &\qquad \act{\{h_2,\partial (l'l)\}}(\act{\act{\partial h_2}l'}\{\act{\partial h_2}l,\{h_2,h_1\}\}\{\act{\partial h_2}l',\act{\partial \act{\partial h_2}l}\{h_2,h_1\}\}) & &\text{\tiny \eqref{eq:HH_H-2CM}} \\
  &\qquad \{h_2|\partial(l'l),h_1\}^{-1} \\
  &= 
  \act{\act{h_2}l'}\left<h_2,l\right>\left<h_2,l'\right>\act{\{h_2,\partial l'\}}\{\act{\partial h_2}l',\{h_2,\partial l\}\} \\
  &\qquad \act{\{h_2,\partial l'\}\act{\partial \act{\partial h_2}l'}\{h_2,\partial l\}}(\act{\act{\partial h_2}l'}\{\act{\partial h_2}l,\{h_2,h_1\}\}\{\act{\partial h_2}l',\act{\partial \act{\partial h_2}l}\{h_2,h_1\}\}) \\
  &\qquad \{h_2|\partial l',\partial l\}^{-1}\{h_2|\partial(l'l),h_1\}^{-1} & &\text{\tiny \eqref{eq:2CM-Peifferid}} \\
  &= 
  \act{\act{h_2}l'}\left<h_2,l\right>\left<h_2,l'\right> \\
  &\qquad \act{\{h_2,\partial l'\}}(\act{\act{\partial h_2}l'\{h_2,\partial l\}}\{\act{\partial h_2}l,\{h_2,h_1\}\}
  \{\act{\partial h_2}l',\{h_2,\partial l\}\}
  \act{\act{\partial \act{\partial h_2}l'}\{h_2,\partial l\}}\{\act{\partial h_2}l',\act{\partial \act{\partial h_2}l}\{h_2,h_1\}\}) & &\text{\tiny \eqref{eq:2CM-Peifferid}} \\
  &\qquad \act{\{h_2,\partial l'\}}(\act{\act{\partial h_2}\partial l'}\{h_2|\partial l,h_1\}^{-1})\{h_2|\partial l',\partial lh_1\}^{-1} & &\text{\tiny \eqref{eq:3-CM-Fund-right-Hom}} \\
  &= 
  \act{\act{h_2}l'}\left<h_2,l\right>\act{\act{h_2}l'\{h_2,\partial l\}}\{\act{\partial h_2}l,\{h_2,h_1\}\}\left<h_2,l'\right> & &\text{\tiny \eqref{eq:2CM-Peifferid}} \\
  &\qquad 
  \act{\{h_2,\partial l'\}}\{\act{\partial h_2}l',\{h_2,\partial l\}\act{\partial \act{\partial h_2}l}\{h_2,h_1\}\} & &\text{\tiny \eqref{eq:H_HH-2CM}} \\
  &\qquad \act{\{h_2,\partial l'\}}(\act{\act{\partial h_2}\partial l'}\{h_2|\partial l,h_1\}^{-1})\{h_2|\partial l',\partial lh_1\}^{-1} \\
  &= 
  \act{\act{h_2}l'}\left<h_2,l\right>\act{\act{h_2}l'\{h_2,\partial l\}}\{\act{\partial h_2}l,\{h_2,h_1\}\}\left<h_2,l'\right> \\
  &\qquad 
  \act{\{h_2,\partial l'\}\act{\partial h_2}l'}\{h_2|\partial l,h_1\}^{-1}\act{\{h_2,\partial l'\}}\{\act{\partial h_2}l',\{h_2,\partial lh_1\}\} & &\text{\tiny \eqref{eq:2CM-2-funct2}} \\
  &\qquad \{h_2|\partial l',\partial lh_1\}^{-1} \\
  &= 
  \act{\act{h_2}l'}\left<h_2,l\right>\act{\act{h_2}l'\{h_2,\partial l\}}\{\act{\partial h_2}l,\{h_2,h_1\}\}\act{\act{h_2}l'}\{h_2|\partial l,h_1\}^{-1} \\
  &\qquad 
  \left<h_2,l'\right> \act{\{h_2,\partial l'\}}\{\act{\partial h_2}l',\{h_2,\partial lh_1\}\}\{h_2|\partial l',\partial lh_1\}^{-1} & &\text{\tiny \eqref{eq:2CM-Peifferid}} \\
  &= \act{\act{h_2}l'}R(h_2;l,h_1)R(h_2;l',\partial lh_1) & &\text{\tiny \eqref{eq:defR}} \\
  &= (\text{R.H.S}).
\end{align*}

\subsection{Tube 1}
\[
  L(\partial ml,h_2;h_1)\act{\{\partial lh_2,h_1\}}(\act{\act{\partial h_2}h_1}m) = mL(l,h_2;h_1)
\]
Starting from the left-hand side, we obtain
\begin{align*}
  &(\text{L.H.S}) \\
  &= 
  \{\partial ml,\{h_2,h_1\}\}\act{\act{\partial l}\{h_2,h_1\}}\left<\partial ml,\act{\partial h_2}h_1\right>\{\partial l,h_2|h_1\}^{-1}
  \act{\{\partial lh_2,h_1\}}(\act{\act{\partial h_2}h_1}m) \\
  &= \act{\partial m}\{l,\{h_2,h_1\}\}\{\partial m,\act{\partial l}\{h_2,h_1\}\} & &\text{\tiny \eqref{eq:HH_H-2CM}} \\
  &\qquad 
  \act{\act{\partial l}\{h_2,h_1\}}(\act{\partial m}\left<l,\act{\partial h_2}h_1\right>\{\partial m,\{\partial l,\act{\partial h_2}h_1\}\}\act{\{\partial l,\act{\partial h_2}h_1\}}\left<\partial m,\act{\partial h_2}h_1\right>\{e_H,\partial l|\act{\partial h_2}h_1\}^{-1}) & &\text{\tiny \eqref{eq:3CM-aux-Pasting1}} \\
  &\qquad 
  \{\partial l,h_2|h_1\}^{-1}\act{\{\partial lh_2,h_1\}}(\act{\act{\partial h_2}h_1}m) \\
  &= m\{l,\{h_2,h_1\}\}m^{-1}\{\partial m,\act{\partial l}\{h_2,h_1\}\}\act{\act{\partial l}\{h_2,h_1\}}m & &\text{\tiny \eqref{eq:2CM-Peifferid}} \\
  &\qquad 
  \act{\act{\partial l}\{h_2,h_1\}}\left<l,\act{\partial h_2}h_1\right>\act{\act{\partial l}\{h_2,h_1\}}m^{-1}\act{\act{\partial l}\{h_2,h_1\}}\{\partial m,\{\partial l,\act{\partial h_2}h_1\}\}\{\partial l,h_2|h_1\}^{-1} & &\tikz[transform shape, scale=.75,baseline=-2pt,inner sep=0pt]{\node {$\begin{array}{ll}\text{lemma \ref{thm:unit}}\\ \text{\eqref{eq:2CM-Peifferid}}\end{array}$};} \\
  &\qquad 
  \act{\{\partial lh_2,h_1\}}(\left<\partial m,\act{\partial h_2}h_1\right>\act{\act{\partial h_2}h_1}m) \\
  &= m\{l,\{h_2,h_1\}\}\act{\act{\partial l}\{h_2,h_1\}}\left<l,\act{\partial h_2}h_1\right>\act{\act{\partial l}\{h_2,h_1\}\{\partial l,\act{\partial h_2}h_1\}}m^{-1}\{\partial l,h_2|h_1\}^{-1} & &\text{\tiny \eqref{eq:LH-2CM}} \\
  &\qquad 
  \act{\{\partial lh_2,h_1\}}(\left<\partial m,\act{\partial h_2}h_1\right>\act{\act{\partial h_2}h_1}m) \\
  &= m\{l,\{h_2,h_1\}\}\act{\act{\partial l}\{h_2,h_1\}}\left<l,\act{\partial h_2}h_1\right>\{\partial l,h_2|h_1\}^{-1}\act{\{\partial lh_2,h_1\}}m^{-1} & &\text{\tiny \eqref{eq:2CM-Peifferid}} \\
  &\qquad 
  \act{\{\partial lh_2,h_1\}}(\left<\partial m,\act{\partial h_2}h_1\right>\act{\act{\partial h_2}h_1}m) \\
  &= mL(l,h_2;h_1) & &\text{\tiny \eqref{eq:defL},\eqref{eq:3CM-aux-HonM}} \\
  &= (\text{R.H.S}).
\end{align*}

\subsection{Tube 2}
\[
  R(h_2;\partial ml,h_1)\act{\{h_2,\partial lh_1\}}(\act{\partial h_2}m) = \act{h_2}mR(h_2;l,h_1)
\]
Starting from the left-hand side, we have
\begin{align*}
  &(\text{L.H.S}) \\
  &= 
  \left<h_2,\partial ml\right>\act{\{h_2,\partial l\}}\{\act{\partial h_2}(\partial ml),\{h_2,h_1\}\}\{h_2|\partial l,h_1\}^{-1}
  \act{\{h_2,\partial lh_1\}}(\act{\partial h_2}m) & &\text{\tiny \eqref{eq:defR}} \\
  &= 
  \act{\act{h_2}\partial m}\left<h_2,l\right>\left<h_2,\partial m\right>\act{\{h_2,\partial \partial m\}}\{\act{\partial h_2}\partial m,\{h_2,\partial l\}\}\{h_2|\partial \partial m,\partial l\}^{-1} & &\text{\tiny \eqref{eq:3CM-aux-Pasting2}} \\
  &\qquad 
  \act{\{h_2,\partial l\}}\act{\partial \act{\partial h_2}m}\{\act{\partial h_2}l,\{h_2,h_1\}\}\act{\{h_2,\partial l\}}\{\partial \act{\partial h_2}m,\act{\partial \act{\partial h_2}l}\{h_2,h_1\}\} \\
  &\qquad 
  \{h_2|\partial l,h_1\}^{-1}
  \act{\{h_2,\partial lh_1\}}(\act{\partial h_2}m) \\
  &= 
  \act{h_2}m\left<h_2,l\right>\act{h_2}m^{-1}\left<h_2,\partial m\right>\{\act{\partial h_2}\partial m,\{h_2,\partial l\}\} & &\text{\tiny \eqref{eq:2CM-Peifferid}} \\
  &\qquad 
  \act{\{h_2,\partial l\}}\act{\partial \act{\partial h_2}m}\{\act{\partial h_2}l,\{h_2,h_1\}\}\act{\{h_2,\partial l\}}\{\partial \act{\partial h_2}m,\act{\partial \act{\partial h_2}l}\{h_2,h_1\}\} \\
  &\qquad 
  \{h_2|\partial l,h_1\}^{-1}
  \act{\{h_2,\partial lh_1\}}(\act{\partial h_2}m) \\
  &= 
  \act{h_2}m\left<h_2,l\right>\act{\partial h_2}m^{-1}\{\act{\partial h_2}\partial m,\{h_2,\partial l\}\} & &\text{\tiny \eqref{eq:3CM-aux-pHonM}} \\
  &\qquad 
  \act{\{h_2,\partial l\}}\act{\partial \act{\partial h_2}m}\{\act{\partial h_2}l,\{h_2,h_1\}\}\act{\{h_2,\partial l\}}\{\partial \act{\partial h_2}m,\act{\partial \act{\partial h_2}l}\{h_2,h_1\}\} \\
  &\qquad 
  \{h_2|\partial l,h_1\}^{-1}
  \act{\{h_2,\partial lh_1\}}(\act{\partial h_2}m) \\
  &= 
  \act{h_2}m\left<h_2,l\right>\act{\{h_2,\partial l\}}(\act{\partial h_2}m)^{-1} & &\text{\tiny \eqref{eq:LH-2CM}} \\
  &\qquad 
  \act{\{h_2,\partial l\}}\act{\partial \act{\partial h_2}m}\{\act{\partial h_2}l,\{h_2,h_1\}\}\act{\{h_2,\partial l\}}\{\partial \act{\partial h_2}m,\act{\partial \act{\partial h_2}l}\{h_2,h_1\}\} \\
  &\qquad 
  \{h_2|\partial l,h_1\}^{-1}
  \act{\{h_2,\partial lh_1\}}(\act{\partial h_2}m) \\
  &= 
  \act{h_2}m\left<h_2,l\right>\act{\{h_2,\partial l\}}\{\act{\partial h_2}l,\{h_2,h_1\}\} \\
  &\qquad 
  \act{\{h_2,\partial l\}}(\act{\partial h_2}m)^{-1}\act{\{h_2,\partial l\}}\{\partial \act{\partial h_2}m,\act{\partial \act{\partial h_2}l}\{h_2,h_1\}\} & &\text{\tiny \eqref{eq:2CM-Peifferid}} \\
  &\qquad 
  \{h_2|\partial l,h_1\}^{-1}
  \act{\{h_2,\partial lh_1\}}(\act{\partial h_2}m) \\
  &= 
  \act{h_2}m\left<h_2,l\right>\act{\{h_2,\partial l\}}\{\act{\partial h_2}l,\{h_2,h_1\}\} \\
  &\qquad 
  \act{\{h_2,\partial l\}\act{\partial \act{\partial h_2}l}\{h_2,h_1\}}(\act{\partial h_2}m)^{-1}
  \{h_2|\partial l,h_1\}^{-1}
  \act{\{h_2,\partial lh_1\}}(\act{\partial h_2}m) & &\text{\tiny \eqref{eq:LH-2CM}} \\
  &= 
  \act{h_2}m\left<h_2,l\right>\act{\{h_2,\partial l\}}\{\act{\partial h_2}l,\{h_2,h_1\}\}\{h_2|\partial l,h_1\}^{-1} & &\text{\tiny \eqref{eq:2CM-Peifferid}} \\
  &= \act{h_2}mR(h_2;l,h_1) & &\text{\tiny \eqref{eq:defR}} \\
  &= \text{(R.H.S)}
\end{align*}

\end{document}